\documentclass[12pt]{amsart}
\usepackage{latexsym ,exscale, enumerate,amscd}
\usepackage{
  a4wide, 
  fullpage,
  amsmath,
  amsfonts,
  amssymb,
  times,
  color,
  hyphenat,
  pinlabel,
  mathtools
}

\usepackage{pstricks}
\usepackage[tiling]{pst-fill}

\psset{linewidth=0.3pt,dimen=middle} \psset{xunit=.70cm,yunit=0.70cm}
\psset{arrowsize=1pt 5,arrowlength=0.6,arrowinset=0.7}

\usepackage[all]{xy}
\SelectTips{cm}{}

\usepackage{graphicx}

\newcommand{\sseq}{{\rm SSeq}}

\newcommand{\kk}{ \textbf{\underline{\textit{k}}}}

\newcommand{\onel}{{\mathbf 1}_{\lambda}}
\newcommand{\onelp}{{\mathbf 1}_{\lambda'}}

\newcommand{\refequal}[1]{\xy {\ar@{=}^{#1}
(-1,0)*{};(1,0)*{}};
\endxy}

\newcommand{\U}{\dot{{\bf U}}(\mathfrak{sl}_n)}
\newcommand{\SD}{\dot{{\bf S}}}

\newcommand{\Ucat}{\cal{U}({\mathfrak sl}_n)}
\newcommand{\Scat}{\mathcal{S}}

\newcommand{\UcatD}{\dot{\mathcal U}(\mathfrak{sl}_n)}
\newcommand{\ScatD}{\dot{\Scat}}

\newcommand{\xsum}[2]{
  \vcenter{\xy
  (0,.4)*{\sum};
  (0,3.8)*{\scs #2};
  (0,-3.2)*{\scs #1};
  \endxy}
}

\newcommand{\Uup}{\xy {\ar (0,-3)*{};(0,3)*{} };(0,0)*{\bullet};(2,0)*{};(-2,0)*{};\endxy}
\newcommand{\Udown}{\xy {\ar (0,3)*{};(0,-3)*{} };(0,0)*{\bullet};(2,0)*{};(-2,0)*{};\endxy}

\newcommand{\Ucupri}{\xy (0,-1)*{\dblue\xybox{(-2,1)*{}; (2,1)*{} **\crv{(-2,-3) & (2,-3)} ?(1)*\dir{>};}} \endxy}

\newcommand{\Ucupli}{\xy (0,-1)*{\dblue\xybox{(2,1)*{}; (-2,1)*{} **\crv{(2,-3) & (-2,-3)}?(1)*\dir{>};}}\endxy}

\newcommand{\Ucapri}{\xy (0,1)*{\dblue\xybox{(-2,-1)*{}; (2,-1)*{} **\crv{(-2,3) & (2,3)}?(1)*\dir{>};}}\endxy\;\;}

\newcommand{\Ucapli}{\xy (0,1)*{\dblue\xybox{(2,-1)*{}; (-2,-1)*{} **\crv{(2,3) &(-2,3) }?(1)*\dir{>};}}\endxy\;}

\newcommand{\Ucrossij}{\xy (0,0)*{\dgreen\xybox{\ar (2.5,-2.5)*{};(-2.5,2.5)*{}}}; (0,0)*{\dblue\xybox{\ar (-2.5,-2.5)*{};(2.5,2.5)*{} }};(4,0)*{};(-4,0)*{};\endxy}

\newcommand{\Ucrossdij}{\xy (0,0)*{\dblue\xybox{\ar (2.5,2.5)*{};(-2.5,-2.5)*{}}}; (0,0)*{\dgreen\xybox{\ar (-2.5,2.5)*{};(2.5,-2.5)*{} }};
(4,0)*{};(-4,0)*{};\endxy}

\newcommand{\BOX}{\hbox {$\sqcap$ \kern -1em $\sqcup$}}

\newcommand{\To}{\Rightarrow}

\newcommand{\Hom}{{\rm Hom}}
\newcommand{\HOM}{{\rm HOM}}
\newcommand{\HOMU}{{\rm HOM_{\Ucat}}}
\newcommand{\HomU}{{\rm Hom_{\Ucat}}}
\newcommand{\HOMUD}{{\rm HOM_{\UcatD}}}

\newcommand{\HOMS}{{\rm HOM_{\Scat(n,d)}}}
\newcommand{\HomS}{{\rm Hom_{\Scat(n,d)}}}
\newcommand{\HOMSD}{{\rm HOM_{\ScatD(n,d)}}}

\newcommand{\End}{{\rm End}}
\newcommand{\END}{{\rm END}}
\newcommand{\ENDU}{{\rm END_{\Ucat}}}

\newcommand{\ENDS}{{\rm END_{\Scat(n,d)}}}
\newcommand{\EndS}{{\rm End_{\Scat(n,d)}}}

\renewcommand{\to}{\rightarrow}
\newcommand{\maps}{\colon}

\newcommand{\id}{{\rm id}}

\newcommand{\scs}{\scriptstyle}

\theoremstyle{definition}
\newtheorem{thm}{Theorem}[section]
\newtheorem{cor}[thm]{Corollary}
\newtheorem{conj}[thm]{Conjecture}
\newtheorem{lem}[thm]{Lemma}
\newtheorem{rem}[thm]{Remark}
\newtheorem{prop}[thm]{Proposition}
\newtheorem{defn}[thm]{Definition}

        \newcommand{\be}{\begin{equation}}
        \newcommand{\ee}{\end{equation}}
        \newcommand{\ba}{\begin{eqnarray}}
        \newcommand{\ea}{\end{eqnarray}}
        \newcommand{\ban}{\begin{eqnarray*}}
        \newcommand{\ean}{\end{eqnarray*}}
        \newcommand{\barr}{\begin{array}}
        \newcommand{\earr}{\end{array}}

\numberwithin{equation}{section}

\def\emph#1{{\sl #1\/}}

\let\tilde=\widetilde

\let\phi=\varphi
\let\epsilon=\varepsilon

\usepackage{bbm}

\def\N{{\mathbbm N}}

\def\Z{{\mathbbm Z}}
\def\Q{{\mathbbm Q}}

\def\cal#1{\mathcal{#1}}%
\def\1{\mathbbm{1}}%
\def\nn{\notag}

\def\mf{\mathfrak}
\def\shuffle{\,\raise 1pt\hbox{$\scriptscriptstyle\cup{\mskip
               -4mu}\cup$}\,}

\newcommand{\lowrru}[1]{\xybox{%
  (-8,0)*{};
  (8,0)*{};
  (-6,-18)*{};(6,-9)*{} **\crv{(-6,-13) & (6,-15)} ?(1)*\dir{>};
  (6,-9)*{};(6,0)*{}  **\dir{-} ?(.3)*\dir{ }+(2,0)*{\scs {\bf j}};
}}

\newcommand{\lowllu}[1]{\xybox{%
  (-8,0)*{};
  (8,0)*{};
  (6,-18)*{};(-6,-9)*{} **\crv{(6,-13) & (-6,-15)} ?(1)*\dir{>};
  (-6,-9)*{};(-6,0)*{}  **\dir{-} ?(.3)*\dir{ }+(-2,0)*{\scs {\bf j}};
}}

\newcommand{\bbe}[1]{\xybox{%
  (-2,0)*{};
  (2,0)*{};
  (0,0);(0,-18) **\dir{-}; ?(.5)*\dir{<}+(2.3,0)*{\scriptstyle{#1}};
}}

\newcommand{\bbsid}{\xybox{%
  (-2,0)*{};
  (2,0)*{};
  (0,10);(0,4) **\dir{-};
}}
\newcommand{\bbpef}[1]{\xybox{%
  (-6,0)*{};
  (6,0)*{};
  (-4,0)*{}="t1";
  (4,0)*{}="t2";
  "t1";"t2" **\crv{(-4,-6) & (4,-6)}; ?(.15)*\dir{>} ?(.9)*\dir{>}
   ?(.5)*\dir{}+(0,-2)*{\scriptstyle{#1}};
}}
\newcommand{\bbpfe}[1]{\xybox{%
  (-6,0)*{};
  (6,0)*{};
  (-4,0)*{}="t1";
  (4,0)*{}="t2";
  "t2";"t1" **\crv{(4,-6) & (-4,-6)}; ?(.15)*\dir{>} ?(.9)*\dir{>}
  ?(.5)*\dir{}+(0,-2)*{\scriptstyle{#1}};
}}

\newcommand{\bbcfe}[1]{\xybox{%
  (-6,0)*{};
  (6,0)*{};
  (-4,0)*{}="t1";
  (4,0)*{}="t2";
  "t1";"t2" **\crv{(-4,6) & (4,6)}; ?(.15)*\dir{>} ?(.9)*\dir{>}
  ?(.5)*\dir{}+(0,2)*{\scriptstyle{#1}};
}}
\newcommand{\bbcef}[1]{\xybox{%
  (-6,0)*{};
  (6,0)*{};
  (-4,0)*{}="t1";
  (4,0)*{}="t2";
  "t2";"t1" **\crv{(4,6) & (-4,6)}; ?(.15)*\dir{>}
  ?(.9)*\dir{>} ?(.5)*\dir{}+(0,2)*{\scriptstyle{#1}};
}}

\newcommand{\ccbub}[2]{
\xybox{%
 (-6,0)*{};
  (6,0)*{};
  (-4,0)*{}="t1";
  (4,0)*{}="t2";
  "t2";"t1" **\crv{(4,6) & (-4,6)}; ?(.7)*\dir{}+(-2,0)*{\scs #2}
  ?(.05)*\dir{>} ?(1)*\dir{>};
  "t2";"t1" **\crv{(4,-6) & (-4,-6)};
   ?(.3)*\dir{}+(0,0)*{\bullet}+(0,-3)*{\scs {#1}};
}}
\newcommand{\cbub}[2]{
\xybox{%
 (-6,0)*{};
  (6,0)*{};
  (-4,0)*{}="t1";
  (4,0)*{}="t2";
  "t2";"t1" **\crv{(4,6) & (-4,6)};?(.7)*\dir{}+(-2,0)*{\scs #2};
   ?(0)*\dir{<} ?(.95)*\dir{<};
  "t2";"t1" **\crv{(4,-6) & (-4,-6)};
   ?(.3)*\dir{}+(0,0)*{\bullet}+(0,-3)*{\scs {#1}};
}}

\newcommand{\bbdl}[1]{\xybox{%
  (2,0);(0,-8) **\crv{(2,-2)&(0,-6)}; ?(.5)*\dir{>}
}}
\newcommand{\bbdlu}[1]{\xybox{%
  (2,0);(0,-8) **\crv{(2,-2)&(0,-6)}; ?(.5)*\dir{<}
}}
\newcommand{\bbdr}[1]{\xybox{%
  (-2,0);(0,-8) **\crv{(-2,-2)&(0,-6)}; ?(.5)*\dir{>}
}}
\newcommand{\bbdru}[1]{\xybox{%
  (-2,0);(0,-8) **\crv{(-2,-2)&(0,-6)}; ?(.5)*\dir{<}
}}

\newcommand\figleft[2]{
\mspace{-16mu}
\;\;\vcenter{\xy 0;/r.16pc/:
(-6,0)*{\dblue\xybox{
 (-4,-15)*{}; (-20,25) **\crv{(-3,-6) & (-20,4)}?(0)*\dir{<}?(.6)*\dir{}+(0,0)*{\bullet};
}};
(-2,0)*{\dred\xybox{
 (-12,-15)*{}; (-4,25) **\crv{(-12,-6) & (-4,0)}?(0)*\dir{<}?(.6)*\dir{}+(.2,0)*{\bullet};
 ?(0)*\dir{<}?(.75)*\dir{}+(.2,0)*{\bullet};?(0)*\dir{<}?(.9)*\dir{}+(0,0)*{\bullet};
}};
(-14,6.8)*{\dgreen\xybox{
 (-28,25)*{}; 
 (-12,25) **\crv{(-28,10) & (-12,10)}?(0)*\dir{<};
  ?(.2)*\dir{}+(0,0)*{\bullet}?(.35)*\dir{}+(0,0)*{\bullet};
}};
(-28,0)*{\dred\xybox{
 (-36,-15)*{}; (-36,25) **\crv{(-34,-6) & (-35,4)}?(1)*\dir{>};
 (-28,-15)*{}; (-42,25) **\crv{(-28,-6) & (-42,4)}?(1)*\dir{>};
}};
(-26,-13)*{\dgreen\xybox{
 (-42,-15)*{}; (-20,-15) **\crv{(-42,-5) & (-20,-5)}?(1)*\dir{>};
}};
 (10,4)*{\dblue \cbub{}{}};
 (-17,-6)*{\dgreen\cbub{}{}};
 (8,-8)*{#1};(-38,-8)*{#2};
 \endxy}\;\;
}
\newcommand\figright[2]{
\mspace{-16mu}
\;\;\vcenter{ \xy 0;/r.16pc/: 
   (-14,8)*{\dred\xybox{
   (-16,-10)*{}; (-8,10)*{} **\crv{(-16,-6) & (-8,6)}?(1)*\dir{}+(.1,0)*{\bullet};
   (-8,-10)*{}; (0,10)*{} **\crv{(-8,-6) & (-0,6)}?(.6)*\dir{}+(.2,0)*{\bullet}?
   (1)*\dir{}+(.1,0)*{\bullet};
   (0,10)*{}; (-16,30)*{} **\crv{(0,14) & (-16,26)}?(1)*\dir{>};
  (-8,10)*{}; (0,30)*{} **\crv{(-8,14) & (-0,26)}?(1)*\dir{>}?(.6)*\dir{}+(.25,0)*{\bullet};
}};
(-14,8)*{\dgreen\xybox{
   (-16,10)*{}; (-8,30)*{} **\crv{(-16,14) & (-8,26)}?(1)*\dir{>};
   (0,-10)*{}; (-16,10)*{} **\crv{(0,-6) & (-16,6)}?(.5)*\dir{};
  }};
 (-26,2)*{#1}; (-2,2)*{#2};
 \endxy} \;\;
}

\newrgbcolor{dblue}{0 0 0.5}
\newrgbcolor{dred}{0.9 0 0}
\newrgbcolor{dgreen}{0 0.6 0}

\newcommand{\lai}{\lambda_{i}}
\newcommand{\laii}{\lambda_{i+1}}
\newcommand{\laiii}{\lambda_{i+2}}
\newcommand{\laj}{\lambda_{j}}

\newcommand{\qra}{\quad\ra\quad}
\newcommand{\bbox}[1]{\framebox{$\scs #1$}}
\newcommand{\und}[1]{\underline{#1}}

\newcommand{\bscs}{\black\scs}
\newcommand{\llambda}{\overline{\lambda}}

\newcommand{\ii}{\underline{\textbf{\textit{i}}}}
\newcommand{\jj}{\underline{\textbf{\textit{j}}}}

\newcommand{\n}{\noindent}

\newcommand{\fbim}{\mathcal{F}_{Bim}}
\newcommand{\fek}{\mathcal{F}_{EK}}
\newcommand{\bim}{\mathbf{Bim}}

\newcommand{\glcat}{\mathcal{U}(\mathfrak{gl}_n)}

\newcommand{\Ugl}{\dot{\bf U}(\mathfrak{gl}_n)}

\DeclareMathOperator{\HomGL}{Hom_{\glcat}}
\DeclareMathOperator{\HOMGL}{HOM_{\glcat}}

\newcommand{\figins}[3] 
{\raisebox{#1pt}{\includegraphics[height=#2 in]{figs/#3}}}

\newcommand{\bN}{\mathbb{N}}
\newcommand{\bZ}{\mathbb{Z}}
\newcommand{\bQ}{\mathbb{Q}}

\newcommand{\cF}{\mathcal{F}}

\newcommand{\ra}{\rightarrow}

\newcommand{\bigb}[1]{
\begin{pspicture}(0,0)
 \rput(0,0){\psframebox[framearc=.5,fillstyle=solid]{\small $#1$}}
\end{pspicture}}

\newcommand{\stccbub}[2]{\xybox{
 (0,0)*{\includegraphics[scale=0.5]{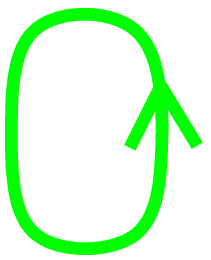}};
 (-5,0)*{\bigb{\pi_{#2}^{\spadesuit}}};(3,-7)*{#1};
}}

\newcommand{\stcbub}[2]{\xybox{
 (0,0)*{\reflectbox{\includegraphics[scale=0.5]{figs/ccbub}}};
 (5,0)*{\bigb{\pi_{#2}^{\spadesuit}}};(-3,-7)*{#1};
}}

\newcommand{\laudaccbub}[2]{
\xybox{%
 (-6,0)*{};
  (6,0)*{};
  (-4,0)*{}="t1";
  (4,0)*{}="t2";
  "t2";"t1" **\crv{(4,6) & (-4,6)}; ?(.7)*\dir{}+(-2,0)*{\scs #2}
  ?(.05)*\dir{>} ?(1)*\dir{>};
  "t2";"t1" **\crv{(4,-6) & (-4,-6)};
   ?(.3)*\dir{}+(0,0)*{\bullet}+(0,-3)*{\scs {#1}};
}}
\newcommand{\laudacbub}[2]{
\xybox{%
 (-6,0)*{};
  (6,0)*{};
  (-4,0)*{}="t1";
  (4,0)*{}="t2";
  "t2";"t1" **\crv{(4,6) & (-4,6)};?(.7)*\dir{}+(-2,0)*{\scs #2};
   ?(0)*\dir{<} ?(.95)*\dir{<};
  "t2";"t1" **\crv{(4,-6) & (-4,-6)};
   ?(.3)*\dir{}+(0,0)*{\bullet}+(0,-3)*{\scs {#1}};
}}

\title{A diagrammatic categorification of the q-Schur algebra}

\author{Marco Mackaay}
\address{Departamento de Matem\'{a}tica\\ Universidade do Algarve\\ 
Campus de Gambelas\\ 8005-139 Faro\\ Portugal and CAMGSD\\Instituto Superior T\'{e}cnico
\\ Avenida Rovisco Pais\\ 
1049-001 Lisboa\\ Portugal}
\email{mmackaay@ualg.pt}
\author{Marko Sto\v si\'c }
\address{Instituto de Sistemas e Rob\'{o}tica and CAMGSD\\Instituto Superior T\'{e}cnico
\\ Avenida Rovisco Pais\\ 
1049-001 Lisboa\\ Portugal and\\
Mathematical Institute SANU\\
Knez Mihailova 36\\
11000 Beograd\\
Serbia}
\email{mstosic@math.ist.utl.pt}
\author{Pedro Vaz}
\address{CAMGSD\\Instituto Superior T\'{e}cnico
\\ Avenida Rovisco Pais\\ 1049-001 Lisboa\\ Portugal}
\email{pedro.forte.vaz@ist.utl.pt}
\begin{document}
%
%
\begin{abstract} 
In this paper we categorify the $q$-Schur algebra $\mathbf{S}_q(n,d)$ 
as a quotient of Khovanov and Lauda's diagrammatic 
$2$-category $\Ucat$~\cite{K-L3}. We also show that our $2$-category contains 
Soergel's~\cite{S} monoidal category of bimodules of type $A$, 
which categorifies the Hecke algebra $H_q(d)$, as a full sub-$2$-category 
if $d\leq n$. For the latter result we use Elias 
and Khovanov's diagrammatic presentation of Soergel's 
monoidal category of type $A$~\cite{E-Kh}.
\end{abstract}
\maketitle
\setcounter{tocdepth}{1}
\tableofcontents
%
%
\section{Introduction}                
\label{sec:intro}                     

There is a well-known relation, called {\em Schur-Weyl duality} or {\em 
reciprocity}, between the polynomial representations of homogeneous degree 
$d$ of the general linear group $\mbox{GL}(n,\mathbb{Q})$ and the finite-dimensional representations of the symmetric group on 
$d$ letters $S_d$. Recall that all irreducible polynomial 
representations of $\mbox{GL}(n,\mathbb{Q})$ of homogeneous degree 
$d$ occur in the decomposition of $V^{\otimes d}$, where 
$V=\mathbb{Q}^n$ is the natural representation of 
$\mbox{GL}(n,\mathbb{Q})$. Instead of the 
$\mbox{GL}(n,\mathbb{Q})$-action, we can consider the 
${\mathbf U}(\mathfrak{gl}_n)$-action, without loss of generality. A key observation 
for Schur-Weyl duality is that the permutation action of $S_d$ on 
$V^{\otimes d}$ commutes with the action of ${\mathbf U}(\mathfrak{gl}_n)$. Furthermore, 
we have 
$$\mathbb{Q}[S_d]\cong \mbox{End}_{{\mathbf U}(\mathfrak{gl}_n)}(V^{\otimes d})$$ 
if $n\geq d$. 

By definition, the {\em Schur algebra} is the other 
centralizer algebra 
$$S(n,d):=\mbox{End}_{S_d}(V^{\otimes d}).$$ It is well known that both 
${\mathbf U}(\mathfrak{sl}_n)$ and ${\mathbf U}(\mathfrak{gl}_n)$ map surjectively onto 
$S(n,d)$, for any $d>0$. Therefore we can also define $S(n,d)$ as the image 
of the map 
$${\mathbf U}(\mathfrak{gl}_n)\to \End_{\mathbb{Q}}(V^{\otimes d}),$$
which is the definition used in this paper. Both $S(n,d)$ and $\mathbb{Q}[S_d]$ are split 
semi-simple finite-dimensional algebras, and the 
{\em double centralizer property} above implies that the categories of 
finite-dimensional modules $S(n,d)-\mbox{mod}$ and $S_d-\mbox{mod}$ are 
equivalent, for $n\geq d$.

There are two more facts of interest to us. The 
first is that there actually exists a concrete functor which 
gives rise to the above mentioned equivalence. For $n\geq d$, there exists 
an embedding of $\mathbb{Q}[S_d]$ in $S(n,d)$, which induces the so called 
{\em Schur functor} 
$$S(n,d)-\mbox{mod}\longrightarrow S_d-\mbox{mod}.$$ 
As it turns out, this functor is an equivalence. 

The second fact of interest 
to us is that the Schur 
algebras $S(n,d)$ for various values of $n$ and $d$ are related. 
If $n\leq m$, then $S(n,d)$ can be embedded into $S(m,d)$. 
A more complicated relation is the following: for any 
$k\in\mathbb{N}$, there is a surjection 
$$S(n,d+nk)\to S(n,d).$$ 
This surjection is compatible with the projections of 
${\mathbf U}(\mathfrak{gl}_n)$ and ${\mathbf U}(\mathfrak{sl}_n)$ onto the Schur algebras. 
With these surjections, the 
Schur algebras form an inverse system. As it turns out, the projections 
of ${\mathbf U}(\mathfrak{sl}_n)$ onto the Schur algebras give rise to an embedding 
$${\mathbf U}(\mathfrak{sl}_n)\subset \oplus_{d=0}^{n-1}\lim_{\leftarrow k}S(n,d+nk).$$ 
To get a similar embedding for ${\mathbf U}(\mathfrak{gl}_n)$, one needs to consider 
generalized Schur algebras. We do not give the details of this 
generalization, because we will not need it. We refer the interested reader 
to~\cite{D-G}.
\vskip0.5cm

All the facts recollected above have $q$-analogues, which involve 
the quantum groups ${\mathbf U}_q(\mathfrak{gl}_n)$ and 
${\mathbf U}_q(\mathfrak{sl}_n)$, the Hecke algebra $H_q(d)$, the 
$q$-Schur algebra $S_q(n,d)$, and their respective 
finite-dimensional representations over $\mathbb{Q}(q)$. 

If one is only interested in the finite-dimensional 
representations of ${\mathbf U}_q(\mathfrak{gl}_n)$ and 
${\mathbf U}_q(\mathfrak{sl}_n)$, which can all be decomposed into 
weight spaces, it is easier to work with Lusztig's 
idempotented version of these quantum groups, denoted 
$\Ugl$ and 
$\U$. In these idempotented versions, 
the Cartan subalgebras are 
``replaced'' by algebras generated by 
orthogonal idempotents corresponding to the weights. 
The kernel of the surjection 
$\Ugl\to S_q(n,d)$ is simply the ideal 
generated by all idempotents corresponding to the 
$\mathfrak{gl}_n$-weights which do not appear in the 
decomposition of $V^{\otimes d}$. The same is true for 
the kernel of $\U\to S_q(n,d)$, using 
$\mathfrak{sl}_n$-weights. We will say more about $\Ugl$ 
and $\U$ in the next section.  
\vskip0.5cm

We are interested in the {\em categorification} of the $q$-algebras above, 
the relations between them and the applications to low-dimensional topology. 
By a categorification of a $q$-algebra 
we mean a monoidal category or a 2-category whose Grothendieck group, 
tensored by $\mathbb{Q}(q)$, is 
isomorphic to that $q$-algebra. 

As a matter of fact, all of 
them have been categorified already, and some of them in more than one way.  
Soergel defined a category of bimodules over polynomial rings in 
$d$ variables, which he proved to categorify $H_q(d)$. Elias and Khovanov gave 
a diagrammatic version of the Soergel category. 
Grojnowski and Lusztig~\cite{G-L} were the first to 
categorify $S_q(n,d)$, using categories of perverse sheaves on products of 
partial flag 
varieties. Subsequently Mazorchuk and Stroppel constructed a categorification 
using representation theoretic techniques~\cite{M-S1} and so did 
Williamson~\cite{Will} for $n=d$ using singular Soergel bimodules. 
Khovanov and Lauda have provided a diagrammatic 2-category $\Ucat$ 
which categorifies $\U$. Rouquier~\cite{R2} 
followed a more representation theoretic approach to the 
categorification of the quantum groups. The precise 
relation of his work with Khovanov and Lauda's remains unclear.
We note that the categorifications mentioned above have 
been obtained for arbitrary root data. However, this paper is only about 
type $A$ and we will not consider other types. 

Our interest is in the diagrammatic approach, by which $H_q(d)$ and 
${\mathbf U}_q(\mathfrak{sl}_n)$ have already been categorified. 
The goal of this paper is to define a diagrammatic categorification of 
$S_q(n,d)$. Recall that the objects of 
$\Ucat$ are the weights of $\mathfrak{sl}_n$, which label 
the regions in the diagrams which constitute the 2-morphisms. 
Our idea is quite simple: define a new 2-category $\glcat$ just 
as $\Ucat$ but switch to $\mathfrak{gl}_n$-weights, which 
we conjecture to give a categorification of $\Ugl$. 
Next we mod out $\glcat$ by all diagrams which have regions labeled by 
weights not appearing in the decomposition of $V^{\otimes d}$. This way 
we obtain a 
2-category $\Scat(n,d)$ and the main result of this paper is the proof 
that it indeed categorifies $S_q(n,d)$. 

There are two good reasons 
for switching to $\mathfrak{gl}_n$-weights, besides giving a conjectural 
categorification of $\Ugl$. It is easier to say 
explicitly which $\mathfrak{gl}_n$-weights 
do not appear in $V^{\otimes d}$, as we will show in the next section. Also, 
while working on our paper we found a sign mistake in 
what Khovanov and Lauda call their signed categorification of 
$\U$~\cite{K-L:err}. 
Fortunately it does not affect their unsigned version, but the corrected 
signed version loses a nice property, the cyclicity. We discovered that with 
$\mathfrak{gl}_n$-weights there is a different sign 
convention which solves the problem, at least for $\Scat(n,d)$.        

On our way of proving the main result of this paper we obtain some other 
interesting results:
\begin{itemize}
\item For $n\geq d$, we define a fully faithful 2-functor 
from Soergel's 
category of bimodules to $\Scat(n,d)$, which categorifies the well-known 
inclusion $H_q(d)\subset S_q(n,d)$ explained in Section~\ref{sec:hecke-schur}. 
\item We define functors $\Scat(n,d)
\to \Scat(m,d)$ when $n\leq m$. We are not (yet) able to prove 
that these are faithful, although we strongly suspect that they are. 
We know that they are not full, but suspect that they 
are ``almost full'' in a sense that we will explain in 
Section~\ref{sec:grothendieck}. 
\item We define essentially surjective full 2-functors 
$$\Scat(n,d+kn)\to \Scat(n,d)$$ 
which categorify the surjections above. 
\item We show that Khovanov and Lauda's 2-representation of $\Ucat$ on 
the equivariant cohomology of flag varieties descends to $\Scat(n,d)$. 
\item We conjecture how to 
categorify the irreducible representations of $S_q(n,d)$ using $\Scat(n,d)$. 
Khovanov and Lauda's categorification of these representations, using the 
so-called cyclotomic quotients, should be equivalent to a quotient of ours.     
\end{itemize}

Understanding the precise relation with the other categorifications of 
$S_q(n,d)$ would be very important, but is left for the future. 
As a matter of fact, Brundan and Stroppel have already established a link 
between the category $\mathcal{O}$ approach to categorification and 
Khovanov and Lauda's (see for example~\cite{B-S}), which perhaps can be used 
to obtain an equivalence between Mazorchuk and Stroppel's categorification 
of the $q$-Schur algebra and ours. For $n=d$, Williamson's 2-category of 
Soergel's singular bimodules is equivalent to Khovanov and Lauda's 
2-category build out of the equivariant cohomology of partial flag 
varieties (of flags in $\mathbb{Q}^d$) and we expect both to be equivalent to 
$\Scat(d,d)$.  
\vskip0.5cm

Besides the intrinsic interest of $\Scat(n,d)$, 
with its combinatorics and its link to representation theory, there is also 
a potential application to knot theory. First recall that 
there is a natural surjection of the braid group onto $H_q(d)$. 
The Jones-Ocneanu trace of the image of a braid in $H_q(d)$ is equal to the so 
called HOMFLYPT knot polynomial of the braid closure. This construction has 
been categorified: Rouquier defined a complex of Soergel bimodules for 
each braid and Khovanov discovered that its Hochschild homology categorifies 
the Jones-Ocneanu trace, showing that in this way one obtains a homology which 
is isomorphic to the Khovanov-Rozansky HOMFLYPT-homology. 
Using Elias and Khovanov's work, Elias and Krasner~\cite{E-Kr} worked out the 
diagrammatic version of Rouquier's complex. Their work still remains to be 
extended to include the Hochschild homology. Besides this approach, 
which is the one most directly related to the 
results in this paper, we should also mention a geometric approach due to 
Webster and Williamson in~\cite{W-W1} and a representation theoretic 
approach due to Mazorchuk and Stroppel~\cite{M-S2}. 

More generally, there is a natural homomorphism from the colored braid group, 
with $n$ strands colored by natural numbers whose sum is equal to $d$, 
to $S_q(n,d)$. It is not as widely advertised as the 
non-colored version, but one can easily obtain it from Lusztig's formulas 
in Section 5.2.1 in~\cite{Lu} or from the second part of the paper by 
Murakami-Ohtsuki-Yamada~\cite{M-O-Y}. 
One can also define a colored version of the 
Jones-Ocneanu trace on $S_q(n,d)$ to obtain the colored HOMFLYPT knot 
invariant. Naturally the question arises how to categorify the 
colored HOMFLYPT knot polynomial. In~\cite{C-R} Chuang and Rouquier 
defined a colored version of Rouquier's complex for a braid, using 
a representation theoretic approach. They proved invariance under the 
second braid-like Reidemeister move and 
conjectured invariance under the third move. In~\cite{M-S-V} we defined a 
complex of singular Soergel bimodules, which is equivalent to the 
Chuang-Rouquier complex. We conjectured that the Hochschild homology of such a 
complex categorifies the colored HOMFLYPT knot polynomial of the braid closure. 
We were only able to prove our conjecture for the colors 1 and 2, due to the 
complexity of the calculations for general colors.
Webster and Williamson subsequently showed our 
conjecture to be true, using a generalization of their geometric 
approach~\cite{W-W2}. Cautis, Kamnitzer and 
Licata~\cite{C-K-L} also studied the Chuang-Rouquier complex from a 
geometric point of view. By the above mentioned 2-representation of 
$\Scat(n,d)$ into singular Soergel bimodules, it is natural to expect that 
one should be able to define the Chuang-Rouquier complex in $\Scat(n,d)$, 
such that its 2-representation gives exactly the complex of singular Soergel 
bimodules which we conjectured. In a forthcoming paper we will come back 
to this. In the meanwhile, papers have appeared in which the colored HOMFLYPT 
homology has been constructed using matrix factorizations 
(see \cite{Wu1,Wu2,Wu3,Yo1,Yo2}). 
\vskip0.5cm
The outline of this paper is as follows: 
\begin{itemize}
\item In Section~\ref{sec:hecke-schur} 
we recall some results on the above mentioned 
$q$-algebras. Our choice has been highly selective in an attempt to 
prevent this paper from becoming too long. We have only included those 
results which we categorify or which we need in order to categorify. We hope 
that this introduction makes up for what we left out. 

\item In Section~\ref{sec:scat} we define the 2-categories  
$\mathcal{U}(\mathfrak{gl}_n)$ and $\Scat(n,d)$. As said 
before, the first one is just a copy of Khovanov and Lauda's definition of 
$\Ucat$, but with a different set of weights and a different sign convention. 
The second one is a quotient of the first one.

\item To understand some of the properties of $\Scat(n,d)$, we first 
define its $2$-representation in the $2$-category of bimodules 
over polynomial rings in Section~\ref{sec:2rep}. Except for the 
different sign convention, it is the factorization of the 2-representation 
of \cite{K-L3} through 
$\Scat(n,d)$. The only new feature is our interpretation of this 
2-representation in terms of the categorified MOY-calculus, which we developed 
in~\cite{M-S-V}.  
     
\item Section~\ref{sec:struct} is devoted 
to comparing the structure of the 2-HOM spaces of $\Ucat$ to those of 
$\Scat(n,d)$. The latter ones remain a bit of a mystery to us and we can only 
prove just enough about them for what we need in the rest of this paper. 

\item In Section~\ref{sec:soergel} we define a fully faithful embedding of 
Soergel's categorification of $H_q(d)$ into $\Scat(n,d)$. We 
have not yet attributed any notation to Soergel's category 
in this introduction, because there are actually two slightly different 
versions of it and we will need both, one for $d=n$ and the other for 
$d<n$. 

\item In Section~\ref{sec:grothendieck} we prove that $\Scat(n,d)$ 
indeed categorifies $S_q(n,d)$. We also conjecture how to categorify 
the Weyl modules of $S_q(n,d)$. 
\end{itemize}

\section{Hecke and $q$-Schur algebras}         
\label{sec:hecke-schur}                        

In this section we recollect some facts about the $q$-algebras mentioned 
in the introduction. For details and proofs see \cite{D} and \cite{Mar} 
unless other references are mentioned. We work over the field 
$\bQ(q)$, where $q$ is a formal parameter. 

\subsection{The quantum general and special linear algebras}

Let us first recall the quantum general and special linear algebras. The 
$\mathfrak{gl}_n$-weight lattice is isomorphic to $\bZ^n$. Let 
$\epsilon_i=(0,\ldots,1,\ldots,0)\in \bZ^n$, with $1$ being on the $i$th 
coordinate, and $\alpha_i=\epsilon_i-\epsilon_{i+1}\in\bZ^{n}$, for 
$i=1,\ldots,n-1$. We also define the Euclidean inner product on $\bZ^n$ by  
$(\epsilon_i,\epsilon_j)=\delta_{i,j}$. 
   
\begin{defn} The {\em quantum general linear algebra} 
${\mathbf U}_q(\mathfrak{gl}_n)$ is 
the associative unital $\bQ(q)$-algebra generated by $K_i,K_i^{-1}$, for $1,\ldots, n$, 
and $E_{\pm i}$, for $i=1,\ldots, n-1$, subject to the relations
\begin{gather*}
K_iK_j=K_jK_i\quad K_iK_i^{-1}=K_i^{-1}K_i=1
\\
E_iE_{-j} - E_{-j}E_i = \delta_{i,j}\dfrac{K_iK_{i+1}^{-1}-K_i^{-1}K_{i+1}}{q-q^{-1}}
\\
K_iE_{\pm j}=q^{\pm (\epsilon_i,\alpha_j)}E_{\pm j}K_i
\\
E_{\pm i}^2E_{\pm j}-(q+q^{-1})E_{\pm i}E_{\pm j}E_{\pm i}+E_{\pm j}E_{\pm i}^2=0
\qquad\text{if}\quad |i-j|=1
\\
E_{\pm i}E_{\pm j}-E_{\pm j}E_{\pm i}=0\qquad\text{else}.
\end{gather*} 
\end{defn}

\begin{defn} 
\label{defn:qsln}
The {\em quantum special linear algebra} 
${\mathbf U}_q(\mathfrak{sl}_n)\subseteq {\mathbf U}_q(\mathfrak{gl}_n)$ is 
the unital $\bQ(q)$-subalgebra generated by $K_iK^{-1}_{i+1}$ and 
$E_{\pm i}$, for $i=1,\ldots, n-1$.
\end{defn}

Recall that the ${\mathbf U}_q(\mathfrak{sl}_n)$-weight lattice is 
isomorphic to $\bZ^{n-1}$. Suppose that $V$ is a 
${\mathbf U}_q(\mathfrak{gl}_n)$-weight representation with 
weights $\lambda=(\lambda_1,\ldots,\lambda_n)\in\bZ^n$, i.e. 
$$V\cong \bigoplus_{\lambda}V_{\lambda}$$ 
and $K_i$ acts as multiplication by 
$q^{\lambda_i}$ on $V_{\lambda}$. Then $V$ is also a 
${\mathbf U}_q(\mathfrak{sl}_n)$-weight representation with weights 
$\overline{\lambda}=(\overline{\lambda}_1,\ldots,\overline{\lambda}_{n-1})\in
\bZ^{n-1}$ such that 
$\overline{\lambda}_j=\lambda_j-\lambda_{j+1}$ for $j=1,\ldots,n-1$. 
Conversely, given a ${\mathbf U}_q(\mathfrak{sl}_n)$-weight 
representation with weights $\mu=(\mu_1,\ldots,\mu_{n-1})$, there is not a 
unique choice of ${\mathbf U}_q(\mathfrak{gl}_n)$-action on $V$. We can 
fix this by choosing the action of $K_1\cdots K_n$. In terms of weights, this 
corresponds to the observation that, for any $d\in\bZ$ the equations 
\begin{align}
\label{eq:sl-gl-wts1}
\lambda_i-\lambda_{i+1}&=\mu_i\\
\label{eq:sl-gl-wts2}
\qquad \sum_{i=1}^{n}\lambda_i&=d
\end{align}  
determine $\lambda=(\lambda_1,\ldots,\lambda_n)$ uniquely, 
if there exists a solution to~\eqref{eq:sl-gl-wts1} and~\eqref{eq:sl-gl-wts2} 
at all. To fix notation, we 
define the map $\phi_{n,d}\colon \bZ^{n-1}\to \bZ^{n}\cup \{*\}$ by 
$$
\phi_{n,d}(\mu)=\lambda 
$$
if~\eqref{eq:sl-gl-wts1} and \eqref{eq:sl-gl-wts2} have a solution, and 
put $\phi_{n,d}(\mu)=*$ otherwise.   

Recall that ${\mathbf U}_q(\mathfrak{gl}_n)$ and 
${\mathbf U}_q(\mathfrak{sl}_n)$ are both Hopf algebras, which implies that 
the tensor product of two of their representations is a representation again. 

Both ${\mathbf U}_q(\mathfrak{gl}_n)$ and ${\mathbf U}_q(\mathfrak{sl}_n)$ 
have plenty of non-weight representations, but we are not interested in them. 
Therefore we can restrict our attention to the 
Beilinson-Lusztig-MacPherson~\cite{B-L-M} idempotented version of these 
quantum groups, denoted $\Ugl$ and $\U$ respectively. 
To understand their definition, recall that $K_i$ acts as $q^{\lambda_i}$ on the 
$\lambda$-weight space of any weight representation. 
For each $\lambda\in\bZ^n$ adjoin an idempotent $1_{\lambda}$ to 
${\mathbf U}_q(\mathfrak{gl}_n)$ and add 
the relations
\begin{align*}
1_{\lambda}1_{\mu} &= \delta_{\lambda,\nu}1_{\lambda}   
\\
E_{\pm i}1_{\lambda} &= 1_{\lambda\pm\alpha_i}E_{\pm i}
\\
K_i1_{\lambda} &= q^{\lambda_i}1_{\lambda}.
\end{align*}
\begin{defn} 
\label{defn:Uglndot}
The idempotented quantum general linear algebra is defined by 
$$\Ugl=\bigoplus_{\lambda,\mu\in\bZ^n}1_{\lambda}{\mathbf U}_q(\mathfrak{gl}_n)1_{\mu}.$$
\end{defn}
\noindent For $\ii=(\alpha_1 i_1,\ldots,\alpha_{n-1}i_{n-1})$, with 
$\alpha_j=\pm$, define $$E_{\ii}:=E_{\alpha_1 i_1}\cdots E_{\alpha_{n-1} i_{n-1}}$$ 
and define $\ii_{\Lambda}\in\bZ^n$ to be the $n$-tuple such that 
$$E_{\ii}1_{\mu}=1_{\mu + \ii_{\Lambda}}E_{\ii}.$$

Similarly for ${\mathbf U}_q(\mathfrak{sl}_n)$, adjoin an idempotent $1_{\mu}$ 
for each $\mu\in\bZ^{n-1}$ and add the relations
\begin{align*}
1_{\mu}1_{\nu} &= \delta_{\mu,\nu}1_{\lambda}   
\\
E_{\pm i}1_{\mu} &= 1_{\mu\pm\overline{\alpha}_i}E_{\pm i}
\\
K_iK^{-1}_{i+1}1_{\mu} &= q^{\mu_i}1_{\mu}.
\end{align*}
\begin{defn} The idempotented quantum special linear algebra is defined by 
$$\U=\bigoplus_{\mu,\nu\in\bZ^{n-1}}1_{\mu}{\mathbf U}_q(\mathfrak{sl}_n)1_{\nu}.$$
\end{defn}
\noindent Note that $\Ugl$ and $\U$ are both non-unital algebras, 
because their units 
would have to be equal to the infinite sum of all their idempotents. 
Furthermore, the only ${\mathbf U}_q(\mathfrak{gl}_n)$ 
and ${\mathbf U}_q(\mathfrak{sl}_n)$-representations which factor through 
$\Ugl$ and $\U$, respectively, are the weight representations. 
Finally, note that there is no embedding of $\U$ into $\Ugl$, because 
there is no embedding of the $\mathfrak{sl}_n$-weights into the 
$\mathfrak{gl}_n$-weights.  

\subsection{The $q$-Schur algebra}

Let $d\in\bN$ and let $V$ be the natural $n$-dimensional representation of 
${\mathbf U}_q(\mathfrak{gl}_n)$. Define  
$$\Lambda(n,d)=\{\lambda\in \bN^n\colon\,\, 
\sum_{i=1}^{n}\lambda_i=d\}$$  
$$\Lambda^+(n,d)=\{\lambda\in\Lambda(n,d)\colon d\geq 
\lambda_1\geq\lambda_2\geq\cdots 
\geq\lambda_n\geq 0\}.$$ 
Recall that the weights in $V^{\otimes d}$ are precisely the elements of 
$\Lambda(n,d)$, and that the highest weights are the elements of 
$\Lambda^+(n,d)$.  
The highest weights correspond exactly to the irreducibles $V_{\lambda}$ 
that show up in the decomposition of $V^{\otimes d}$. 

As explained 
in the introduction, we can define the $q$-Schur algebra as follows: 
\begin{defn}
The $q$-Schur algebra $S_q(n,d)$ is the image 
of the representation 
$\psi_{n,d}\colon {\mathbf U}_q(\mathfrak{gl}_n)\to \End_{\mathbb{Q}}(V^{\otimes d})$.
\end{defn}
For each $\lambda\in\Lambda^+(n,d)$, the 
${\mathbf U}_q(\mathfrak{gl}_n)$-action on $V_{\lambda}$ factors through 
the projection $\psi_{n,d}\colon {\mathbf U}_q(\mathfrak{gl}_n)\to S_q(n,d)$. 
This way we obtain all irreducible representations of $S_q(n,d)$. Note that 
this also implies that all representations of $S_q(n,d)$ have a 
weight decomposition. As a matter of fact, it is well known that 
$$S_q(n,d)\cong \prod_{\lambda\in\Lambda^+(n,d)}\End_{\mathbb{Q}}(V_{\lambda}).$$
Therefore $S_q(n,d)$ is a finite-dimensional split semi-simple 
unital algebra and its dimension is equal to 
$$\sum_{\lambda\in\Lambda^+(n,d)}\dim(V_{\lambda})^2=\binom{n^2+d-1}{d}.$$ 

Since $V^{\otimes d}$ is a weight representation, 
$\psi_{n,d}$ gives rise to a homomorphism 
$\Ugl\to S_q(n,d)$, for 
which we use the same notation. This map is still surjective and 
Doty and Giaquinto, in Theorem 2.4 of~\cite{D-G}, showed that 
the kernel of $\psi_{n,d}$ is equal to the ideal generated by all 
idempotents $1_{\lambda}$ such that 
$\lambda\not\in\Lambda(n,d)$. Let $\SD(n,d)$ be 
the quotient of $\Ugl$ by the kernel of $\psi_{n,d}$. Clearly we have 
$\SD(n,d)\cong S_q(n,d)$. By the above observations, 
we see that $\SD(n,d)$ has a Serre presentation. As a matter of fact, 
by Corollary 4.3.2 in~\cite{C-G}, this presentation is simpler than 
that of $\Ugl$: one does not need to impose the last two Serre relations, 
involving cubical terms, because they are implied by the other relations 
and the finite 
dimensionality.\footnote{We thank Rapha\"{e}l Rouquier for pointing 
this out to us and giving us the reference.}  
\begin{lem} 
$\SD(n,d)$ is isomorphic to the associative unital 
$\bQ(q)$-algebra generated by $1_{\lambda}$, for $\lambda\in\Lambda(n,d)$, 
and $E_{\pm i}$, for $i=1,\ldots,n-1$, subject to the relations
\begin{align*}
1_{\lambda}1_{\mu} &= \delta_{\lambda,\mu}1_{\lambda} 
\\[0.5ex]
\sum_{\lambda\in\Lambda(n,d)}1_{\lambda} &= 1
\\[0.5ex]
E_{\pm i}1_{\lambda} &= 1_{\lambda\pm\alpha_i}E_{\pm i}
\\[0.5ex]
E_iE_{-j}-E_{-j}E_i &= \delta_{ij}\sum\limits_{\lambda\in\Lambda(n,d)}
[\overline{\lambda}_i]1_{\lambda}.
\end{align*}
We use the convention that $1_{\mu}X1_{\lambda}=0$, if $\mu$ 
or $\lambda$ is not contained in $\Lambda(n,d)$. Recall that $[a]$ is the 
$q$-integer $(q^a-q^{-a})/(q-q^{-1})$.
\end{lem}

Although there is no embedding of $\U$ into $\Ugl$, the projection 
$$\psi_{n,d}\colon{\mathbf U}_q(\mathfrak{gl}_n)\to S_q(n,d)$$ 
can be restricted to ${\mathbf U}_q(\mathfrak{sl}_n)$ and is still surjective. 
This gives rise to the surjection 
$$\psi_{n,d}\colon \U\to \SD(n,d),$$
defined by 
\begin{equation}
\label{eq:psi}
\psi_{n,d}(E_{\pm i}1_{\lambda})=E_{\pm i}1_{\phi_{n,d}(\lambda)},
\end{equation}
where $\phi_{n,d}$ was defined below equations~\eqref{eq:sl-gl-wts1} and 
\eqref{eq:sl-gl-wts2}. By convention we put $1_{*}=0$.   

As mentioned in the introduction, the $q$-Schur algebras for various values 
of $n$ and $d$ are related. Let $m\geq n$ and $d$ be arbitrary. There is an 
obvious embedding of the set of ${\mathbf U}_q(\mathfrak{gl}_n)$-weights into 
the set of ${\mathbf U}_q(\mathfrak{gl}_m)$-weights, given by 
$$(\lambda_1,\ldots,\lambda_n)\mapsto(\lambda_1,\ldots,\lambda_n,0,\ldots,0).$$
For fixed $d$, this gives an inclusion $\Lambda(n,d)\subseteq\Lambda(m,d)$, 
which we can use to define 
$$\xi_{n,m}=\sum_{\lambda\in\Lambda(n,d)}1_{\lambda}\in \SD(m,d).$$ 
Note that $\xi_{n,m}\ne 1$ unless $n=m$. 
\begin{defn}
There is a well-defined homomorphism 
$${\iota}_{n,m}\colon \SD(n,d)\to\xi_{n,m}\SD(m,d)\xi_{n,m}$$
given by 
$$
E_{\pm i}\mapsto\xi_{n,m}E_{\pm i}\xi_{n,m}\qquad\mbox{and}\qquad
1_{\lambda}\mapsto \xi_{n,m}1_{\lambda}\xi_{n,m}=1_{\lambda}.
$$
\end{defn}
\noindent It is easy to see that this is an 
isomorphism.

\begin{defn}  Suppose $d'=d+nk$, for a certain $k\in\bN$. Then we define a 
homomorphism $$\pi_{d',d}\colon \SD(n,d')\to \SD(n,d)$$ by 
$$1_{\lambda}\mapsto 1_{\lambda-(k^n)}\quad\mbox{and}\quad 
E_{\pm i}\mapsto E_{\pm i}.$$  
\end{defn}
\noindent It is easy to check that $\pi_{d',d}$ is well-defined and surjective. It is 
also easy to see that 
$$\pi_{d',d}\psi_{n,d'}=\psi_{n,d}$$ 
and that $\pi_{d',d}$ induces a linear isomorphism 
$$V_{\lambda}\to V_{\lambda-(k^n)},$$ 
which intertwines the $\SD(n,d')$ and $\SD(n,d)$ actions, 
if $\lambda-(k^n)\in\Lambda^+(n,d)$. Of course $V_{\lambda}$ and 
$V_{\lambda-(k^n)}$ are isomorphic as ${\mathbf U}_q(\mathfrak{sl}_n)$ 
representations. Furthermore, note that for any 
$d=0,\ldots,n-1$ the set 
\begin{equation}
\label{eq:inversesystem}
\left(S_q(n,d+nk),\pi_{d+nk,d}\right)_{k\in\bN}
\end{equation}
\noindent forms an inverse system, 
so we can form the inverse limit algebra 
$$\lim_{\longleftarrow k}S_q(n,d+nk).$$
\noindent The following lemma is perhaps a bit surprising.
\begin{lem}
\label{lem:inverselimit}
The map 
$\sum_d\prod_k\psi_{n,d+nk}$, with $d=0,\ldots, n-1$ and 
$k\in\bN$, gives an embedding
$${\mathbf U}_q(\mathfrak{sl}_n)\subset \bigoplus_{d=0}^{n-1}
\lim_{\longleftarrow k}S_q(n,d+nk).$$
\end{lem}
\noindent We also have 
\begin{equation}
\label{eq:inverselimit}
\U\subset\bigoplus_{d=0}^{n-1}
\lim_{\longleftarrow k}S_q(n,d+nk).
\end{equation}
\noindent The reader should remember this embedding when reading 
Corollary~\ref{cor:inverselimit}. The results in this paragraph were 
taken from~\cite{B-L-M}.  

We need to recall two more facts about $q$-Schur algebras and their 
representations. The first is that the irreducibles $V_{\lambda}$, 
for $\lambda\in\Lambda^+(n,d)$, can be 
constructed as subquotients of $\SD(n,d)$, called Weyl modules. Let $<$ denote  
the lexicographic order on $\Lambda(n,d)$.  
\begin{lem} 
\label{lem:weyl}
For any $\lambda\in\Lambda^+(n,d)$, we have 
$$V_{\lambda}\cong \SD(n,d)1_{\lambda}/[\mu>\lambda].$$
Here $[\mu>\lambda]$ is the ideal generated by all elements of the form 
$1_{\mu}x1_{\lambda}$, for some $x\in \SD(n,d)$ and $\mu>\lambda$.
\end{lem} 

Finally, we recall a well known anti-involution on $\SD(n,d)$, which we 
will need in this paper. 
\begin{defn}
\label{defn:tau}
We define an algebra anti-involution 
$$\tau\colon \SD(n,d)\to \SD(n,d)^{\mbox{\scriptsize op}}$$
\noindent by 
$$
\tau(1_{\lambda})=1_{\lambda},
\quad \tau(1_{\lambda+\alpha_i}E_i1_{\lambda})= 
q^{-1-\overline{\lambda}_i}1_{\lambda}E_{-i}1_{\lambda+\alpha_i},
\quad \tau(1_{\lambda}E_{-i}1_{\lambda+\alpha_i})= q^{1+\overline{\lambda}_i}
1_{\lambda+\alpha_i}E_i1_{\lambda}.
$$ 
\end{defn}
\noindent Note that up to a shift $t'$, we have 
\begin{align*}
1_{\mu}E_{s_1}E_{s_2}\cdots E_{s_{m-1}}E_{s_m}1_{\lambda}q^t &\mapsto 
1_{\lambda}E_{-s_m}E_{-s_{m-1}}\cdots E_{-s_2}E_{-s_1}1_{\mu}q^{-t+t'}.
\end{align*}
\noindent Our $\tau$ is the analogue of the one in \cite{K-L3}.

\subsection{The Hecke algebra}

Recall that $H_q(n)$ is a $q$-deformation of the group algebra of 
the symmetric group on $n$ letters.
\begin{defn}
The Hecke algebra $H_q(n)$ is the unital associative 
$\bQ(q)$-algebra generated by the 
elements $T_i$, $i=1,\ldots, n-1$, subject to the relations
\begin{align*}
T_i^2 &= (q^2-1)T_i+q^2
\\
T_iT_j &= T_jT_i\qquad\text{if}\quad|i-j|>1
\\
T_iT_{i+1}T_i &=T_{i+1}T_iT_{i+1}.
\end{align*}
\end{defn}
\noindent Note that some people write $q$ where we write $q^2$ and 
use $v=q^{-1}$ in their presentation of the Hecke algebra. It is also not 
uncommon to find $t$ instead of our $q$. 

For $q=1$ we recover the presentation of $\bQ[S_n]$ in terms of the simple 
transpositions $\sigma_i$. For any element $\sigma\in S_n$ we can define 
$T_{\sigma}=T_{i_1}\cdots T_{i_k}$, choosing a reduced expression 
$\sigma=\sigma_{i_1}\cdots \sigma_{i_k}$. The relations above guarantee that 
 all reduced expressions of $\sigma$ give the same element $T_{\sigma}$. 
The $T_{\sigma}$, for $\sigma\in S_n$, form a linear basis of $H_q(n)$.

There is a simple change of generators, 
which is convenient for categorification 
purposes. Write $b_i=q^{-1}(T_i+1)$. Then the relations above become 
\begin{align*}
b_i^2 &= (q+q^{-1})b_i
\\
b_ib_j &= b_jb_i \qquad\text{if}\quad|i-j|>1
\\
b_ib_{i+1}b_i+b_{i+1} &= b_{i+1}b_ib_{i+1}+b_i.
\end{align*}
\noindent These generators are the simplest elements of the so called 
{\em Kazhdan-Lusztig basis}. Although the change of generators is simple, the 
whole change of linear bases is very complicated. 

As mentioned in the introduction, there is a $q$-version of Schur-Weyl duality. 
There is a $q$-permutation action of $H_q(d)$ on $V^{\otimes d}$, which 
is induced by the $R$-matrix of 
${\mathbf U}_q(\mathfrak{gl}_n)$ or ${\mathbf U}_q(\mathfrak{sl}_n)$ and 
commutes with the actions of these quantum enveloping algebras. With respect to 
these actions, $H_q(d)$ and $\SD(n,d)$ have the double centralizer property. 
Furthermore, their respective categories of finite-dimensional representations 
are equivalent. 

Suppose $n\geq d$. We explicitly recall the embedding of $H_q(d)$ into 
$\SD(n,d)$. Let $1_d=1_{(1^d)}$. Note that the ${\mathbf U}_q(\mathfrak{gl}_n)$-weight 
$(1^d)$ gives the zero ${\mathbf U}_q(\mathfrak{sl}_n)$-weight, for $n=d$, and 
a fundamental ${\mathbf U}_q(\mathfrak{sl}_n)$-weight for $n>d$. 
We define the following map 
$$\sigma_{n,d}\colon H_q(d)\to 1_{d}\SD(n,d)1_{d}$$
by 
$$\sigma_{n,d}(b_i)=1_dE_{-i}E_i1_d=1_dE_iE_{-i}1_d,$$
for $i=1,\ldots,d-1$.
It is easy to check that $\sigma_{n,d}$ is well-defined. It turns out that 
$\sigma_{n,d}$ is actually an isomorphism, which induces the 
{\em $q$-Schur functor} $\SD(n,d)-\mbox{mod}\to H_q(d)-\mbox{mod}$, where mod 
denotes the category of finite-dimensional modules. This functor is an 
equivalence. Let us state explicitly an easy implication of this 
equivalence, which we need in the sequel. 
\begin{lem} 
\label{lem:emb}
Let $0<d\leq n$ and let $A$ be a unital associative 
$\bQ(q)$-algebra. Suppose $\pi\colon \SD(n,d)\to A$ is a surjection of 
$\bQ(q)$-algebras, 
such that $\pi\circ \sigma_{n,d}\colon H_q(d)\to A$ is an embedding. Then 
$A\cong \SD(n,d)$.
\end{lem}   
\begin{proof} Recall that 
$$\SD(n,d)\cong \prod_{\lambda\in\Lambda^+(n,d)}\mbox{End}_{\Q(q)}(V_{\lambda}).$$ 
The fact that the $q$-Schur functor is an equivalence means that the 
projection of 
$\sigma_{n,d}(H_q(d))$ onto $\mbox{End}_{\Q(q)}(V_{\lambda})$ is non-zero, for any  
$\lambda\in\Lambda^+(n,d)$. Since all $\mbox{End}_{\Q(q)}(V_{\lambda})$ are simple 
algebras, $A$ has to be isomorphic to the product 
$$\prod_{\lambda\in\Lambda'}\mbox{End}_{\Q(q)}(V_{\lambda}),$$ 
for a certain subset $\Lambda'\subseteq \Lambda^+(n,d)$. 
But $\pi\circ \sigma_{n,d}$ is an embedding, so $\Lambda'=\Lambda^+(n,d)$ 
has to hold. 
\end{proof}

\section{The 2-categories $\glcat$ and $\Scat(n,d)$}  
\label{sec:scat}                                      

In this section we define two $2$-categories, 
$\mathcal{U}(\mathfrak{gl}_n)$ and $\Scat(n,d)$, 
using a graphical calculus analogous to Khovanov and Lauda's in~\cite{K-L3}. 
We thank Khovanov and Lauda for letting us copy their definition of 
$\mathcal{U}_{\to}(\mathfrak{sl}_n)$. 
Taking their definition, we first introduce a change of weights to obtain 
$\mathcal{U}(\mathfrak{gl}_n)$. Then we divide by an ideal to obtain 
$\Scat(n,d)$. 

As remarked in the introduction, our signs are 
slightly different from 
those in~\cite{K-L3}. Khovanov 
and Lauda~\cite{K-L:err} corrected their sign convention in 
$\mathcal{U}_{\to}(\mathfrak{sl}_n)$. 
As it turns out, 
the corrected $\mathcal{U}_{\to}(\mathfrak{sl}_n)$ is no longer cyclic, 
which makes working with that sign convention awkward. 
Fortunately Khovanov and Lauda's 
non-signed version, $\Ucat$, is still correct and cyclic and is isomorphic to 
the corrected $\mathcal{U}_{\to}(\mathfrak{sl}_n)$~\cite{K-L3, K-L:err}. However, the sign 
convention in $\Ucat$ is not so practical for the 
2-representation into bimodules, so we have decided to stick to our own 
sign convention in this paper. To get from our signs back to 
Khovanov and Lauda's (corrected) signs in $\mathcal{U}_{\to}(\mathfrak{sl}_n)$, 
apply the $2$-isomorphism which is the 
identity on all objects, $1$- and 
$2$-morphisms except the left cups and caps, on which it is given by   
\begin{equation}
\label{eq:signs}
\Ucapli_{i,\lambda}\mapsto (-1)^{\lambda_{i+1}+1}\,\,\Ucapli_{i,\lambda}
\quad\mbox{and}\quad \Ucupli_{i,\lambda}\mapsto (-1)^{\lambda_{i+1}}\,\,
\Ucupli_{i,\lambda}.
\end{equation}  

The various parts of our definition of $\mathcal{U}(\mathfrak{gl}_n)$ and 
$\Scat(n,d)$ below have exactly the same order as the corresponding parts of 
Khovanov and Lauda's definition of $\mathcal{U}_{\to}(\mathfrak{sl}_n)$, so 
the reader can compare them in detail. From 
now on we will always write $\Ucat$, instead of $\mathcal{U}_{\to}(\mathfrak{sl}_n)$, for the 
corrected signed categorification of $\U$. Since we will never work with 
the unsigned version, there should be no confusion. 

%
%
\subsection{The 2-category $\glcat$}
As already remarked in the introduction, the idea underlying the 
definition of $\mathcal{U}(\mathfrak{gl}_n)$ is very simple: 
it is obtained from $\Ucat$ by passing from $\mathfrak{sl}_n$-weights 
to $\mathfrak{gl}_n$-weights.

From now on let $n\in\N_{>1}$ be arbitrary but fixed and 
let $I=\{1,2,\ldots,n-1\}$. In the sequel we use {\em signed sequences} 
$\ii=(\alpha_1i_1,\ldots,\alpha_mi_m)$, 
for any $m\in\N$, $\alpha_j\in\{\pm 1\}$ and $i_j\in I$. 
The set of signed sequences 
we denote $\sseq$. For $\ii=(\alpha_1i_1,\ldots,\alpha_mi_m)\in\sseq$ we 
define $\ii_{\Lambda}:=\alpha_1 (i_1)_{\Lambda}+\cdots+\alpha_m (i_m)_{\Lambda}$, where 
$$(i_j)_{\Lambda}=(0,0,\ldots,1,-1,0\ldots,0),$$
such that the vector starts with $i_j-1$ and ends with $k-1-i_j$ zeros. 
To understand these definitions, the reader should recall our definition of 
$E_{\ii}$ and $\ii_{\Lambda}$ below Definition~\ref{defn:Uglndot}. We also 
define the symmetric $\Z$-valued bilinear form on $\Q[I]$ 
by $i\cdot i=2$, $i\cdot (i+1)=-1$ and $i\cdot j=0$, for $\vert i-j\vert>1$. 
Recall that $\overline{\lambda}_i=\lambda_i-\lambda_{i+1}$.

\begin{defn} \label{def_glcat} $\glcat$ is an
additive $\Q$-linear 2-category. The 2-category $\glcat$ consists of
\begin{itemize}
  \item objects:  $\lambda\in\bZ^n$.
\end{itemize}
The hom-category $\glcat(\lambda,\lambda')$ between two objects 
$\lambda$, $\lambda'$ is an additive $\Q$-linear category 
consisting of:
\begin{itemize}
  \item objects\footnote{We refer to objects of the category
$\glcat(\lambda,\lambda')$ as 1-morphisms of $\glcat$.  Likewise, the morphisms of
$\glcat(\lambda,\lambda')$ are called 2-morphisms in $\glcat$. } of
$\glcat(\lambda,\lambda')$: a 1-morphism in $\glcat$ from $\lambda$ to $\lambda'$
is a formal finite direct sum of 1-morphisms
  \[
 \cal{E}_{\ii} \onel\{t\} = \onelp \cal{E}_{\ii} \onel\{t\}
:= \cal{E}_{\alpha_1 i_1}\dotsm\cal{E}_{\alpha_m i_m} \onel\{t\}
  \]
for any $t\in \Z$ and signed sequence $\ii\in\sseq$ such that 
$\lambda'=\lambda+\ii_{\Lambda}$ and $\lambda$, $\lambda'\in\bZ^n$. 
  \item morphisms of $\glcat(\lambda,\lambda')$: for 1-morphisms $\cal{E}_{\ii} \onel\{t\}$
and  $\cal{E}_{\jj} \onel\{t'\}$ in $\glcat$, the hom
sets $\glcat(\cal{E}_{\ii} \onel\{t\},\cal{E}_{\jj} \onel\{t'\})$ of
$\glcat(\lambda,\lambda')$ are graded $\Q$-vector spaces given by linear
combinations of degree $t-t'$ diagrams, modulo certain relations, built from
compo\-sites of:
\begin{enumerate}[i)]
  \item  Degree zero identity 2-morphisms $1_x$ for each 1-morphism $x$ in
$\glcat$; the identity 2-morphisms $1_{\cal{E}_{+i} \onel}\{t\}$ and
$1_{\cal{E}_{-i} \onel}\{t\}$, for $i \in I$, are represented graphically by
\[
\begin{array}{ccc}
  1_{\cal{E}_{+i} \onel\{t\}} &\quad  & 1_{\cal{E}_{-i} \onel\{t\}} \\ \\
   \xy
 (0,0)*{\dblue\xybox{(0,8);(0,-8); **\dir{-} ?(.5)*\dir{>}+(2.3,0)*{\scriptstyle{}};}};
 (-1,-11)*{ i};(-1,11)*{ i};
 (6,2)*{ \lambda};
 (-8,2)*{ \lambda +i_{\Lambda}};
 (-10,0)*{};(10,0)*{};
 \endxy
 & &
 \;\;   
   \xy
 (0,0)*{\dblue\xybox{(0,8);(0,-8); **\dir{-} ?(.5)*\dir{<}+(2.3,0)*{\scriptstyle{}};}};
 (-1,-11)*{ i};(-1,11)*{ i};
 (6,2)*{ \lambda};
 (-8.5,2)*{ \lambda -i_{\Lambda}};
 (-12,0)*{};(12,0)*{};
 \endxy
\\ \\
   \;\;\text{ {\rm deg} 0}\;\;
 & &\;\;\text{ {\rm deg} 0}\;\;
\end{array}
\]
for any $\lambda + i_{\Lambda} \in\bZ^n$ and any 
$\lambda - i_{\Lambda} \in \bZ^n$, respectively.

More generally, for a signed sequence $\ii=(\alpha_1i_1, \alpha_2i_2, \ldots
\alpha_mi_m)$, the identity $1_{\cal{E}_{\ii} \onel\{t\}}$ 2-morphism is
represented as
\begin{equation*}
\begin{array}{ccc}
  \xy
 (-12,0)*{\dblue\xybox{(-12,8);(-12,-8); **\dir{-};}};
 (-4,0)*{\dred\xybox{(-4,8);(-4,-8); **\dir{-};}};
 (4,0)*{\cdots};
 (12,0)*{\dgreen\xybox{(12,8);(12,-8); **\dir{-};}};
 (-12,11)*{i_1}; (-4,11)*{ i_2};(12,11)*{ i_m };
  (-12,-11)*{ i_1}; (-4,-11)*{ i_2};(12,-11)*{ i_m};
 (18,2)*{ \lambda}; (-20,2)*{ \lambda+\ii_{\Lambda}};
 \endxy
\end{array}
\end{equation*}
where the strand labeled $i_{k}$ is oriented up if $\alpha_{k}=+$
and oriented down if $\alpha_{k}=-$. We will often place labels with no
sign on the side of a strand and omit the labels at the top and bottom.  The
signs can be recovered from the orientations on the strands.

\item For each $\lambda \in \bZ^n$ the 2-morphisms 
\[
\begin{tabular}{|l|c|c|c|c|}
\hline
 {\bf Notation:} \xy (0,-5)*{};(0,7)*{}; \endxy&
 $\dblue\Uup_{\black i,\lambda}$  &  $\dblue\Udown_{\black i,\lambda}$  
 &$\Ucrossij_{i,j,\lambda}$
 &$\Ucrossdij_{i,j,\lambda}$  \\
 \hline
 {\bf 2-morphism:} &   \xy
 (0,0)*{\dblue\xybox{(0,7);(0,-7); **\dir{-} ?(.75)*\dir{>}+(2.3,0)*{\scriptstyle{}}
 ?(.1)*\dir{ }+(2,0)*{\black \scs i};
 (0,-2)*{\txt\large{$\bullet$}};}};
 (4,4)*{ \lambda};
 (-8,4)*{ \lambda +i_{\Lambda}};
 (-10,0)*{};(10,0)*{};
 \endxy
 &
     \xy
 (0,0)*{\dblue\xybox{(0,7);(0,-7); **\dir{-} ?(.75)*\dir{<}+(2.3,0)*{\scriptstyle{}}
 ?(.1)*\dir{ }+(2,0)*{\black\scs i};
 (0,-2)*{\txt\large{$\bullet$}};}};
 (-6,4)*{ \lambda};
 (8,3.9)*{ \lambda +i_{\Lambda}};
 (-10,0)*{};(10,9)*{};
 \endxy
 &
   \xy
  (0,0)*{\xybox{
    (0,0)*{\dblue\xybox{(-4,-4)*{};(4,4)*{} **\crv{(-4,-1) & (4,1)}?(1)*\dir{>} ;}};
    (0,0)*{\dgreen\xybox{(4,-4)*{};(-4,4)*{} **\crv{(4,-1) & (-4,1)}?(1)*\dir{>};}};
    (-5,-3)*{\scs i};
     (5.1,-3)*{\scs j};
     (8,1)*{ \lambda};
     (-12,0)*{};(12,0)*{};
     }};
  \endxy
 &
   \xy
  (0,0)*{\xybox{
    (0,0)*{\dgreen\xybox{(-4,4)*{};(4,-4)*{} **\crv{(-4,1) & (4,-1)}?(1)*\dir{>} ;}};
    (0,0)*{\dblue\xybox{(4,4)*{};(-4,-4)*{} **\crv{(4,1) & (-4,-1)}?(1)*\dir{>};}};
    (-6,-3)*{\scs i};
     (6,-3)*{\scs j};
     (8,1)*{ \lambda};
     (-12,0)*{};(12,0)*{};
     }};
  \endxy
\\ & & & &\\
\hline
 {\bf Degree:} & \;\;\text{  $i\cdot i$ }\;\;
 &\;\;\text{  $i\cdot i$}\;\;& \;\;\text{  $-i \cdot j$}\;\;
 & \;\;\text{  $-i \cdot j$}\;\; \\
 \hline
\end{tabular}
\]

\[
\begin{tabular}{|l|c|c|c|c|}
\hline
  {\bf Notation:} \xy (0,-5)*{};(0,7)*{}; \endxy
 & \text{$\Ucupri_{i,\lambda}$} 
 & \text{$\Ucupli_{i,\lambda}$} 
 & \text{$\Ucapli_{i,\lambda}$} 
 & \text{$\Ucapri_{i,\lambda}$} \\
 \hline
  {\bf 2-morphism:} & \xy
    (0,-3)*{\dblue\bbpef{\black i}};
    (8,-3)*{ \lambda};
    (-12,0)*{};(12,0)*{};
    \endxy
  & \xy
    (0,-3)*{\dblue\bbpfe{\black i}};
    (8,-3)*{ \lambda};
    (-12,0)*{};(12,0)*{};
    \endxy
  & \xy
    (0,0)*{\dblue\bbcef{\black i}};
    (8,3)*{ \lambda};
    (-12,0)*{};(12,0)*{};
    \endxy 
  & \xy
    (0,0)*{\dblue\bbcfe{\black i}};
    (8,3)*{ \lambda};
    (-12,0)*{};(12,0)*{};(8,8)*{};
    \endxy\\& & &  &\\ \hline
 {\bf Degree:} & \;\;\text{  $1+\llambda_i$}\;\;
 & \;\;\text{ $1-\llambda_i$}\;\;
 & \;\;\text{ $1+\llambda_i$}\;\;
 & \;\;\text{  $1-\llambda_i$}\;\;
 \\
 \hline
\end{tabular}
\]
\end{enumerate}

\item Biadjointness and cyclicity:
\begin{enumerate}[i)]
\item\label{it:sl2i}  $\mathbf{1}_{\lambda+i_{\Lambda}}\cal{E}_{+i}\onel$ and
$\onel\cal{E}_{-i}\mathbf{1}_{\lambda+i_{\Lambda}}$ are biadjoint, up to grading shifts:
\begin{equation} \label{eq_biadjoint1}
\text{$
  \xy   0;/r.18pc/:
    (0,0)*{\dblue\xybox{
    (-8,0)*{}="1";
    (0,0)*{}="2";
    (8,0)*{}="3";
    (-8,-10);"1" **\dir{-};
    "1";"2" **\crv{(-8,8) & (0,8)} ?(0)*\dir{>} ?(1)*\dir{>};
    "2";"3" **\crv{(0,-8) & (8,-8)}?(1)*\dir{>};
    "3"; (8,10) **\dir{-};}};
    (12,-9)*{\lambda};
    (-6,9)*{\lambda+i_{\Lambda}};
    \endxy
    \; =
    \;
\xy   0;/r.18pc/:
    (0,0)*{\dblue\xybox{
    (-8,0)*{}="1";
    (0,0)*{}="2";
    (8,0)*{}="3";
    (0,-10);(0,10)**\dir{-} ?(.5)*\dir{>};}};
    (5,8)*{\lambda};
    (-9,8)*{\lambda+i_{\Lambda}};
    \endxy
\qquad \quad  \xy   0;/r.18pc/:
    (0,0)*{\dblue\xybox{
    (-8,0)*{}="1";
    (0,0)*{}="2";
    (8,0)*{}="3";
    (-8,-10);"1" **\dir{-};
    "1";"2" **\crv{(-8,8) & (0,8)} ?(0)*\dir{<} ?(1)*\dir{<};
    "2";"3" **\crv{(0,-8) & (8,-8)}?(1)*\dir{<};
    "3"; (8,10) **\dir{-};}};
    (12,-9)*{\lambda+i_{\Lambda}};
    (-6,9)*{ \lambda};
    \endxy
    \; =
    \;
\xy   0;/r.18pc/:
    (0,0)*{\dblue\xybox{
    (-8,0)*{}="1";
    (0,0)*{}="2";
    (8,0)*{}="3";
    (0,-10);(0,10)**\dir{-} ?(.5)*\dir{<};}};
    (9,8)*{\lambda+i_{\Lambda}};
    (-6,8)*{ \lambda};
    \endxy
$}
\end{equation}

\begin{equation}
\label{eq_biadjoint2}
\text{$
 \xy   0;/r.18pc/:
    (0,0)*{\dblue\xybox{
    (8,0)*{}="1";
    (0,0)*{}="2";
    (-8,0)*{}="3";
    (8,-10);"1" **\dir{-};
    "1";"2" **\crv{(8,8) & (0,8)} ?(0)*\dir{>} ?(1)*\dir{>};
    "2";"3" **\crv{(0,-8) & (-8,-8)}?(1)*\dir{>};
    "3"; (-8,10) **\dir{-};}};
    (12,9)*{\lambda};
    (-5,-9)*{\lambda+i_{\Lambda}};
    \endxy
    \; =
    \;
      \xy 0;/r.18pc/:
    (0,0)*{\dblue\xybox{
    (8,0)*{}="1";
    (0,0)*{}="2";
    (-8,0)*{}="3";
    (0,-10);(0,10)**\dir{-} ?(.5)*\dir{>};}};
    (5,-8)*{\lambda};
    (-9,-8)*{\lambda+i_{\Lambda}};
    \endxy
\qquad \quad \xy  0;/r.18pc/:
    (0,0)*{\dblue\xybox{
    (8,0)*{}="1";
    (0,0)*{}="2";
    (-8,0)*{}="3";
    (8,-10);"1" **\dir{-};
    "1";"2" **\crv{(8,8) & (0,8)} ?(0)*\dir{<} ?(1)*\dir{<};
    "2";"3" **\crv{(0,-8) & (-8,-8)}?(1)*\dir{<};
    "3"; (-8,10) **\dir{-};}};
    (12,9)*{\lambda+i_{\Lambda}};
    (-6,-9)*{ \lambda};
    \endxy
    \; =
    \;
\xy  0;/r.18pc/:
    (0,0)*{\dblue\xybox{
    (8,0)*{}="1";
    (0,0)*{}="2";
    (-8,0)*{}="3";
    (0,-10);(0,10)**\dir{-} ?(.5)*\dir{<};}};
    (9,-8)*{\lambda+i_{\Lambda}};
    (-6,-8)*{ \lambda};
    \endxy
$}
\end{equation}
\item
\begin{equation} \label{eq_cyclic_dot}
\text{$
    \xy
    (0,0)*{\dblue\xybox{
    (-8,5)*{}="1";
    (0,5)*{}="2";
    (0,-5)*{}="2'";
    (8,-5)*{}="3";
    (-8,-10);"1" **\dir{-};
    "2";"2'" **\dir{-} ?(.5)*\dir{<};
    "1";"2" **\crv{(-8,12) & (0,12)} ?(0)*\dir{<};
    "2'";"3" **\crv{(0,-12) & (8,-12)}?(1)*\dir{<};
    "3"; (8,10) **\dir{-};
    (0,4)*{\txt\large{$\bullet$}};}};
    (15,-9)*{ \lambda+i_{\Lambda}};
    (-12,9)*{\lambda};
    (10,8)*{\scs };
    (-10,-8)*{\scs i};
    \endxy
    \quad = \quad
      \xy
 (0,0)*{\dblue\xybox{
 (0,10);(0,-10); **\dir{-} ?(.75)*\dir{<}+(2.3,0)*{\scriptstyle{}}
 ?(.1)*\dir{ }+(2,0)*{\scs };
 (0,0)*{\txt\large{$\bullet$}};}};
 (-6,5)*{ \lambda};
 (8,5)*{ \lambda +i_{\Lambda}};
 (-10,0)*{};(10,0)*{};(-2,-8)*{\scs i};
 \endxy
    \quad = \quad
    \xy
    (0,0)*{\dblue\xybox{
    (8,5)*{}="1";
    (0,5)*{}="2";
    (0,-5)*{}="2'";
    (-8,-5)*{}="3";
    (8,-10);"1" **\dir{-};
    "2";"2'" **\dir{-} ?(.5)*\dir{<};
    "1";"2" **\crv{(8,12) & (0,12)} ?(0)*\dir{<};
    "2'";"3" **\crv{(0,-12) & (-8,-12)}?(1)*\dir{<};
    "3"; (-8,10) **\dir{-};
    (0,4)*{\txt\large{$\bullet$}};}};
    (15,9)*{\lambda+i_{\Lambda}};
    (-12,-9)*{\lambda};
    (-10,8)*{\scs };
    (10,-8)*{\scs i};
    \endxy
$}
\end{equation}
\item All 2-morphisms are cyclic with respect to the above biadjoint
   structure.\footnote{See \cite{L1} 
and the references therein for
  the definition of a cyclic 2-morphism with respect to a biadjoint 
structure.} This is ensured by the relations \eqref{eq_cyclic_dot}, and the
   relations
\begin{equation} \label{eq_cyclic_cross-gen}
\text{$
\xy 0;/r.19pc/:
(0,0)*{\dred\xybox{
     (-4,-4)*{};(4,4)*{} **\crv{(-4,-1) & (4,1)}?(1)*\dir{>};
     (-4,-4)*{};(18,-4)*{} **\crv{(-4,-16) & (18,-16)} ?(1)*\dir{<}?(0)*\dir{<};
     (18,-4);(18,12) **\dir{-};
     (4,4)*{};(-18,4)*{} **\crv{(4,16) & (-18,16)} ?(1)*\dir{>};  
     (-18,4);(-18,-12) **\dir{-}; 
}};
(0,0)*{\dblue\xybox{
     (4,-4)*{};(-4,4)*{} **\crv{(4,-1) & (-4,1)};
     (12,-4);(12,12) **\dir{-};
     (-12,4);(-12,-12) **\dir{-};
     (4,-4)*{};(12,-4)*{} **\crv{(4,-10) & (12,-10)}?(1)*\dir{<}?(0)*\dir{<};
      (-4,4)*{};(-12,4)*{} **\crv{(-4,10) & (-12,10)}?(1)*\dir{>}?(0)*\dir{>};
}};
     (8,1)*{ \lambda};
     (-10,0)*{};(10,0)*{};
      (20,11)*{\scs j};(10,11)*{\scs i};
      (-20,-11)*{\scs j};(-10,-11)*{\scs i};
  \endxy
\quad =  \quad \xy
(0,0)*{\dred\xybox{
     (-4,-4)*{};(4,4)*{} **\crv{(-4,-1) & (4,1)}?(0)*\dir{<};}};
(0,0)*{\dblue\xybox{
     (4,-4)*{};(-4,4)*{} **\crv{(4,-1) & (-4,1)}?(0)*\dir{<};}};
     (-5,3)*{\scs i};
     (5.1,3)*{\scs j};
     (-8,0)*{ \lambda};
     (-12,0)*{};(12,0)*{};
  \endxy \quad :=  \quad
 \xy 0;/r.19pc/:
(0,0)*{\dblue\xybox{
      (4,-4)*{};(-4,4)*{} **\crv{(4,-1) & (-4,1)}?(1)*\dir{>};
      (-4,4)*{};(18,4)*{} **\crv{(-4,16) & (18,16)} ?(1)*\dir{>};
      (4,-4)*{};(-18,-4)*{} **\crv{(4,-16) & (-18,-16)} ?(1)*\dir{<}?(0)*\dir{<};
      (18,4);(18,-12) **\dir{-};
      (-18,-4);(-18,12) **\dir{-};
}};
(0,0)*{\dred\xybox{
      (-4,-4)*{};(4,4)*{} **\crv{(-4,-1) & (4,1)}; 
      (12,4);(12,-12) **\dir{-};
      (-10,0)*{};(10,0)*{};
      (-4,-4)*{};(-12,-4)*{} **\crv{(-4,-10) & (-12,-10)}?(1)*\dir{<}?(0)*\dir{<};
      (4,4)*{};(12,4)*{} **\crv{(4,10) & (12,10)}?(1)*\dir{>}?(0)*\dir{>};
      (-12,-4);(-12,12) **\dir{-};
}};
  (8,1)*{ \lambda};
 (-20,11)*{\scs i};(-10,11)*{\scs j};
  (20,-11)*{\scs i};(10,-11)*{\scs j};
  \endxy
$}
\end{equation}
Note that we can take either the first or the last diagram above as the 
definition of the up-side-down crossing. We have chosen the last one 
above, because it is the one which matches Khovanov and Lauda's signs. 
The cyclic condition on 2-morphisms expressed by \eqref{eq_cyclic_dot} and
\eqref{eq_cyclic_cross-gen} ensures that diagrams related by isotopy represent
the same 2-morphism in $\glcat$.

It will be convenient to introduce degree zero 2-morphisms:
\begin{equation} \label{eq_crossl-gen}
\text{$
  \xy
(0,0)*{\dred\xybox{
    (-4,-4)*{};(4,4)*{} **\crv{(-4,-1) & (4,1)}?(1)*\dir{>};}};
(0,0)*{\dblue\xybox{
    (4,-4)*{};(-4,4)*{} **\crv{(4,-1) & (-4,1)}?(0)*\dir{<};}};
    (-5,-3)*{\scs j};
     (-5,3)*{\scs i};
     (8,2)*{ \lambda};
     (-12,0)*{};(12,0)*{};
  \endxy
:=
 \xy 0;/r.19pc/:
(0,0)*{\dred\xybox{
    (4,-4)*{};(-4,4)*{} **\crv{(4,-1) & (-4,1)}?(1)*\dir{>};
    (-4,4);(-4,12) **\dir{-}; 
    (4,-4);(4,-12) **\dir{-};
}};
(0,0)*{\dblue\xybox{
    (-4,-4)*{};(4,4)*{} **\crv{(-4,-1) & (4,1)};
    (-12,-4);(-12,12) **\dir{-};
    (12,4);(12,-12) **\dir{-};
    (-10,0)*{};(10,0)*{};
    (-4,-4)*{};(-12,-4)*{} **\crv{(-4,-10) & (-12,-10)}?(1)*\dir{<}?(0)*\dir{<};
    (4,4)*{};(12,4)*{} **\crv{(4,10) & (12,10)}?(1)*\dir{>}?(0)*\dir{>};
}};
    (16,1)*{\lambda};
    (-14,11)*{\scs i};(-2,11)*{\scs j};
    (14,-11)*{\scs i};(2,-11)*{\scs j};
 \endxy
  \quad = \quad
  \xy 0;/r.19pc/:
(0,0)*{\dblue\xybox{
    (-4,-4)*{};(4,4)*{} **\crv{(-4,-1) & (4,1)}?(1)*\dir{<};
    (4,4);(4,12) **\dir{-};(-4,-4);(-4,-12) **\dir{-};
}};
(0,0)*{\dred\xybox{
    (4,-4)*{};(-4,4)*{} **\crv{(4,-1) & (-4,1)};
     (12,-4);(12,12) **\dir{-};
     (-12,4);(-12,-12) **\dir{-};
     (10,0)*{};(-10,0)*{};
     (4,-4)*{};(12,-4)*{} **\crv{(4,-10) & (12,-10)}?(1)*\dir{>}?(0)*\dir{>};
      (-4,4)*{};(-12,4)*{} **\crv{(-4,10) & (-12,10)}?(1)*\dir{<}?(0)*\dir{<};
}};
     (16,1)*{\lambda};
     (14,11)*{\scs j};(2,11)*{\scs i};
     (-14,-11)*{\scs j};(-2,-11)*{\scs i};
  \endxy
$}
\end{equation}
\begin{equation} \label{eq_crossr-gen}
\text{$
  \xy
(0,0)*{\dred\xybox{
    (-4,-4)*{};(4,4)*{} **\crv{(-4,-1) & (4,1)}?(0)*\dir{<};}};
(0,0)*{\dblue\xybox{
    (4,-4)*{};(-4,4)*{} **\crv{(4,-1) & (-4,1)}?(1)*\dir{>};}};
    (5.1,-3)*{\scs i};
     (5.1,3)*{\scs j};
     (-8,2)*{ \lambda};
     (-12,0)*{};(12,0)*{};
  \endxy
:=
 \xy 0;/r.19pc/:
(0,0)*{\dblue\xybox{
    (-4,-4)*{};(4,4)*{} **\crv{(-4,-1) & (4,1)}?(1)*\dir{>};
    (4,4);(4,12) **\dir{-};
    (-4,-4);(-4,-12) **\dir{-};
}};
(0,0)*{\dred\xybox{
     (4,-4)*{};(-4,4)*{} **\crv{(4,-1) & (-4,1)};
     (12,-4);(12,12) **\dir{-};
     (-12,4);(-12,-12) **\dir{-};
     (10,0)*{};(-10,0)*{};
     (4,-4)*{};(12,-4)*{} **\crv{(4,-10) & (12,-10)}?(1)*\dir{<}?(0)*\dir{<};
     (-4,4)*{};(-12,4)*{} **\crv{(-4,10) & (-12,10)}?(1)*\dir{>}?(0)*\dir{>};
}};
     (-16,1)*{\lambda};
     (14,11)*{\scs j};(2,11)*{\scs i};
     (-14,-11)*{\scs j};(-2,-11)*{\scs i};
\endxy
  \quad = \quad
  \xy 0;/r.19pc/:
(0,0)*{\dred\xybox{
     (4,-4)*{};(-4,4)*{} **\crv{(4,-1) & (-4,1)}?(1)*\dir{<};
     (-4,4);(-4,12) **\dir{-};     (4,-4);(4,-12) **\dir{-};
}};
(0,0)*{\dblue\xybox{
     (-4,-4)*{};(4,4)*{} **\crv{(-4,-1) & (4,1)};
     (-12,-4);(-12,12) **\dir{-};
     (12,4);(12,-12) **\dir{-};
     (-10,0)*{};(10,0)*{};
     (-4,-4)*{};(-12,-4)*{} **\crv{(-4,-10) & (-12,-10)}?(1)*\dir{>}?(0)*\dir{>};
     (4,4)*{};(12,4)*{} **\crv{(4,10) & (12,10)}?(1)*\dir{<}?(0)*\dir{<};
}};
     (-16,1)*{\lambda};
     (-14,11)*{\scs i};(-2,11)*{\scs j};
     (14,-11)*{\scs i};(2,-11)*{\scs j};
  \endxy
$}
\end{equation}
where the second equality in \eqref{eq_crossl-gen} and \eqref{eq_crossr-gen}
follow from \eqref{eq_cyclic_cross-gen}. Again we have indicated which choice 
of twists we use to define the sideways crossings, which is exactly the choice 
which matches Khovanov and Lauda's sign conventions.  

\item All dotted bubbles of negative degree are zero. That is,
\begin{equation} \label{eq_positivity_bubbles}
 \xy
 (-12,0)*{\dblue\cbub{\black m}{\black i}};
 (-8,8)*{\lambda};
 \endxy
  = 0
 \qquad
  \text{if $m<\llambda_i-1$} \qquad
 \xy
 (-12,0)*{\dblue\ccbub{\black m}{\black i}};
 (-8,8)*{\lambda};
 \endxy = 0\quad
  \text{if $m< -\llambda_i-1$}
\end{equation}
for all $m \in \Z_+$, where a dot carrying a label $m$ denotes the
$m$-fold iterated vertical composite of $\Uup_{i,\lambda}$ or
$\Udown_{i,\lambda}$ depending on the orientation.  A dotted bubble of degree
zero equals $\pm 1$:
\begin{equation}\label{eq:bubb_deg0}
\xy 0;/r.18pc/:
 (0,-1)*{\dblue\cbub{\black\llambda_i-1}{\black i}};
  (4,8)*{\lambda};
 \endxy
  = (-1)^{\laii} \quad \text{for $\llambda_i \geq 1$,}
  \qquad \quad
  \xy 0;/r.18pc/:
 (0,-1)*{\dblue\ccbub{\black -\llambda_i-1}{\black i}};
  (4,8)*{\lambda};
 \endxy
  = (-1)^{\laii-1} \quad \text{for $\llambda_i \leq -1$.}
\end{equation}
\item For the following relations we employ the convention that all summations
are increasing, so that a summation of the form $\sum_{f=0}^{m}$ is zero 
if $m < 0$.
\begin{eqnarray}
\label{eq:redtobubbles}
  \text{$\xy 0;/r.18pc/:
  (10,8)*{\lambda};
  (0,-3)*{\dblue\xybox{
  (-3,-8)*{};(3,8)*{} **\crv{(-3,-1) & (3,1)}?(1)*\dir{>};?(0)*\dir{>};
    (3,-8)*{};(-3,8)*{} **\crv{(3,-1) & (-3,1)}?(1)*\dir{>};
  (-3,-12)*{\bbsid};
  (-3,8)*{\bbsid};
  (3,8)*{}="t1";
  (9,8)*{}="t2";
  (3,-8)*{}="t1'";
  (9,-8)*{}="t2'";
   "t1";"t2" **\crv{(3,14) & (9, 14)};
   "t1'";"t2'" **\crv{(3,-14) & (9, -14)};
   "t2'";"t2" **\dir{-} ?(.5)*\dir{<};}};
   (9,0)*{}; (-7.5,-12)*{\scs i};
 \endxy$} \; = \; -\sum_{f=0}^{-\llambda_i}
   \xy
  (19,4)*{\lambda};
  (0,0)*{\dblue\bbe{}};(-2,-8)*{\scs i};
  (12,-2)*{\dblue\cbub{\black\llambda_i-1+f}{\black i}};
  (0,6)*{\dblue\bullet}+(6,1)*{\scs -\llambda_i-f};
 \endxy
\qquad \quad
  \text{$ \xy 0;/r.18pc/:
  (-12,8)*{\lambda};
   (0,-2)*{\dblue\xybox{
   (-3,-8)*{};(3,8)*{} **\crv{(-3,-1) & (3,1)}?(1)*\dir{>};?(0)*\dir{>};
    (3,-8)*{};(-3,8)*{} **\crv{(3,-1) & (-3,1)}?(1)*\dir{>};
  (3,-12)*{\bbsid};
  (3,8)*{\bbsid}; 
  (-9,8)*{}="t1";
  (-3,8)*{}="t2";
  (-9,-8)*{}="t1'";
  (-3,-8)*{}="t2'";
   "t1";"t2" **\crv{(-9,14) & (-3, 14)};
   "t1'";"t2'" **\crv{(-9,-14) & (-3, -14)};
  "t1'";"t1" **\dir{-} ?(.5)*\dir{<};}};(7.5,-11)*{\scs i};
 \endxy$} \; = \;
 \sum_{g=0}^{\llambda_i}
   \xy
  (-12,8)*{\lambda};
  (0,0)*{\dblue\bbe{}};(2,-8)*{\scs i};
  (-12,-2)*{\dblue\ccbub{\black -\llambda_i-1+g}{\black i}};
  (0,6)*{\dblue\bullet}+(8,-1)*{\scs \llambda_i-g};
 \endxy
\end{eqnarray}
%
%
\begin{eqnarray}
\label{eq:EF}
 \vcenter{\xy 0;/r.18pc/:
  (0,0)*{\dblue\xybox{
  (-8,0)*{};
  (8,0)*{};
  (-4,10)*{}="t1";
  (4,10)*{}="t2";
  (-4,-10)*{}="b1";
  (4,-10)*{}="b2";
  "t1";"b1" **\dir{-} ?(.5)*\dir{<};
  "t2";"b2" **\dir{-} ?(.5)*\dir{>};}};
  (-6,-8)*{\scs i};
  (6,-8)*{\scs i};
  (10,2)*{\lambda};
  (-10,2)*{\lambda};
  \endxy}
&\quad = \quad&
 \vcenter{   \xy 0;/r.18pc/:
    (0,0)*{\dblue\xybox{
    (-4,-4)*{};(4,4)*{} **\crv{(-4,-1) & (4,1)}?(1)*\dir{>};
    (4,-4)*{};(-4,4)*{} **\crv{(4,-1) & (-4,1)}?(1)*\dir{<};?(0)*\dir{<};
    (-4,4)*{};(4,12)*{} **\crv{(-4,7) & (4,9)};
    (4,4)*{};(-4,12)*{} **\crv{(4,7) & (-4,9)}?(1)*\dir{>};}};
    (8,8)*{\lambda};(-6,-7)*{\scs i};(6.8,-7)*{\scs i};
 \endxy}
  \quad - \quad
   \sum_{f=0}^{\llambda_i-1} \sum_{g=0}^{f}
    \vcenter{\xy 0;/r.18pc/:
    (-10,10)*{\lambda};
  (0,0)*{\dblue\xybox{
  (-8,0)*{}; (8,0)*{};
  (-4,-15)*{}="b1";
  (4,-15)*{}="b2";
  "b2";"b1" **\crv{(5,-8) & (-5,-8)}; ?(.05)*\dir{<} ?(.93)*\dir{<}
  ?(.8)*\dir{}+(0,-.1)*{\bullet}+(-5,2)*{\black\scs f-g};
  (-4,15)*{}="t1";
  (4,15)*{}="t2";
  "t2";"t1" **\crv{(5,8) & (-5,8)}; ?(.15)*\dir{>} ?(.95)*\dir{>}
  ?(.4)*\dir{}+(0,-.2)*{\bullet}+(3,-2)*{\black\scs \mspace{38mu}\;\;\; \llambda_i-1-f};
  (0,0)*{\ccbub{\black\scs \quad\;\;\;-\llambda_i-1+g}{i}};}};
  \endxy}
\label{eq_ident_decomp0}
\\[2ex]
 \vcenter{\xy 0;/r.18pc/:
  (0,0)*{\dblue\xybox{
  (-8,0)*{};(8,0)*{};
  (-4,10)*{}="t1";
  (4,10)*{}="t2";
  (-4,-10)*{}="b1";
  (4,-10)*{}="b2";
  "t1";"b1" **\dir{-} ?(.5)*\dir{>};
  "t2";"b2" **\dir{-} ?(.5)*\dir{<};}};
  (-6,-8)*{\scs i};(6,-8)*{\scs i};
  (10,2)*{\lambda};
  (-10,2)*{\lambda};
  \endxy}
&\quad = \quad&
   \vcenter{\xy 0;/r.18pc/:
    (0,0)*{\dblue\xybox{
    (-4,-4)*{};(4,4)*{} **\crv{(-4,-1) & (4,1)}?(1)*\dir{<};?(0)*\dir{<};
    (4,-4)*{};(-4,4)*{} **\crv{(4,-1) & (-4,1)}?(1)*\dir{>};
    (-4,4)*{};(4,12)*{} **\crv{(-4,7) & (4,9)}?(1)*\dir{>};
    (4,4)*{};(-4,12)*{} **\crv{(4,7) & (-4,9)};}};
    (8,8)*{\lambda};(-6.8,-7)*{\scs i};(6,-7)*{\scs i};
 \endxy}
  \quad - \quad
\sum_{f=0}^{-\llambda_i-1} \sum_{g=0}^{f}
    \vcenter{\xy 0;/r.18pc/:
  (0,0)*{\dblue\xybox{ 
  (-8,0)*{}; (8,0)*{};
  (-4,-15)*{}="b1";
  (4,-15)*{}="b2";
  "b2";"b1" **\crv{(5,-8) & (-5,-8)}; ?(.1)*\dir{>} ?(.95)*\dir{>}
  ?(.8)*\dir{}+(0,-.1)*{\bullet}+(-5,2)*{\black\scs f-g};
  (-4,15)*{}="t1";
  (4,15)*{}="t2";
  "t2";"t1" **\crv{(5,8) & (-5,8)}; ?(.15)*\dir{<} ?(.97)*\dir{<}
  ?(.4)*\dir{}+(0,-.2)*{\bullet}+(3,-2)*{\black\scs \mspace{32mu}\;\;-\llambda_i-1-f};
  (0,0)*{\cbub{\black\scs \quad\; \llambda_i-1+g}{i}};}};
  (-10,10)*{\lambda};
  \endxy} \label{eq_ident_decomp}
\end{eqnarray}
for all $\lambda\in \bZ^n$
(see~\eqref{eq_crossl-gen} and~\eqref{eq_crossr-gen} for the definition of sideways
crossings). 
Notice that for some values of 
$\lambda$ the dotted
bubbles appearing above have negative labels. A composite of $\dblue\Uup_{\black\!\! i,\lambda}$
or $\dblue\Udown_{\black i,\lambda}$ with itself a negative number of times does not make
sense. These dotted bubbles with negative labels, called {\em fake bubbles}, are
formal symbols inductively defined by the equation
\begin{equation}
 \makebox[0pt]{ $
\left(\ \xy 0;/r.15pc/:
 (0,0)*{\dblue\xybox{
 (0,0)*{\ccbub{\black\mspace{-32mu}-\llambda_i-1}{\black i}};}};
  (4,8)*{\lambda};
 \endxy
 +
 \xy 0;/r.15pc/:
 (0,0)*{\dblue\xybox{
 (0,0)*{\ccbub{\black\mspace{-12mu}-\llambda_i-1+1}{\black i}};}};
  (4,8)*{\lambda};
 \endxy t
 + \cdots +
 \xy 0;/r.15pc/:
 (0,0)*{\dblue\xybox{
 (0,0)*{\ccbub{\black\mspace{-12mu}-\llambda_i-1+r}{\black i}};}};
  (4,8)*{\lambda};
 \endxy t^{r}
 + \cdots
\right)
%
\left( \xy 0;/r.15pc/:
(0,0)*{\dblue\xybox{
 (0,0)*{\cbub{\black\mspace{-22mu}\llambda_i-1}{\black i}};}};
  (4,8)*{\lambda};
 \endxy
 + \cdots +
 \xy 0;/r.15pc/:
 (0,0)*{\dblue\xybox{
 (0,0)*{\cbub{\black\mspace{-8mu}\llambda_i-1+r}{\black i}};}};
 (4,8)*{\lambda};
 \endxy t^{r}
 + \cdots
\right) =-1 $ } 
\label{eq_infinite_Grass}
\end{equation}
and the additional condition
\[
\xy 0;/r.18pc/:
 (0,0)*{\dblue\cbub{\black -1}{\black i}};
  (4,8)*{\lambda};
 \endxy
 \quad = (-1)^{\laii},
 \qquad
  \xy 0;/r.18pc/:
 (0,0)*{\dblue\ccbub{\black -1}{\black i}};
  (4,8)*{\lambda};
 \endxy
  \quad = (-1)^{\laii-1} \qquad \text{if $\llambda_i =0$.}
\]
Although the labels are negative for fake bubbles, one can check that the overall
degree of each fake bubble is still positive, so that these fake bubbles do not
violate the positivity of dotted bubble axiom. The above equation, called the
infinite Grassmannian relation, remains valid even in high degree when most of
the bubbles involved are not fake bubbles.  See \cite{L1} for more details.

\item NilHecke relations:
 \begin{equation}
  \vcenter{\xy 0;/r.18pc/:
    (0,0)*{\dblue\xybox{
    (-4,-4)*{};(4,4)*{} **\crv{(-4,-1) & (4,1)}?(1)*\dir{>};
    (4,-4)*{};(-4,4)*{} **\crv{(4,-1) & (-4,1)}?(1)*\dir{>};
    (-4,4)*{};(4,12)*{} **\crv{(-4,7) & (4,9)}?(1)*\dir{>};
    (4,4)*{};(-4,12)*{} **\crv{(4,7) & (-4,9)}?(1)*\dir{>};}};
  (8,8)*{\lambda};(-5,-7)*{\scs i};(5.1,-7)*{\scs i};
 \endxy}
 =0, \qquad \quad
 \vcenter{
 \xy 0;/r.18pc/:
 (0,0)*{\dblue\xybox{
    (-4,-4)*{};(4,4)*{} **\crv{(-4,-1) & (4,1)}?(1)*\dir{>};
    (4,-4)*{};(-4,4)*{} **\crv{(4,-1) & (-4,1)}?(1)*\dir{>};
    (4,4)*{};(12,12)*{} **\crv{(4,7) & (12,9)}?(1)*\dir{>};
    (12,4)*{};(4,12)*{} **\crv{(12,7) & (4,9)}?(1)*\dir{>};
    (-4,12)*{};(4,20)*{} **\crv{(-4,15) & (4,17)}?(1)*\dir{>};
    (4,12)*{};(-4,20)*{} **\crv{(4,15) & (-4,17)}?(1)*\dir{>};
    (-4,4)*{}; (-4,12) **\dir{-};
    (12,-4)*{}; (12,4) **\dir{-};
    (12,12)*{}; (12,20) **\dir{-};}}; 
     (-9.5,-11)*{\scs i};(1.5,-11)*{\scs i};(9.5,-11)*{\scs i};
  (12,0)*{\lambda};
\endxy}
 \;\; =\;\;
 \vcenter{
 \xy 0;/r.18pc/:
    (0,0)*{\dblue\xybox{
    (4,-4)*{};(-4,4)*{} **\crv{(4,-1) & (-4,1)}?(1)*\dir{>};
    (-4,-4)*{};(4,4)*{} **\crv{(-4,-1) & (4,1)}?(1)*\dir{>};
    (-4,4)*{};(-12,12)*{} **\crv{(-4,7) & (-12,9)}?(1)*\dir{>};
    (-12,4)*{};(-4,12)*{} **\crv{(-12,7) & (-4,9)}?(1)*\dir{>};
    (4,12)*{};(-4,20)*{} **\crv{(4,15) & (-4,17)}?(1)*\dir{>};
    (-4,12)*{};(4,20)*{} **\crv{(-4,15) & (4,17)}?(1)*\dir{>};
    (4,4)*{}; (4,12) **\dir{-};
    (-12,-4)*{}; (-12,4) **\dir{-};
    (-12,12)*{}; (-12,20) **\dir{-};}};
  (-1.5,-11)*{\scs i};(9.5,-11)*{\scs i};(-9.5,-11)*{\scs i};
  (12,0)*{\lambda};
\endxy} \label{eq_nil_rels}
  \end{equation}
\begin{eqnarray}
  \xy
  (0,1)*{\dblue\xybox{
  (4,4);(4,-4) **\dir{-}?(0)*\dir{<}+(2.3,0)*{};
  (-4,4);(-4,-4) **\dir{-}?(0)*\dir{<}+(2.3,0)*{};}};
  (6,2)*{\lambda};(-6.2,-2)*{\scs i}; (4,-2)*{\scs i};
 \endxy
 \quad =
\xy
  (0,0)*{\dblue\xybox{
    (-4,-4)*{};(4,4)*{} **\crv{(-4,-1) & (4,1)}?(1)*\dir{>}?(.25)*{\bullet};
    (4,-4)*{};(-4,4)*{} **\crv{(4,-1) & (-4,1)}?(1)*\dir{>};
    (-5,-3)*{\black\scs i};
     (5.1,-3)*{\black\scs i};
     (8,1)*{\black\lambda};
     (-10,0)*{};(10,0)*{};
     }};
  \endxy
 \;\; -
 \xy
  (0,0)*{\dblue\xybox{
    (-4,-4)*{};(4,4)*{} **\crv{(-4,-1) & (4,1)}?(1)*\dir{>}?(.75)*{\bullet};
    (4,-4)*{};(-4,4)*{} **\crv{(4,-1) & (-4,1)}?(1)*\dir{>};
    (-5,-3)*{\black\scs i};
     (5.1,-3)*{\black\scs i};
     (8,1)*{\black\lambda};
     (-10,0)*{};(10,0)*{};
     }};
  \endxy
 \;\; =
\xy
  (0,0)*{\dblue\xybox{
    (-4,-4)*{};(4,4)*{} **\crv{(-4,-1) & (4,1)}?(1)*\dir{>};
    (4,-4)*{};(-4,4)*{} **\crv{(4,-1) & (-4,1)}?(1)*\dir{>}?(.75)*{\bullet};
    (-5,-3)*{\black\scs i};
     (5.1,-3)*{\black\scs i};
     (8,1)*{\black\lambda};
     (-10,0)*{};(10,0)*{};
     }};
  \endxy
 \;\; -
  \xy
  (0,0)*{\dblue\xybox{
    (-4,-4)*{};(4,4)*{} **\crv{(-4,-1) & (4,1)}?(1)*\dir{>} ;
    (4,-4)*{};(-4,4)*{} **\crv{(4,-1) & (-4,1)}?(1)*\dir{>}?(.25)*{\bullet};
    (-5,-3)*{\black\scs i};
     (5.1,-3)*{\black\scs i};
     (8,1)*{\black\lambda};
     (-10,0)*{};(10,0)*{};
     }};
  \endxy 
 \label{eq_nil_dotslide}
\end{eqnarray}
We will also include \eqref{eq_cyclic_cross-gen} for $i =j$ as an
$\mf{sl}_2$-relation.
\end{enumerate}

\item For $i \neq j$
\begin{equation} \label{eq_downup_ij-gen}
\text{$
 \vcenter{   \xy 0;/r.18pc/:
(0,0)*{\dblue\xybox{
    (-4,-4)*{};(4,4)*{} **\crv{(-4,-1) & (4,1)}?(1)*\dir{>};
    (4,4)*{};(-4,12)*{} **\crv{(4,7) & (-4,9)}?(1)*\dir{>};
}};
(0,0)*{\dred\xybox{
    (4,-4)*{};(-4,4)*{} **\crv{(4,-1) & (-4,1)}?(1)*\dir{<};?(0)*\dir{<};
    (-4,4)*{};(4,12)*{} **\crv{(-4,7) & (4,9)};
}};
   (7,4)*{\lambda};(-6,-7)*{\scs i};(6.5,-7)*{\scs j};
 \endxy}
 \;\;= \;\;
\xy 0;/r.18pc/:
(3.5,0)*{\dred\xybox{
    (3,9);(3,-9) **\dir{-}?(.55)*\dir{>}+(2.3,0)*{};}};
(-3.5,0)*{\dblue\xybox{
    (-3,9);(-3,-9) **\dir{-}?(.5)*\dir{<}+(2.3,0)*{};}};
    (7,2)*{\lambda};(-6,-6)*{\scs i};(3.8,-6)*{\scs j};
 \endxy
 \qquad
    \vcenter{\xy 0;/r.18pc/:
(0,0)*{\dblue\xybox{
    (-4,-4)*{};(4,4)*{} **\crv{(-4,-1) & (4,1)}?(1)*\dir{<};?(0)*\dir{<};
    (4,4)*{};(-4,12)*{} **\crv{(4,7) & (-4,9)};
}}; 
(0,0)*{\dred\xybox{
    (4,-4)*{};(-4,4)*{} **\crv{(4,-1) & (-4,1)}?(1)*\dir{>};
    (-4,4)*{};(4,12)*{} **\crv{(-4,7) & (4,9)}?(1)*\dir{>};
}};
    (7,4)*{\lambda};(-6.5,-7)*{\scs i};(6,-7)*{\scs j};
 \endxy}
 \;\;=\;\;
\xy 0;/r.18pc/:
(3.5,0)*{\dred\xybox{
    (3,9);(3,-9) **\dir{-}?(.5)*\dir{<}+(2.3,0)*{};
}};
(-3.5,0)*{\dblue\xybox{
    (-3,9);(-3,-9) **\dir{-}?(.55)*\dir{>}+(2.3,0)*{};
}};
    (7,2)*{\lambda};(-6.5,-6)*{\scs i};(3.6,-6)*{\scs j};
 \endxy
$}
\end{equation}

\item The analogue of the $R(\nu)$-relations:
\begin{enumerate}[i)]
\item For $i \neq j$
\begin{eqnarray}
  \vcenter{\xy 0;/r.18pc/:
(0,0)*{\dblue\xybox{
    (-4,-4)*{};(4,4)*{} **\crv{(-4,-1) & (4,1)}?(1)*\dir{>};
    (4,4)*{};(-4,12)*{} **\crv{(4,7) & (-4,9)}?(1)*\dir{>};
}};
(0,0)*{\dred\xybox{
    (4,-4)*{};(-4,4)*{} **\crv{(4,-1) & (-4,1)}?(1)*\dir{>};
    (-4,4)*{};(4,12)*{} **\crv{(-4,7) & (4,9)}?(1)*\dir{>};
}};
    (8,8)*{\lambda};(-5,-6)*{\scs i};
    (5.3,-6)*{\scs j};
 \endxy}
 \qquad = \qquad
 \left\{
 \begin{array}{ccc}
     \xy 0;/r.18pc/:
(4,0)*{\dred\xybox{
  (3,9);(3,-9) **\dir{-}?(.5)*\dir{<}+(2.3,0)*{};
}};
(-2.5,0)*{\dblue\xybox{
  (-3,9);(-3,-9) **\dir{-}?(.5)*\dir{<}+(2.3,0)*{};
}};
  (8,2)*{\lambda};(-5,-6)*{\scs i};(5.1,-6)*{\scs j};
 \endxy &  &  \text{if $i \cdot j=0$,}\\ \\
  (i-j)\left(
  \vcenter{\xy 0;/r.18pc/:
(4.5,0)*{\dred\xybox{
  (3,9);(3,-9) **\dir{-}?(.5)*\dir{<}+(2.3,0)*{};
}};
(-2.5,0)*{\dblue\xybox{
  (-3,9);(-3,-9) **\dir{-}?(.5)*\dir{<}+(2.3,0)*{};
  (-3,4)*{\bullet};
}};
  (8,2)*{\black\lambda}; 
  (-5,-6)*{\bscs i};     (5.1,-6)*{\bscs j};
 \endxy} \quad
 - \quad
 \vcenter{\xy 0;/r.18pc/:
(4,0)*{\dred\xybox{
  (3,9);(3,-9) **\dir{-}?(.5)*\dir{<}+(2.3,0)*{}; (3,4)*{\bullet};
}};
(-2,0)*{\dblue\xybox{
  (-3,9);(-3,-9) **\dir{-}?(.5)*\dir{<}+(2.3,0)*{};
}};
  (9,2)*{\black\lambda};
  (-5,-6)*{\bscs i};     (5.1,-6)*{\bscs j};
 \endxy}\right)
   &  & \text{if $i \cdot j =-1$.}
 \end{array}
 \right. 
\label{eq_r2_ij-gen}
\end{eqnarray}
\n Notice that $(i-j)$ is just a sign, which takes into account the standard 
orientation of the Dynkin diagram.

\begin{eqnarray} \label{eq_dot_slide_ij-gen}
\xy
(0,0)*{\dblue\xybox{
    (-4,-4)*{};(4,4)*{} **\crv{(-4,-1) & (4,1)}?(1)*\dir{>}?(.75)*{\bullet};
}};
(0,0)*{\dred\xybox{
    (4,-4)*{};(-4,4)*{} **\crv{(4,-1) & (-4,1)}?(1)*\dir{>};
}};
    (-5,-3)*{\scs i};
    (5.1,-3)*{\scs j};
    (8,1)*{ \lambda};
    (-10,0)*{};(10,0)*{};
\endxy
 \;\; =
\xy
(0,0)*{\dblue\xybox{
    (-4,-4)*{};(4,4)*{} **\crv{(-4,-1) & (4,1)}?(1)*\dir{>}?(.25)*{\bullet};
}};
(0,0)*{\dred\xybox{
    (4,-4)*{};(-4,4)*{} **\crv{(4,-1) & (-4,1)}?(1)*\dir{>};
}}; 
     (-5,-3)*{\scs i};
     (5.1,-3)*{\scs j};
     (8,1)*{ \lambda};
     (-10,0)*{};(10,0)*{};
\endxy
\qquad  \xy
(0,0)*{\dblue\xybox{
    (-4,-4)*{};(4,4)*{} **\crv{(-4,-1) & (4,1)}?(1)*\dir{>};
}};
(0,0)*{\dred\xybox{
    (4,-4)*{};(-4,4)*{} **\crv{(4,-1) & (-4,1)}?(1)*\dir{>}?(.75)*{\bullet};
}};
    (-5,-3)*{\scs i};
    (5.1,-3)*{\scs j};
    (8,1)*{ \lambda};
    (-10,0)*{};(10,0)*{};
  \endxy
\;\;  =
  \xy
(0,0)*{\dblue\xybox{
    (-4,-4)*{};(4,4)*{} **\crv{(-4,-1) & (4,1)}?(1)*\dir{>} ;
}};
(0,0)*{\dred\xybox{
    (4,-4)*{};(-4,4)*{} **\crv{(4,-1) & (-4,1)}?(1)*\dir{>}?(.25)*{\bullet};
}};
    (-5,-3)*{\scs i};
     (5.1,-3)*{\scs j};
     (8,1)*{ \lambda};
     (-10,0)*{};(12,0)*{};
   \endxy
\end{eqnarray}

\item Unless $i = k$ and $i \cdot j=-1$
\begin{equation}
\text{$
 \vcenter{
 \xy 0;/r.18pc/:
(0,0)*{\dblue\xybox{
    (-4,-4)*{};(4,4)*{} **\crv{(-4,-1) & (4,1)}?(1)*\dir{>};
    (4,4)*{};(12,12)*{} **\crv{(4,7) & (12,9)}?(1)*\dir{>};
    (12,12)*{}; (12,20) **\dir{-};
}};
(-4,0)*{\dred\xybox{
    (4,-4)*{};(-4,4)*{} **\crv{(4,-1) & (-4,1)}?(1)*\dir{>};
    (-4,12)*{};(4,20)*{} **\crv{(-4,15) & (4,17)}?(1)*\dir{>};
    (-4,4)*{}; (-4,12) **\dir{-};
}};
(0,0)*{\dgreen\xybox{
    (12,4)*{};(4,12)*{} **\crv{(12,7) & (4,9)}?(1)*\dir{>};
    (4,12)*{};(-4,20)*{} **\crv{(4,15) & (-4,17)}?(1)*\dir{>};
    (12,-4)*{}; (12,4) **\dir{-};
}};
  (12,0)*{\lambda};
  (-10,-11)*{\scs i};
  (  2,-11)*{\scs j};
  (10.5,-11)*{\scs k};
\endxy}
 \;\; =\;\;
 \vcenter{
 \xy 0;/r.18pc/:
(0,0)*{\dgreen\xybox{
    (4,-4)*{};(-4,4)*{} **\crv{(4,-1) & (-4,1)}?(1)*\dir{>};
    (-4,4)*{};(-12,12)*{} **\crv{(-4,7) & (-12,9)}?(1)*\dir{>};
    (-12,12)*{}; (-12,20) **\dir{-};
}};
(4,0)*{\dred\xybox{
    (-4,-4)*{};(4,4)*{} **\crv{(-4,-1) & (4,1)}?(1)*\dir{>};
    (4,12)*{};(-4,20)*{} **\crv{(4,15) & (-4,17)}?(1)*\dir{>};
    (4,4)*{}; (4,12) **\dir{-};
}};
(0,0)*{\dblue\xybox{
    (-12,4)*{};(-4,12)*{} **\crv{(-12,7) & (-4,9)}?(1)*\dir{>};
    (-4,12)*{};(4,20)*{} **\crv{(-4,15) & (4,17)}?(1)*\dir{>};
    (-12,-4)*{}; (-12,4) **\dir{-};
}};
  (12,0)*{\lambda};
  (10,-11)*{\scs k};
  (-1.5,-11)*{\scs j};
  (-9.5,-11)*{\scs i};
\endxy} \label{eq_r3_easy-gen}
$}
\end{equation}

For $i \cdot j =-1$
\begin{equation}
\text{$
 \vcenter{
 \xy 0;/r.18pc/:
(0,0)*{\dblue\xybox{
    (-4,-4)*{};(4,4)*{} **\crv{(-4,-1) & (4,1)}?(1)*\dir{>};
    (4,4)*{};(12,12)*{} **\crv{(4,7) & (12,9)}?(1)*\dir{>};
    (12,4)*{};(4,12)*{} **\crv{(12,7) & (4,9)}?(1)*\dir{>};
    (4,12)*{};(-4,20)*{} **\crv{(4,15) & (-4,17)}?(1)*\dir{>};
    (12,-4)*{}; (12,4) **\dir{-};
    (12,12)*{}; (12,20) **\dir{-};
}};
(-4,0)*{\dred\xybox{
    (4,-4)*{};(-4,4)*{} **\crv{(4,-1) & (-4,1)}?(1)*\dir{>};
    (-4,12)*{};(4,20)*{} **\crv{(-4,15) & (4,17)}?(1)*\dir{>};
    (-4,4)*{}; (-4,12) **\dir{-};
}};
  (12,0)*{\lambda};
  (-10,-11)*{\scs i};
  (1.5,-11)*{\scs j};
  (9.5,-11)*{\scs i};
\endxy}
\quad - \quad
 \vcenter{
 \xy 0;/r.18pc/:
(0,0)*{\dblue\xybox{
    (4,-4)*{};(-4,4)*{} **\crv{(4,-1) & (-4,1)}?(1)*\dir{>};
    (-4,4)*{};(-12,12)*{} **\crv{(-4,7) & (-12,9)}?(1)*\dir{>};
    (-12,4)*{};(-4,12)*{} **\crv{(-12,7) & (-4,9)}?(1)*\dir{>};
    (-4,12)*{};(4,20)*{} **\crv{(-4,15) & (4,17)}?(1)*\dir{>};
    (-12,-4)*{}; (-12,4) **\dir{-};
    (-12,12)*{}; (-12,20) **\dir{-};
}};
(4,0)*{\dred\xybox{
    (-4,-4)*{};(4,4)*{} **\crv{(-4,-1) & (4,1)}?(1)*\dir{>};
    (4,12)*{};(-4,20)*{} **\crv{(4,15) & (-4,17)}?(1)*\dir{>};
    (4,4)*{}; (4,12) **\dir{-};
}};
  (12,0)*{\lambda};
  (10,-11)*{\scs i};
  (-1.5,-11)*{\scs j};
  (-9.5,-11)*{\scs i};
\endxy}
 \;\; =\;\;\;\; (i-j)
\xy 0;/r.18pc/:
(10,0)*{\dblue\xybox{
  (4,12);(4,-12) **\dir{-}?(.5)*\dir{<};
  (22,12);(22,-12) **\dir{-}?(.5)*\dir{<}?(.25)*\dir{}+(0,0)*{}+(10,0)*{\scs};
}};
(3.5,0)*{\dred\xybox{
  (-4,12);(-4,-12) **\dir{-}?(.5)*\dir{<}?(.25)*\dir{}+(0,0)*{}+(-3,0)*{\scs };
}};
  (20,0)*{\lambda}; 
  (-5.6,-11)*{\scs i};
  (3.1,-11)*{\scs j};
  (15.2,-11)*{\scs i};
 \endxy
 \label{eq_r3_hard-gen}
$}.
\end{equation}
\end{enumerate}

\item The additive $\Z$-linear composition functor $\glcat(\lambda,\lambda')
 \times \glcat(\lambda',\lambda'') \to \glcat(\lambda,\lambda'')$ is given on
 1-morphisms of $\glcat$ by
\begin{equation}
  \cal{E}_{\jj}\mathbf{1}_{\lambda'}\{t'\} \times \cal{E}_{\ii}\onel\{t\} \mapsto
  \cal{E}_{\jj\ii}\onel\{t+t'\}
\end{equation}
for $\ii_{\Lambda}=\lambda-\lambda'$, and on 2-morphisms of $\glcat$ by juxtaposition of
diagrams
\[\text{$
 \left(\figleft{\lambda'}{\lambda''}\right)
\;\; \times \;\;
\left(\figright{\lambda'}{\lambda}\right)
 \;\;\mapsto \;\
\figleft{}{\lambda''}
\figright{}{\lambda}
$} . \]
\end{itemize}
\end{defn}

\noindent This concludes the definition of $\glcat$. In the next subsection we 
will show some further relations, which are easy consequences of the ones 
above. 

\subsubsection{Further relations in $\glcat$}

\medskip

The following $\glcat$-relations follow from the relations in 
Definition~\ref{def_glcat} and are going to be used in the sequel.


\n\emph{Bubble slides}:
\begin{equation}
\label{eq:bub_slides}
\text{$ 
   \xy
  (14,8)*{\lambda};
  (0,0)*{\dgreen\bbe{}};
  (0,-12)*{\scs j};
  (12,-2)*{\dblue\ccbub{\black -\llambda_i-1+m}{\black i}};
  (0,6)*{ }+(7,-1)*{\scs  };
 \endxy
$}
\quad = \quad
 \begin{cases}
 \ \xsum{f=0}{m}(f-m-1)
   \xy
  (0,8)*{\lambda+j_{\Lambda}};
  (12,0)*{\dgreen\bbe{}};
  (12,-12)*{\scs j};
  (0,-2)*{\dblue\ccbub{\black - \overline{(\lambda + j_\Lambda)}_i -1 + f}{\black i}};
  (12,6)*{\dgreen\bullet}+(5,-1)*{\scs m-f};
 \endxy
    &  \text{if $i=j$} 
\\  \\
\ \qquad \qquad  \xy
  (0,8)*{\lambda+j_{\Lambda}};
  (12,0)*{\dgreen\bbe{}};
  (12,-12)*{\scs j};
  (0,-2)*{\dblue\ccbub{\black -\overline{(\lambda + j_\Lambda)_i} -1+m }{\black i}};
 \endxy  &  \text{if $i \cdot j=0$}
 \end{cases}
\end{equation}
\begin{align}
\label{eq:2ndbubbslide}
\text{$ 
   \xy
  (14,8)*{\lambda};
  (0,0)*{\dred\bbe{}};
  (0,-12)*{\scs i+1};
  (12,-2)*{\dblue\ccbub{\black -\llambda_i-1+m}{\black i}};
  (0,6)*{ }+(7,-1)*{\scs  };
 \endxy
$}
&= \quad
   \xy
  (-4,8)*{\lambda+(i+1)_{\Lambda}};
  (12,0)*{\dred\bbe{}};
  (12,-12)*{\scs i+1};
  (-6,-2)*{\dblue\ccbub{\black -(\overline{\lambda +(i+1)_{\Lambda}})_i-2+m}{\black i}};
  (12,6)*{\dred\bullet}+(5,-1)*{\scs };
 \endxy
 \quad - \quad
  \xy
  (-4,8)*{\lambda+(i+1)_{\Lambda}};
  (12,0)*{\red\bbe{}};
  (11,-12)*{\scs i+1};
  (-6,-2)*{\dblue\ccbub{\black -(\overline{\lambda+(i+1)_{\Lambda}})_i -1+m}{\black i}};
 \endxy
\\ \displaybreak[0]
\text{$ 
   \xy
  (6,8)*{\lambda};
  (12,0)*{\dred\bbe{}};
  (12,-12)*{\scs i+1};
  (0,-2)*{\dblue\ccbub{\black -\llambda_i-1+m}{\black i}};
  (-12,6)*{ }+(7,-1)*{\scs  };
 \endxy
$}
&= 
-\sum\limits_{f+g=m}
   \xy
  (18,8)*{\lambda-(i+1)_{\Lambda}};
  (0,0)*{\dred\bbe{}};
  (0,-12)*{\scs i+1};
  (16,-2)*{\dblue\ccbub{\black -(\overline{\lambda -(i+1)_{\Lambda}})_i-2+g}{\black i}};
  (0,6)*{\dred\bullet}+(-3,-1)*{\scs f};
 \endxy
\\\nn\\ 
\label{eq:extrabubble4}
\displaybreak[0]
\text{$ 
   \xy
  (14,8)*{\lambda};
  (0,0)*{\dred\bbe{}};
  (0,-12)*{\scs i+1};
  (12,-2)*{\dblue\cbub{\black \llambda_i-1+m}{\black i}};
  (0,6)*{ }+(7,-1)*{\scs  };
 \endxy
$}
&= -\sum\limits_{f+g=m}\ \
   \xy
  (-4,8)*{\lambda+(i+1)_{\Lambda}};
  (12,0)*{\dred\bbe{}};
  (12,-12)*{\scs i+1};
  (-6,-2)*{\dblue\ccbub{\black (\overline{\lambda +(i+1)_{\Lambda}})_i-1+g}{\black i}};
  (12,6)*{\dred\bullet}+(2,-1)*{\scs f};
 \endxy
\\\nn \\ \displaybreak[0]
\text{$ 
   \xy
  (6,8)*{\lambda};
  (12,0)*{\dred\bbe{}};
  (12,-12)*{\scs i+1};
  (0,-2)*{\dblue\cbub{\black \llambda_i-1+m}{\black i}};
  (-12,6)*{ }+(7,-1)*{\scs  };
 \endxy
$}
&= 
\quad
   \xy
  (18,8)*{\lambda-(i+1)_{\Lambda}};
  (0,0)*{\dred\bbe{}};
  (0,-12)*{\scs i+1};
  (16,-2)*{\dblue\cbub{\black (\overline{\lambda -(i+1)_{\Lambda}})_i-2+m}{\black i}};
  (0,6)*{\dred\bullet}+(-3,-1)*{\scs };
 \endxy
\quad-\quad
   \xy
  (18,8)*{\lambda-(i+1)_{\Lambda}};
  (0,0)*{\dred\bbe{}};
  (0,-12)*{\scs i+1};
  (16,-2)*{\dblue\cbub{\black (\overline{\lambda -(i+1)_{\Lambda}})_i-1+m}{\black i}};
 \endxy
\end{align}
If we switch labels $i$ and $i+1$, then the r.h.s. of the above equations gets 
a minus sign. Bubble slides with the vertical strand oriented downwards 
can easily be obtained from the ones above by rotating the diagrams 180 
degrees. 

\medskip

\n\emph{More Reidemeister 3 like relations}. 
Unless $i=k=j$ we have
\begin{equation} \label{eq_other_r3_1}
\text{$
 \vcenter{
 \xy 0;/r.18pc/:
(0,0)*{\dblue\xybox{
    (-4,-4)*{};(4,4)*{} **\crv{(-4,-1) & (4,1)}?(1)*\dir{>};
    (4,4)*{};(12,12)*{} **\crv{(4,7) & (12,9)}?(1)*\dir{>};
    (12,12)*{}; (12,20) **\dir{-};
}};
(-4,0)*{\dred\xybox{
    (4,-4)*{};(-4,4)*{} **\crv{(4,-1) & (-4,1)}?(0)*\dir{<};
    (-4,12)*{};(4,20)*{} **\crv{(-4,15) & (4,17)}?(0)*\dir{<};
    (-4,4)*{}; (-4,12) **\dir{-};
}};
(0,0)*{\dgreen\xybox{
    (12,4)*{};(4,12)*{} **\crv{(12,7) & (4,9)}?(1)*\dir{>};
    (4,12)*{};(-4,20)*{} **\crv{(4,15) & (-4,17)}?(1)*\dir{>};
    (12,-4)*{}; (12,4) **\dir{-};
}};
  (12,0)*{\lambda};
  ( -10,-11)*{\scs i};
  ( 2.5,-11)*{\scs j};
  (10.5,-11)*{\scs k};
\endxy}
 \;\; =\;\;
 \vcenter{
 \xy 0;/r.18pc/:
(0,0)*{\dgreen\xybox{
    (4,-4)*{};(-4,4)*{} **\crv{(4,-1) & (-4,1)}?(1)*\dir{>};
    (-4,4)*{};(-12,12)*{} **\crv{(-4,7) & (-12,9)}?(1)*\dir{>};
    (-12,12)*{}; (-12,20) **\dir{-};
}};
(4,0)*{\dred\xybox{
    (-4,-4)*{};(4,4)*{} **\crv{(-4,-1) & (4,1)}?(0)*\dir{<};
    (4,12)*{};(-4,20)*{} **\crv{(4,15) & (-4,17)}?(0)*\dir{<};
    (4,4)*{}; (4,12) **\dir{-};
}};
(0,0)*{\dblue\xybox{
    (-12,4)*{};(-4,12)*{} **\crv{(-12,7) & (-4,9)}?(1)*\dir{>};
    (-4,12)*{};(4,20)*{} **\crv{(-4,15) & (4,17)}?(1)*\dir{>};
    (-12,-4)*{}; (-12,4) **\dir{-};
}};
  (12,0)*{\lambda};
  (10,-11)*{\scs k};
  (-2.5,-11)*{\scs j};
  (-9.5,-11)*{\scs i};
\endxy}
$}
\end{equation}
and when $i=j=k$ we have  
\begin{equation} \label{eq_r3_extra}
\text{$
 \vcenter{
 \xy 0;/r.17pc/:
 (0,0)*{\dblue\xybox{
    (-4,-4)*{};(4,4)*{} **\crv{(-4,-1) & (4,1)}?(1)*\dir{>};
    (4,-4)*{};(-4,4)*{} **\crv{(4,-1) & (-4,1)}?(1)*\dir{<};
    ?(0)*\dir{<};
    (4,4)*{};(12,12)*{} **\crv{(4,7) & (12,9)}?(1)*\dir{>};
    (12,4)*{};(4,12)*{} **\crv{(12,7) & (4,9)}?(1)*\dir{>};
    (-4,12)*{};(4,20)*{} **\crv{(-4,15) & (4,17)};
    (4,12)*{};(-4,20)*{} **\crv{(4,15) & (-4,17)}?(1)*\dir{>};
    (-4,4)*{}; (-4,12) **\dir{-};
    (12,-4)*{}; (12,4) **\dir{-};
    (12,12)*{}; (12,20) **\dir{-};
}};
  (12,0)*{\lambda};
  (-10,-11)*{\scs i};
  ( 2.5,-11)*{\scs i};
  ( 9.5,-11)*{\scs i};
\endxy}
-\;
   \vcenter{
 \xy 0;/r.17pc/:
(0,0)*{\dblue\xybox{ 
   (4,-4)*{};(-4,4)*{} **\crv{(4,-1) & (-4,1)}?(1)*\dir{>};
    (-4,-4)*{};(4,4)*{} **\crv{(-4,-1) & (4,1)}?(0)*\dir{<};
    (-4,4)*{};(-12,12)*{} **\crv{(-4,7) & (-12,9)}?(1)*\dir{>};
    (-12,4)*{};(-4,12)*{} **\crv{(-12,7) & (-4,9)}?(1)*\dir{>};
    (4,12)*{};(-4,20)*{} **\crv{(4,15) & (-4,17)};
    (-4,12)*{};(4,20)*{} **\crv{(-4,15) & (4,17)}?(1)*\dir{>};
    (4,4)*{}; (4,12) **\dir{-} ?(.5)*\dir{<};
    (-12,-4)*{}; (-12,4) **\dir{-};
    (-12,12)*{}; (-12,20) **\dir{-};
}};
  (12,0)*{\lambda};
  (10  ,-11)*{\scs i};
  (-2.5,-11)*{\scs i};
  (-9.5,-11)*{\scs i};
\endxy}
  \; = \;
 \sum_{} \; \xy 0;/r.17pc/:
(0,0)*{\dblue\xybox{ 
   (-4,12)*{}="t1";
    (4,12)*{}="t2";
  "t2";"t1" **\crv{(5,5) & (-5,5)}; ?(.15)*\dir{} ?(.9)*\dir{>}
  ?(.2)*\dir{}+(0,-.2)*{\bullet}+(3,-2)*{\bscs f_1};
    (-4,-12)*{}="t1";
    (4,-12)*{}="t2";
  "t2";"t1" **\crv{(5,-5) & (-5,-5)}; ?(.05)*\dir{} ?(.9)*\dir{<}
  ?(.15)*\dir{}+(0,-.2)*{\bullet}+(3,2)*{\bscs f_3};
    (-8.5,0.5)*{\ccbub{\bscs -\llambda_i-3+f_4}{\black i}};
    (13,12)*{};(13,-12)*{} **\dir{-} ?(.5)*\dir{<};
    (13,8)*{\bullet}+(3,2)*{\bscs f_2};
}};
(18,-6)*{\lambda};
  \endxy
+\;
  \sum_{}
\; \xy 0;/r.17pc/: 
(0,0)*{\dblue\xybox{
  (-10,12)*{};(-10,-12)*{} **\dir{-} ?(.5)*\dir{<};
  (-10,8)*{\bullet}+(-3,2)*{\bscs g_2};
  (-4,12)*{}="t1";
  (4,12)*{}="t2";
  "t1";"t2" **\crv{(-4,5) & (4,5)}; ?(.15)*\dir{>} ?(.9)*\dir{>}
  ?(.4)*\dir{}+(0,-.2)*{\bullet}+(3,-2)*{\bscs \;\; g_1};
  (-4,-12)*{}="t1";
  (4,-12)*{}="t2";
  "t2";"t1" **\crv{(4,-5) & (-4,-5)}; ?(.12)*\dir{>} ?(.97)*\dir{>}
  ?(.8)*\dir{}+(0,-.2)*{\bullet}+(1,4)*{\bscs g_3};
  (16.6,-4.5)*{\cbub{\bscs \llambda_i-1+g_4}{\black i}};
}};
  (18,6)*{\lambda};
  \endxy
$}
\end{equation}
where the first sum is over all $f_1, f_2, f_3, f_4 \geq 0$ with
$f_1+f_2+f_3+f_4=\llambda_i$ and the second sum is over all $g_1, g_2,
g_3, g_4 \geq 0$ with $g_1+g_2+g_3+g_4=\llambda_i -2$. Note that the 
first summation is zero if $\llambda_i<0$ and the second is zero 
when $\llambda_i<2$.

Reidemeister 3 like relations for all other orientations are determined from
\eqref{eq_r3_easy-gen}, \eqref{eq_r3_hard-gen}, and the above relations using
duality.


\subsubsection{Enriched $\Hom$ spaces}

For any shift $t$, there are 2-morphisms
\begin{eqnarray}
  \xy
(0,0)*{\dblue\xybox{
 (0,7);(0,-7); **\dir{-} ?(.75)*\dir{>}+(2.3,0)*{\scriptstyle{}};
 (0.1,-2)*{\txt\large{$\bullet$}};
 (6,4)*{ \black\lambda};
 (-10,0)*{};(10,0)*{};(0,-9)*{\bscs i };}};
 \endxy
 \maps  \cal{E}_{+i}\onel\{t\} \To \cal{E}_{+i}\onel\{t-2\}\quad
\xy
(0,0)*{\dblue\xybox{
    (-4,-6)*{};(4,6)*{} **\crv{(-4,-1) & (4,1)}?(1)*\dir{>};
}};
(0,0)*{\dred\xybox{
    (4,-6)*{};(-4,6)*{} **\crv{(4,-1) & (-4,1)}?(1)*\dir{>};
}};
    (-4,-8)*{\bscs i};
     (4,-8)*{\bscs j};
     (8,1)*{ \lambda};
     (-12,0)*{};(12,0)*{};
 \endxy
  \maps \cal{E}_{+i+j}\onel\{t\} \To \cal{E}_{+j+i}\onel\{t-i\cdot j\} \nn \\
\text{$
  \xy
    (0,-3)*{\dblue\bbpef{\black i}};
    (8,-5)*{ \lambda};
    (-12,0)*{};(12,0)*{};
    \endxy \maps
    \onel\{t\} \To \cal{E}_{-i+i}\onel\{t-(1+\llambda_i)\} \quad \xy
    (0,0)*{\dblue\bbcfe{\black i}};
    (7,4)*{ \lambda};
    (-12,0)*{};(12,0)*{};
    \endxy \maps
    \cal{E}_{-i+i}\onel\{t\} \To \onel\{t-(1-\llambda_i)\} 
$}\nn
\end{eqnarray}
in $\glcat$, and the diagrammatic relation
\[\text{$
\vcenter{
 \xy 0;/r.18pc/:
   (0,0)*{\dblue\xybox{
    (-4,-4)*{};(4,4)*{} **\crv{(-4,-1) & (4,1)}?(1)*\dir{>};
    (4,-4)*{};(-4,4)*{} **\crv{(4,-1) & (-4,1)}?(1)*\dir{>};
    (4,4)*{};(12,12)*{} **\crv{(4,7) & (12,9)}?(1)*\dir{>};
    (12,4)*{};(4,12)*{} **\crv{(12,7) & (4,9)}?(1)*\dir{>};
    (-4,12)*{};(4,20)*{} **\crv{(-4,15) & (4,17)}?(1)*\dir{>};
    (4,12)*{};(-4,20)*{} **\crv{(4,15) & (-4,17)}?(1)*\dir{>};
    (-4,4)*{}; (-4,12) **\dir{-};
    (12,-4)*{}; (12,4) **\dir{-};
    (12,12)*{}; (12,20) **\dir{-}; (-5.5,-3)*{\bscs i};
     (5.5,-3)*{\bscs i};(14,-3)*{\bscs i};
  (18,8)*{\black\lambda};}};
\endxy}
 \;\; =\;\;
 \vcenter{
 \xy 0;/r.18pc/:
   (0,0)*{\dblue\xybox{
    (4,-4)*{};(-4,4)*{} **\crv{(4,-1) & (-4,1)}?(1)*\dir{>};
    (-4,-4)*{};(4,4)*{} **\crv{(-4,-1) & (4,1)}?(1)*\dir{>};
    (-4,4)*{};(-12,12)*{} **\crv{(-4,7) & (-12,9)}?(1)*\dir{>};
    (-12,4)*{};(-4,12)*{} **\crv{(-12,7) & (-4,9)}?(1)*\dir{>};
    (4,12)*{};(-4,20)*{} **\crv{(4,15) & (-4,17)}?(1)*\dir{>};
    (-4,12)*{};(4,20)*{} **\crv{(-4,15) & (4,17)}?(1)*\dir{>};
    (4,4)*{}; (4,12) **\dir{-};
    (-12,-4)*{}; (-12,4) **\dir{-};
    (-12,12)*{}; (-12,20) **\dir{-};(-5.5,-3)*{\bscs i};
     (5.5,-3)*{\bscs i};(-14,-3)*{\bscs i};
  (10,8)*{\black\lambda};}};
\endxy}$}
\]
gives rise to relations in
$\glcat\big(\cal{E}_{iii}\onel\{t\},\cal{E}_{iii}\onel\{t+3i\cdot i\}\big)$ for
all $t\in \Z$.


\medskip

Note that for two $1$-morphisms $x$ and $y$ in 
$\glcat$ the 2hom-space $\HomGL(x,y)$ 
only contains 2-morphisms of degree zero and is therefore finite-dimensional. 
Following Khovanov and Lauda we introduce the graded 2hom-space 
$$\HOMGL(x,y)=\oplus_{t\in\Z}\HomGL(x\{t\},y),$$
which is infinite-dimensional. We also define the $2$-category 
$\glcat^*$ which has the same objects and $1$-morphisms as $\glcat$, 
but for two $1$-morphisms $x$ and $y$ the vector space of 2-morphisms is 
defined by 
\begin{equation}
\label{eq:ast}
\glcat^*(x,y)=\HOMGL(x,y).
\end{equation}

%
%
\subsection{The 2-category $\Scat(n,d)$}

Fix $d\in\N_{>0}$. 
As explained in Section~\ref{sec:hecke-schur}, the $q$-Schur algebra $\SD(n,d)$
can be seen as a quotient of $\glcat$ by the ideal generated by 
all idempotents corresponding to the weights that do not belong to 
$\Lambda(n,d)$. 
It is then natural to define the 2-category $\Scat(n,d)$ as a quotient of 
$\glcat$ as follows. 

\begin{defn}
The 2-category $\Scat(n,d)$ is the quotient of $\glcat$ by the ideal 
generated by all 2-morphisms containing
a region with a label not in $\Lambda(n,d)$. 
\end{defn}

We remark that we only put real bubbles, whose interior has a label outside 
$\Lambda(n,d)$, equal to zero. To see what happens to a fake bubble, one 
first has to write it in terms of real bubbles with the opposite orientation 
using the infinite Grassmannian relation~\eqref{eq_infinite_Grass}.

\section{A 2-representation of $\Scat(n,d)$}   
\label{sec:2rep}                               

In this section we define a $2$-functor
$$\fbim: \Scat{(n,d)}^{*}\to {\bim}^{*},$$
where $\bim$ is the graded $2$-category of bimodules over polynomial 
rings with rational 
coefficients. Recall that in the previous 
section (formula~\eqref{eq:ast}), we have defined the $^*$ version of a 
graded $2$-category, as the $2$-category with the same objects
and $1$-morphisms, while the $2$-morphisms between two $1$-morphisms 
can have arbitrary degree.
 
In~\cite{K-L3} Khovanov and Lauda defined a 2-functor $\Gamma^G_d$ from 
$\Ucat$ to a $2$-category equivalent to a sub-2-category of $\bim^*$. 
As one can easily 
verify, $\Gamma^G_d$ kills any diagram with labels outside 
$\Lambda(n,d)$, so it descends to $\Scat(n,d)$. In this section we have 
rewritten this 2-functor, which we denote $\fbim$, in terms of 
categorified MOY-diagrams, because we think it might help some people to 
understand its definition more easily. For further comments see 
Section~\ref{sec:2-functor}.

\subsection{Categorified MOY diagrams}

Before proceeding with the definition of $\fbim$, we first specify our 
notation for MOY diagrams and their categorification.

A colored MOY diagram \cite{M-O-Y}, is an oriented trivalent graph whose edges are 
labeled by natural numbers 
(this label is also called the \emph{color} or the \emph{thickness} of the 
corresponding edge). 
At each trivalent vertex we have at least one incoming and one outgoing edge, and we 
require that at each vertex the sum of the labels of the incoming edges is equal to the 
sum of the labels of the outgoing edges. Moreover, in this paper 
we assume that all edges in MOY diagrams are oriented upwards. \\

To obtain a bimodule corresponding to a given colored MOY diagram, 
we proceed in the following way: 
To each edge labeled $a$, we associate $a$ variables, 
say $\underline{x}=(x_1,\ldots,x_a)$, and to different edges 
we associate different variables. 
At every vertex (like the ones in Figure~\ref{fig:triv}), we impose the relations 
\begin{align*}
e_i(z_1,\ldots,z_{a+b}) 
&= e_i(x_1,\ldots,x_a,y_1,\ldots,y_{b})\\
e_i(z'_1,\ldots,z'_{a+b})
&= e_i(x'_1,\ldots,x'_a,y'_1,\ldots,y'_{b})
\end{align*}  
for all $i\in\{1,\ldots,a+b\}$, where $e_i$ is the $i$th elementary 
symmetric polynomial.
In other words, at every vertex we require that an arbitrary symmetric polynomial 
in the variables corresponding to the incoming edges, is equal to the same 
symmetric polynomial in the variables corresponding to the outgoing edges. \\

\begin{figure}[h]
\begin{equation*}
\labellist
\tiny\hair 2pt
\pinlabel $a$ at 16 93  \pinlabel $b$ at 100 95  \pinlabel $a+b$ at 90 42
\pinlabel $x_1,\dotsc,x_a$ at -8 155 \pinlabel $y_1,\dotsc,y_b$ at 128 155
\pinlabel $z_1,\dotsc,z_{a+b}$ at 68 -10
\endlabellist
\figins{0}{0.7}{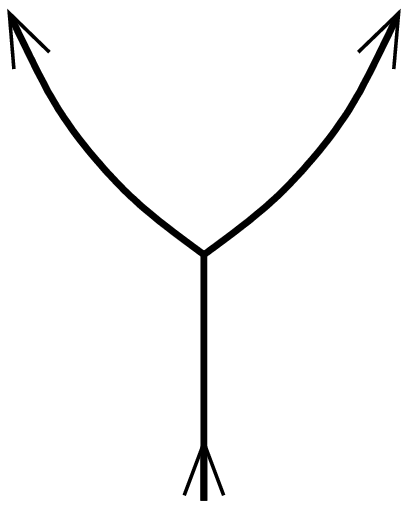}
\mspace{140mu}
\labellist
\tiny\hair 2pt
\pinlabel $a$ at 12 42  \pinlabel $b$ at 114 44  \pinlabel $a+b$ at 94 95
\pinlabel $x'_1,\dotsc,x'_a$ at -8 -10 \pinlabel $y'_1,\dotsc,y'_b$ at 128 -10
\pinlabel $z'_1,\dotsc,z'_{a+b}$ at 68 155
\endlabellist
\figins{0}{0.7}{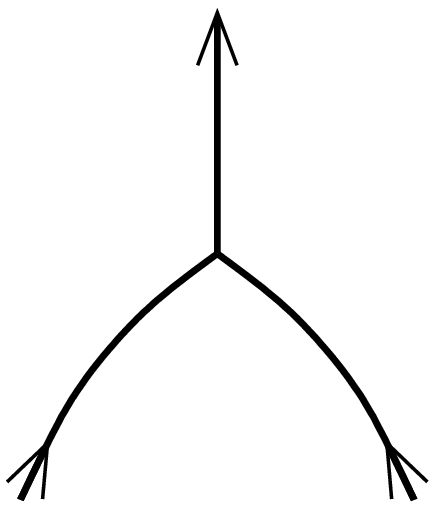}
\end{equation*}
\caption{trivalent vertices}\label{fig:triv}
\end{figure}

Now, to an arbitrary diagram $\Gamma$, we associate the ring $R_{\Gamma}$ of 
polynomials over $\bQ$ which are symmetric in the variables on each strand 
separately, modded out by the relations corresponding to 
all trivalent vertices. 

In particular, to a graph without trivalent vertices (just strands):
\begin{equation*}
\labellist
\tiny\hair 2pt
\pinlabel $c$ at -10 189    \pinlabel $b$ at 220 191    \pinlabel $a$ at 330 189 
\pinlabel $\dotsc$ at 125 100 
\pinlabel $\underline{z}$ at   7 -20    
\pinlabel $\underline{y}$ at 235 -20
\pinlabel $\underline{x}$ at 349 -20
\endlabellist
\figins{-21}{0.7}{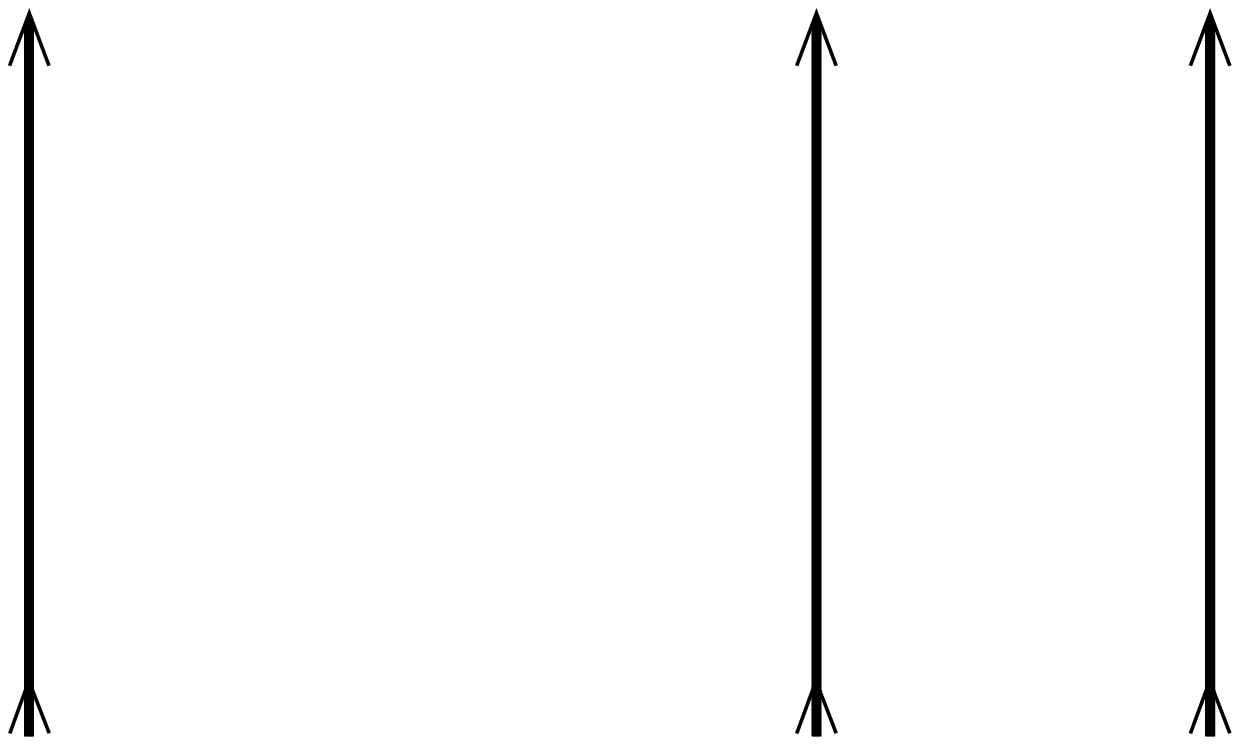}\vspace*{2ex}
\end{equation*}

\noindent we associate the ring of partially symmetric polynomials  
$\bQ[\underline{x},\underline{y},\ldots,\underline{z}]^{S_a\times S_b\times\cdots\times S_c}$.

In this way, the ring $R_{\Gamma}$ associated to a MOY diagram $\Gamma$, 
is a bimodule over the rings of partially symmetric polynomials associated to 
the top (right action) and bottom end (left action) strands, respectively 
(remember that we are assuming that all MOY diagrams are oriented upwards, 
so they have a top and a bottom end). 
Bimodules are graded by setting the degree of any variable equal to 2.

In the rest of the paper, we will often identify the MOY diagram and the 
corresponding bimodule.
Also, by abuse of notation, we shall call the elements of the bimodule $R_{\Gamma}$ 
polynomials.\\

There is another way to describe these bimodules associated to MOY diagrams 
(see e.g.~\cite{Kh,M-S-V,Will}). Fix the polynomial ring 
$R:=\mathbb{Q}[x_1,\ldots,x_d]$. For any $(a_1,\ldots,a_n)\in\Lambda(n,d)$, 
let $R^{a_1,\ldots,a_n}$ be the sub-ring of polynomials which are invariant 
under $S_{a_1}\times\cdots\times S_{a_n}$. To the first diagram in 
Figure~\ref{fig:triv} one associates the $R^{a+b}-R^{a,b}$-bimodule
$$\mbox{Res}_{R^{a,b}}^{R^{a+b}}R^{a,b},$$
where one simply restricts the left action on $R^{a,b}$ to 
$R^{a+b}\subseteq R^{a,b}$.
To the second diagram in Figure~\ref{fig:triv} 
one associates the $R^{a,b}-R^{a+b}$-bimodule
$$\mbox{Ind}_{R^{a+b}}^{R^{a,b}}R^{a+b}:=R^{a,b}\otimes_{R^{a+b}}R^{a+b}.$$
In this way, to every MOY-diagram $\Gamma$ one 
associates a tensor product of bimodules, which is isomorphic to 
the bimodule $R_{\Gamma}$ that we described in the paragraph above. 

In this paper we always use $R_{\Gamma}$, since it is computationally 
easier to use 
polynomials than to use tensor products of polynomials.

\subsection{Definition of $\fbim$}
Now we can proceed with the definition of $\fbim \colon \Scat(n,d)^*\to \bim^*$. 

Let $z_1,\ldots,z_d$ be variables. For convenience we shall use 
Khovanov and Lauda's 
notation $k_i=\lambda_1+\cdots+\lambda_i$, for $i=1,\ldots,n$. 

On objects $\lambda\in\Lambda(n,d)$, the 2-functor $\fbim$ is given by:
$$\lambda=(\lambda_1,\cdots,\lambda_n)
\mapsto \bQ[z_1,\ldots ,z_d]^{S_{\lambda_1}\times\cdots\times S_{\lambda_n}}.$$

On $1$-morphisms we define $\fbim$ as follows:
$$1_{\lambda}\{t\}\mapsto \bQ[z_1,\ldots,z_d]^{S_{\lambda_1}\times\cdots\times S_{\lambda_n}}\{t\}.$$

In terms of MOY diagrams this is presented by:
\begin{equation*}
1_{\lambda}\{t\}\quad \longmapsto\qquad
\labellist
\tiny\hair 2pt
\pinlabel $\lambda_n$ at -17 189    
\pinlabel $\lambda_2$ at 210 189    
\pinlabel $\lambda_1$ at 322 189 
\pinlabel $\dotsc$ at 125 100 
\endlabellist
\figins{-21}{0.7}{IDweb}
\end{equation*}

Note that we are drawing the entries of $\lambda$ from right to left, which 
is compatible with Khovanov and Lauda's convention. 

The remaining generating $1$-morphisms are mapped as follows:
\begin{align*}
\cal{E}_{+i}\onel\{t\}\ \
&\mapsto
\qquad
\labellist
\tiny\hair 2pt
\pinlabel $\lambda_n$ at -22 193  
\pinlabel $\laii$ at 150 193 
\pinlabel $\lai$ at 275 193       
\pinlabel $\lambda_1$ at 460 193 
\pinlabel $1$ at 243 74
\pinlabel $\laii-1$ at 150 -15
\pinlabel $\lai+1$ at 310 -15 
\pinlabel $\dotsc$ at 120 100 
\pinlabel $\dotsc$ at 400 100   
\endlabellist
\figins{-23}{0.7}{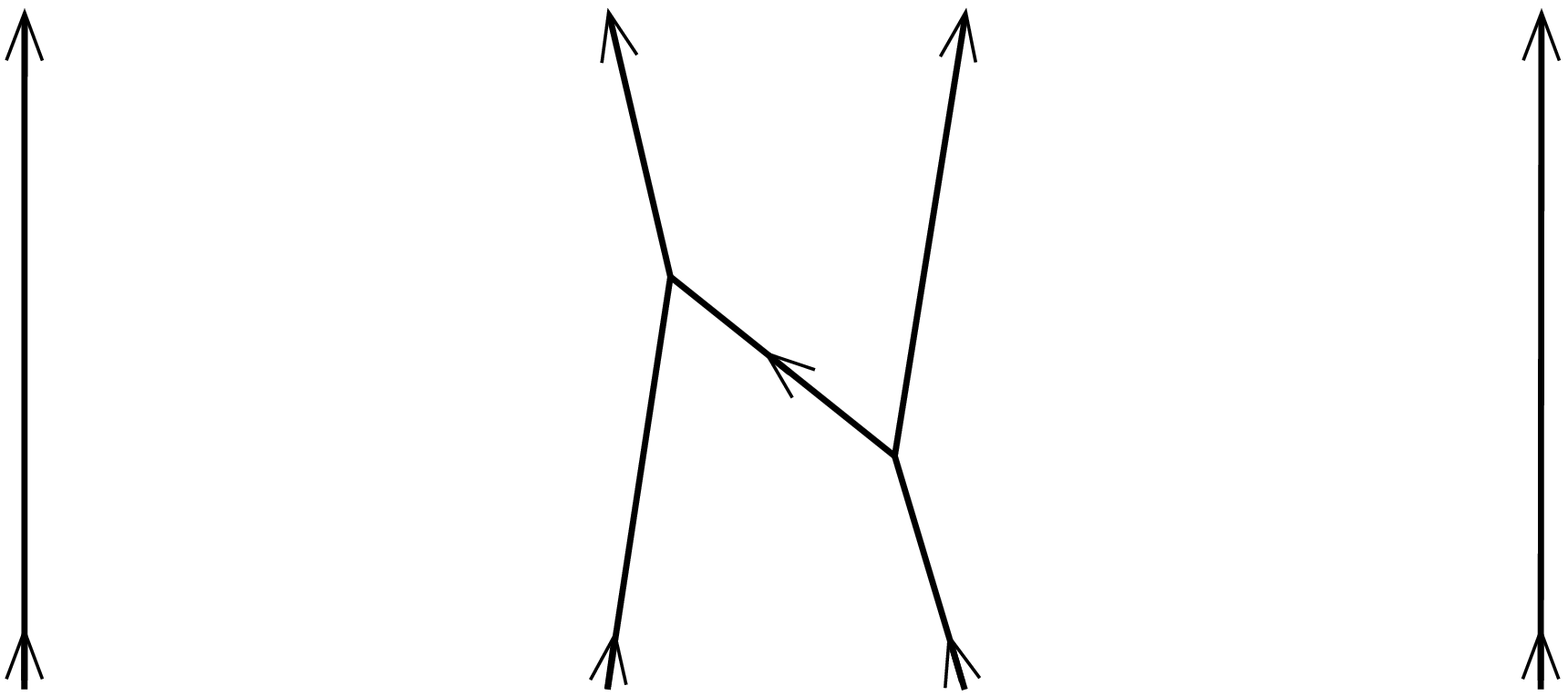}\quad \{t+1+k_{i-1}+k_i-k_{i+1}\}
\displaybreak[0] \\[4ex]
\cal{E}_{-i}\onel\{t\}\ \
&\mapsto
\qquad
\labellist
\tiny\hair 2pt
\pinlabel $\lambda_n$ at -22 193  
\pinlabel $\laii$ at 150 193 
\pinlabel $\lai$ at 275 193       
\pinlabel $\lambda_1$ at 460 193 
\pinlabel $1$ at 243 74
\pinlabel $\laii+1$ at 150 -15
\pinlabel $\lai-1$ at 310 -15 
\pinlabel $\dotsc$ at 120 100 
\pinlabel $\dotsc$ at 400 100  
\endlabellist
\figins{-23}{0.7}{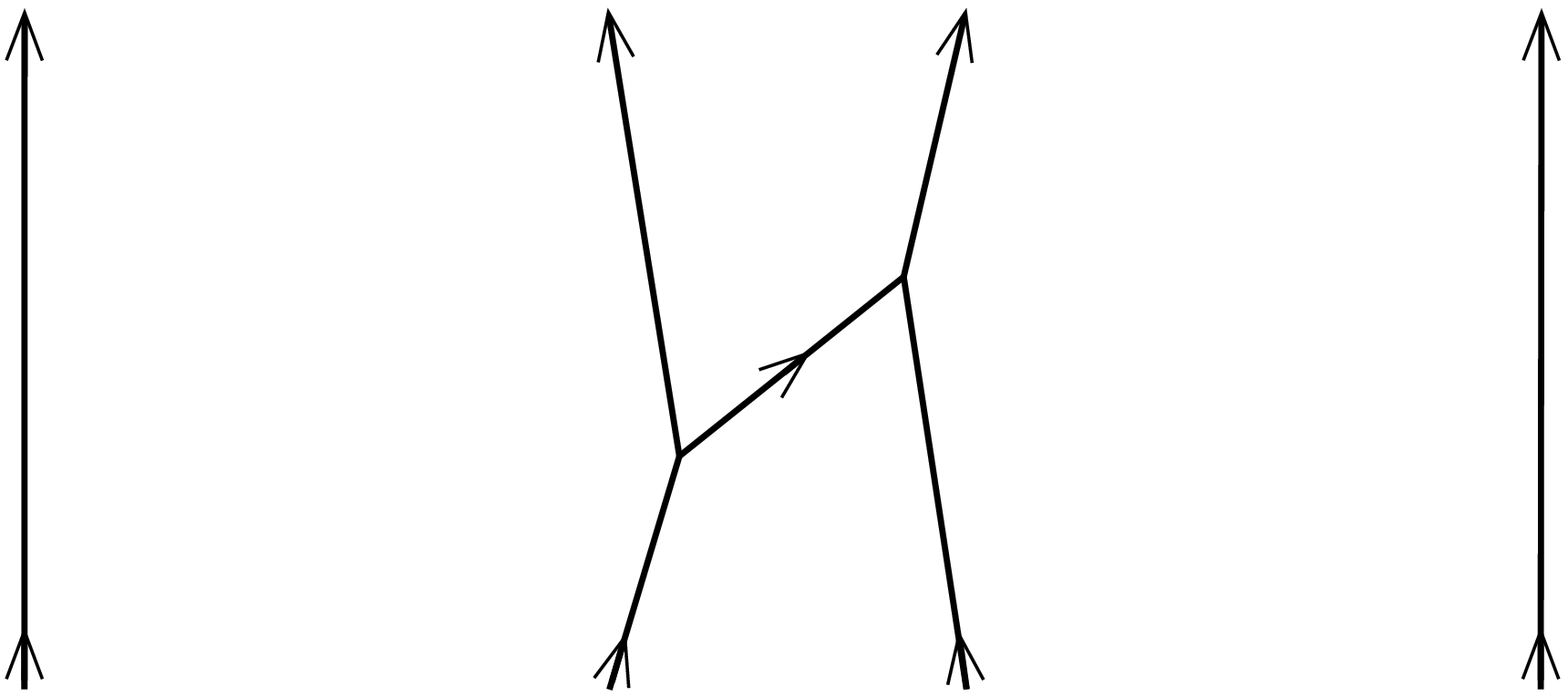}\quad \{t+1-k_i\}
\end{align*}
\medskip

In both cases, the partition corresponding to the bottom strands is 
$\lambda+j_{\Lambda}$
(with $j$ being $+i$ or $-i$).  Thus, the condition we imposed on $\Scat(n,d)$ 
that all regions have labels from $\Lambda(n,d)$ 
(i.e. no region can have labels with negative entries), ensures that 
on the RHS above we really have MOY diagrams.

The composite $\fbim(\mathcal{E}_i1_{\lambda+j_{\Lambda}}\mathcal{E}_j1_{\lambda})$ is given by 
stacking the MOY diagram corresponding to $\mathcal{E}_j1_{\lambda}$ on top of the one 
corresponding to $\mathcal{E}_i1_{\lambda+j_{\Lambda}}$. The shifts add under composition.\\

To define $\fbim$ on $2$-morphisms, we give the image of the 
generating $2$-morphisms. In the definitions the divided difference operator $\partial_{xy}$ is used. 
For $p\in Q[x,y,\ldots]$ it is given by
\begin{equation}
\partial_{xy}p=\frac{p-p_{\mid x\leftrightarrow y}}{x-y},
\label{novo}
\end{equation}
where 
$p_{\mid x\leftrightarrow y}$ is the polynomial obtained from $p$ by swapping 
the variables $x$ and $y$. 
Moreover, for $\underline{x}=(x_1,\ldots,x_a)$, we use the shorthand notation
\begin{align}
\partial_{\underline{x}y}
&= \partial_{x_1y}\partial_{x_2y}\cdots\partial_{x_ay}\\
\partial_{y\underline{x}}
&= \partial_{yx_1}\partial_{yx_2}\cdots\partial_{yx_a}.
\end{align}
\vskip 0.3cm

Before listing the definition of $\fbim$, we explain the notation we are using. 
We denote a bimodule map as a  pair, the first term showing the corresponding MOY 
diagrams (of the source and target 1-morphism), and the second being an explicit 
formula of the map in terms of the (classes of) polynomials that are the elements of 
the corresponding rings. In a few cases we have added an intermediate 
MOY-diagram, in order to clarify the definition. Finally, in order to simplify the pictures, 
in each formula we only draw 
the strands 
that are affected, while on the others we just set the identity. Also in every line we 
require that the polynomial rings corresponding to the top (respectively bottom) end 
strands are the same throughout the movie. Furthermore, we only write 
explicitly the variables of the strands that are relevant in the definition of the corresponding 
bimodule map.

\begin{align*}
  \xy
(0,-1.5)*{\dblue\xybox{
 (0,7);(0,-7); **\dir{-} ?(1)*\dir{>};
 (4,4)*{ \bscs \lambda};
 (0,-9)*{\bscs i };}};
 \endxy
\quad
& \mapsto
\quad
\id
\left(\ \
\labellist
\tiny\hair 2pt
\pinlabel $\laii$ at -30 183  \pinlabel $\lai$ at 150 185  \pinlabel $1$ at 62 74
\endlabellist
\quad
\figins{-19}{0.6}{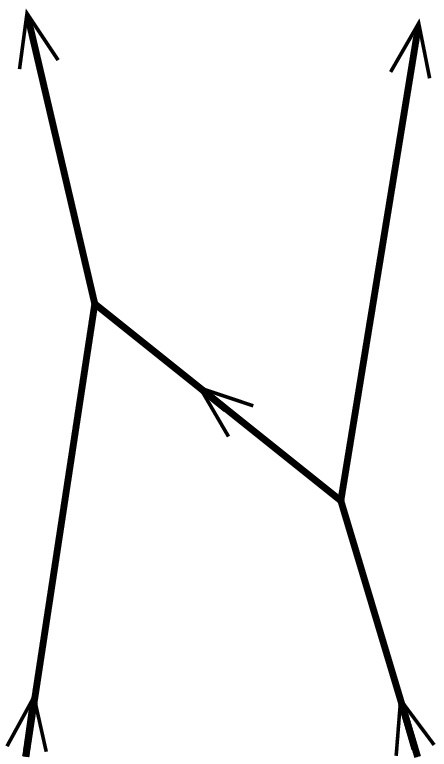}\ \quad
\right)
\displaybreak[0] 
\\[1.5ex]
  \xy
(0,-1.5)*{\dblue\xybox{
 (0,7);(0,-7); **\dir{-} ?(1)*\dir{>};
 (0.1,-2)*{\txt\large{$\bullet$}}+(2.3,1)*{\bscs r};
 (4,4)*{ \bscs \lambda};
 (0,-9)*{\bscs i };}};
 \endxy
\quad
& \mapsto 
\quad
\left(\quad\ \
\labellist
\tiny\hair 2pt
\pinlabel $\laii$ at   -30 183 \pinlabel $\lai$ at 150 185  \pinlabel $1$ at 62 74
\pinlabel $\color[rgb]{.5,.5,.5}{\downarrow}$ at 56 134 \pinlabel $x$ at 56 162
\endlabellist
\figins{-19}{0.6}{Hwebl}
\qra
\labellist
\tiny\hair 2pt
\pinlabel $\laii$ at   -30 183  
\pinlabel $\lai$ at 150 185  \pinlabel $1$ at 62 74
\endlabellist
\figins{-19}{0.6}{Hwebl}\ \
,\quad
p  \mapsto x^rp
\right)
\displaybreak[0]
\\[1.5ex]
  \xy
(0,-1.5)*{\dblue\xybox{
 (0,7);(0,-7); **\dir{-} ?(0)*\dir{<};
 (4,4)*{ \bscs \lambda};
 (0,-9)*{\bscs i };}};
 \endxy
\quad
& \mapsto
\quad
\id
\left(\quad\ \
\labellist
\tiny\hair 2pt
\pinlabel $\laii$ at  -33 183 \pinlabel $\lai$ at 146 185  \pinlabel $1$ at 66 74
\endlabellist
\figins{-19}{0.6}{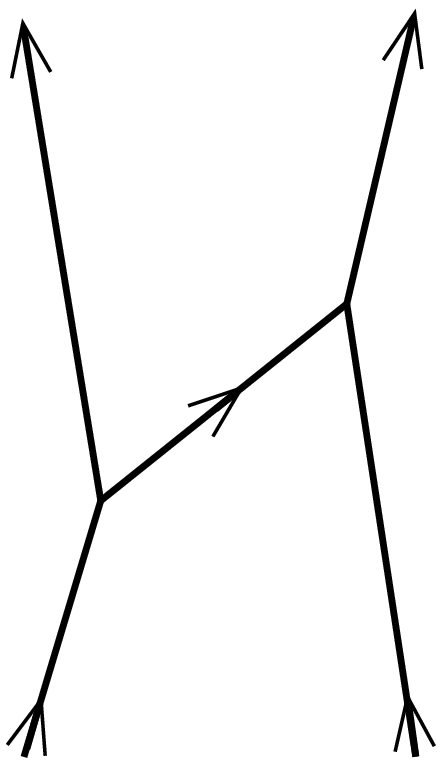}\ \quad
\right)
\displaybreak[0]
\\[1.5ex]
  \xy
(0,-1.5)*{\dblue\xybox{
 (0,7);(0,-7); **\dir{-} ?(0)*\dir{<};
 (0.1,0)*{\txt\large{$\bullet$}}+(2.3,-1)*{\bscs r};
 (4,4)*{ \bscs \lambda};
 (0,-9)*{\bscs i };}};
 \endxy
\quad
& \mapsto 
\quad
\left(\quad\ \
\labellist
\tiny\hair 2pt
\pinlabel  $\laii$ at  -33 183 \pinlabel $\lai$ at 146 185  \pinlabel $1$ at 66 74
\pinlabel $\color[rgb]{.5,.5,.5}{\downarrow}$ at 70 134 \pinlabel $x$ at 70 162
\endlabellist
\figins{-19}{0.6}{Hwebr}\ \
\qra
\labellist
\tiny\hair 2pt
\pinlabel $\laii$ at  -33 183 \pinlabel $\lai$ at 146 185  \pinlabel $1$ at 66 74
\endlabellist
\figins{-19}{0.6}{Hwebr}\ \
,\quad
p\mapsto x^rp
\right)
\end{align*}
\begin{align*}
\xy
(0,0)*{\dblue\xybox{
    (-4,-4)*{};(4,4)*{} **\crv{(-4,-1) & (4,1)}?(1)*\dir{>}; 
    (4,-4)*{};(-4,4)*{} **\crv{(4,-1) & (-4,1)}?(1)*\dir{>};
}};
    (-5,-3)*{\scs i};
    (5.1,-3)*{\scs i};
    (7,1)*{\scs \lambda};
\endxy
\quad
& \mapsto 
\quad
\left(\quad
\ \
\labellist
\tiny\hair 2pt
\pinlabel $\laii$ at -35 193 \pinlabel $\lai$ at 138 195
\pinlabel $\color[rgb]{.5,.5,.5}{\downarrow}$ at 50 169 \pinlabel $x_2$ at 50 197
\pinlabel $\color[rgb]{.5,.5,.5}{\uparrow}$ at 50 52 \pinlabel $x_1$ at 50 20
\endlabellist
\figins{-19}{0.6}{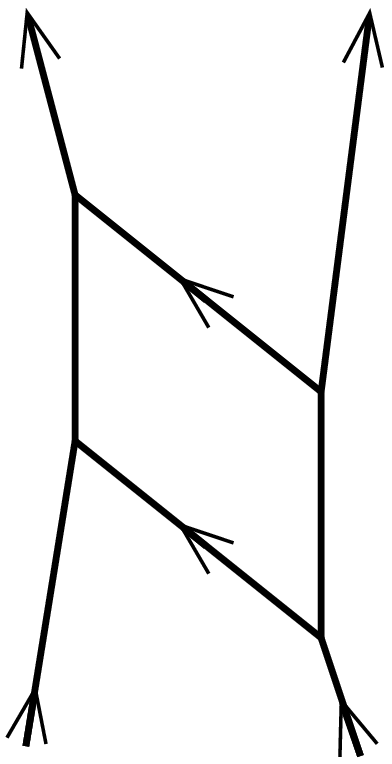}
\qra
\labellist
\tiny\hair 2pt
\pinlabel  $\laii$ at  -35 193 \pinlabel $\lai$ at 146 195  \pinlabel $2$ at 62 74
\endlabellist
\figins{-19}{0.6}{Hwebl}
\qra
\labellist
\tiny\hair 2pt
\pinlabel $\laii$ at -35 193 \pinlabel $\lai$ at 140 195
\pinlabel $\color[rgb]{.5,.5,.5}{\downarrow}$ at 50 169 \pinlabel $x_1$ at 50 197
\pinlabel $\color[rgb]{.5,.5,.5}{\uparrow}$ at 50 52 \pinlabel $x_2$ at 50 20
\endlabellist
\figins{-19}{0.6}{HHwebll}
\ \
, \quad
p\mapsto \partial_{x_1x_2}p
\right)
\displaybreak[0]
\\[1.5ex]
\xy
(0,0)*{\dblue\xybox{
    (-4,-4)*{};(4,4)*{} **\crv{(-4,-1) & (4,1)}?(0)*\dir{<}; 
    (4,-4)*{};(-4,4)*{} **\crv{(4,-1) & (-4,1)}?(0)*\dir{<};
}};
    (-6,-3)*{\scs i};
    (6.1,-3)*{\scs i};
    (7,1)*{\scs\lambda};
\endxy
\quad
&\mapsto
\quad
\left(\quad
\ \
\labellist
\tiny\hair 2pt
\pinlabel $\laii$ at -36 193 \pinlabel $\lai$ at 130 195
\pinlabel $\color[rgb]{.5,.5,.5}{\downarrow}$ at 65 171 \pinlabel $x_1$ at 65 199
\pinlabel $\color[rgb]{.5,.5,.5}{\uparrow}$ at 65 52 \pinlabel $x_2$ at 65 20
\endlabellist
\figins{-19}{0.6}{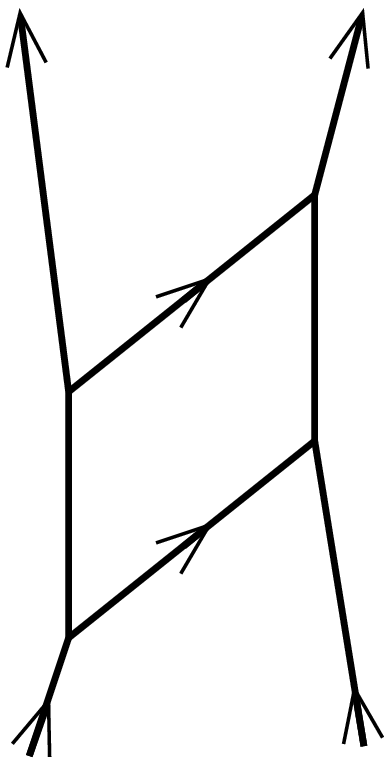}
\qra
\labellist
\tiny\hair 2pt
\pinlabel  $\laii$ at  -35 193 \pinlabel $\lai$ at 142 195  \pinlabel $2$ at 66 74
\endlabellist
\figins{-19}{0.6}{Hwebr}
\qra
\labellist
\tiny\hair 2pt
\pinlabel $\laii$ at -35 193 \pinlabel $\lai$ at 130 195
\pinlabel $\color[rgb]{.5,.5,.5}{\downarrow}$ at 65 171 \pinlabel $x_2$ at 65 199
\pinlabel $\color[rgb]{.5,.5,.5}{\uparrow}$ at 65 52 \pinlabel $x_1$ at 65 20
\endlabellist
\figins{-19}{0.6}{HHwebrr}
\ \
, \quad
p\mapsto \partial_{x_1x_2}p
\right)
\displaybreak[0]
\\[1.5ex]
\xy
(0,0)*{\dblue\xybox{
    (-4,-4)*{};(4,4)*{} **\crv{(-4,-1) & (4,1)}?(1)*\dir{>}; 
    (4,-4)*{};(-4,4)*{} **\crv{(4,-1) & (-4,1)}?(0)*\dir{<};
}};
    (-6,-3)*{\scs i};
    (6.1,-3)*{\scs i};
    (7,1)*{\scs\lambda};
\endxy
\quad
& \mapsto
\quad
\left(\quad
\ \
\labellist
\tiny\hair 2pt
\pinlabel $\laii$ at -36 193 \pinlabel $\lai$ at 130 195
\pinlabel $\color[rgb]{.5,.5,.5}{\downarrow}$ at 65 176 \pinlabel $x_1$ at 65 204
\pinlabel $\color[rgb]{.5,.5,.5}{\uparrow}$ at 65 42 \pinlabel $y$ at 65 10
\endlabellist
\figins{-19}{0.6}{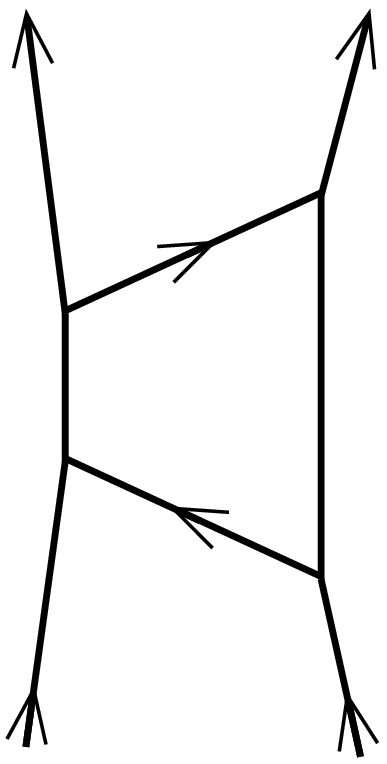}
\qra
\labellist
\tiny\hair 2pt
\pinlabel $\laii$ at -35 193 \pinlabel $\lai$ at 130 195
\pinlabel $\color[rgb]{.5,.5,.5}{\downarrow}$ at 70 171 \pinlabel $y$ at 70 201
\pinlabel $\color[rgb]{.5,.5,.5}{\uparrow}$ at 70 52 \pinlabel $x_2$ at 70 20
\endlabellist
\figins{-19}{0.6}{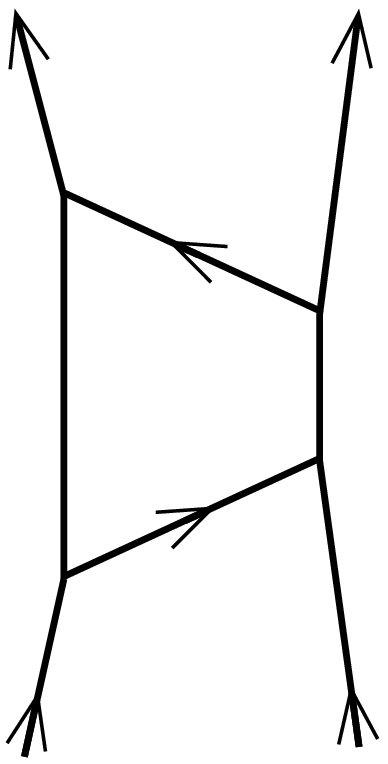}
\ \
, \quad
p\mapsto p\vert_{x_1\mapsto x_2}
\right)
\displaybreak[0]
\\[1.5ex]
\xy
(0,0)*{\dblue\xybox{
    (-4,-4)*{};(4,4)*{} **\crv{(-4,-1) & (4,1)}?(0)*\dir{<}; 
    (4,-4)*{};(-4,4)*{} **\crv{(4,-1) & (-4,1)}?(1)*\dir{>};
}};
    (-6,-3)*{\scs i};
    (6.1,-3)*{\scs i};
    (7,1)*{\scs\lambda};
\endxy
\quad
& \mapsto
\quad
\left(\quad
\ \
\labellist
\tiny\hair 2pt
\pinlabel $\laii$ at -36 193 \pinlabel $\lai$ at 130 195
\pinlabel $\color[rgb]{.5,.5,.5}{\downarrow}$ at 65 176 \pinlabel $x_1$ at 65 204
\pinlabel $\color[rgb]{.5,.5,.5}{\uparrow}$ at 65 42 \pinlabel $y$ at 65 10
\endlabellist
\figins{-19}{0.6}{sqwebrl}
\qra
\labellist
\tiny\hair 2pt
\pinlabel $\laii$ at -35 193 \pinlabel $\lai$ at 130 195
\pinlabel $\color[rgb]{.5,.5,.5}{\downarrow}$ at 65 176 \pinlabel $y$ at 65 206
\pinlabel $\color[rgb]{.5,.5,.5}{\uparrow}$ at 65 42 \pinlabel $x_2$ at 65 10
\endlabellist
\figins{-19}{0.6}{sqweblr}
\ \
, \quad
p\mapsto p\vert_{x_1\mapsto x_2}
\right)
\end{align*}

\medskip

\begin{align}
\xy
(0,0)*{\dblue\xybox{
    (-4,-4)*{};(4,4)*{} **\crv{(-4,-1) & (4,1)}?(1)*\dir{>} }};
(0,0)*{\dred\xybox{
    (4,-4)*{};(-4,4)*{} **\crv{(4,-1) & (-4,1)}?(1)*\dir{>};
}};
    (-5,-3)*{\scs i};
    (6.9,-3)*{\scs i+1};
(9,1)*{ \scs\lambda};
\endxy
\quad
&
\mapsto 
\quad
\left(\ \ 
\quad
\labellist
\tiny\hair 2pt
\pinlabel $\laiii$ at -25 189 \pinlabel $\laii$ at 170 191 \pinlabel $\lai$ at 255 189 
\pinlabel $1$ at 66 135 \pinlabel $1$ at 165 45
\endlabellist
\figins{-19}{0.6}{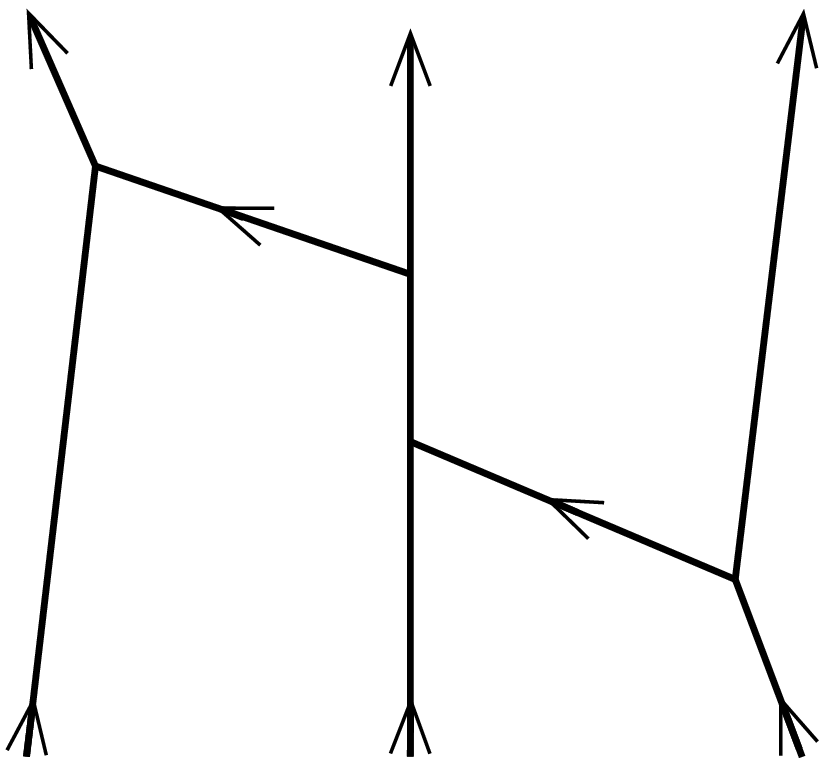}
\qra
\labellist
\tiny\hair 2pt
\pinlabel $\laiii$ at -35 189 \pinlabel $\laii$ at 70 191 \pinlabel $\lai$ at 255 189 
\pinlabel $1$ at 166 125 \pinlabel $1$ at 66 50
\endlabellist
\figins{-19}{0.6}{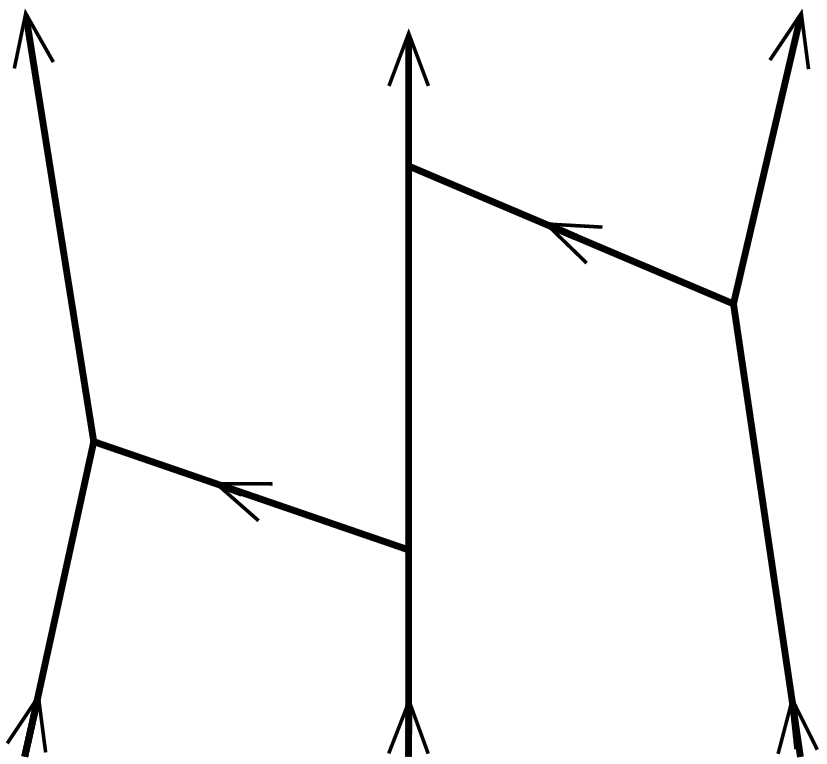}\ \
,\ \
p\mapsto p
\right)
\nn\displaybreak[0]
\\[1.5ex]
\xy
(0,0)*{\dred\xybox{
    (-4,-4)*{};(4,4)*{} **\crv{(-4,-1) & (4,1)}?(1)*\dir{>} }};
(0,0)*{\dblue\xybox{
    (4,-4)*{};(-4,4)*{} **\crv{(4,-1) & (-4,1)}?(1)*\dir{>};
}};
    (-6.9,-3)*{\scs i+1};
    (5.1,-3)*{\scs i};
(9,1)*{ \scs\lambda};
\endxy
\quad
&
\mapsto
\quad
\left(\ \ 
\quad
\labellist
\tiny\hair 2pt
\pinlabel $\laiii$ at -35 189 \pinlabel $\laii$ at 70 191 \pinlabel $\lai$ at 250 189 
\pinlabel $1$ at 166 125 \pinlabel $1$ at 66 50
\pinlabel $\color[rgb]{.5,.5,.5}{\downarrow}$ at 50 114 \pinlabel $x$ at 50 142
\pinlabel $\color[rgb]{.5,.5,.5}{\downarrow}$ at 150 184 \pinlabel $y$ at 150 214
\endlabellist
\figins{-19}{0.6}{hHwebl}
\qra
\labellist
\tiny\hair 2pt
\pinlabel $\laiii$ at -35 189 \pinlabel $\laii$ at 165 191 \pinlabel $\lai$ at 255 189 
\pinlabel $1$ at 66 135 \pinlabel $1$ at 165 45
\pinlabel $\color[rgb]{.5,.5,.5}{\downarrow}$ at 50 186 \pinlabel $x$ at 50 214
\pinlabel $\color[rgb]{.5,.5,.5}{\downarrow}$ at 146 108 \pinlabel $y$ at 146 138
\endlabellist
\figins{-19}{0.6}{Hhwebl}\ \
,\ \
p\mapsto (x-y)p
\right)
\nn\displaybreak[0]
\\[1.5ex]
\xy
(0,0)*{\dred\xybox{
    (-4,-4)*{};(4,4)*{} **\crv{(-4,-1) & (4,1)}?(0)*\dir{<} }};
(0,0)*{\dblue\xybox{
    (4,-4)*{};(-4,4)*{} **\crv{(4,-1) & (-4,1)}?(0)*\dir{<};
}};
    (-7.9,-3)*{\scs i+1};
    (6.1,-3)*{\scs i};
(9,1)*{ \scs\lambda};
\endxy
\quad
&
\mapsto
\quad
\left(\ \ 
\quad
\labellist
\tiny\hair 2pt
\pinlabel $\laiii$ at -35 189 \pinlabel $\laii$ at 70 191 \pinlabel $\lai$ at 250 189 
\pinlabel $1$ at 166 125 \pinlabel $1$ at 66 40
\endlabellist
\figins{-19}{0.6}{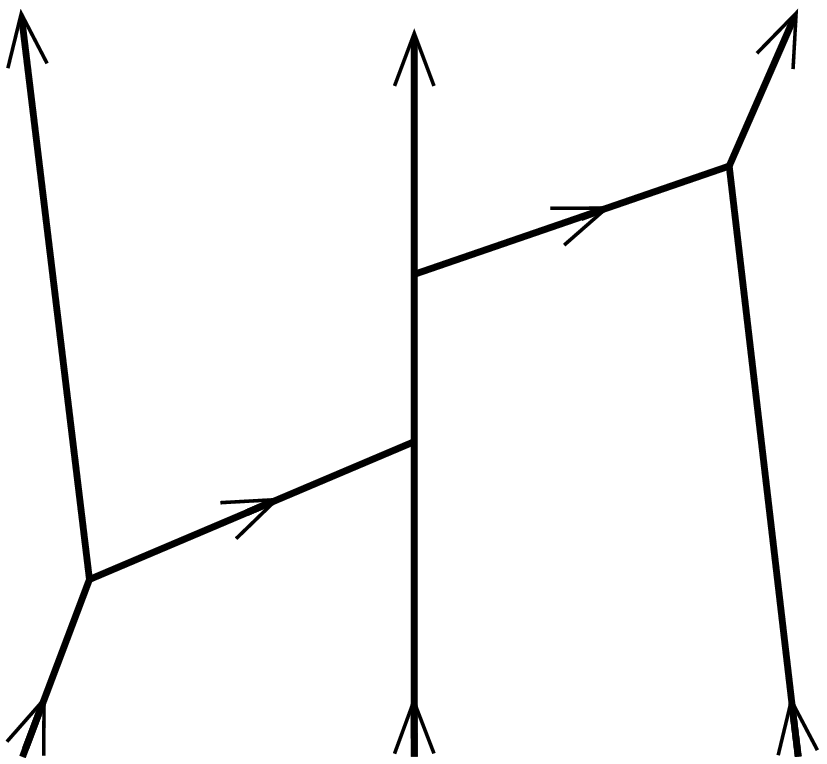}
\qra
\labellist
\tiny\hair 2pt
\pinlabel $\laiii$ at -34 189 \pinlabel $\laii$ at 166 191 \pinlabel $\lai$ at 255 189 
\pinlabel $1$ at 66 125 \pinlabel $1$ at 165 45
\endlabellist
\figins{-19}{0.6}{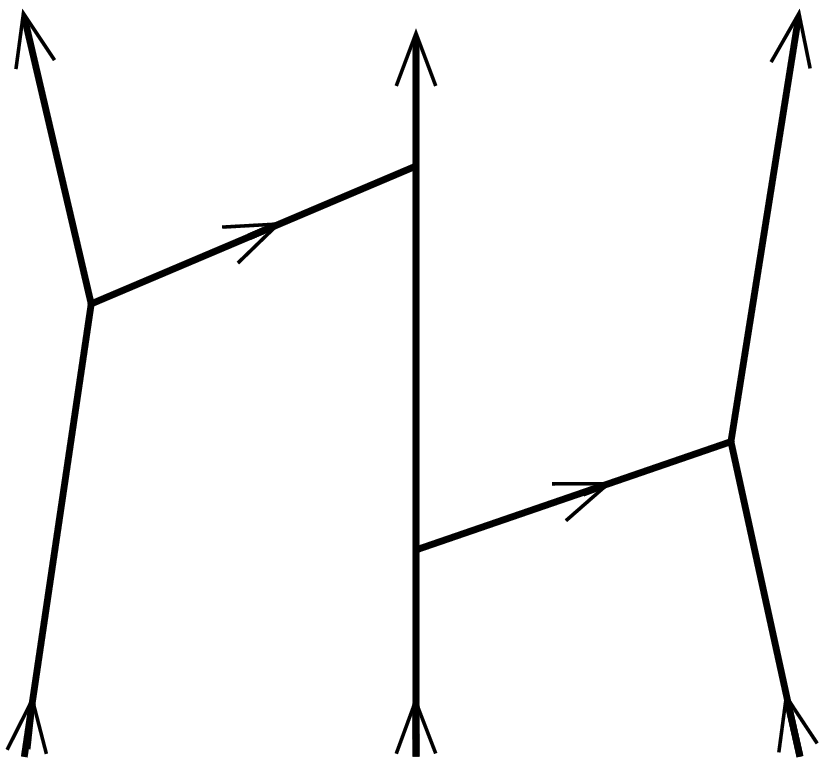}\ \
,\ \
p\mapsto p
\right)
\nn\displaybreak[0]
\\[1.5ex]
\xy
(0,0)*{\dblue\xybox{
    (-4,-4)*{};(4,4)*{} **\crv{(-4,-1) & (4,1)}?(0)*\dir{<} }};
(0,0)*{\dred\xybox{
    (4,-4)*{};(-4,4)*{} **\crv{(4,-1) & (-4,1)}?(0)*\dir{<};
}};
    (-6.4,-3)*{\scs i};
    (7.9,-3)*{\scs i+1};
(9,1)*{ \scs\lambda};
\endxy
\quad
&
\mapsto
\quad
\left(\ \ 
\quad
\labellist
\tiny\hair 2pt
\pinlabel $\laiii$ at -25 189 \pinlabel $\laii$ at 170 191 \pinlabel $\lai$ at 255 189 
\pinlabel $\color[rgb]{.5,.5,.5}{\downarrow}$ at 55 181 \pinlabel $x$ at 55 214
\pinlabel $\color[rgb]{.5,.5,.5}{\downarrow}$ at 151 103 \pinlabel $y$ at 151 136
\pinlabel $1$ at 66 125 \pinlabel $1$ at 165 45
\endlabellist
\figins{-19}{0.6}{Hhwebr}
\qra
\labellist
\tiny\hair 2pt
\pinlabel $\laiii$ at -35 189 \pinlabel $\laii$ at 70 191 \pinlabel $\lai$ at 255 189 
\pinlabel $\color[rgb]{.5,.5,.5}{\downarrow}$ at 151 181 \pinlabel $y$ at 151 214
\pinlabel $\color[rgb]{.5,.5,.5}{\downarrow}$ at 55 103 \pinlabel $x$ at 55 136
\pinlabel $1$ at 163 125 \pinlabel $1$ at 71 45
\endlabellist
\figins{-19}{0.6}{hHwebr}\ \
,\ \
p\mapsto (x-y)p
\right)
\nn\displaybreak[0]
\\[1.5ex]
\label{eq:assoc1}
\xy
(0,0)*{\dblue\xybox{
    (-4,-4)*{};(4,4)*{} **\crv{(-4,-1) & (4,1)}?(1)*\dir{>} }};
(0,0)*{\dred\xybox{
    (4,-4)*{};(-4,4)*{} **\crv{(4,-1) & (-4,1)}?(0)*\dir{<};
}};
    (-5,-3)*{\scs i};
    (7.8,-3.1)*{\scs i+1};
(9,1)*{ \scs\lambda};
\endxy
\quad
&
\mapsto
\quad
\left(\ \ 
\quad
\labellist
\tiny\hair 2pt
\pinlabel $\laiii$ at -35 195 \pinlabel $\laii$ at 168 197 \pinlabel $\lai$ at 258 195
\pinlabel $1$ at 70 132 \pinlabel $1$ at 168 46
\pinlabel $\color[rgb]{.5,.5,.5}{\downarrow}$ at  80 189 \pinlabel $x_1$ at  80 217
\pinlabel $\color[rgb]{.5,.5,.5}{\downarrow}$ at 160 104 \pinlabel $y$   at 160 132
\endlabellist
\figins{-19}{0.6}{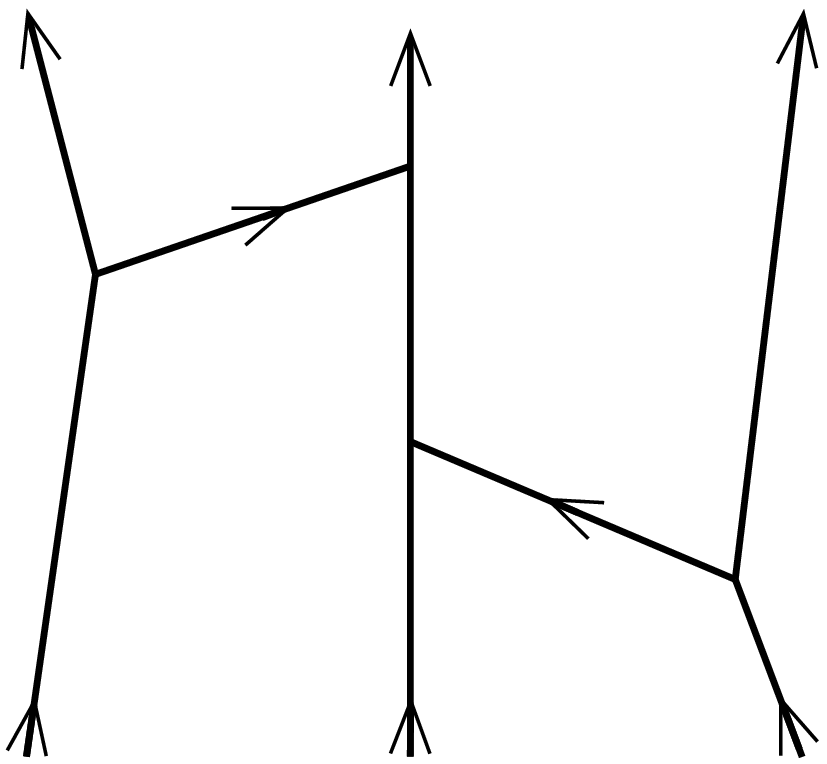}
\qra
\labellist
\tiny\hair 2pt
\pinlabel $\laiii$ at -40 195 \pinlabel $\laii$ at 73 198 \pinlabel $\lai$ at 255 195
\pinlabel $1$ at 70 50 \pinlabel $1$ at 168 126
\pinlabel $\color[rgb]{.5,.5,.5}{\downarrow}$ at  80 109 \pinlabel $x_2$ at  80 137
\pinlabel $\color[rgb]{.5,.5,.5}{\downarrow}$ at 160 189 \pinlabel $y$   at 160 217
\endlabellist
\figins{-19}{0.6}{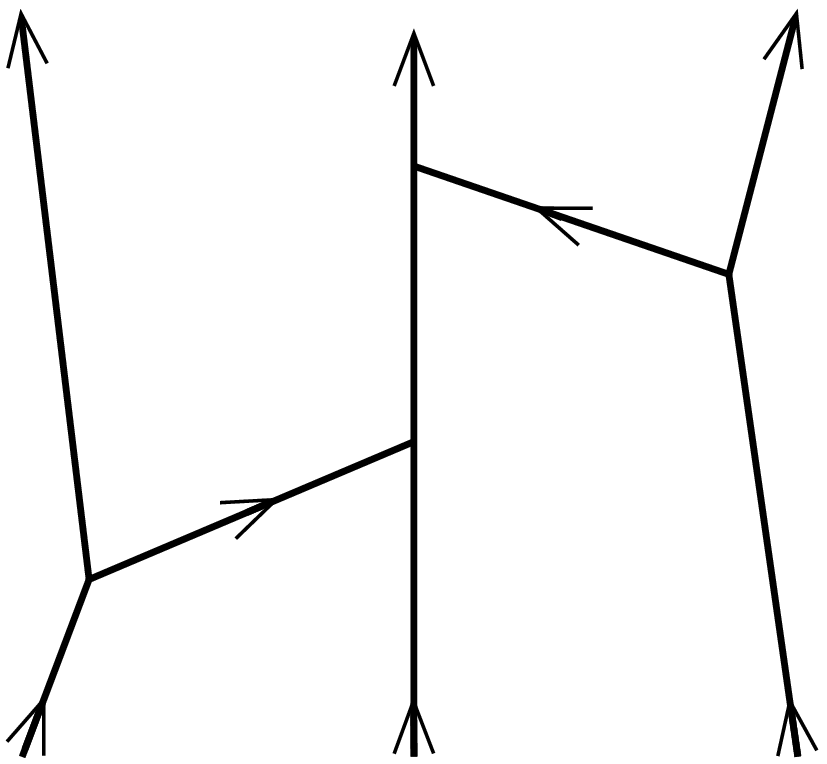}
,\ \
p\mapsto p\vert_{x_1\mapsto x_2}
\right)
\displaybreak[0]
\\[1.5ex]
\label{eq:assoc2}
\xy
(0,0)*{\dred\xybox{
    (-4,-4)*{};(4,4)*{} **\crv{(-4,-1) & (4,1)}?(1)*\dir{>} }};
(0,0)*{\dblue\xybox{
    (4,-4)*{};(-4,4)*{} **\crv{(4,-1) & (-4,1)}?(0)*\dir{<};
}};
    (-6.9,-3)*{\scs i+1};
    (6.0,-3)*{\scs i};
(9,1)*{ \scs\lambda};
\endxy
\quad
&
\mapsto
\quad
\left(\ \ 
\quad
\labellist
\tiny\hair 2pt
\pinlabel $\laiii$ at -35 189 \pinlabel $\laii$ at 75 191 \pinlabel $\lai$ at 250 189 
\pinlabel $1$ at 166 115 \pinlabel $1$ at 66 50
\pinlabel $\color[rgb]{.5,.5,.5}{\downarrow}$ at 50 114 \pinlabel $y$ at 50 143
\pinlabel $\color[rgb]{.5,.5,.5}{\downarrow}$ at 176 179 \pinlabel $x_1$ at 176 207
\endlabellist
\figins{-19}{0.6}{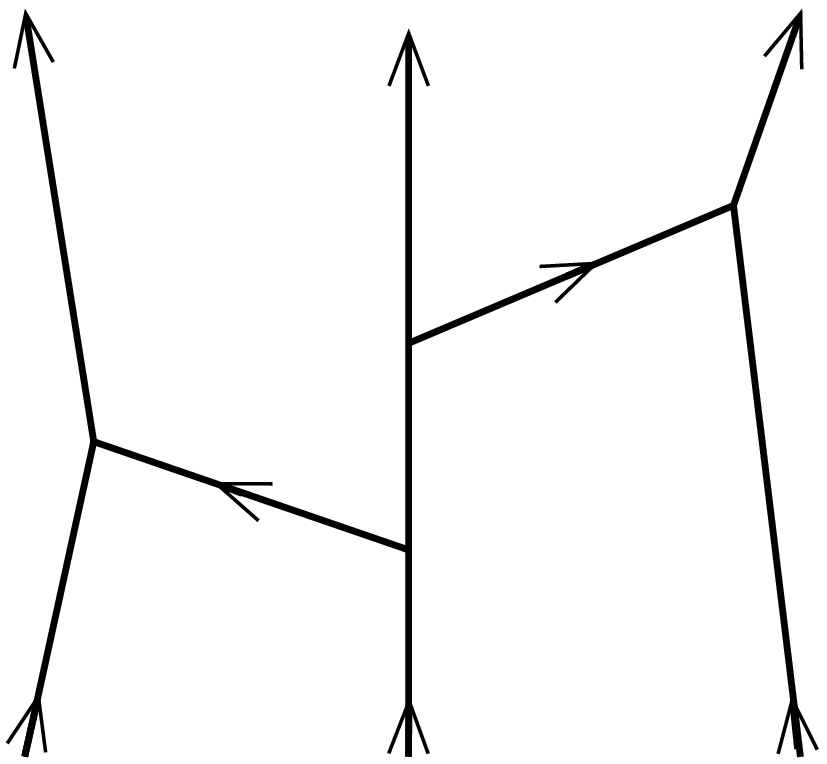}\
\ra\quad
\labellist
\tiny\hair 2pt
\pinlabel $\laiii$ at -35 189 \pinlabel $\laii$ at 165 191 \pinlabel $\lai$ at 255 189 
\pinlabel $1$ at 66 115 \pinlabel $1$ at 165 45
\pinlabel $\color[rgb]{.5,.5,.5}{\downarrow}$ at  50 176 \pinlabel $y$ at 50 205
\pinlabel $\color[rgb]{.5,.5,.5}{\downarrow}$ at 176 108 \pinlabel $x_2$ at 176 136
\endlabellist
\figins{-19}{0.6}{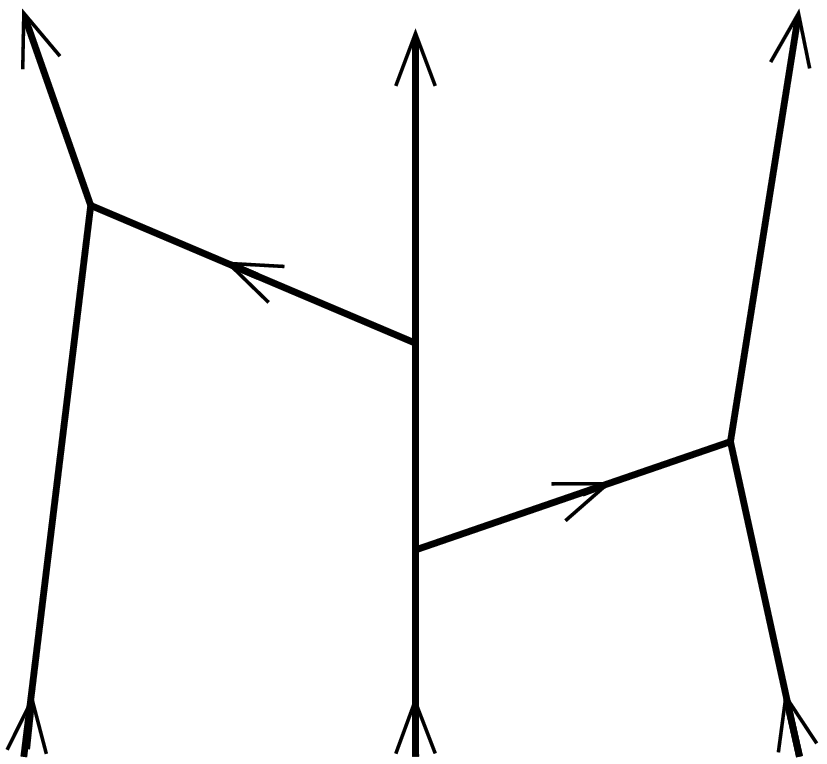}\ \
,\ \
p\mapsto p_{\vert {x_1\mapsto x_2}}
\right)
\displaybreak[0]
\\[1.5ex]
\xy
(0,0)*{\dblue\xybox{
    (-4,-4)*{};(4,4)*{} **\crv{(-4,-1) & (4,1)}?(0)*\dir{<} }};
(0,0)*{\dred\xybox{
    (4,-4)*{};(-4,4)*{} **\crv{(4,-1) & (-4,1)}?(1)*\dir{>};
}};
    (-6.5,-3)*{\scs i};
    (7.5,-3.1)*{\scs i+1};
(9,1)*{ \scs\lambda};
\endxy
\quad
&
\mapsto
\quad
\left(\ \ 
\quad
\labellist
\tiny\hair 2pt
\pinlabel $\laiii$ at -35 195 \pinlabel $\laii$ at 168 197 \pinlabel $\lai$ at 258 195
\pinlabel $1$ at 70 112 \pinlabel $1$ at 168 46
\pinlabel $\color[rgb]{.5,.5,.5}{\downarrow}$ at  75 169 \pinlabel $x_1$ at  75 197
\pinlabel $\color[rgb]{.5,.5,.5}{\downarrow}$ at 160 104 \pinlabel $y$   at 160 132
\endlabellist
\figins{-19}{0.6}{Hhweblr}
\qra
\labellist
\tiny\hair 2pt
\pinlabel $\laiii$ at -40 195 \pinlabel $\laii$ at 73 198 \pinlabel $\lai$ at 255 195
\pinlabel $1$ at 70 50 \pinlabel $1$ at 168 111
\pinlabel $\color[rgb]{.5,.5,.5}{\downarrow}$ at  80 109 \pinlabel $x_2$ at  80 137
\pinlabel $\color[rgb]{.5,.5,.5}{\downarrow}$ at 165 174 \pinlabel $y$   at 165 202
\endlabellist
\figins{-19}{0.6}{hHweblr}
,\ \
p\mapsto p_{\vert {x_1\mapsto x_2}}
\right)
\nn\displaybreak[0]
\\[1.5ex]
\xy
(0,0)*{\dred\xybox{
    (-4,-4)*{};(4,4)*{} **\crv{(-4,-1) & (4,1)}?(0)*\dir{<} }};
(0,0)*{\dblue\xybox{
    (4,-4)*{};(-4,4)*{} **\crv{(4,-1) & (-4,1)}?(1)*\dir{>};
}};
    (-7.9,-3)*{\scs i+1};
    (6.0,-3)*{\scs i};
(9,1)*{ \scs\lambda};
\endxy
\quad
&
\mapsto
\quad
\left(\ \ 
\quad
\labellist
\tiny\hair 2pt
\pinlabel $\laiii$ at -35 189 \pinlabel $\laii$ at 75 191 \pinlabel $\lai$ at 250 189 
\pinlabel $1$ at 166 125 \pinlabel $1$ at 66 50
\pinlabel $\color[rgb]{.5,.5,.5}{\downarrow}$ at 60 104 \pinlabel $y$ at 60 140
\pinlabel $\color[rgb]{.5,.5,.5}{\downarrow}$ at 176 179 \pinlabel $x_1$ at 176 207
\endlabellist
\figins{-19}{0.6}{hHwebrl}\
\ra\quad
\labellist
\tiny\hair 2pt
\pinlabel $\laiii$ at -35 189 \pinlabel $\laii$ at 165 191 \pinlabel $\lai$ at 255 189 
\pinlabel $1$ at 66 125 \pinlabel $1$ at 165 45
\pinlabel $\color[rgb]{.5,.5,.5}{\downarrow}$ at  50 176 \pinlabel $y$ at 50 204
\pinlabel $\color[rgb]{.5,.5,.5}{\downarrow}$ at 176 103 \pinlabel $x_2$ at 176 131
\endlabellist
\figins{-19}{0.6}{Hhwebrl}\ \
,\ \
p\mapsto p_{\vert {x_1\mapsto x_2}}
\right)
\nn
\end{align}

For $\vert i-j\vert\geq 2$:
\begin{equation*}
\xy
(0,0)*{\dblue\xybox{
    (-4,-4)*{};(4,4)*{} **\crv{(-4,-1) & (4,1)}?(1)*\dir{>} }};
(0,0)*{\dgreen\xybox{
    (4,-4)*{};(-4,4)*{} **\crv{(4,-1) & (-4,1)}?(1)*\dir{>};
}};
    (-5,-3)*{\scs i};
    (5.1,-3)*{\scs j};
(9,1)*{ \scs\lambda};
\endxy
\quad
\mapsto
\quad
\id\left(\ \ 
\quad
\labellist
\tiny\hair 2pt
\pinlabel $\laii$ at -35 189 \pinlabel $\lai$ at 142 191 \pinlabel $1$ at 55 65
\pinlabel $\dotsm$ at 200 110
\pinlabel $\lambda_{j+1}$ at 210 191 \pinlabel $\laj$ at 394 191 \pinlabel $1$ 
at 305 115
\endlabellist
\figins{-19}{0.64}{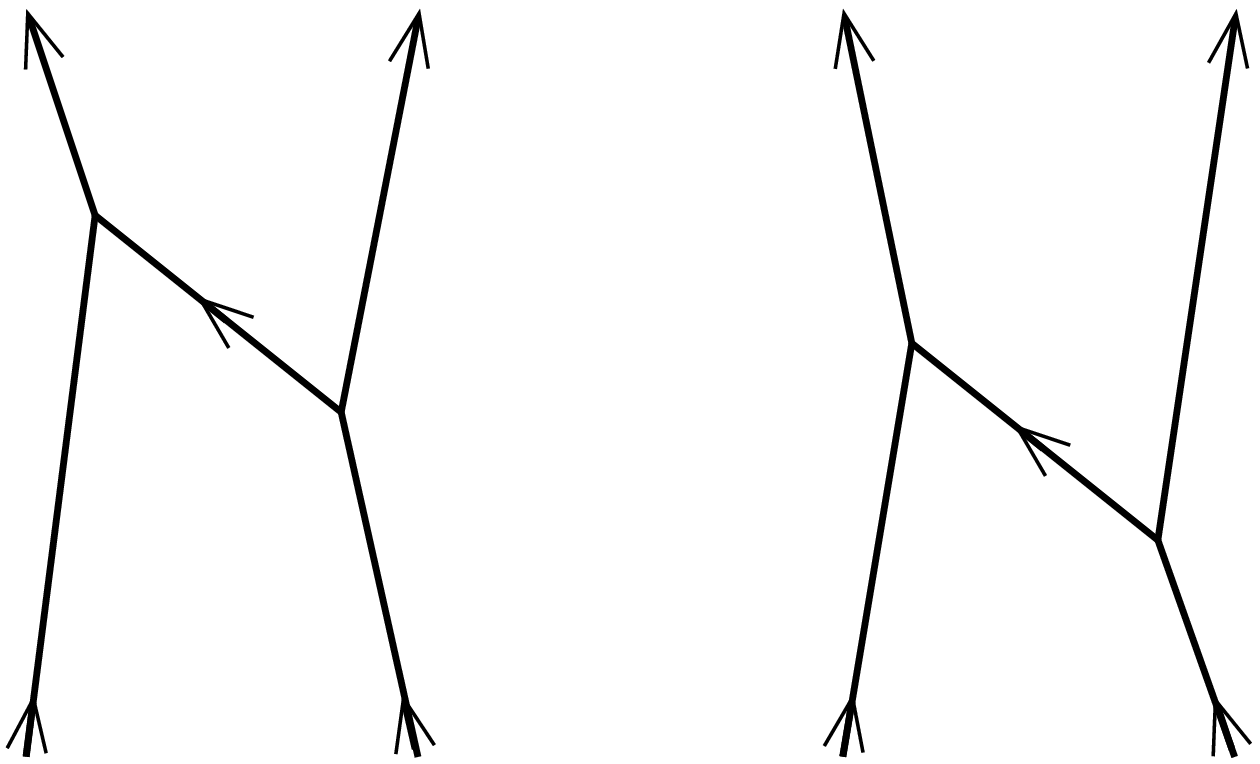}
\quad \right)
\end{equation*}
\begin{equation*}
\xy
(0,0)*{\dblue\xybox{
    (-4,-4)*{};(4,4)*{} **\crv{(-4,-1) & (4,1)}?(0)*\dir{<} }};
(0,0)*{\dgreen\xybox{
    (4,-4)*{};(-4,4)*{} **\crv{(4,-1) & (-4,1)}?(0)*\dir{<};
}};
    (-6.1,-3)*{\scs i};
    (6,-3)*{\scs j};
(9,1)*{ \scs\lambda};

\endxy
\quad\mapsto\quad
\id\left(\ \ 
\quad 
\labellist 
\tiny\hair 2pt
\pinlabel $\laii$ at -35 189 \pinlabel $\lai$ at 146 191 \pinlabel $1$ at 65 65
\pinlabel $\dotsm$ at 210 110
\pinlabel $\lambda_{j+1}$ at 240 191 \pinlabel $\laj$ at 440 191 \pinlabel $1$ at 355 100
\endlabellist
\figins{-19}{0.6}{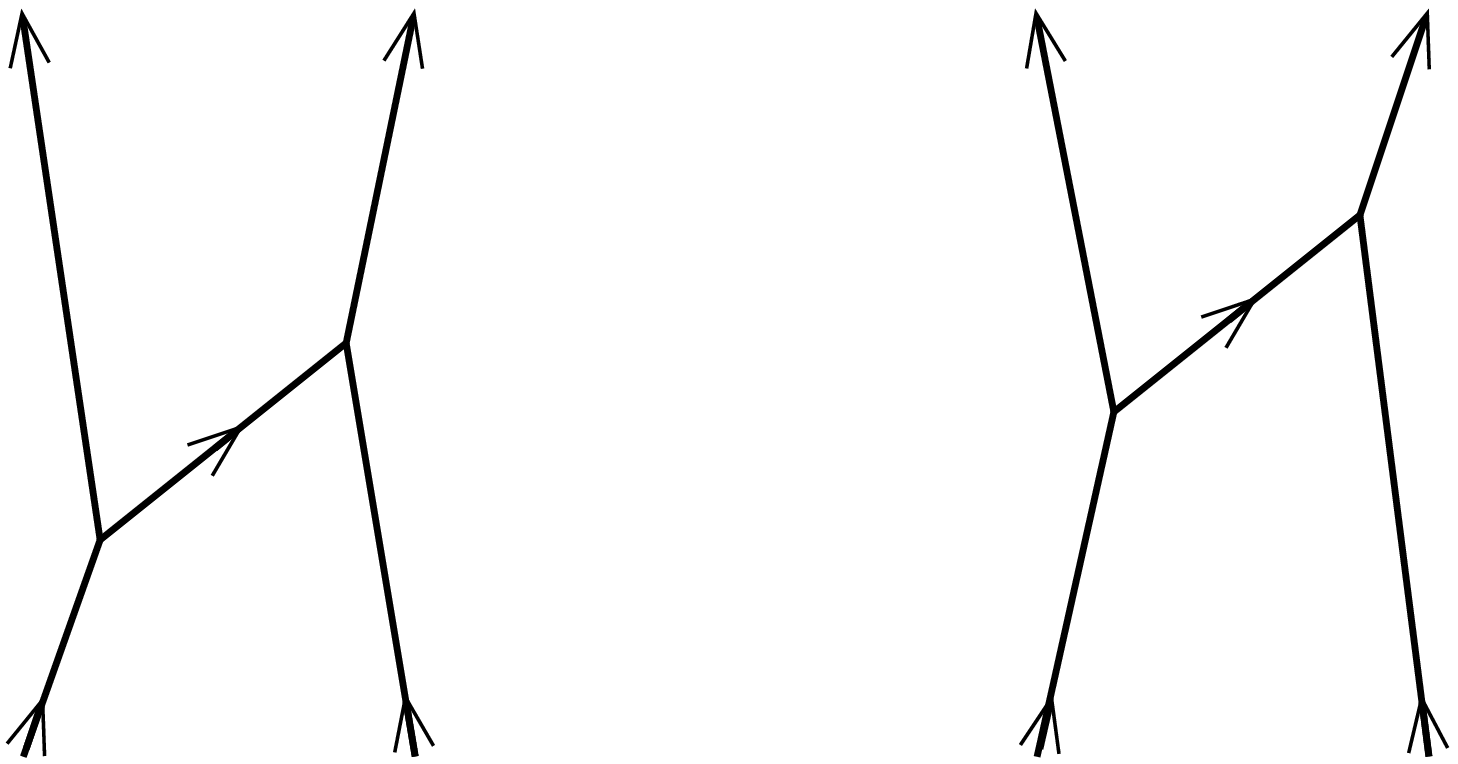}
\quad
\right)
\end{equation*}
Sideways crossings for $\vert i-j\vert\geq 2$ are defined in the same way as in the
case of $\vert i-j\vert=1$.

\begin{align*}
 \xy
    (0,0)*{\dblue\bbpef{\bscs i}};
    (6,0)*{\scs\lambda};
 \endxy
\quad
\mapsto &
\quad
\left(\ \
\begin{aligned}
&
\labellist
\tiny\hair 2pt
\pinlabel $\laii$ at -38 189 \pinlabel $\lai$ at 150 191 
\endlabellist
\quad
\figins{-19}{0.6}{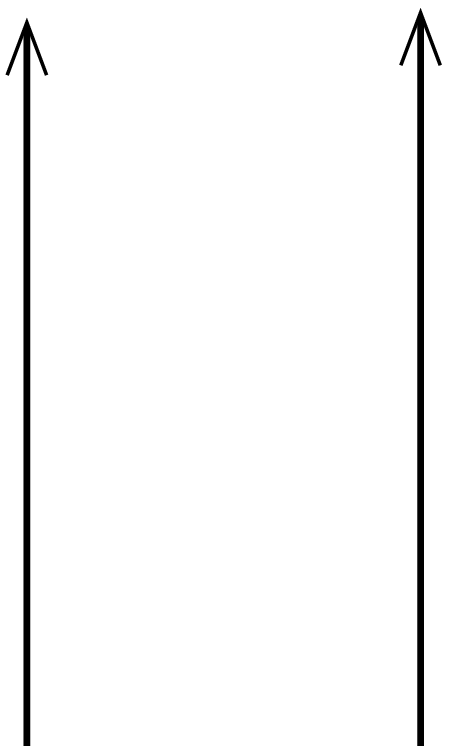} \ \
\qra\ \
\labellist
\tiny\hair 2pt
\pinlabel $\laii$ at -30 189 \pinlabel $\lai$ at 125 191 
\pinlabel $1$ at 60 50 \pinlabel $1$ at 50 130
\pinlabel $\color[rgb]{.5,.5,.5}{\downarrow}$ at 55 191 \pinlabel $x$ at 55 230
\pinlabel $\color[rgb]{.5,.5,.5}{\swarrow}$ at 130 41 
\pinlabel $\und{t}$ at 163 60
\endlabellist
\figins{-19}{0.6}{sqwebrl}\qquad\
\\[1.0ex]
&  
p\mapsto \sum\limits_{\ell=0}^{\lai}(-1)^{\ell}x^{\lai-\ell}
e_{\ell}(\und{t})p
\end{aligned}
\right)
\end{align*}

\begin{align*}
  \xy
    (0,0)*{\dblue\bbpfe{\bscs i}};
    (6,0)*{\scs\lambda};
    \endxy
\quad
\mapsto &
\quad
\left(\ \ 
\begin{aligned}
&
\quad
\labellist
\tiny\hair 2pt
\pinlabel $\laii$ at -37 189 \pinlabel $\lai$ at 150 191 
\endlabellist
\figins{-19}{0.6}{id2web}
\qra\quad
\labellist
\tiny\hair 2pt
\pinlabel $\laii$ at -35 189 \pinlabel $\lai$ at 130 191 
\pinlabel $1$ at 60 90 \pinlabel $1$ at 50 170
\pinlabel $\color[rgb]{.5,.5,.5}{\uparrow}$ at 55 41 \pinlabel $y$ at 53 7
\pinlabel $\color[rgb]{.5,.5,.5}{\nearrow}$ at -15 146 
\pinlabel $\und{z}$ at -50 105
\endlabellist
\figins{-19}{0.6}{sqweblr}
\quad
\\[1.0ex]
&  
p\mapsto \sum\limits_{\ell=0}^{\lambda_{i+1}}(-1)^{\ell}e_{\laii-\ell}(\und{z})y^{\ell}p\ \ 
\end{aligned}
\right)
\end{align*}
\begin{align*}
 \xy
    (0,0)*{\dblue\bbcef{\bscs i}};
    (6,0)*{\scs\lambda};
    \endxy
\quad
\mapsto &
\quad
\left(\ \ 
\begin{aligned}
&
\quad
\labellist
\tiny\hair 2pt
\pinlabel $\laii$ at -30 189 \pinlabel $\lai$ at 130 191 
\pinlabel $1$ at 50 90 \pinlabel $1$ at 50 130
\pinlabel $\color[rgb]{.5,.5,.5}{\downarrow}$ at 55 181 \pinlabel $y$ at 55 212
\pinlabel $\color[rgb]{.5,.5,.5}{\uparrow}$ at 55 31 \pinlabel $x$ at 54 -10
\pinlabel $\color[rgb]{.5,.5,.5}{\ra}$ at -10 106 
\pinlabel $\und{u}$ at -50 105
\endlabellist
\figins{-19}{0.6}{sqwebrl}
\qra\quad
\labellist
\tiny\hair 2pt
\pinlabel $\laii$ at -38 189 \pinlabel $\lai$ at 150 191 
\endlabellist
\figins{-19}{0.6}{id2web}\quad
\\[1.0ex]
&
\qquad p\mapsto \partial_{\und{u}x}(p_{\vert {y=x}})
\end{aligned}
\right)
\end{align*}
\begin{align*}
 \xy
    (0,0)*{\dblue\bbcfe{\bscs i}};
    (6,0)*{\scs\lambda};
 \endxy
\quad
\mapsto &
\quad
\left(\ \ 
\begin{aligned}
&
\quad
\labellist
\tiny\hair 2pt
\pinlabel $\laii$ at -35 189 \pinlabel $\lai$ at 130 191 
\pinlabel $1$ at 65 90 \pinlabel $1$ at 65 130
\pinlabel $\color[rgb]{.5,.5,.5}{\downarrow}$ at 55 181 \pinlabel $y$ at 55 212
\pinlabel $\color[rgb]{.5,.5,.5}{\uparrow}$ at 55 41 \pinlabel $x$ at 55 10
\pinlabel $\color[rgb]{.5,.5,.5}{\leftarrow}$ at 123 86 
\pinlabel $\und{u}$ at 158 85
\endlabellist
\figins{-19}{0.6}{sqweblr}
\quad\qra\quad
\labellist
\tiny\hair 2pt
\pinlabel $\laii$ at -38 189 \pinlabel $\lai$ at 150 191 
\endlabellist
\figins{-19}{0.6}{id2web}\quad
\\[1.0ex]
& \qquad\ \ 
p\mapsto \partial_{x\und{u}}(p_{\vert_{y=x}})
\end{aligned}
\right)
\end{align*}

This ends the definition of $\fbim$. Without giving any details, we remark 
that the bimodule maps above can be obtained as composites of elementary ones, 
called {\em zip, unzip, associativity, digon creation} and {\em annihilation}, 
which can be found in~\cite{M-S-V}.


\subsection{$\fbim$ is a $2$-functor}
\label{sec:2-functor}

We are now able to explain the relation between our $\fbim$ and Khovanov and 
Lauda's (see Subsection 6.3 in \cite{K-L3})
$$\Gamma^G_d\colon \Ucat^*\to 
\mbox{\bf EqFLAG}^*_d\subset \bim^*.$$ 

In the first place, we categorify the homomorphism $\psi_{n,d}$ 
from Section~\ref{sec:hecke-schur}. Note that all the relations in 
$\Scat(n,d)$ only 
depend on $\mathfrak{sl}_n$-weights, except the value of the 
degree zero bubbles, which truly depend on $\mathfrak{gl}_n$-weights. 

\begin{defn} We define a 2-functor  
\label{defn:Psi}
$$\Psi_{n,d}\colon \Ucat\to \Scat(n,d).$$
On objects and $1$-morphisms $\Psi_{n,d}$ is defined just as $\psi_{n,d}\colon 
\U \to \SD(n,d)$ in~\eqref{eq:psi}. 
On $2$-morphisms we define $\Psi_{n,d}$ as follows. Let $D$ be a 
string diagram representing a $2$-morphism in $\Ucat$ (from now on we will 
simply say that $D$ is a diagram in $\Ucat$). Then $\Psi_{n,d}$ maps $D$ to the 
same diagram, multiplied by a power of $-1$ depending 
on the left cups and caps in $D$ according to the rule in~\eqref{eq:signs}. 
The labels in $\bZ^{n-1}$ of the regions of $D$ are mapped by $\phi_{n,d}$ 
to labels in $\bZ^n$ of the corresponding regions of $\Psi_{n,d}(D)$, or to 
$*$. This means that, if $D$ has a region labeled by $\lambda$ such that 
$\phi_{n,d}(\lambda)\not\in\Lambda(n,d)$, then $\Psi_{n,d}(D)=0$ by definition.  
Finally, extend this definition to all $2$-morphisms by linearity.  
\end{defn}
\noindent It is easy to see that $\Psi_{n,d}$ is well-defined, 
full and essentially surjective.

In the second place, recall that there is a well-known isomorphism 
$$\bQ[x_1,\ldots,x_d]^{S_{\lambda_1}\times\cdots\times S_{\lambda_n}}\cong 
H_{GL(d)}(Fl(\underline{k})),$$ with 
$\underline{k}=(k_0,k_1,k_2,k_3,\ldots,k_n)= (0,\lambda_1,\lambda_1+\lambda_2,
\lambda_1+\lambda_2+\lambda_3,\ldots,d)$, for any $\lambda\in\Lambda(n,d)$ 
(see (6.25) in~\cite{K-L3}, for example). Using this isomorphism, it 
is straightforward 
to check that the following lemma holds by comparing the images of 
the generators. Recall that $\Gamma_d^G$ kills all diagrams with labels outside 
$\Lambda(n,d)$.  
\begin{lem} 
\label{lem:rep}
The following triangle is commutative
\begin{equation*}
\xymatrix{
\Ucat^*\ar[rr]^{\Gamma^G_d}\ar[dr]_{\Psi_{n,d}} && \bim^*
\\
& \Scat(n,d)^*\ar[ur]_{\fbim} &
}
\end{equation*}
\end{lem}

The following result is now an immediate consequence of Khovanov and Lauda's 
Theorem 6.13. 

\begin{prop}\label{prop:fbim}
$\fbim$ defines a 2-functor from $\Scat(n,d)^*$ to $\bim^*$.
\end{prop}

One could of course prove Proposition~\ref{prop:fbim} by hand. 
We will just give two sample calculations. The 
result of the second one, the image of the dotted bubbles, will be 
needed in a later section. 

\subsubsection{Examples of the direct proof of Proposition~\ref{prop:fbim}}

We shall give the proof for the zig-zag relation of biadjointness 
and compute the images of the bubbles by $\fbim$. 

Before proceeding, we give some useful relations that are used in the computations. 
First of all, both the kernel and the image of the divided difference operator 
$\partial_{xy}$ consist of the polynomials that are symmetric in the variables $x$ and $y$.
If $p$ is symmetric in the variables $x$ and $y$ then 
$$\partial_{xy}(p\,q)=p\,\partial_{xy}q$$
for any polynomial $q$. Also, note that $\partial_{yx}=-\partial_{xy}$. 

We shall frequently use the following useful identities 
(see for example~\cite{F} for the proofs). For $\underline{x}=(x_1,\ldots,x_k)$, let $h_j(\underline{x})$ denote the 
$j$-th complete symmetric polynomial in the variables $x_1,\ldots,x_k$. Then we have
\begin{equation}\label{use1}
\partial_{y\underline{x}}(y^{N})=h_{N-k}(y,\underline{x}),
\end{equation}
and
\begin{equation}\label{use2}
\sum_{j=0}^k(-1)^je_j(\underline{x})h_{k-j}(\underline{x})=\delta_{k,0}.
\end{equation}

Moreover, if $x$, $\underline{u}=(u_1,\ldots,u_{a})$ and $\underline{t}=(t_1,\ldots,t_{a+1})$ 
are variables such that
\[
e_l(x,\underline{u})=e_l(\underline{t}),\quad l=1,\ldots,a+1,
\]
then for every $l=1,\ldots,a+1$, we have  
\begin{equation}\label{use3}
e_l(\underline{u})=\sum_{j=0}^l(-1)^jx^je_{l-j}(\underline{t}),
\end{equation}
and 
\begin{equation}\label{use4}
e_l(\underline{t})=e_l(\underline{u})+xe_{l-1}(\underline{u}).
\end{equation}

$\bullet$ The zig-zag relations.\\

In order to reduce the number of subindices (to keep the notation as concise as possible), 
we denote $\lambda_i=a$ and $\lambda_{i+1}=b$.

Then the left hand side of the first of the relations~\eqref{eq_biadjoint1} is mapped 
by $\fbim$ as 
follows:
\begin{align*}
 \xy   0;/r.18pc/:
    (0,0)*{\dblue\xybox{
    (-8,0)*{}="1";
    (0,0)*{}="2";
    (8,0)*{}="3";
    (-8,-10);"1" **\dir{-};
    "1";"2" **\crv{(-8,8) & (0,8)} ?(0)*\dir{>} ?(1)*\dir{>};
    "2";"3" **\crv{(0,-8) & (8,-8)}?(1)*\dir{>};
    "3"; (8,10) **\dir{-};}};
    (12,5)*{\lambda};(-10,-8)*{\scs i};
    \endxy
\quad
\mapsto &
\quad
\left(
\begin{aligned}
&
\quad
\labellist
\tiny\hair 2pt
\pinlabel $b$ at -12 193 \pinlabel $a$ at 128 191 
\pinlabel $1$ at 55 85
\pinlabel $\color[rgb]{.5,.5,.5}{\uparrow}$ at 55 31 \pinlabel $x_1$ at 55 0
\pinlabel $\color[rgb]{.5,.5,.5}{\searrow}$ at -16 51
\pinlabel $\und{u}$ at -46 70
\pinlabel $\color[rgb]{.5,.5,.5}{\swarrow}$ at 126 41
\pinlabel $\und{t}$ at 153 62
\endlabellist
\figins{-19}{0.7}{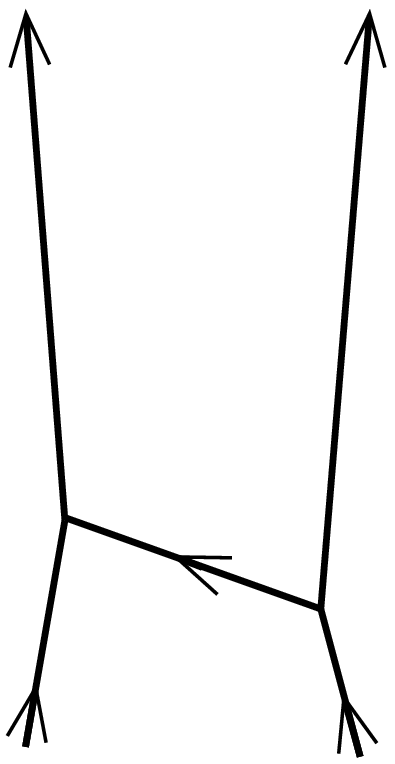}
\quad \qra
\labellist
\tiny\hair 2pt
\pinlabel $b$ at -10 193 \pinlabel $a$ at 124 191 
\pinlabel $1$ at 38 45 \pinlabel $1$ at 72 94 \pinlabel $1$ at 38 190
\pinlabel $\color[rgb]{.5,.5,.5}{\leftarrow}$ at 120 71 
\pinlabel $\und{v}$ at 154 70
\pinlabel $\color[rgb]{.5,.5,.5}{\downarrow}$ at 70 182
\pinlabel $x_2$ at 70 210
\pinlabel $\color[rgb]{.5,.5,.5}{\uparrow}$ at 70 30
\pinlabel $x_1$ at 70 -5
\pinlabel $y$ at 40 120 
\pinlabel $\color[rgb]{.5,.5,.5}{\searrow}$ at -16 41
\pinlabel $\und{u}$ at -46 60
\pinlabel $\color[rgb]{.5,.5,.5}{\leftarrow}$ at 132 11
\pinlabel $\und{t}$ at 161 12
\endlabellist
\figins{-19}{0.7}{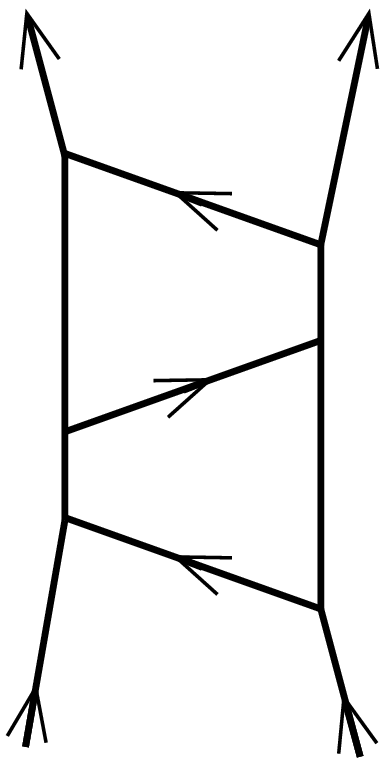}
\qquad\ra\quad
\labellist
\tiny\hair 2pt
\pinlabel $b$ at -10 193 \pinlabel $a$ at 124 191 
\pinlabel $1$ at 38 150
\pinlabel $\color[rgb]{.5,.5,.5}{\uparrow}$ at 70 135
\pinlabel $x_2$ at 70 105
\endlabellist
\figins{-19}{0.7}{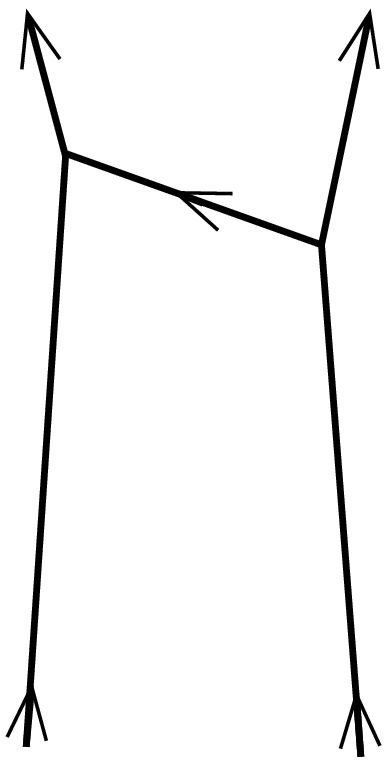}\ \ \
\\[1.0ex]
& \qquad
p\mapsto 
\partial_{x_1\und{v}}
\left(\sum\limits_{\ell=0}^{a}(-1)^{\ell}x_2^{a-\ell}
e_{\ell}(\und{v}) p\right)
\end{aligned}
\right)
\end{align*}

Note that $p=p(x_1,\underline{u},\underline{t})$ is symmetric in the variables 
$\underline{u}$ and $\underline{t}$ separately. 
Also, the lowest trivalent vertex on the right strand in the middle picture of the movie, 
implies that $e_l(x_1,\underline{v})=e_l(\underline{t})$, 
for every $l=1,\ldots,a+1$. So, $x_1^j$ for $j>a$ is a symmetric polynomial in 
the variables $\underline{t}$ (e.g. this follows from (\ref{use3}) for 
$l=a+1$). Thus we can write $p$ as:
\begin{equation}
p=\sum_{j=0}^a {x_1^jq_j(\underline{u},\underline{t})},
\end{equation}
where $q_j=q_j(\underline{u},\underline{t})$, $j=0,\ldots,a$, are polynomials symmetric 
in $\underline{u}$ and $\underline{t}$ separately.

Then we have:
\[
\partial_{x_1\underline{v}}(\sum_{l=0}^a(-1)^lx_2^{a-l}
e_l(\underline{v})p)=\sum_{l=0}^a(-1)^lx_2^{a-l}\partial_{x_1\underline{v}}
(e_l(\underline{v})\sum_{j=0}^ax_1^jq_j){=}_{(l\mapsto a-l)} 
\]
\begin{equation}\label{pom11}
=\sum_{l=0}^a\sum_{j=0}^a(-1)^{a-l}x_2^lq_j\partial_{x_1\underline{v}}(x_1^je_{a-l}
(\underline{v})). 
\end{equation}

Since $e_l(x_1,\underline{v})=e_l(\underline{t})$, 
for every $l=1,\ldots,a+1$, by (\ref{use3})
we have 
$e_{a-l}(\underline{v})=\sum_{k=0}^{a-l}(-1)^k x_1^k e_{a-l-k}(\underline{t})$. 
After replacing this in (\ref{pom11}), we get 

\begin{align*}
&= \sum_{l=0}^a\sum_{j=0}^ax_2^lq_j\sum_{k=0}^{a-l}(-1)^{a-l-k}e_{a-l-k}
(\underline{t})\partial_{x_1\underline{v}}(x_1^{j+k})=_{(\ref{use1})}\\
&= \sum_{l=0}^a\sum_{j=0}^ax_2^lq_j\sum_{k=0}^{a-l}(-1)^{a-l-k}e_{a-l-k}
(\underline{t})h_{j+k-a}(x_1,\underline{v})=\\
&= \sum_{l=0}^a\sum_{j=0}^ax_2^lq_j\sum_{k=0}^{a-l}(-1)^{a-l-k}e_{a-l-k}
(\underline{t})h_{j+k-a}(\underline{t})=_{(k\mapsto a-l-k)}\\
&= \sum_{l=0}^a\sum_{j=0}^ax_2^lq_j\sum_{k=0}^{a-l}(-1)^{k}e_{k}
(\underline{t})h_{j-l-k}(\underline{t}).
\end{align*}
Since $h_p(\underline{t})=0$ for $p<0$, we must have $k\le j-l(\le a-l)$ in the innermost summation, and so by (\ref{use2}) the last expression above 
is equal to
\begin{align*}
&= \sum_{l=0}^a\sum_{j=0}^ax_2^lq_j\sum_{k=0}^{j-l}(-1)^{k}
e_{k}(\underline{t})h_{j-l-k}(\underline{t})=\sum_{l=0}^a\sum_{j=0}^ax_2^lq_j 
\delta_{j-l,0}=\\
&= \sum_{j=0}^ax_2^jq_j=p_{\mid x_1\mapsto x_2},
\end{align*}
which is just the identity map, as wanted.\\

$\bullet$ Images of bubbles by $\fbim$. \\

Again we denote $\lambda_i=a$ and $\lambda_{i+1}=b$.

The clockwise oriented bubble with $r\ge 0$ dots on it is mapped by $\fbim$ as follows
\begin{align*}
\xy 0;/r.18pc/:
 (0,0)*{\dblue\cbub{\black r}{\black i}};
  (4,8)*{\scs\lambda};
 \endxy
\quad
\mapsto &
\quad
\left(
\begin{aligned}
&
\quad
\labellist
\tiny\hair 2pt
\pinlabel $b$ at -20 193 \pinlabel $a$ at 150 191 
\pinlabel $\color[rgb]{.5,.5,.5}{\nearrow}$ at -18 58
\pinlabel $\und{t}$ at -52 20
\pinlabel $\color[rgb]{.5,.5,.5}{\nwarrow}$ at 148 58
\pinlabel $\und{u}$ at 185 20
\endlabellist
\figins{-19}{0.6}{id2web}
\qquad\ra\qquad
\labellist
\tiny\hair 2pt
\pinlabel $b$ at -15 197 \pinlabel $a$ at 128 196 
\pinlabel $\color[rgb]{.5,.5,.5}{\swarrow}$ at 130 30 
\pinlabel $\und{u}$ at 162 48
\pinlabel $\color[rgb]{.5,.5,.5}{\swarrow}$ at 120 115 
\pinlabel $\und{v}$ at 160 138
\pinlabel $x$ at 55 45
\pinlabel $y$ at 65 170 
\pinlabel $\color[rgb]{.5,.5,.5}{\nearrow}$ at -18 148
\pinlabel $\und{t}$ at -50 110
\endlabellist
\figins{-19}{0.6}{sqweblr}
\qquad\ra\quad
\labellist
\tiny\hair 2pt
\pinlabel $b$ at -15 191 \pinlabel $a$ at 144 193 
\endlabellist
\figins{-19}{0.6}{id2web}\ \ \
\\[1.0ex]
& \qquad\ 
p\mapsto 
\partial_{x\und{v}}
\biggl(\sum\limits_{\ell=0}^{b}(-1)^{\ell}e_{b-\ell}(\und{t})x^{\ell+r} p\biggr)
\end{aligned}
\right)
\end{align*}

The polynomial $p=p(\underline{t},\underline{u})$ is symmetric in the variables 
$\underline{t}$ and $\underline{u}$ separately. In particular, 
we have $\partial_{x\underline{v}} (p\,q)=p\,\partial_{x\underline{v}}(q)$, for any polynomial $q$.
We have:
\begin{align*}
\partial_{x\underline{v}}(\sum_{l=0}^b {(-1)^l e_{b-l}(\underline{t})x^{l+r}p})
&= \sum_{l=0}^b {(-1)^l 
e_{b-l}(\underline{t})p\partial_{x\underline{v}}(x^{l+r})}=_{(\ref{use1})}\\
&= p\sum_{l=0}^b {(-1)^l e_{b-l}(\underline{t})h_{l+r-a+1}(\underline{u})}=_{(l\mapsto b-l)}\\
&= p (-1)^b\sum_{l=0}^b {(-1)^l e_{l}(\underline{t})h_{b-a+r+1-l}(\underline{u})}.
\end{align*}
Since $e_l(\underline{t})=0$ for $l>b$, and $h_{b-a+r+1-l}(\underline{u})=0$, 
for $l>b-a+r+1$, we have that the clockwise oriented bubble is mapped by $\fbim$ to the 
following bimodule map:
\begin{equation}\label{pom12}
p\mapsto p (-1)^b\sum_{l=0}^{b-a+r+1} {(-1)^l e_{l}
(\underline{t})h_{b-a+r+1-l}(\underline{u})}.
\end{equation}
In particular, if $b-a+r+1<0$, i.e. if $r<a-b-1$, the bubble is mapped to zero, and if 
$r=a-b-1$, the bubble is mapped to $(-1)^b$ times the identity 
(note that $a-b=\lambda_i-\lambda_{i+1}$ is $\mathfrak{sl}_n$ weight). 
Also, $r$ can be naturally extended to $r\ge a-b-1$ (in (\ref{pom12})), i.e. to 
include fake bubbles in the case $a\le b$.\\

The counter-clockwise oriented bubble with $r\ge 0$ dots on it is mapped by $\fbim$ as 
follows\\

\begin{align*}
  \xy 0;/r.18pc/:
 (0,0)*{\dblue\ccbub{\black r}{\black i}};
  (4,8)*{\scs\lambda};
 \endxy
\quad
\mapsto &
\quad
\left(
\begin{aligned}
&
\quad
\labellist
\tiny\hair 2pt
\pinlabel $b$ at -15 194 \pinlabel $a$ at 144 192 
\pinlabel $\color[rgb]{.5,.5,.5}{\nearrow}$ at -18 58
\pinlabel $\und{t}$ at -52 20
\pinlabel $\color[rgb]{.5,.5,.5}{\nwarrow}$ at 148 58
\pinlabel $\und{u}$ at 183 20
\endlabellist
\figins{-19}{0.6}{id2web}
\qquad\ra\qquad
\labellist
\tiny\hair 2pt
\pinlabel $b$ at -15 198 \pinlabel $a$ at 130 197 
\pinlabel $\color[rgb]{.5,.5,.5}{\searrow}$ at -15 35 
\pinlabel $\und{t}$ at -54 58
\pinlabel $\color[rgb]{.5,.5,.5}{\nwarrow}$ at 130 155 
\pinlabel $\und{u}$ at 166 122
\pinlabel $\color[rgb]{.5,.5,.5}{\uparrow}$ at 58 38 
\pinlabel $x$ at 58 0
\pinlabel $y$ at 45 172 
\pinlabel $\color[rgb]{.5,.5,.5}{\searrow}$ at -10 121
\pinlabel $\und{v}$ at -50 140
\endlabellist
\figins{-19}{0.6}{sqwebrl}
\qquad\ra\quad
\labellist
\tiny\hair 2pt
\pinlabel $b$ at -15 191 \pinlabel $a$ at 145 193 
\endlabellist
\figins{-19}{0.6}{id2web}\quad
\\[1.0ex]
&  
\qquad
p\mapsto 
\partial_{\und{v}x}
\left(\sum\limits_{\ell=0}^{a}(-1)^{\ell}x^{a-\ell+r}
e_{\ell}(\und{u})p\right)
\end{aligned}
\right)
\end{align*}

Completely analogously as above, we have that the counter-clockwise oriented bubble is 
mapped by $\fbim$ to the following bimodule map:
\begin{equation}\label{pom13}
p\mapsto p (-1)^{b+1}\sum_{l=0}^{a-b+r+1} {(-1)^l e_{l}(\underline{u})h_{a-b+r+1-l}
(\underline{t})}.
\end{equation}
Again, from $r<b-a-1$, the bubble is mapped to zero, and if $r=b-a-1$, it is mapped to 
$(-1)^{b+1}$ times the identity. Moreover, $r$ can be naturally extended to $r\ge b-a-1$, 
i.e. to include fake bubbles in the case $b\le a$.\\

\begin{rem}
Our reason for changing the signs from~\cite{K-L3}, was to make the 
signs in the image of the degree zero bubbles, i.e. $(-1)^b$ for the 
clockwise bubble and $(-1)^{b+1}$ for the counter-clockwise bubble, 
coincide with those of ~\eqref{eq:bubb_deg0}.   
\end{rem}

Finally, by the Giambelli and the dual Giambelli formulas (see e.g. 
\cite{F}), from (\ref{pom12}) and (\ref{pom13}) the infinite 
Grassmannian relation follows directly.

\section{Comparisons with ${\mathcal U}(\mathfrak{sl}_n)$}  
\label{sec:struct}                                          

In this section we show the analogues for  
$\Scat(n,d)$ of some of Khovanov and Lauda's results on the structure of 
$\Ucat$. Our results are far from complete. More work will need to be 
done to understand the structure of $\Scat(n,d)$ better. 

To simplify terminology, by a $2$-functor 
we will always mean an additive $\Q$-linear degree preserving 
$2$-functor. 

\subsection{Categorical inclusions and projections}
In the first place, we categorify the homomorphisms $\pi_{d',d}$ 
from Section~\ref{sec:hecke-schur}. 

\begin{defn} Let $d'=d+kn$, with $k\in\bN$. We define a $2$-functor 
$$\Pi_{d',d}\colon \Scat(n,d')\to\Scat(n,d).$$
On objects and $1$-morphisms $\Pi_{d',d}$ is defined as $\pi_{d',d}$. 
On $2$-morphisms $\Pi_{d',d}$ is defined as follows. For any diagram $D$ in 
$\Scat(n,d')$ with regions labeled 
$\lambda\in\Lambda(n,d')$ such that $\lambda-(k^n)\in\Lambda(n,d)$, let 
$\Pi_{d',d}(D)$ be given by the same diagram with labels of the form 
$\lambda-(k^n)$, multiplied by $(-1)^k$ for every left cap and left cup in 
$D$. For any other diagram $D$, let $\Pi_{d',d}(D)=0$. Extend this definition 
to all $2$-morphisms by linearity.  
\end{defn}
\noindent Note that $\Pi_{d',d}$ is well-defined, because 
$\overline{\lambda}=\overline{\lambda-(k^n)}$. The extra $(-1)^k$ for left 
cups and caps is necessary to match our normalization of the degree zero 
bubbles. It also ensures that we have 
$$\Pi_{d',d}\Psi_{n,d'}=\Psi_{n,d},$$
where $$\Psi_{n,d}\colon \Ucat\to \Scat(n,d)$$
is the 2-functor defined in Definition~\ref{defn:Psi}.

Note also that the $\Pi_{d',d}$ form something like an inverse system of 
$2$-functors between $2$-categories, 
because 
$$\Pi_{d',d}\Pi_{d'',d'}=\Pi_{d'',d}$$ 
(compare to \eqref{eq:inversesystem}). We say ``something like'' an inverse system, because we have not been able to 
find a precise definition of such a structure in the literature on 
$n$-categories. Also one would have to think carefully if the 
 ``inverse limit'' of the $\Scat(n,d)$ would still be Krull-Schmidt. Finally, 
there appears to be no general theorem that says that 
the Grothendieck group of an inverse limit is the inverse limit of the 
Grothendieck groups (even for algebras there is no such theorem). 
So we cannot (yet) reasonably conjecture the categorification of the 
embedding \eqref{eq:inverselimit}. 
All we can say at the moment is the following:   
\begin{cor} We have:
\label{cor:inverselimit}
\begin{enumerate}
\item Let $f1_{\alpha}$ be a $2$-morphism in 
$\Ucat$. Let $d_0>0$ be the minimum value such 
that $\alpha=\overline{\beta}$ with $\beta\in\Lambda(n,d_0)$. Then $f=0$ if 
and only if $\Psi_{n,d_0+nk}(f)=0$ for any $k\geq 0$.
\item Let $\{f_i1_{\alpha}\}_{i=1}^{s}$ be a finite set of $2$-morphisms in 
$\HomU(x,y)$. Then the 
$f_i1_{\alpha}$ are linearly independent if and only if 
there exists a $d>0$ such that the $\Psi_{n,d}(f_i1_{\lambda})$ are linearly 
independent in $\HomS(x,y)$. 
\end{enumerate} 
\end{cor}
The proof of Corollary~\ref{cor:inverselimit} follows from 
Khovanov and Lauda's Lemma 6.16 in~\cite{K-L3}, 
which implies Theorem 1.3 in~\cite{K-L3}, our Lemma~\ref{lem:rep} 
and the remarks above Corollary~\ref{cor:inverselimit}. 

The main reason for trying to categorify~\eqref{eq:inverselimit} 
is the following: if the inverse limit of the $\Scat(n,d)$ turns out to exist, 
perhaps it contains a sub-2-category which categorifies 
${\mathbf U}_q(\mathfrak{sl}_n)$. 

\subsection{The structure of the 2HOM-spaces}
We now turn our attention to the structure of the 2HOM-spaces in 
$\Scat(n,d)$. The reader should compare our results to Khovanov and 
Lauda's in~\cite{K-L3}. We first show the analogue of Lemma 6.15. 
\subsubsection{Bubbles for $n=2$}
For starters 
suppose that $n=2$. Let $\lambda=(a,b)\in\Lambda(2,d)$. Recall that a 
partially symmetric polynomial 
$p(\underline{x},\underline{y})=p(\underline{x},\underline{y})\in 
\Q[x_1,\ldots,x_a,y_1,\ldots,y_b]^{S_a\times S_b}$ 
is called 
{\em supersymmetric} if the substitution $x_1=t=y_1$ gives a polynomial 
independent of $t$ (see~\cite{F-P} and~\cite{McD} for example). 
We let $R_{a,b}^{ss}$ denote the ring of supersymmetric 
polynomials. The {\em elementary} supersymmetric polynomials are 
$$e_{j}(\underline{x},\underline{y})=\sum_{s=0}^j (-1)^s h_{j-s}(\underline{x})\varepsilon_s(\underline{y}),$$
where $h_{j-s}(\underline{x})$ is the $j-s$th complete symmetric polynomial in $a$ 
variables and $\varepsilon_s(\underline{y})$ the $s$th elementary symmetric polynomial in $b$ 
variables, which we put equal to zero if $s>b$ by convention. It is easy 
to see that $e_j(\underline{x},\underline{y})$ is supersymmetric, because we 
have 
$$\prod_{r=1}^a\prod_{s=1}^b\dfrac{1-y_rZ}{1-x_sZ}=\sum_j e_j(\underline{x},\underline{y})Z^j.$$ 
Using the supersymmetric analogue of the Giambelli formula we can define 
the {\em supersymmetric Schur polynomials}
$$\pi_{\alpha}(\underline{x},\underline{y})=\det(e_{\alpha_i+j-i}(\underline{x},\underline{y}))$$
for $1\leq i,j\leq m$ and $\alpha$ a partition of length $m$. In the following 
lemma we give the basic facts about supersymmetric Schur polynomials, which 
are of interest to us in this paper. For the proofs see~\cite{F-P, McD} and 
the references therein. Let $\Gamma(a,b)$ be the set of 
partitions $\alpha$ such that 
$\alpha_j\leq b$ for all $j>a$.
\begin{lem} 
\label{lem:ss}
We have
\begin{enumerate}
\item If $\alpha\not\in\Gamma(a,b)$, then $\pi_{\alpha}(\underline{x},\underline{y})=0$.
\item The set $\{\pi_{\alpha}(\underline{x},\underline{y})\,|\,\alpha\in\Gamma(a,b)\}$ is a linear 
basis of $R_{a,b}^{ss}$.
\item We have 
$$\pi_{\alpha}(\underline{x},\underline{y})\pi_{\beta}(\underline{x},\underline{y})=\sum_{\gamma}C_{\alpha\beta}^{\gamma}\pi_{\gamma}(\underline{x},\underline{y}),$$ 
where $C_{\alpha\beta}^{\gamma}$ are the Littlewood-Richardson coefficients.
\item We have 
$$\pi_{\alpha}(\underline{x},\underline{y})=(-1)^{|\alpha|}\pi_{\alpha'}(\underline{y},\underline{x}),$$ 
where $|\alpha|=\sum_i\alpha_i$ and $\alpha'$ is the conjugate partition.   
\item We also get the ordinary Schur polynomials as special cases 
\begin{align*}
\pi_{\alpha}(\underline{x},0) &= \pi_{\alpha}(\underline{x})
\\
\pi_{\beta}(0,\underline{y}) &= (-1)^{|\beta|}\pi_{\beta'}(\underline{y}).
\end{align*}
\end{enumerate}
\end{lem}

In \cite{K-L-M-S} the {\em extended calculus} in 
${\mathcal U}(\mathfrak{sl}_2)$ was 
developed. Here we only use a little part of it. 
Below, for partitions $\alpha,\beta$ with length $m$, we write 
$\alpha^{\spadesuit}=\alpha-(a-b)-m$ for counter-clockwise oriented bubbles of 
thickness $m$ in a region labeled $(a,b)$, 
and $\beta^{\spadesuit}=\beta+(a-b)-m$ for clockwise oriented 
bubbles of thickness $m$. Recall that thick bubbles
labeled by a spaded Schur polynomial can be written as Giambelli type 
determinants (see Equations (3.33) and (3.34) in~\cite{K-L-M-S}, but bear 
our sign conventions in mind):
\begin{equation} 
\label{eq_ccbub_det}
 \xy
 (0,0)*{\stccbub{m}{\alpha}};
 (0,9)*{\scs (a,b)};
 \endxy \quad := \quad 
\left|
\begin{array}{ccccc}
  \xy 0;/r.18pc/: (0,0)*{\laudaccbub{\spadesuit+\alpha_1}{}};(5,7)*{\scs (a,b)}; \endxy &
  \xy 0;/r.18pc/: (0,0)*{\laudaccbub{\spadesuit+\alpha_1+1}{}};(5,7)*{\scs (a,b)}; \endxy &
  \xy 0;/r.18pc/: (0,0)*{\laudaccbub{\spadesuit+\alpha_1+2}{}};(5,7)*{\scs (a,b)};
\endxy & \cdots &
  \xy 0;/r.18pc/:
(0,0)*{\laudaccbub{\spadesuit+\alpha_1+(m-1)}{}};(5,7)*{\scs (a,b)}; \endxy \\ \\
 \xy 0;/r.18pc/: (0,0)*{\laudaccbub{\spadesuit+\alpha_2-1}{}};(5,7)*{\scs (a,b)}; \endxy &
  \xy 0;/r.18pc/: (0,0)*{\laudaccbub{\spadesuit+\alpha_2}{}};(5,7)*{\scs (a,b)}; \endxy &
  \xy 0;/r.18pc/: (0,0)*{\laudaccbub{\spadesuit+\alpha_2+1}{}};(5,7)*{\scs (a,b)};
\endxy & \cdots &
  \xy 0;/r.18pc/:
(0,0)*{\laudaccbub{\spadesuit+\alpha_2+(m-2)}{}};(5,7)*{\scs (a,b)}; \endxy \\ \\
  \cdots & \cdots & \cdots & \cdots & \cdots \\ \\
 \xy 0;/r.18pc/: (0,0)*{\laudaccbub{\spadesuit+\alpha_m-m+1}{}};(5,7)*{\scs (a,b)}; \endxy &
  \xy 0;/r.18pc/: (0,0)*{\laudaccbub{\spadesuit+\alpha_m-m+2}{}};(5,7)*{\scs (a,b)}; \endxy &
  \xy 0;/r.18pc/: (0,0)*{\laudaccbub{\spadesuit+\alpha_m-m+3}{}};(5,7)*{\scs (a,b)};
\endxy & \cdots &
  \xy 0;/r.18pc/: (0,0)*{\laudaccbub{\spadesuit+\alpha_m}{}};(5,7)*{\scs (a,b)}; \endxy
\end{array}
\right|
\end{equation}

\begin{equation} \label{eq_cbub_det}
 \xy
 (0,0)*{\stcbub{m}{\beta}};
 (0,9)*{\scs (a,b)};
 \endxy \quad := \quad 
\left|
\begin{array}{ccccc}
  \xy 0;/r.18pc/: (0,0)*{\laudacbub{\spadesuit+\beta_1}{}};(5,7)*{\scs (a,b)}; \endxy &
  \xy 0;/r.18pc/: (0,0)*{\laudacbub{\spadesuit+\beta_1+1}{}};(5,7)*{\scs (a,b)}; \endxy &
  \xy 0;/r.18pc/: (0,0)*{\laudacbub{\spadesuit+\beta_1+2}{}};(5,7)*{\scs (a,b)};
\endxy & \cdots &
  \xy 0;/r.18pc/: (0,0)*{\laudacbub{\spadesuit+\beta_1+(m-1)}{}};(5,7)*{\scs (a,b)};
\endxy \\ \\
 \xy 0;/r.18pc/: (0,0)*{\laudacbub{\spadesuit+\beta_2-1}{}};(5,7)*{\scs (a,b)}; \endxy &
  \xy 0;/r.18pc/: (0,0)*{\laudacbub{\spadesuit+\beta_2}{}};(5,7)*{\scs (a,b)}; \endxy &
  \xy 0;/r.18pc/: (0,0)*{\laudacbub{\spadesuit+\beta_2+1}{}};(5,7)*{\scs (a,b)};
\endxy & \cdots &
  \xy 0;/r.18pc/: (0,0)*{\laudacbub{\spadesuit+\beta_2+(m-2)}{}};(5,7)*{\scs (a,b)};
\endxy \\ \\
  \cdots & \cdots & \cdots & \cdots & \cdots \\ \\
 \xy 0;/r.18pc/: (0,0)*{\laudacbub{\spadesuit+\beta_m-m+1}{}};(5,7)*{\scs (a,b)}; \endxy &
  \xy 0;/r.18pc/: (0,0)*{\laudacbub{\spadesuit+\beta_m-m+2}{}};(5,7)*{\scs (a,b)}; \endxy &
  \xy 0;/r.18pc/: (0,0)*{\laudacbub{\spadesuit+\beta_m-m+3}{}};(5,7)*{\scs (a,b)};
\endxy & \cdots &
  \xy 0;/r.18pc/: (0,0)*{\laudacbub{\spadesuit+\beta_m}{}};(5,7)*{\scs (a,b)}; \endxy
\end{array}
\right|.
\end{equation}

The reader unfamiliar with \cite{K-L-M-S} can interpret the above simply as definitions. 
In Proposition 4.10 in \cite{K-L-M-S} it is proved that the 
clockwise thick bubbles form a 
linear basis of $\ENDU(1_{a-b})$ and that they obey the 
Littlewood-Richardson rule under multiplication. Of course the 
counter-clockwise thick bubbles form another basis and also obey the 
L-R rule. Proposition 4.10 in \cite{K-L-M-S} also shows the relation between 
the two bases (recall that we have slightly different sign conventions in 
this paper and that $\alpha'$ is the partition conjugate to $\alpha$):
\begin{equation}
\label{eq:c-cc}
 \xy
 (0,0)*{\stcbub{m}{\alpha}};
 (0,9)*{\scs (a,b)};
 \endxy \quad =(-1)^{|\alpha|+m}\quad
 \xy
 (0,0)*{\stccbub{m}{\alpha'}};
 (0,9)*{\scs (a,b)};
 \endxy. 
\end{equation}
Therefore, in our case the non-zero clockwise thick bubbles also form a 
nice basis of $\ENDS(1_{(a,b)})$. 
\begin{lem} 
\label{lem:superschur} 
$\fbim\colon \ENDS(1_{(a,b)})\to R_{a,b}^{ss}$ is a ring isomorphism, 
mapping the clockwise thick bubbles to the corresponding supersymmetric Schur 
polynomials. 
\end{lem}  
\begin{proof}
It is clear that the thick bubbles generate $\ENDS(1_{(a,b)})$, 
because they are the image of the thick bubbles in $\ENDU(1_{a-b})$, 
which form a linear basis. Since $\Psi_{n,d}$ is a $2$-functor, we see that 
the multiplication of bubbles in $\ENDS(1_{(a,b)})$ satisfies the 
Littlewood-Richardson rule. In Section~\ref{sec:2rep} we showed that using $\fbim$ we get  
\begin{align*} 
\xy 0;/r.18pc/:
 (0,0)*{\dblue\cbub{\black r}{\black i}};
  (4,8)*{\scs(a,b)};
 \endxy
\quad
\mapsto &\quad 
(-1)^b e_{-(a-b)+1+r}(\underline{x},\underline{y}).
\end{align*}
This implies that 
\begin{equation*} 
 \xy
 (0,0)*{\stcbub{m}{\beta}};
 (0,9)*{\scs (a,b)};
 \endxy \qquad\mapsto\quad (-1)^{mb}\pi_{\beta}(\underline{x},\underline{y}).
\end{equation*}
Therefore, by Lemma~\ref{lem:ss}, all we have to show is that 
\begin{equation*} 
 \xy
 (0,0)*{\stcbub{m}{\beta}};
 (0,9)*{\scs (a,b)};
 \endxy\qquad = \quad 0 
\end{equation*}
if 
$\beta\not\in\Gamma(a,b)$. We proceed by induction on $m$. 
Note that if $m<a+1$, then $\beta\in\Gamma(a,b)$, so the induction starts at 
$m=a+1$. If $m=a+1$, then $\beta_{a+1}=\beta_m>b$ implies that 
$\beta_i>b$ holds for all $i=1,\ldots,m$, because $\beta$ is a partition. 
Therefore, for any $i=1,\ldots,m$, we have 
$$\beta_i+a-b-m=\beta_i+a-b-(a+1)=\beta_i-b-1\geq 0.$$
Thus the bubble is real and equals zero because its inner region is labeled 
$(-1,a+b+1)\not\in\Lambda(2,d)$.   

Suppose that $m>a+1$ and that the result has 
been proved for bubbles of thickness $<m$. Using induction, we will prove that 
it holds for bubbles of thickness $m$. The 
trouble is that in this case the bubble can be fake, so we cannot repeat the 
argument above. Instead we use a second induction, this time on $\beta_m$. 
Write $\beta'=(\beta_1,\ldots,\beta_{m-1})$. First suppose $\beta_m=0$. Then 

\begin{equation*} 
 \xy
 (0,0)*{\stcbub{m}{\beta}};
 (0,9)*{\scs (a,b)};
 \endxy\qquad = \quad (-1)^b
\xy
 (0,0)*{\stcbub{m-1\mspace{40mu}}{\beta'}};
 (0,9)*{\scs (a,b)};
 \endxy\qquad = \quad 0
\end{equation*}
\noindent by induction on 
$m$. Now suppose $\beta_m>0$. Then we have 
\begin{equation*} 
 \xy
 (0,0)*{\stcbub{m-1\mspace{40mu}}{\beta'}};
 (8,9)*{\scs (a,b)};
 (20,0)*{\dblue\ccbub{\black a-b-1+\beta_m}{}};
 \endxy\ = \quad 0
\end{equation*}
by induction on $m$. By Pieri's rule, the left-hand side equals  
\begin{equation*}
\qquad \sum_{\beta<\gamma\leq\beta+(\beta_m)}
\xy
 (0,0)*{\stcbub{m}{\gamma}};
 (8,9)*{\scs (a,b)};
 \endxy
\ +\quad
 \xy
 (0,0)*{\stcbub{m}{\beta}};
 (8,9)*{\scs (a,b)},
 \endxy
\end{equation*}
where $\beta+(\beta_m)=(\beta_1+\beta_m,\beta_2,\ldots,\beta_{m-1},0)$. 
Note that for any $\beta<\gamma\leq \beta+(\beta_m)$, 
we have $\gamma\not\in\Gamma(a,b)$ and $\gamma_m<\beta_m$. Thus, 
by induction on $\beta_m$, all the thick bubbles labeled with 
$\pi^{\spadesuit}_\gamma$ are zero. This implies that 
\begin{equation*}
\xy
 (0,0)*{\stcbub{m}{\beta}};
 (8,9)*{\scs (a,b)}.
 \endxy\ =\quad 0.
\end{equation*}
\end{proof} 

Note that for bubbles with the opposite orientation we have

\begin{align*} 
\xy 0;/r.18pc/:
 (0,0)*{\dblue\ccbub{\black r}{\black i}};
  (4,8)*{\scs(a,b)};
 \endxy
\
\mapsto &\quad 
(-1)^{b+1} e_{(a-b)+1+r}(y,x).
\end{align*}
This implies that 

\begin{equation} 
 \xy
 (0,0)*{\stccbub{m}{\alpha}};
 (0,9)*{\scs (a,b)};
 \endxy \quad \mapsto \quad (-1)^{m(b+1)}\pi_{\alpha}(y,x).
\end{equation}
This way we get another isomorphism between 
$\ENDS(1_{(a,b)})$ and $R^{ss}_{a,b}$. 

\subsubsection{Bubbles for $n>2$}

For $n>2$, we get polynomials in thick bubbles of $n-1$ colors.  
Unfortunately we have not been able to find anything in the literature on 
a generalization of supersymmetric polynomials to more than two alphabets. 
Nor has the extended calculus for $\Ucat$ been worked out and written up 
for $n>2$ so far. Therefore all we can say is the following.
Let 
$$S\Pi_{\lambda}=\bigotimes_{i=1}^{n-1}R_{\lambda_i,\lambda_{i+1}}^{ss}.$$
There is a surjective homomorphism 
$$S\Pi_{\lambda}\to 
\ENDS(1_{\lambda})$$ 
sending supersymmetric polynomials to the corresponding clockwise oriented 
thick bubbles. Note that 
$\Psi_{n,d}\colon \Pi_{\overline{\lambda}}\cong \ENDU(1_{\overline{\lambda}})\to 
\ENDS(1_{\lambda})$ factors through $S\Pi_{\lambda}$. Recall that 
$$\Pi_{\overline{\lambda}}\cong \bigotimes_{i=1}^{n-1}\Lambda(\underline{x}),$$
where $\Lambda(\underline{x})$ is the ring of symmetric functions in 
infinitely many variables $\underline{x}=(x_1,x_2,\ldots)$ 
(see (3.24) and Lemma 6.15 in \cite{K-L3}). The map  
$\Pi_{\overline{\lambda}}\to S\Pi_{\lambda}$ referred to above is defined by 
$$\pi^i_{\alpha}(\underline{x})\mapsto \pi_{\alpha}(\underline{x},\underline{y}),$$
where 
$(\underline{x},\underline{y})
=(x_1,\ldots,x_{\lambda_i},y_1,\ldots,y_{\lambda_{i+1}})$ 
and $\pi^i_{\alpha}(\underline{x})=1\otimes\cdots\otimes 
\pi_{\alpha}(\underline{x})\otimes\cdots\otimes 1 $ belongs to 
the $i$-th tensor 
factor and . 

Note also that the projection 
$$S\Pi_{\lambda}\to 
\ENDS(1_{\lambda})$$ 
is not an isomorphism in general. For example, with blue bubbles colored 1 and 
red bubbles colored 2, we have   
\begin{align} 
\xy 0;/r.18pc/:
 (0,0)*{\dred\cbub{\black 1}{2}};
  (4,9)*{\scs(0,1,0)};
 \endxy
- \ \
\xy 0;/r.18pc/:
 (0,-1)*{\dblue\cbub{\black -1}{1}};
  (4,8)*{\scs(0,1,0)};
 \endxy
\
=&\quad 0.
\label{eq:twistbub3}
\end{align}
To see why this holds, first use  
\begin{equation} 
\label{eq:non-inj}
\xy 0;/r.18pc/:
 (0,0)*{\dred\cbub{\black 1}{2}};
  (4,9)*{\scs(0,1,0)};
 \endxy
=\quad 
\xy 0;/r.18pc/:
 (0,0)*{\dblue\ccbub{\black 0}{1}};
 (14,0)*{\dred\cbub{\black 1}{2}};
  (7,9)*{\scs(0,1,0)};
 \endxy
\end{equation}
This equation holds because 
$$
\xy 0;/r.18pc/:
(0,0)*{\dblue\ccbub{\black 0}{1}};
(4,9)*{\scs(0,1,0)};
\endxy
=1.
$$ 
Then slide the red bubble inside the blue one on the r.h.s. 
of~\eqref{eq:non-inj} with bubble-slide~\eqref{eq:extrabubble4}. Note that 
we have to switch the colors $i$ and $i+1$ in~\eqref{eq:extrabubble4}, but 
that only changes the sign on the r.h.s. of that bubble-slide, as remarked 
below the list of bubble-slides. After doing that, only one blue bubble with 
one dot survives, because in the 
interior of that bubble, which is labeled $(1,0,0)$, only a 
degree zero red bubble is non-zero. This holds because the red bubbles of 
positive degree are real bubbles and their interior is labeled 
$(1,-1,1)\not\in\Lambda(3,1)$. The degree zero red bubble is equal to 
1, by~\eqref{eq:bubb_deg0}. Thus we have obtained 
\begin{equation}
\xy 0;/r.18pc/:
 (0,0)*{\dred\cbub{\black 1}{2}};
  (4,9)*{\scs(0,1,0)};
 \endxy
=
\xy 0;/r.18pc/:
 (0,0)*{\dblue\ccbub{\black 1}{1}};
  (4,9)*{\scs(0,1,0)};
 \endxy,
\end{equation}
which is equal to 
\begin{equation}
\xy 0;/r.18pc/:
 (0,-1)*{\dblue\cbub{\black -1}{1}};
  (4,8)*{\scs(0,1,0)};
 \endxy
\end{equation}
by the infinite Grassmannian 
relation~\eqref{eq_infinite_Grass} and relation~\eqref{eq:bubb_deg0}. 

The relation above between bubbles of different colors generalizes. 
Using the extended calculus for $\Scat(n,d)$~\cite{K-L-M-S}, 
we can see that 
whenever $\lambda$ is of the form $(\ldots,0,\lambda_i,0,\ldots)$, bubbles of the same degree of 
colors $i-1$ and $i$ are equal up to a sign. 
This also has to do with the fact that compositions of $d$ of the form 
$(\ldots,a,0,\dots)$ and $(\ldots,0,a,\ldots)$ are equivalent as objects 
in the Karoubi envelope $\ScatD(n,d)$. We will explain this in 
Remark~\ref{rem:zeros}. Here we just leave a conjecture about 
$\EndS(1_{\lambda})$.
\begin{conj} 
\label{conj:bubbles}
Let $\lambda\in\Lambda(n,d)$ be arbitrary and let 
$\mu\in\Lambda(n,d)$ be obtained from $\lambda$ by placing all zero 
entries of $\lambda$ at the end, but without changing the relative order of the 
non-zero entries, e.g. for $\lambda=(2,0,1)$ we get $\mu=(2,1,0)$. 
Then we conjecture that 
$$
\EndS(1_{\lambda})\cong S\Pi_{\mu}.
$$
\end{conj}
\noindent Note that if $\mu_k\ne 0$ and $\mu_{k+1}=0$ for a 
certain $1\leq k\leq n-1$ in Conjecture~\ref{conj:bubbles}, then 
$S\Pi_{\mu}$ is isomorphic to the 
algebra of all partially symmetric polynomials 
$\Q[x_1,\ldots,x_d]^{S_{\mu_1}\times\cdots\times S_{\mu_k}}$. This follows from the fact 
that $R^{ss}_{\mu_k,0}$ is the algebra of symmetric polynomials in $\mu_k$ 
variables. For example, suppose $\mu=(1,1,0)$. Then $R^{ss}_{1,1}\cong \Q[x-y]$ 
and $R^{ss}_{1,0}\cong\Q[y]$, so $S\Pi_{(1,1,0)}\cong \Q[x-y]\otimes\Q[y]$. 
The latter algebra is isomorphic to $\Q[x,y]$ by
$$(x-y)\otimes 1+1\otimes y \leftrightarrow x,\qquad 
1\otimes y\leftrightarrow y.$$ 

\subsubsection{More general 2-morphisms}
There is not all that much that we know about more general $2$-hom spaces in 
$\Scat(n,d)$. Let us give a conjecture about an ``analogue'' of Lemma 3.9 
from~\cite{K-L3} for $\Scat(n,d)$. Let $\nu\in\N[I]$ and ${\mathbf i},{\mathbf j}\in\nu$. 
Recall (see Section 2 in \cite{K-L1} and Subsection 3.2.2 in \cite{K-L3}) that 
${}_{\mathbf i}R(\nu)_{\mathbf j}$ is the vector space of upwards 
oriented braid-like diagrams as in $\Ucat$ whose lower boundary is labeled by 
$\mathbf i$ and upper boundary by $\mathbf j$, modulo the braid-like relations 
in $\Ucat$. Note that all strands of such a diagram have labels uniquely 
determined by $\mathbf i$ and $\mathbf j$. Note also that the braid-like 
relations in $\Ucat$ are independent of the weights, so the definition of 
${}_{\mathbf i}R(\nu)_{\mathbf j}$ does not involve weights. Unfortunately, we 
cannot define 
the analogue of ${}_{\mathbf i}R(\nu)_{\mathbf j}$ for $\Scat(n,d)$, 
because there the braid-like diagrams with a region labeled by a weight outside 
$\Lambda(n,d)$ are equal to zero, creating a weight dependence. However, we 
will be able to use ${}_{\mathbf i}R(\nu)_{\mathbf j}$ and the fact that 
$\Psi_{n,d}$ is full. Khovanov and Lauda~(Lemma 3.9, Definition 3.15 and 
the remarks thereafter, and Theorem 1.3 in \cite{K-L3}) 
showed that the obvious map  
$$
\Psi_{{\mathbf i},{\mathbf j},\overline{\lambda}}\colon 
{}_{\mathbf i}R(\nu)_{\mathbf j}\otimes \Pi_{\overline{\lambda}}\to 
\HOMU({\mathcal E}_{\mathbf i}1_{\overline{\lambda}},
{\mathcal E}_{\mathbf j}1_{\overline{\lambda}})
$$
is an isomorphism. Unfortunately it is also impossible to factor  
$\HOMS({\mathcal E}_{\mathbf i}1_{\lambda}, {\mathcal E}_{\mathbf j}1_{\lambda})$ 
so nicely into braid-like diagrams and 
bubbles. For example, let $\lambda=(0,1)$ and 
look at the following reduction to bubble relation  

$$ 
0 \; =\; \text{$\xy 0;/r.18pc/:
  (12,8)*{\scs (0,1)};
  (0,-3)*{\dblue\xybox{
  (-3,-8)*{};(3,8)*{} **\crv{(-3,-1) & (3,1)}?(1)*\dir{>};?(0)*\dir{>};
    (3,-8)*{};(-3,8)*{} **\crv{(3,-1) & (-3,1)}?(1)*\dir{>};
  (-3,-12)*{\bbsid};
  (-3,8)*{\bbsid};
  (3,8)*{}="t1";
  (9,8)*{}="t2";
  (3,-8)*{}="t1'";
  (9,-8)*{}="t2'";
   "t1";"t2" **\crv{(3,14) & (9, 14)};
   "t1'";"t2'" **\crv{(3,-14) & (9, -14)};
   "t2'";"t2" **\dir{-} ?(.5)*\dir{<};}};
   (9,0)*{}; (-7.5,-12)*{\scs i};
 \endxy$} \; = \; -\sum_{f=0}^{1}
   \xy
  (19,4)*{\scs (0,1)};
  (0,0)*{\dblue\bbe{}};(-2,-8)*{\scs i};
  (12,-2)*{\dblue\cbub{\black f}{\black i}};
  (0,6)*{\dblue\bullet}+(6,1)*{\scs 1-f};
 \endxy.
$$
This result generalizes to any $\lambda$, using the extended calculus 
in~\cite{K-L-M-S}. Thus, 
given any $\lambda$, there is an upper bound $t_r$ for the number of dots 
on the arcs of the $r$-strands. Any braid-like diagram in 
$\HOMS({\mathcal E}_{\mathbf i}1_{\lambda},
{\mathcal E}_{\mathbf j}1_{\lambda})$ with more than 
$t_r$ dots on an $r$-colored strand can be written as a linear combination of 
braid-like diagrams whose $r$-strands have $\leq t_r$ dots with 
coefficients in $\ENDS(1_{\lambda})$.
By the fullness of $\Psi_{n,d}$ and the fact that 
${}_{\mathbf i}R(\nu)_{\mathbf j}$ has a basis ${}_{\mathbf i}B_{\mathbf j}$ which 
only contains a finite number of braid-like diagrams if one forgets the dots 
(see Theorem 2.5 in \cite{K-L1}), it follows that 
$\HOMS({\mathcal E}_{\mathbf i}1_{\lambda},
{\mathcal E}_{\mathbf j}1_{\lambda})$ is finitely generated over 
$\ENDS(1_{\lambda})$. In Section~\ref{sec:soergel} 
we will say a little more about 
the image of 
$$B_{{\mathbf i},{\mathbf j},\overline{\lambda}}=
\Psi_{{\mathbf i},{\mathbf j},\overline{\lambda}}
({}_{\mathbf i}B_{\mathbf j})\subseteq \HOMU({\mathcal E}_{\mathbf i}
1_{\overline{\lambda}},{\mathcal E}_{\mathbf j}1_{\overline{\lambda}})$$ 
in $\HOMS({\mathcal E}_{\mathbf i}1_{\lambda},
{\mathcal E}_{\mathbf j}1_{\lambda})$ under $\Psi_{n,d}$.
Recall again that ${}_{\mathbf i}B_{\mathbf j}$ is Khovanov and Lauda's basis 
of ${}_{\mathbf i}R(\nu)_{\mathbf j}$ in Theorem 2.5 in \cite{K-L1}. 
Unfortunately, all we can give for now is a conjecture. 
\begin{conj} 
\label{conj:struct}
We conjecture that 
$\HOMS({\mathcal E}_{\mathbf i}1_{\lambda},
{\mathcal E}_{\mathbf j}1_{\lambda})$ is a free right module of finite rank over 
$\ENDS(1_{\lambda})$.
\end{conj} 
Note that if 
${\mathcal E}_{\mathbf i}1_{\lambda}=1_{\mu}{\mathcal E}_{\mathbf i}$ and 
${\mathcal E}_{\mathbf j}1_{\lambda}=1_{\mu}{\mathcal E}_{\mathbf j}$, 
then we also conjecture that $\HOMS(1_{\mu}{\mathcal E}_{\mathbf i},
1_{\mu}{\mathcal E}_{\mathbf j})$ is a free left module of finite rank over 
$\ENDS(1_{\mu})$. However, it is not hard to give examples which show that, 
if the conjectures are true at all, the ranks of 
$\HOMS(1_{\mu}{\mathcal E}_{\mathbf i}1_{\lambda},1_{\mu}{\mathcal E}_{\mathbf j}1_{\lambda})$ as a right $\ENDS(1_{\lambda})$-module and as 
a left $\ENDS(1_{\mu})$-module are not equal in general. 
This is not surprising, because the graded 
dimensions of $\ENDS(1_{\lambda})$ and $\ENDS(1_{\mu})$ are not equal 
in general either. 

\subsection{The categorical anti-involution}
The last part of this section is dedicated to the categorification of 
the anti-involution $\tau\colon \SD(n,d)\to\SD(n,d)^{\mbox{\scriptsize op}}$ in 
Section~\ref{sec:hecke-schur}. 
We simply follow Khovanov and Lauda's Subsection 3.3.2. Let 
$\Scat(n,d)^{\mbox{\scriptsize coop}}$ denote the 2-category which the same 
objects as $\Scat(n,d)$, but with the directions of the 1- and 2-morphisms 
reversed. We define a strict degree preserving 2-functor 
$\tilde{\tau}\colon \Scat(n,d)\to \Scat(n,d)^{\mbox{\scriptsize coop}}$ by 
\begin{align*}
\lambda&\mapsto\lambda\\
1_{\mu}\mathcal{E}_{s_1}\mathcal{E}_{s_2}\cdots\mathcal{E}_{s_{m-1}}
\mathcal{E}_{s_m}1_{\lambda}\{t\}&\mapsto 
1_{\lambda}\mathcal{E}_{-s_m}\mathcal{E}_{-s_{m-1}}\cdots\mathcal{E}_{-s_2}
\mathcal{E}_{-s_1}1_{\mu}\{-t+t'\}\\
\zeta&\mapsto\zeta^*.
\end{align*}
Let $D$ be a diagram, then $D^*$ is obtained from $D$ by rotating the latter  
$180^{\circ}$. Since $\Scat(n,d)$ is cyclic, it does not matter in which way 
you rotate. By linear extension this defines $\zeta^*$ for any $2$-morphism. 
The shift $t'$ is defined by requiring that $\tilde{\tau}$ be degree 
preserving. One can easily check that $\tilde{\tau}$ is well-defined. 
For more details on the analogous $\tilde{\tau}$ defined on $\Ucat$ 
see Subsection 3.3.2 in \cite{K-L3}. As a matter 
of fact $\tilde{\tau}$ is a functorial anti-involution. 
The most important result about $\tilde{\tau}$ is the analogue 
of Remark 3.20 in~\cite{K-L3}.
\begin{lem}
\label{lem:tildetau}
There are degree zero isomorphisms of graded $\Q$-vector spaces
\begin{align*}
\HOMS(fx,y)&\cong \HOMS(x,\tilde{\tau}(f)y)\\
\HOMS(xg,y)&\cong \HOMS(x,y\tilde{\tau}(g)),
\end{align*} 
for any $1$-morphisms $x,y,f,g$.
\end{lem}

\section{The diagrammatic Soergel categories and $\Scat(n,d)$}  
\label{sec:soergel}                                             

\subsection{The diagrammatic Soergel category revisited}
\label{ssec:soergel}

In this subsection we recall the diagrammatics for Soergel categories 
introduced by Elias and Khovanov in~\cite{E-Kh}. Actually we first 
recall the version sketched by Elias and Khovanov in Section 4.5 and 
used by Elias and Krasner in~\cite{E-Kr}. After that we will comment on how 
to alter it in order to get the original version by Elias and Khovanov. 
Note that both versions categorify the Hecke algebra, although they are 
not equivalent as categories. In this paper we will need both versions.  

Fix a positive integer $n$. The category $\mathcal{SC}_1(n)$ is the category 
whose objects are finite length sequences of points on the real line, where 
each point is colored by an 
integer between $1$ and $n-1$. 
We read sequences of points from left to right. 
Two colors $i$ and $j$ are called {\em adjacent} 
if $\vert i-j\vert=1$ and {\em distant} if 
$\vert i-j\vert >1$.  
The morphisms of $\mathcal{SC}_1(n)$ are given by generators modulo relations. 
A morphism of $\mathcal{SC}_1(n)$ is a $\Q$-linear combination of planar 
diagrams 
constructed by horizontal and vertical gluings of the following generators 
(by convention no label means a generic color $j$):
\begin{itemize}
\item Generators involving only one color:
\begin{equation*}
\xymatrix@R=1.0mm{
\figins{-15}{0.5}{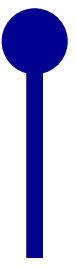}
&
\figins{-9}{0.5}{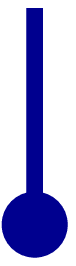}
&
\figins{-15}{0.55}{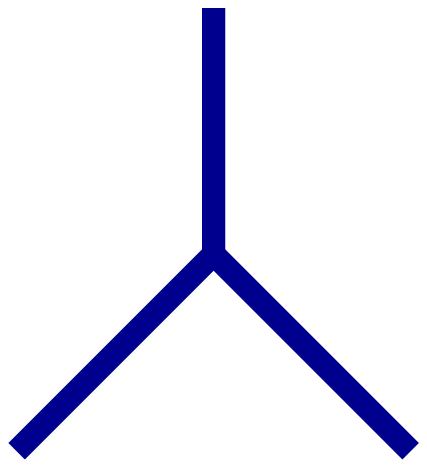}
&
\figins{-15}{0.55}{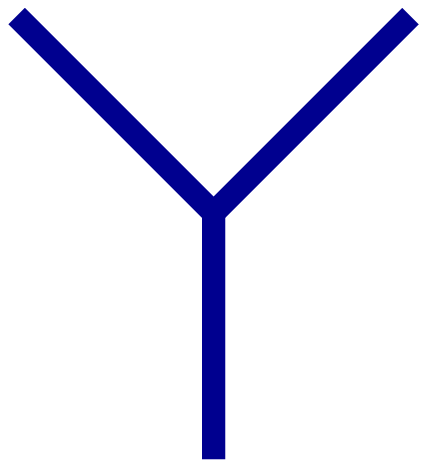}
\\
\text{EndDot} & \text{StartDot} & \text{Merge} & \text{Split}
}
\end{equation*}

\medskip

It is useful to define the cap and cup as
\begin{equation*}
\figins{-17}{0.55}{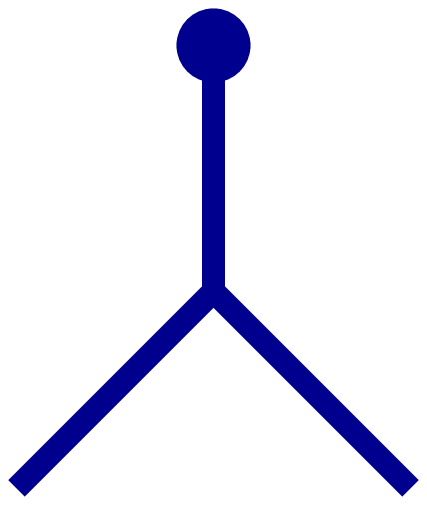}\
\equiv\
\figins{-17}{0.55}{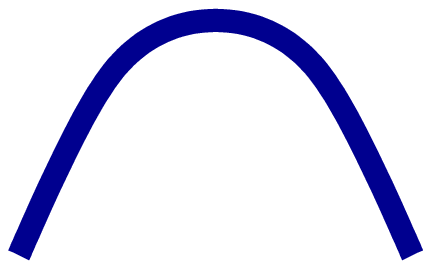}\
\mspace{50mu}
\figins{-17}{0.55}{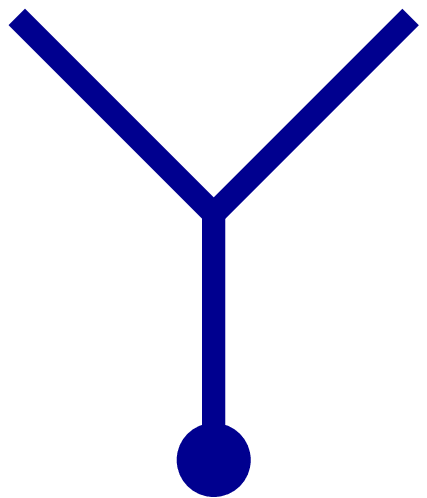}\
\equiv\
\figins{-17}{0.55}{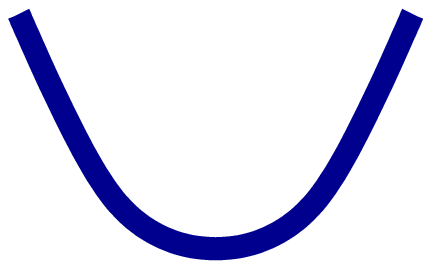}
\end{equation*}

\medskip

\item Generators involving two colors:
\begin{itemize}
\item The 4-valent vertex, with distant colors,
\begin{equation*}
\labellist
\tiny\hair 2pt
\pinlabel $i$ at  -4 -10
\pinlabel $j$ at 134 -10
\pinlabel $i$ at 134 140
\pinlabel $j$ at -4 140
\endlabellist
\figins{-15}{0.55}{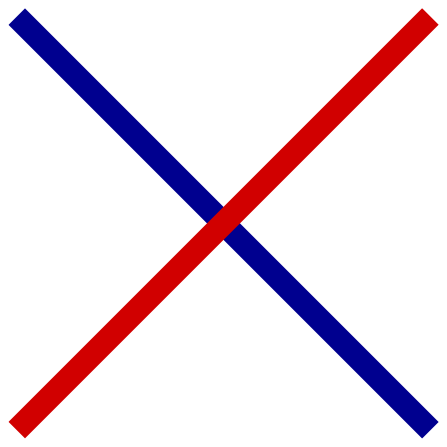}\vspace{1.5ex}
\end{equation*}
\item and the 6-valent vertex, with adjacent colors $i$ and $j$
\begin{equation*}
\labellist
\tiny\hair 2pt
\pinlabel $i$   at  -4 -10 
\pinlabel $j$ at  66 -12
\pinlabel $i$ at 136 -10
\pinlabel $j$ at -4 140
\pinlabel $i$ at 66 140
\pinlabel $j$ at 136 140
\pinlabel $j$ at 250 -10
\pinlabel $i$ at 320 -10
\pinlabel $j$ at 390 -10
\pinlabel $i$ at 250 140
\pinlabel $j$ at 320 140
\pinlabel $i$ at 390 140
\endlabellist
\figins{-17}{0.55}{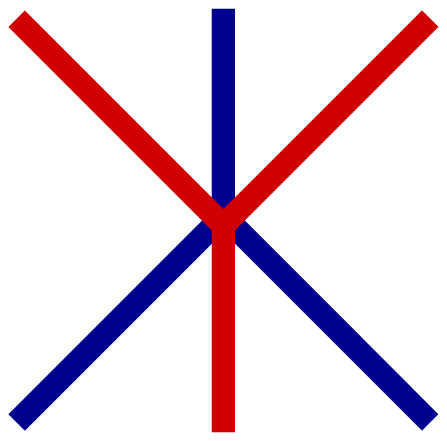}
\mspace{55mu}
\figins{-17}{0.55}{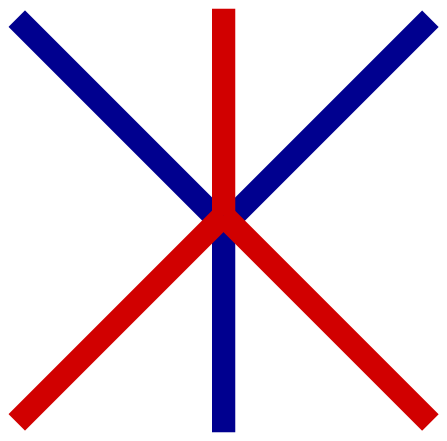}\ .
\vspace{1.5ex}
\end{equation*}
\end{itemize}
\end{itemize}

\n In this setting a diagram represents a morphism from the 
bottom boundary to the top.
We can add a new colored point to a sequence and this endows $\mathcal{SC}_1(n)$ with a 
monoidal structure on objects, which is extended to morphisms in the obvious way. 
Composition of morphisms consists of stacking one diagram on top of the other.

We consider our diagrams modulo the following relations.

\n\emph{''Isotopy'' relations:}
\begin{equation}\label{eq:adj}
\figins{-17}{0.55}{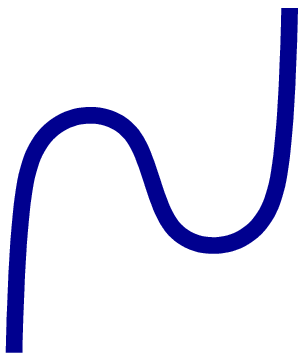}\
=\
\figins{-17}{0.55}{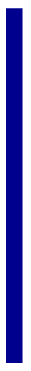}\
=\
\figins{-17}{0.55}{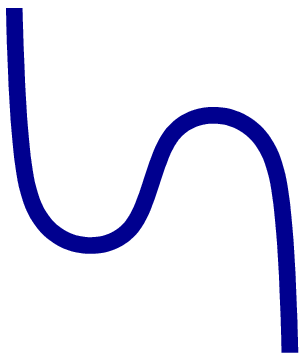}
\end{equation}

\begin{equation}\label{eq:curldot}
\figins{-17}{0.55}{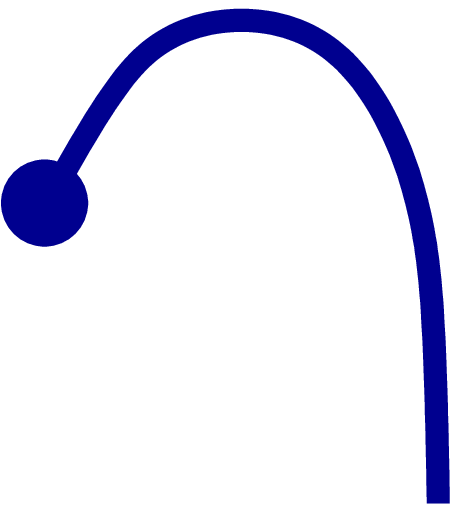}\
=\
\figins{-17}{0.55}{enddot}\
=\
\figins{-17}{0.55}{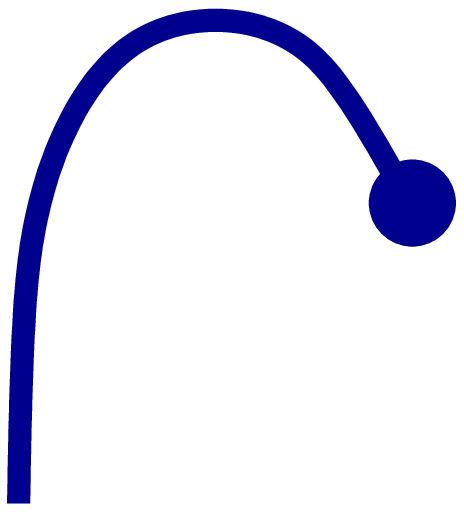}
\end{equation}

\begin{equation}\label{eq:v3rot}
\figins{-17}{0.55}{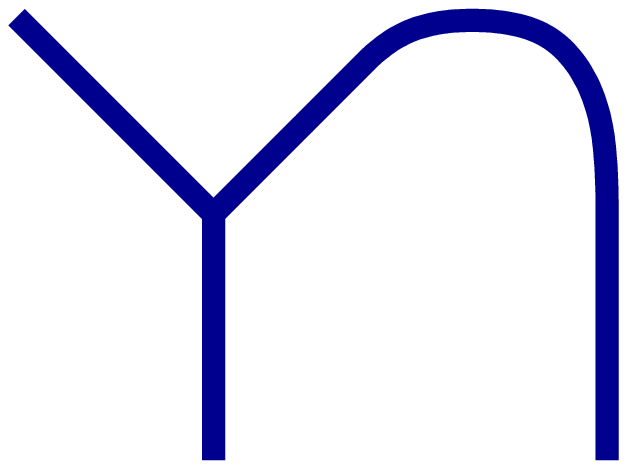}\
=\
\figins{-17}{0.55}{merge}\
=\
\figins{-17}{0.55}{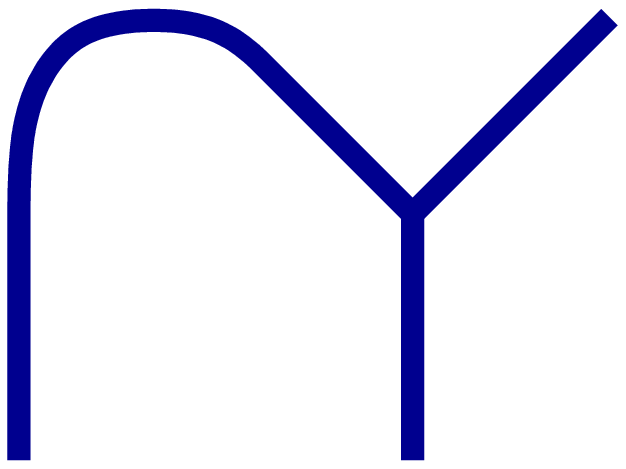}
\end{equation}

\begin{equation}\label{eq:v4rot}
\figins{-17}{0.55}{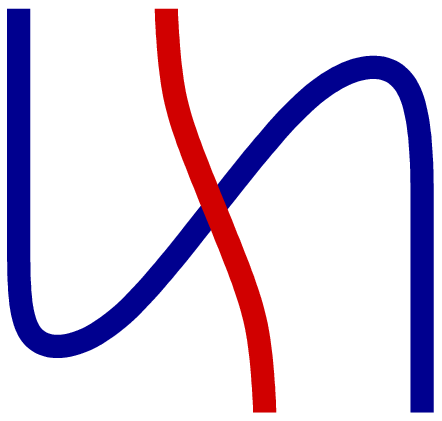}\
=\
\figins{-17}{0.55}{4vert}\
=\
\figins{-17}{0.55}{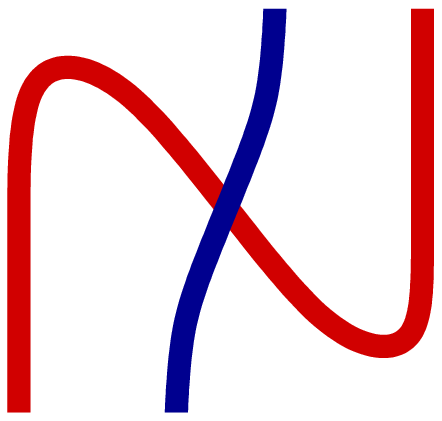}
\end{equation}

\begin{equation}\label{eq:v6rot}
\figins{-17}{0.55}{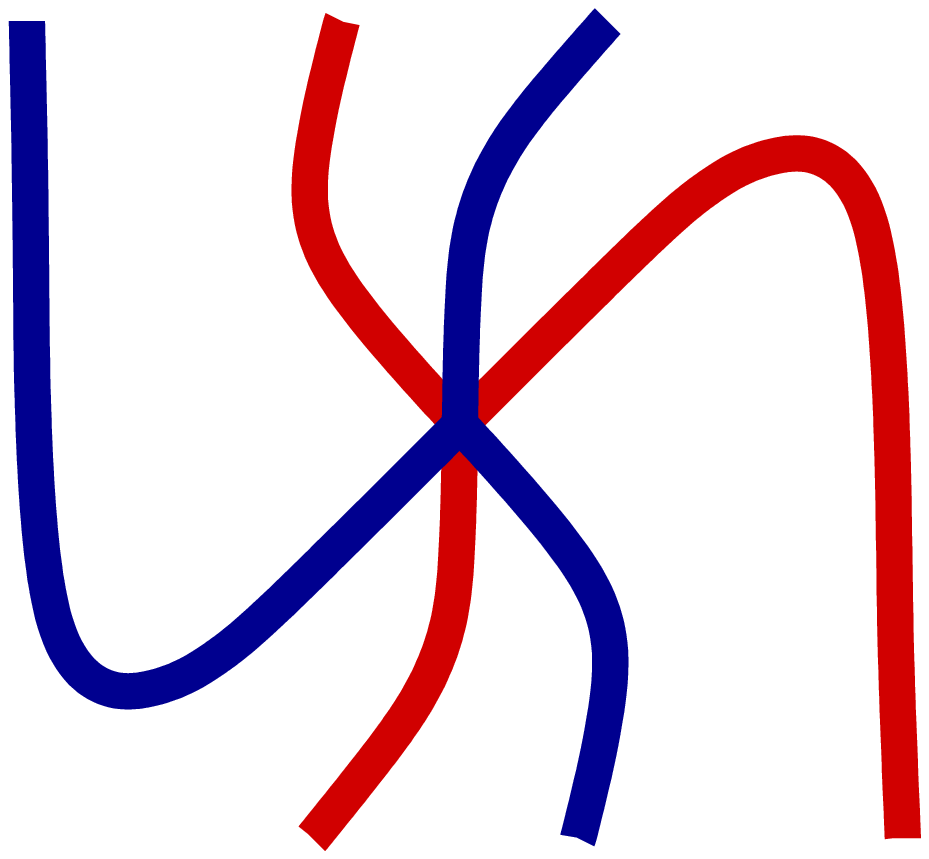}\
=\
\figins{-17}{0.55}{6vertu}\
=\
\figins{-17}{0.55}{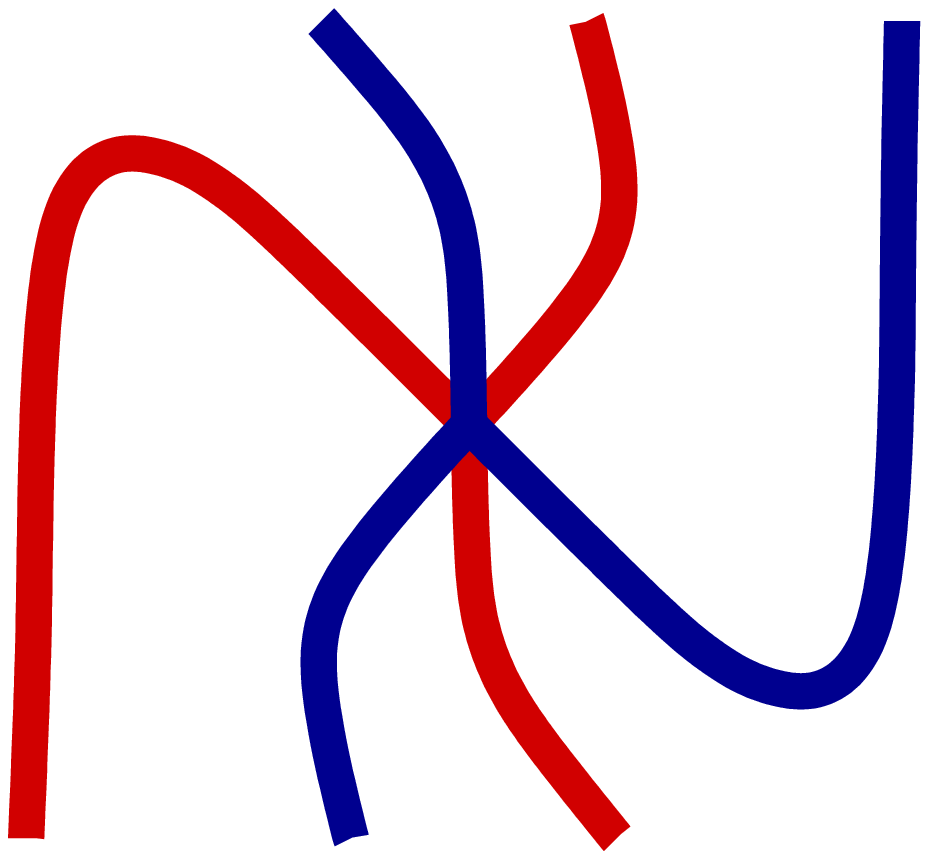}
\end{equation}

The relations are presented in terms of diagrams with generic colorings. 
Because of isotopy invariance, one may draw a diagram with a boundary on the side,
and view it as a morphism in $\mathcal{SC}_1(n)$ by either bending strands up 
or down.
By the same reasoning, a horizontal line corresponds to a sequence of cups and caps.

\bigskip\medskip

\n\emph{One color relations:}

\begin{equation}\label{eq:dumbrot}
\figins{-16}{0.5}{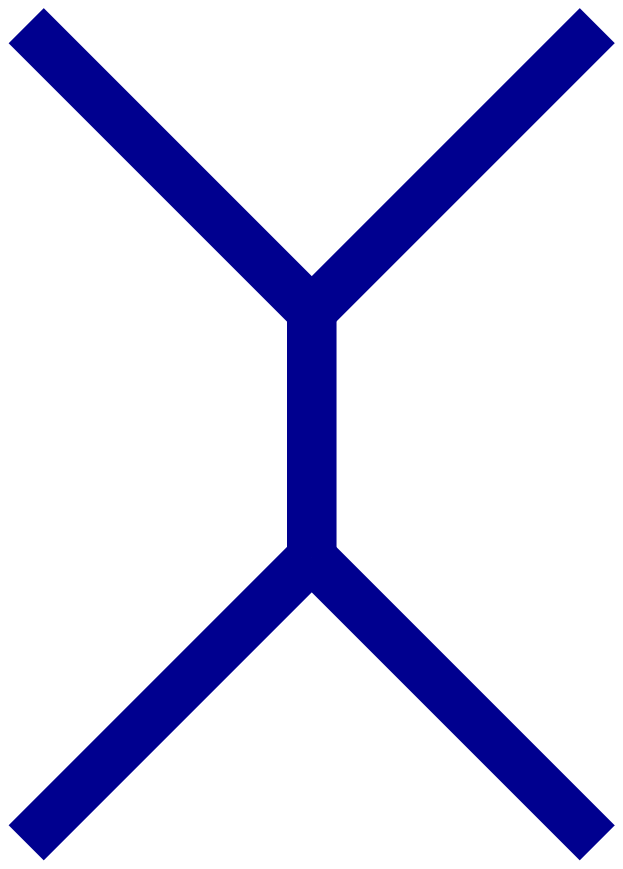}\
=\
\figins{-14}{0.45}{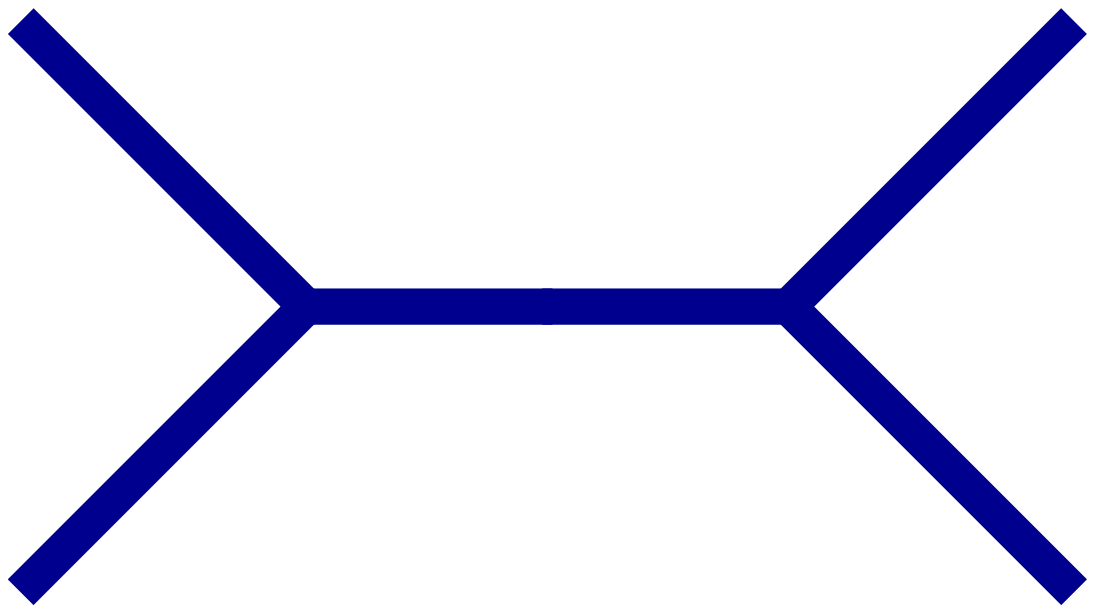}
\end{equation}

\begin{equation}\label{eq:lollipop}
\figins{-17}{0.55}{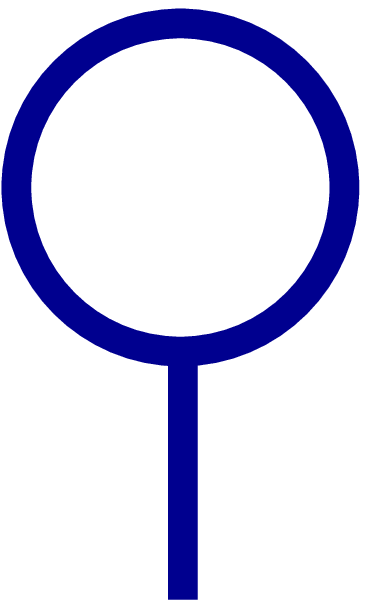}\
=\
0
\end{equation}

\begin{equation}\label{eq:deltam}
\figins{-17}{0.55}{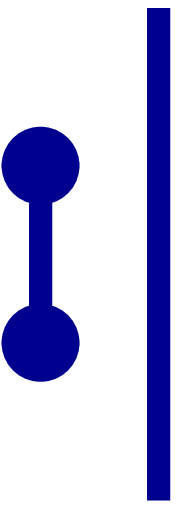}\
+\
\figins{-17}{0.55}{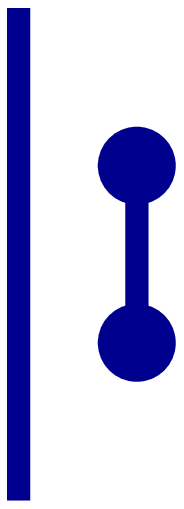}\
=\ 2\ 
\figins{-17}{0.55}{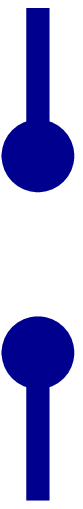}
\end{equation}

\medskip
\n\emph{Two distant colors:}
\begin{equation}\label{eq:reid2dist}
\figins{-32}{0.9}{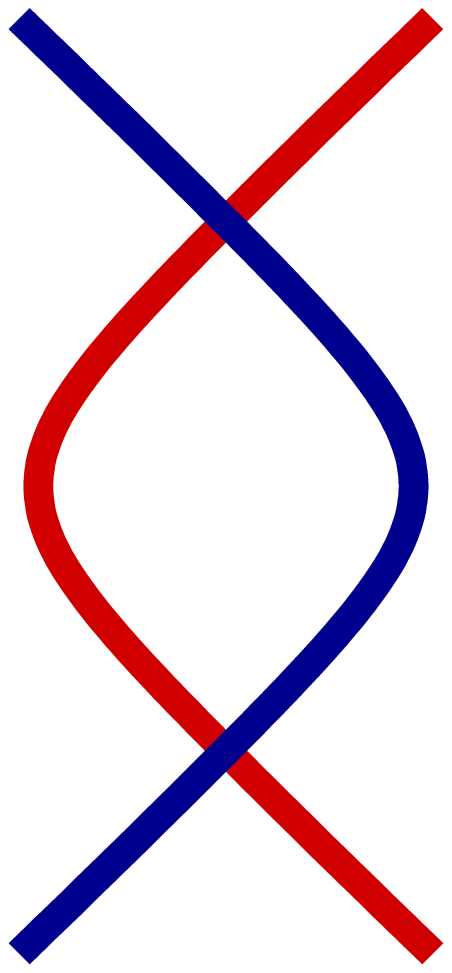}\
=\
\figins{-32}{0.9}{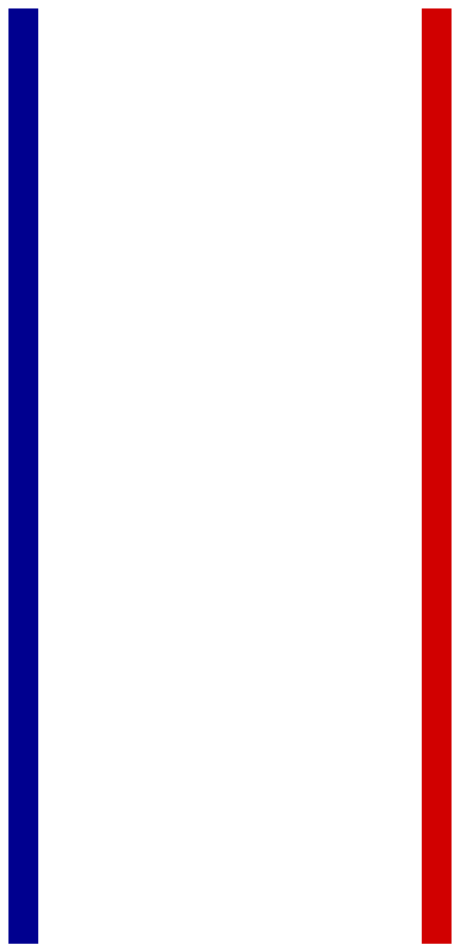}
\end{equation}

\begin{equation}\label{eq:slidedotdist}
\figins{-16}{0.5}{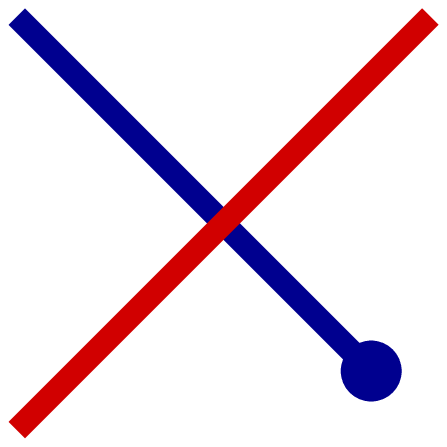}\
=\
\figins{-16}{0.5}{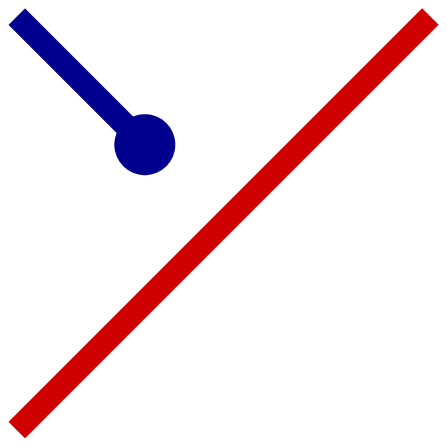}
\end{equation}

\begin{equation}\label{eq:slide3v}
\figins{-17}{0.55}{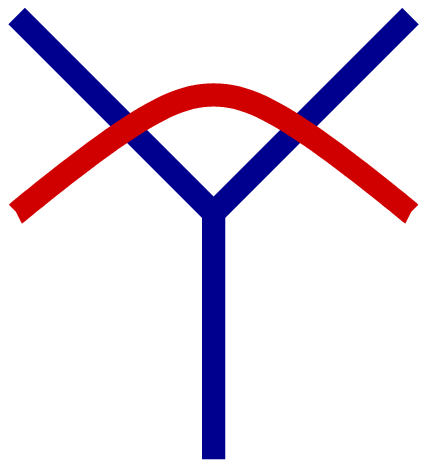}\
=\
\figins{-17}{0.55}{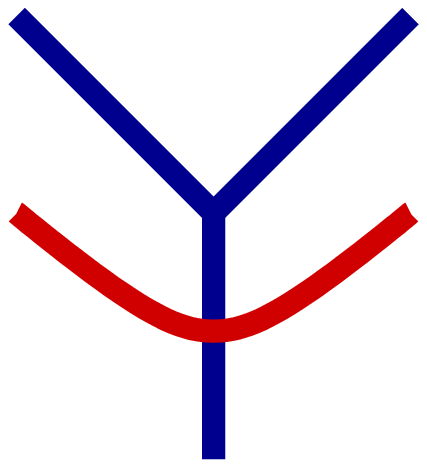}
\end{equation}

\medskip
\n\emph{Two adjacent colors:}
\begin{equation}\label{eq:dot6v}
\figins{-16}{0.5}{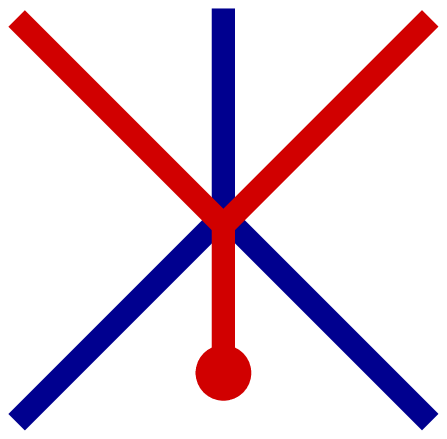}\
=\
\figins{-16}{0.5}{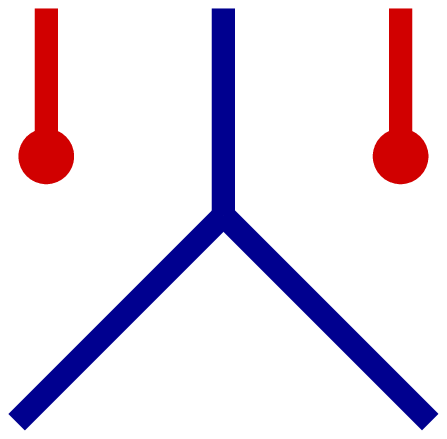}\
+\
\figins{-16}{0.5}{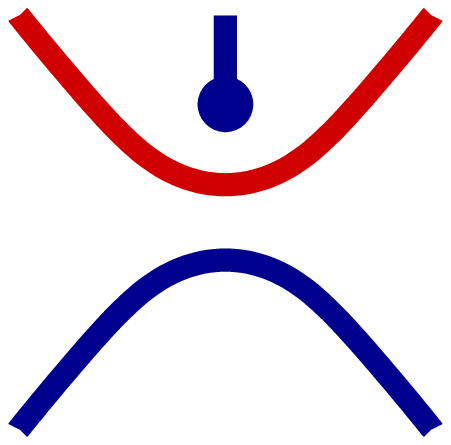}
\end{equation}

\begin{equation}\label{eq:reid3}
\figins{-30}{0.85}{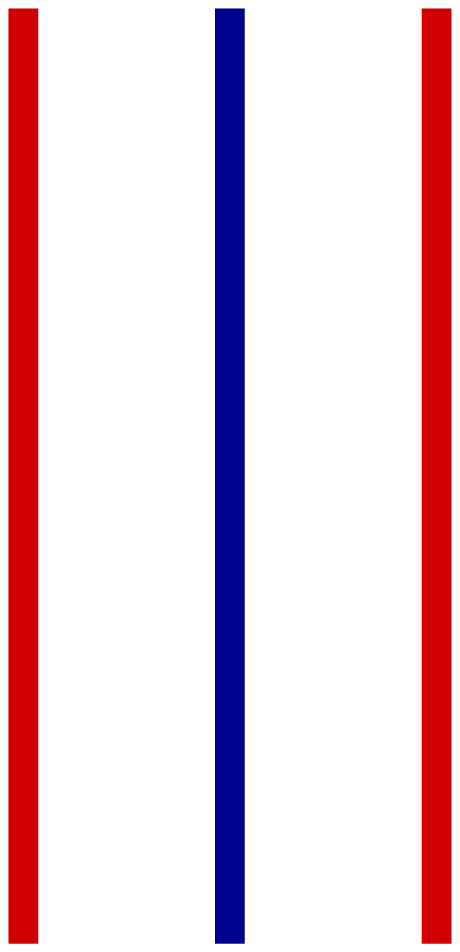}\
=\
\figins{-30}{0.85}{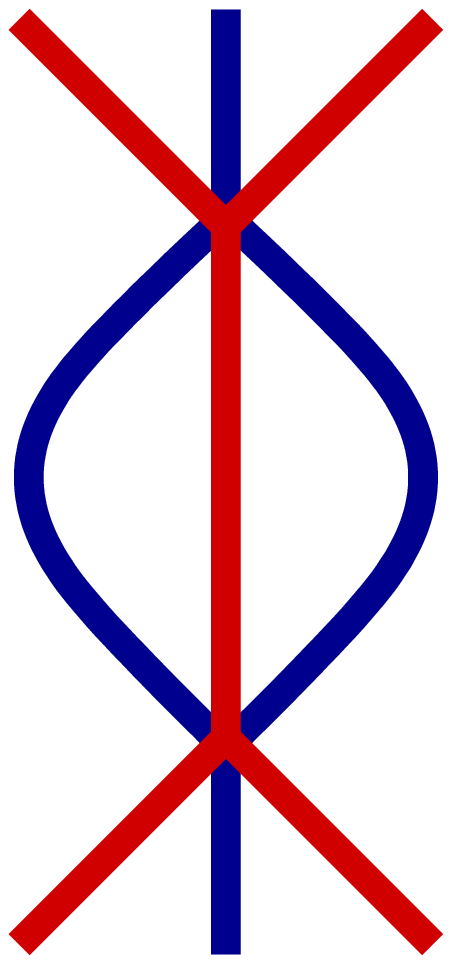}\
-\
\figins{-30}{0.85}{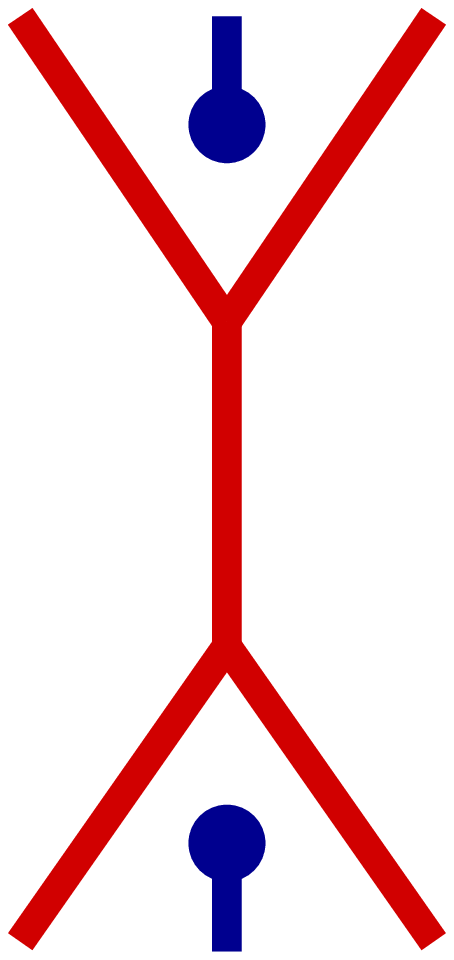}
\end{equation}

\begin{equation}\label{eq:dumbsq}
\figins{-30}{0.85}{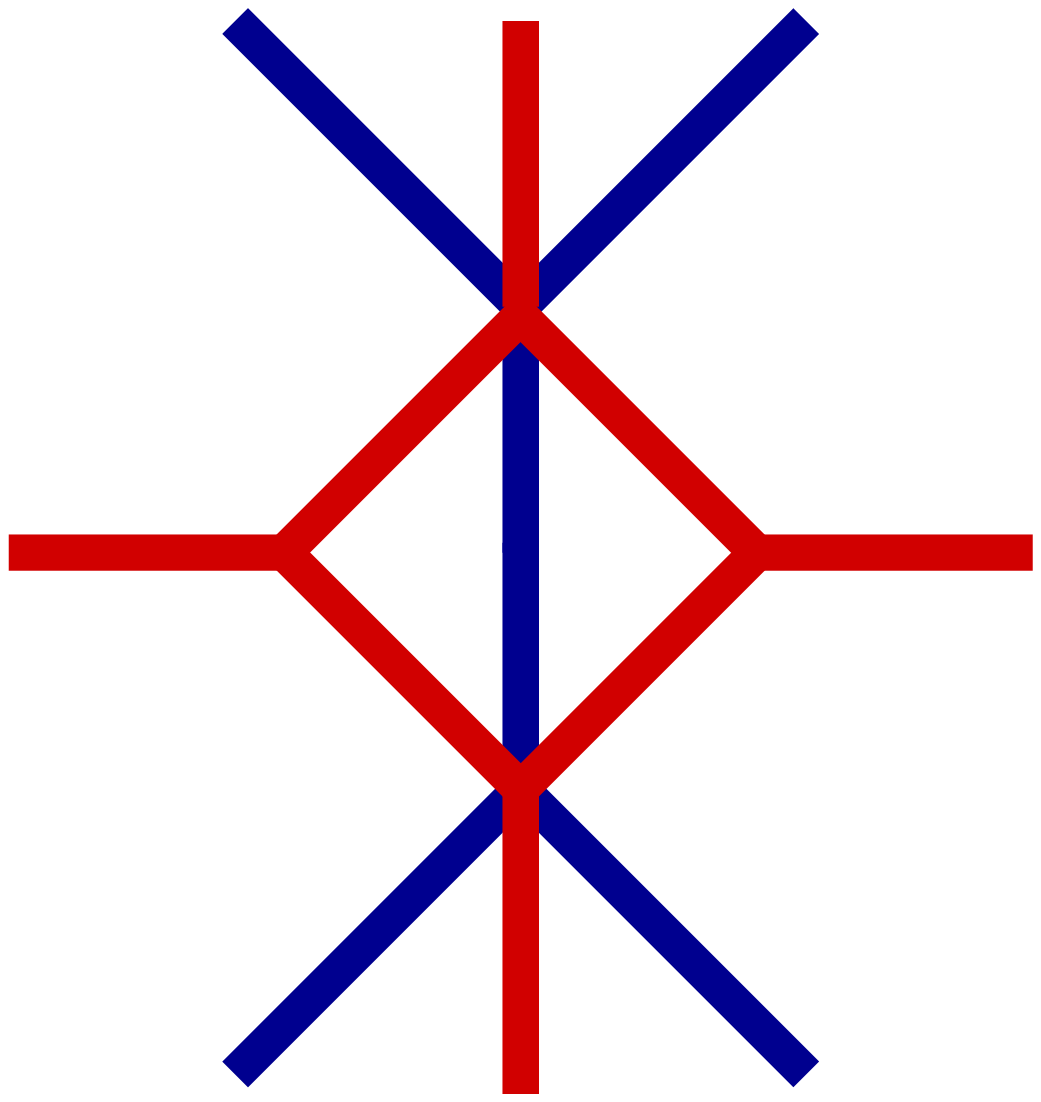}\
=\
\figins{-30}{0.85}{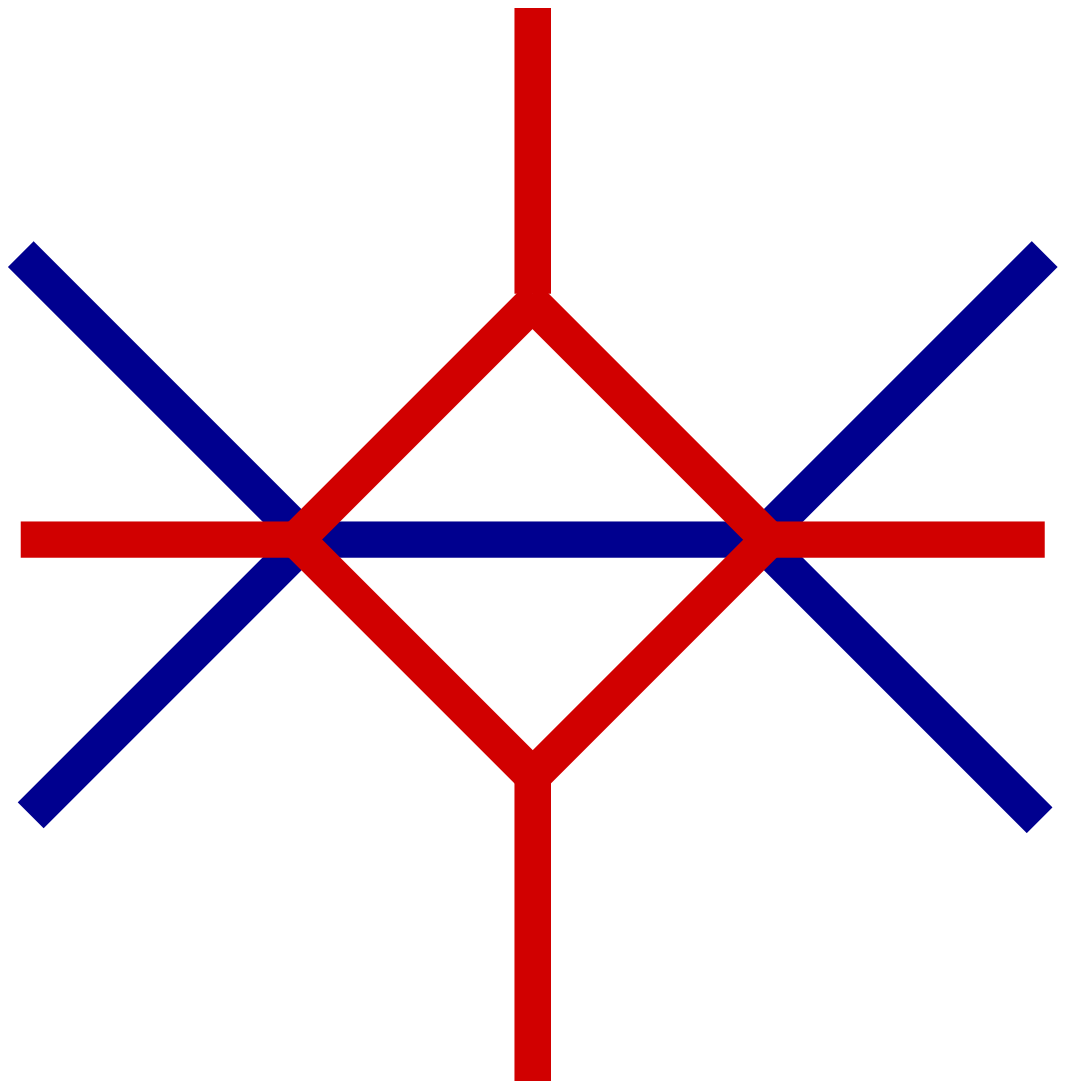}
\end{equation}

\begin{equation}\label{eq:slidenext}
\labellist
\tiny\hair 2pt
\pinlabel $j$ at -15  35
\pinlabel $i$ at  46 -10
\endlabellist
\figins{-17}{0.55}{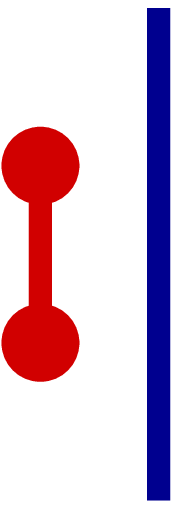}\
-\
\labellist
\tiny\hair 2pt
\pinlabel $j$ at  63  35
\pinlabel $i$ at   5 -10
\endlabellist
\figins{-17}{0.55}{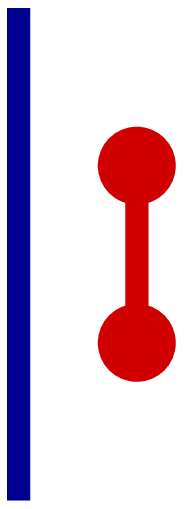}\quad
=\
\frac{1}{2}
\Biggl(\
\labellist
\tiny\hair 2pt
\pinlabel $i$   at  58  35
\pinlabel $i$   at   5 -10
\endlabellist
\figins{-17}{0.55}{edge-startenddot}\
-\
\labellist
\tiny\hair 2pt
\pinlabel $i$   at  -5   35
\pinlabel $i$   at  48 -10
\endlabellist
\figins{-17}{0.55}{startenddot-edge}\
\Biggr)
\end{equation}

\medskip
\n\emph{Relations involving three colors:}
(adjacency is determined by the vertices which appear)
\begin{equation}\label{eq:slide6v}
\figins{-18}{0.6}{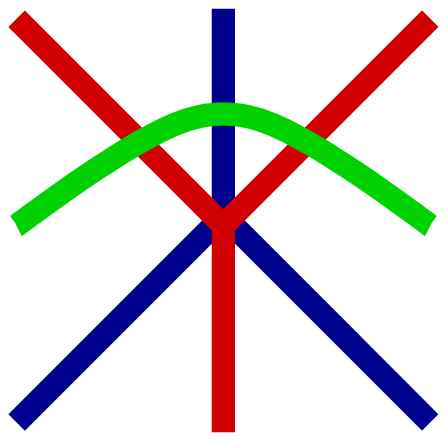}\
=\
\figins{-18}{0.6}{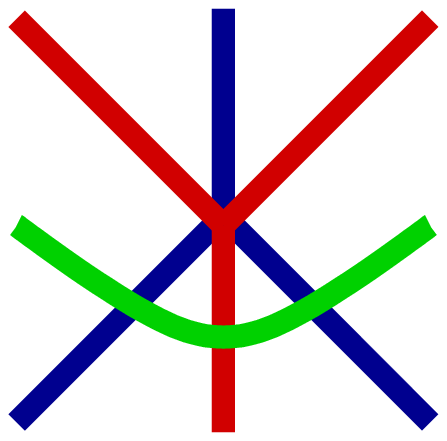}
\end{equation}

\begin{equation}\label{eq:slide4v}
\figins{-18}{0.6}{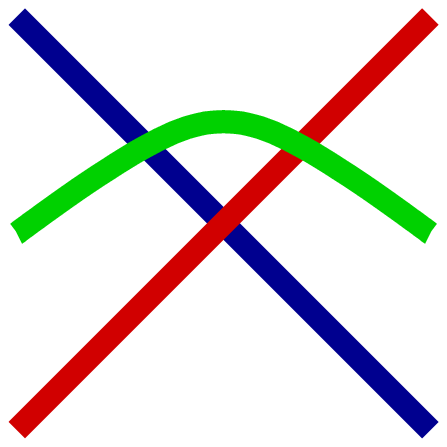}\
=\
\figins{-18}{0.6}{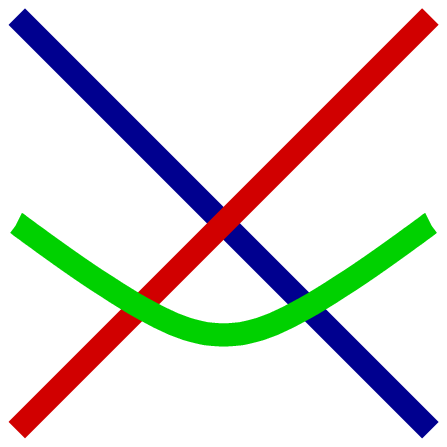}
\end{equation}

\begin{equation}\label{eq:dumbdumbsquare}
\figins{-30}{0.85}{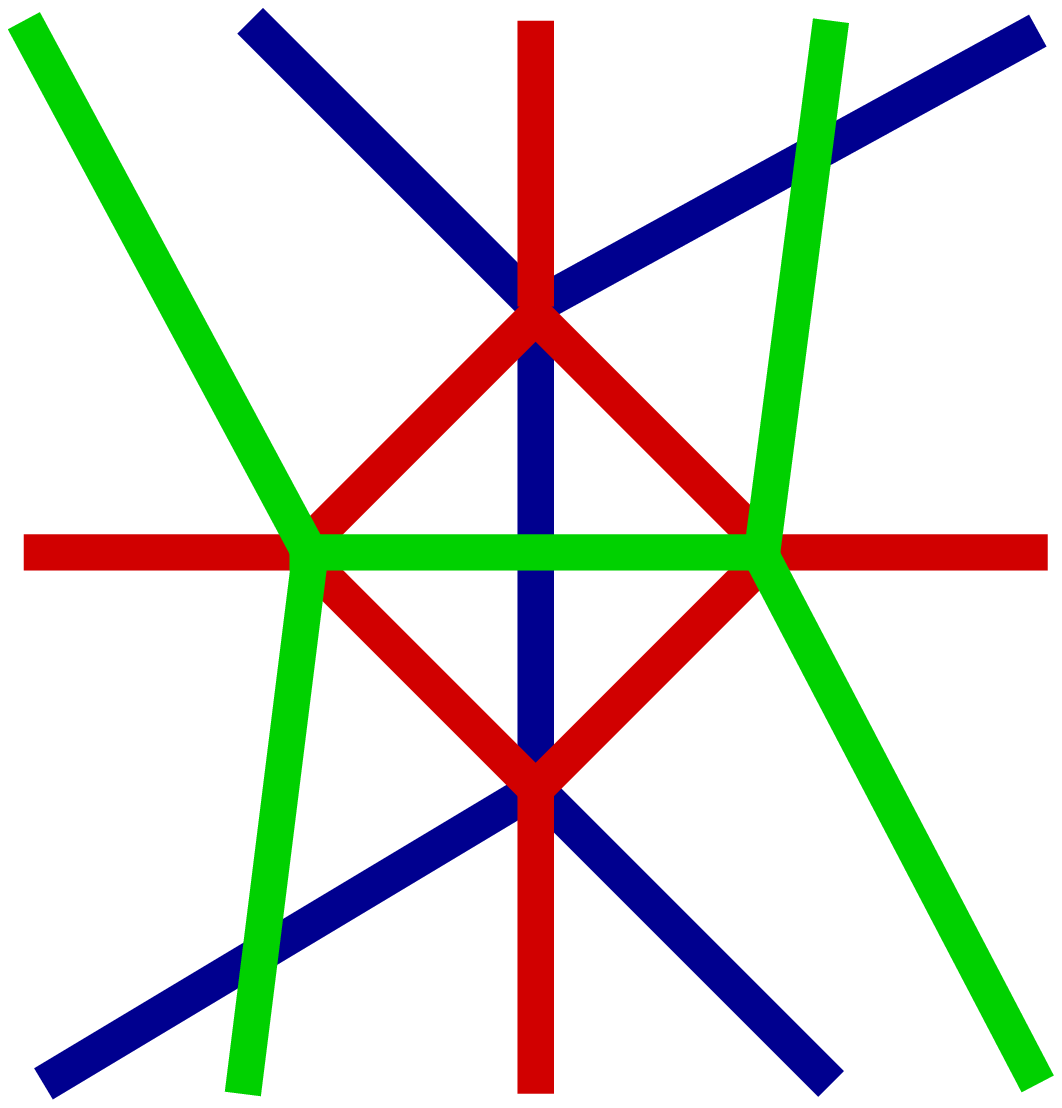}\
=\
\figins{-30}{0.85}{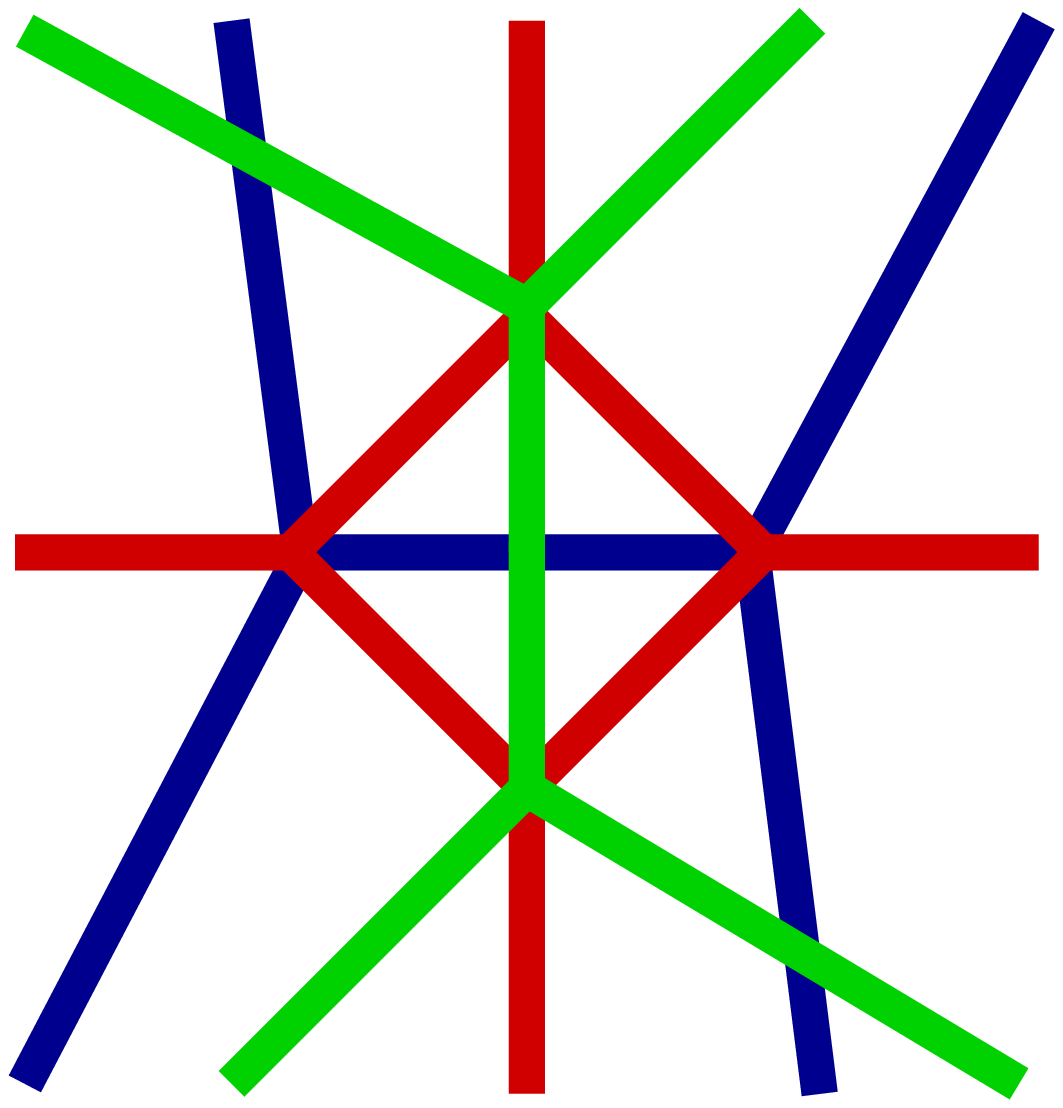}.
\end{equation}

Introduce a grading on $\mathcal{SC}_1(n)$ by declaring dots to have degree 
$1$, trivalent vertices degree $-1$ 
and $4$- and $6$-valent vertices degree $0$.

\begin{defn}
The category $\mathcal{SC}_2(n)$ is the category containing all direct sums and 
grading 
shifts of objects in 
$\mathcal{SC}_1(n)$ and whose morphisms are the grading preserving morphisms from 
$\mathcal{SC}_1(n)$.
\end{defn}

\begin{defn}
The category $\mathcal{SC}(n)$ is the Karoubi envelope of the category 
$\mathcal{SC}_2(n)$.
\end{defn}


\subsection{The extension $\mathcal{SC}'(n)$ of $\mathcal{SC}(n)$}
\label{ssec:usoergel}

In~\cite{E-Kh} Elias and Khovanov give a slightly different diagrammatic 
Soergel category, denoted $\mathcal{SC}'(n)$, which is a faithful extension  
of $\mathcal{SC}(n)$. The objects of $\mathcal{SC}'_1(n)$ are the same
as those of $\mathcal{SC}_1(n)$. The vector spaces of morphisms are an 
extension of the ones of $\mathcal{SC}_1(n)$ in the following sense. 
Regions can be decorated with boxes colored by $i$ 
for $1\leq i\leq n$, which we depict as
\begin{equation*}
{\bbox{i}}
\end{equation*}
For $f$ a polynomial in the set of boxes colored from $1$ to $n$ we use the 
shorthand notation
\begin{equation*}
{\bbox{f}}
\end{equation*}
The set of boxes is therefore in bijection with the polynomial ring in $n$ 
variables. 
Let $s_i$ be the transposition that switches $i$ and $i+1$.
Define the formal symbol
\begin{equation*}
\bbox{\partial_{i}f}
=\bbox{\partial_{x_ix_{i+1}}f}
\end{equation*}
where $\partial_{x_ix_{i+1}}$ was defined in Equation~\eqref{novo}.
This way any box $\bbox{f}$ can be written as
\begin{equation*}
\bbox{f}=
\bbox{P_i(f)} + \bbox{i}\,\bbox{\partial_if}
\end{equation*}
where $P_i(f)$ is a polynomial which is symmetric in $\bbox{i}$ and 
$\bbox{i+1}$ (we will take this formula as a definition of $P_i(f)$).

\medskip

The boxes are related to the previous calculus by the \emph{box relations}
\begin{align}
\label{eq:box1}
\figins{-6}{0.25}{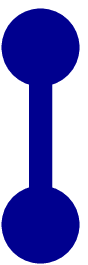}_{\,i}\
&= \
\bbox{i} -\bbox{i+1}
\\[1ex] \displaybreak[0]
\biggl(
\bbox{i}\ 
 +\
\bbox{i+1} 
\biggr)\
\underset{i}{\figins{-20}{0.65}{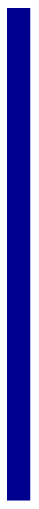}}
&=
\underset{i}{\figins{-20}{0.65}{vedge}}
\biggl(
\bbox{i}\
+\
\bbox{i+1}
\biggr)
\\[1ex] \displaybreak[0]
\bbox{i}\ \bbox{i+1}\
\underset{i}{\figins{-20}{0.65}{vedge}}
&=
\underset{i}{\figins{-20}{0.65}{vedge}}\ \bbox{i}\ \bbox{i+1}
\\[1ex] \displaybreak[0]
\bbox{j}\
\underset{i}{\figins{-20}{0.65}{vedge}} 
&=
\underset{i}{\figins{-20}{0.65}{vedge}}\ \bbox{j}
\rlap{\hspace*{10ex} for $j\neq i, i+1$.}
\end{align}

It is clear that $\mathcal{SC}(n)$ is a faithful monoidal subcategory of 
$\mathcal{SC}'(n)$. As explained in Section 4.5 of~\cite{E-Kh},  
the category $\mathcal{SC}(n)$
is also isomorphic to the quotient of $\mathcal{SC}'(n)$ by the central 
morphism 
\begin{equation*}
\raisebox{1pt}{\bbox{e_1}}
\overset{def}{=}
\sum\limits_{i=1}^n{\bbox{i}}.
\end{equation*}
This result depends subtly on the base field, which in our case is 
$\bQ$.

The category $\mathcal{SC}'_1(n)$ has a grading induced by the one of 
$\mathcal{SC}_1(n)$, if we declare that a box colored $i$ has degree 2 
for all $1\geq i\geq n$.


\begin{defn}
The category $\mathcal{SC}'_2(n)$ is the category containing all direct sums and grading 
shifts of objects in $\mathcal{SC}'_1(n)$ and whose morphisms are the grading preserving 
morphisms from $\mathcal{SC}'_1(n)$. 
The category $\mathcal{SC}'(n)$ is the Karoubi envelope of the category 
$\mathcal{SC}'_2(n)$.
\end{defn}

Elias and Khovanov's main result in \cite{E-Kh} is that 
$\mathcal{SC}(n)$ and $\mathcal{SC}'(n)$ are equivalent to the corresponding 
Soergel categories. A corollary to that is the following theorem, 
where $K_0$ is the split Grothendieck group and  
$K_0^{\bQ(q)}(-)=K_0(-)\otimes_{\bZ[q,q^{-1}]}\bQ(q)$. 
\begin{thm}[Elias-Khovanov, Soergel]
\label{thm:e-k-s}
We have 
$$K_0^{\bQ(q)}(\mathcal{SC}(n))\cong K_0^{\bQ(q)}(\mathcal{SC}'(n))\cong H_q(n).$$
\end{thm}
\n As explained in~\cite{E-Kh}, this result also depends on the fact that 
we are working over $\bQ$.
Recall that $\mathcal{SC}(n)$ and $\mathcal{SC}'(n)$ are monoidal categories, 
with the monoidal structure defined by concatenation. Therefore 
their Grothendieck groups are algebras indeed.

Let $\bim(n)^*=\mbox{End}_{\bim^*}(\Q[x_1,\ldots,x_n])$.
Elias and Khovanov defined functors from 
$\mathcal{SC}(n)$ and $\mathcal{SC}'(n)$ to $\bim(n)^*$ (see~\cite{E-Kh, E-Kr}) 
which we denote by $\fek$ and $\fek'$ respectively.

\subsection{A functor from $\mathcal{SC}(n)$ to $\Scat(n,n)^*((1^n),(1^n))$}
\label{ssec:scqs}
Let $n\geq 1$ be arbitrary but fixed. In this subsection we define an 
additive $\Q$-linear monoidal functor 
$$\Sigma_{n,n}\colon \mathcal{SC}_1(n)\to \Scat(n,n)^*((1^n),(1^n)),$$
where the target is the monoidal category whose objects are 
the $1$-endomorphisms of $(1^n)$ in $\Scat(n,n)^*$ and whose morphisms 
are the $2$-morphisms between such $1$-morphisms in $\Scat(n,n)^*$. This 
monoidal functor categorifies the homomorphism $\sigma_{n,n}$ from 
Section~\ref{sec:hecke-schur}. 

\n\emph{On objects:} $\Sigma_{n,n}$ sends the empty sequence in 
$\mathcal{SC}_1(n)$ 
to $1_n=1_{(1^n)}$ in $\Scat(n,n)^*$ and 
the one-term sequence $(i)$ to $\mathcal{E}_{-i}\mathcal{E}_{+i}1_n$, 
with $\Sigma_{n,n}(jk)$ given by the 
horizontal composite  
$\mathcal{E}_{-j}\mathcal{E}_{+j}\mathcal{E}_{-k}\mathcal{E}_{+k}1_n$.

\n\emph{On morphisms:}

\begin{itemize}
\item The empty diagram is sent to the 
empty diagram in the region labeled $(1^n)$.

\item The vertical line coloured $i$ is sent to the identity $2$-morphism on 
$\mathcal{E}_{-i}\mathcal{E}_{+i}1_n$.
\begin{equation*}
\labellist
\tiny\hair 2pt
\pinlabel $i$   at -10  60
\endlabellist
\figins{-16}{0.5}{line}\ \
\longmapsto\ \ \
\text{$
 \xy 
 (0,0)*{\dblue\xybox{
 (-5,7);(-5,-7); **\dir{-} ?(.5)*\dir{>}+(2.3,0)*{\scriptstyle{}};}};
 (-6.3,-9)*{\scs i};
 (0,0)*{\dblue\xybox{
 (10,7);(10,-7); **\dir{-} ?(.5)*\dir{<}+(12.3,0)*{\scriptstyle{}};}};
 ( -1.2,-9)*{\scs i};
 (6,0)*{ (1^n)};
  \endxy
$}
\vspace*{2ex}
\end{equation*}

\item The \emph{StartDot} and \emph{EndDot} morphisms are sent to the cup and 
the cap respectively:
\begin{equation*}
\labellist
\tiny\hair 2pt
\pinlabel $i$   at -10  60
\endlabellist
\figins{-16}{0.5}{startdot}
\longmapsto\ 
    {\dblue \xy
    (0,2)*{\bbpef{\black i}};
    (10,2)*{\black (1^n) };
    \endxy}
\mspace{140mu}
\labellist
\tiny\hair 2pt
\pinlabel $i$   at -10  60
\endlabellist
\figins{-16}{0.5}{enddot}
\longmapsto\
    {\dblue \xy
    (0,-2.5)*{\bbcef{\black i}};
    (8,0.5)*{ \black (1^n) };
    \endxy}
\vspace*{2ex}
\end{equation*}

\item \emph{Merge} and \emph{Split} are sent to diagrams involving cups and 
caps:
\begin{equation*}
\labellist
\hair 2pt
\pinlabel $\scs i$   at  45  95
\endlabellist
\figins{-16}{0.6}{merge}
\longmapsto\
\xy 0;/r.16pc/; 
    (0,-1.5)*{\dblue\bbcfe{\black i}};
    (14,4)*{(1^n) };    
    (-12,3)*{};(12,3)*{};
    ( 7.5,2)*{\dblue\xybox{
    (-3,-5)*{}; (-10,8.5) **\crv{(-3,1) & (-10,3)}?(1)*\dir{>};}};
    ( -7.5,2)*{\dblue\xybox{
    ( 3,-5)*{}; ( 10,8.5) **\crv{( 3,1) & ( 10,3)}?(0)*\dir{<};}};
    (-11,-7)*{\scs i};
    ( 11,-7)*{\scs i};
    \endxy
\mspace{80mu}
\labellist
\hair 2pt
\pinlabel $\scs i$   at  45  45
\endlabellist
\figins{-16}{0.6}{split}
\longmapsto \
\xy 0;/r.16pc/; 
    (0,5)*{\dblue\bbpfe{\black i}};
    (14,0)*{(1^n) };    
    (-12,3)*{};(12,3)*{};
    (-7.5,2)*{\dblue\xybox{
    (-3,-5)*{}; (-10,8.5) **\crv{(-3,1) & (-10,3)}?(0)*\dir{<};}};
    ( 7.5,2)*{\dblue\xybox{
    ( 3,-5)*{}; ( 10,8.5) **\crv{( 3,1) & ( 10,3)}?(1)*\dir{>};}};
    (-4,-7)*{\scs i};
    ( 4,-7)*{\scs i};
    \endxy
\vspace*{2ex}
\end{equation*}

\item The \emph{4-valent vertex} with distant colors. For $i$ and $j$ distant we have:
\begin{equation*}
\labellist
\hair 2pt
\pinlabel $\scs j$   at  -5 -12
\pinlabel $\scs i$   at 128 -10
\endlabellist
\figins{-16}{0.6}{4vert}
\longmapsto\ \ \ 
\text{$
\xy 0;/r.16pc/; 
    (14,0)*{(1^n) };    
    ( 0,0)*{\dblue\xybox{
    ( 0,-11)*{}; (-15,8.5) **\crv{( 0,-11) & (-15,8.5)}?(0)*\dir{<};
    ( 5,-11)*{}; (-10,8.5) **\crv{( 5,-11) & (-10,8.5)}?(1)*\dir{>};}};
    ( 0,0)*{\dred\xybox{  
    ( 0,-11)*{}; ( 15,8.5) **\crv{( 0,-11) & ( 15,8.5)}?(0)*\dir{<};
    ( 5,-11)*{}; ( 20,8.5) **\crv{( 5,-11) & ( 20,8.5)}?(1)*\dir{>};}};
    (-10,-12)*{\scs j}; (-5,-12)*{\scs j};
    ( 10,-12)*{\scs i}; ( 5,-12)*{\scs i};
    \endxy
$}
\vspace*{2ex}
\end{equation*}

\item For the \emph{6-valent vertices} we have:
\begin{equation}\label{eq:sixval}
\labellist
\tiny\hair 2pt
\pinlabel $i+1$ at -5 -10
\pinlabel $i$   at 65 -10
\endlabellist
\figins{-18}{0.6}{6vertu}\
\longmapsto\ \ \
\text{$
\xy 0;/r.16pc/; 
    (16,0)*{(1^n) };    
    ( 0,0)*{\dblue\xybox{
    (-7.5,10)*{}; (    5,-10) **\crv{(-4.5, 7) & ( 7.5,0) & ( 5,-9)}?(0)*\dir{<};
    (12.5,10)*{}; (    0,-10) **\crv{( 9.5, 7) & (-2.5,0) & ( 0,-9)}?(1)*\dir{>};
    (17.5,10)*{}; (-12.5, 10) **\crv{(   8, 0) & ( 2.5,-6) & (-3,0)}?(0)*\dir{<};
}};
    ( 0,0)*{\dred\xybox{  
    (-10,-20)*{};( 10,-20) **\crv{(-9,-19) & (0,-12) & (8,-19)}?(.2)*\dir{>} ?(.8)*\dir{>};
    ( 2.5, 0)*{};( 15,-20) **\crv{( 2.5,0) & ( 2,-10) & ( 15,-20)}?(0)*\dir{<};
    (-2.5, 0)*{};(-15,-20) **\crv{(-2.5,0) & (-2,-10) & (-15,-20)}?(1)*\dir{>};
}};  
    ( -17,-12)*{\scs i+1}; (-10,-12)*{\scs i+1};
    (-2.5,-12)*{\scs i }; (2.5,-12)*{\scs i };
    ( 16,-12)*{\scs i+1}; 
    ( 16, 12)*{\scs i};
    \endxy
$}
\vspace*{2ex}
\end{equation}
and
\begin{equation*}
\labellist
\tiny\hair 2pt
\pinlabel $i$ at -5 -10
\pinlabel $i+1$   at 65 -10
\endlabellist
\figins{-18}{0.6}{6vertd}\
\longmapsto\ \ \
\text{$
\xy 0;/r.16pc/; 
    (16,0)*{(1^n) };    
    ( 0,0)*{\dblue\xybox{
    (-7.5,-10)*{}; (  5, 10) **\crv{(-4.5,-7) & ( 7.5,0) & ( 5,9)}?(1)*\dir{>};
    (12.5,-10)*{}; (  0, 10) **\crv{( 9.5,-7) & (-2.5,0) & ( 0,9)}?(0)*\dir{<};
    (17.5,-10)*{}; (-12.5,-10) **\crv{(  8,  0) & (2.5,6) & (-3, 0)}?(1)*\dir{>};
}};
    ( 0,0)*{\dred\xybox{  
    (-10,0)*{};(10,0) **\crv{(-9,-1) & (0,-8) & (8,-1)}?(.2)*\dir{<} ?(.8)*\dir{<};
    ( 2.5,-20)*{}; (  15, 0) **\crv{( 2.5,-20) & ( 2,-10) & ( 15,0)}?(1)*\dir{>};
    (-2.5,-20)*{}; ( -15, 0) **\crv{(-2.5,-20) & (-2,-10) & (-15,0)}?(0)*\dir{<};
}};  
    ( -16,-12)*{\scs   i}; (-11,-12)*{\scs   i};
    (-3.5,-12)*{\scs i+1}; (3.5,-12)*{\scs i+1};
    (  11,-12)*{\scs i  }; ( 16,-12)*{\scs   i}; 
    (  10, 12)*{\scs i+1};
    \endxy
$}
\vspace*{2ex}
\end{equation*}
\end{itemize}

It is clear that $\Sigma_{n,n}$ respects the gradings of the morphisms.
Moreover, let us remark that, in the decategorified picture, the image of 
$H_q(n)$ lies in the projection of the zero weight space of $\U$ onto 
$\SD(n,n)$, so we have $E_iE_{-i}=E_{-i}E_i$. 
Using the $2$-isomorphism $\mathcal{E}_i\mathcal{E}_{-i}\cong \mathcal{E}_{-i}
\mathcal{E}_i$ given by the crossing, we obtain a $2$-functor naturally 
isomorphic to $\Sigma_{n,n}$. However, this $2$-functor cannot be obtained by 
simply inverting the orientation of the diagrams defining $\Sigma_{n,n}$, 
as can be easily checked. As a matter of fact, inverting the orientations 
does not even give a $2$-functor, e.g. 
relation~\eqref{eq:dot6v} is not preserved.

\begin{lem}
$\Sigma_{n,n}$ is a monoidal functor.
\end{lem}

\begin{proof}
The assignment given by $\Sigma_{n,n}$ clearly respects the monoidal structures of 
the categories $\mathcal{SC}_1(n)$ and 
$\mbox{End}_{\Scat(n,n)^*}(1^n)$. 
So we only need to show that $\Sigma_{n,n}$ is a functor, i.e. it respects the 
relations~\eqref{eq:adj} to~\eqref{eq:dumbdumbsquare}.

\medskip
\n\emph{''Isotopy relations'':} 
Relations~\eqref{eq:adj} to~\eqref{eq:v6rot} are straightforward to check and correspond 
to isotopies of their images under $\Sigma_{n,n}$. 

\medskip
\n\emph{One color relations:} 
To check the one color relations we only need to use the $\mathfrak{sl}_2$ relations. 
Relation~\eqref{eq:dumbrot} corresponds to an easy isotopy of diagrams in 
$\Scat(n,n)$.
For relation~\eqref{eq:lollipop} we have
\begin{equation*}
\labellist
\hair 2pt
\pinlabel $\scs i$ at 26 45            
\endlabellist
\Sigma_{n,n}\Biggl(\
\figins{-16}{0.5}{lollipop-u}\
\Biggr)
=\ \
\text{$
\xy
(0,3)*{\dblue\xybox{%
    (-6,0)*{};
    (6,0)*{};
    (4,0);(-4,0) **\crv{(4,6) & (-4,6)};?(.7)*\dir{}+(-2,0)*{\bscs i};
    ?(0)*\dir{<} ?(.95)*\dir{<};
    (4,0);(-4,0) **\crv{(4,-6) & (-4,-6)};}};
(0,3)*{\dblue\xybox{
    (8,0);(-8,0) **\crv{(8,12) & (-8,12)};
    (8,0);(4,-12) **\crv{(8,-4) & (4,-6)};?(0.8)*\dir{<}+(2,-1)*{\bscs i};
    (-8,0);(-4,-12) **\crv{(-8,-4) & (-4,-6)};
}};
(12,8)*{\scs (1^n)};(0,-5)*{\scs (1_{+i}^n)};
\endxy
=\
0
$}
\end{equation*}
because the bubble in the diagram on the r.h.s. has negative degree. 
We have used the notation $1_{+i}^n=(1,\ldots,2,0,1,\ldots,1)$, with the $2$ 
on the $i$th coordinate. 

Relation~\eqref{eq:deltam} requires some more work. First notice that
from relations~\eqref{eq:bubb_deg0} and~\eqref{eq:EF} it follows that
\begin{equation}
\text{$
0\ =\
\xy
(1,0)*{\dblue\xybox{%
    (4,4);(-4,4) **\crv{(4,2) & (8,-4) & (0,-9) & (-8,-4) & (-4,2)};
?(.7)*\dir{}+(-1,8)*{\bscs i};?(0)*\dir{<} ?(.95)*\dir{<};
    (4,-12);(-4,-12) **\crv{(4,-10) & (8,-4) & (0,1) & (-8,-4) & (-4,-10)};
?(.7)*\dir{}+(-1,-8)*{\bscs i};?(1)*\dir{>} ?(.05)*\dir{>};
}};
(12,4)*{\scs (1^n)};(3,0)*{}
\endxy\ 
=\
\xy
(1,0)*{\dblue\xybox{%
    (4,4);(-4,4) **\crv{(4,-3) & (-4,-3)};?(.7)*\dir{}+(-4,3.5)*{\bscs i};
    ?(0)*\dir{<} ?(.95)*\dir{<};
    (4,-12);(-4,-12) **\crv{(4,-5) & (-4,-5)};?(.7)*\dir{}+(-4,-3.5)*{\bscs i};
    ?(1)*\dir{>} ?(.05)*\dir{>};
}};
(6,0)*{\scs (1^n)};
\endxy\ 
-\ 
\xy
(4,0)*{\dblue\xybox{%
    ( 3,9);( 3,-9) **\crv{( 4,1) & ( 4,-1)};?(0)*\dir{<}; (-3.75,0)*{\bullet};
    (-3,9);(-3,-9) **\crv{(-4,1) & (-4,-1)};?(1)*\dir{>};
}};
(12,4)*{\scs (1^n)};(-1.5,-7)*{\scs i}; (9.2,-7)*{\scs i}
\endxy
-\
\xy
(4,0)*{\dblue\xybox{%
    ( 3,9);( 3,-9) **\crv{( 4,1) & ( 4,-1)};?(0)*\dir{<}; (3.75,0)*{\bullet};
    (-3,9);(-3,-9) **\crv{(-4,1) & (-4,-1)};?(1)*\dir{>};
}};
(12,4)*{\scs (1^n)};(-1.5,-7)*{\scs i}; (8.2,-7)*{\scs i}
\endxy
+\ \ \ \ \
\xy
(0,0)*{\dblue\xybox{%
    (3,0);(-3,0) **\crv{(3,4.2) & (-3,4.2)};?(.7)*\dir{}+(-2,0)*{\bscs i};
    ?(.05)*\dir{>} ?(1)*\dir{>}; (0,-3.1)*{\bullet}+(-0.4,-2)*{\bscs -2};
    (3,0);(-3,0) **\crv{(3,-4.2) & (-3,-4.2)};
    (3,9);(3,-9) **\crv{(4,6) & (6,4) & (6,-3) & (4,-6)};?(0)*\dir{<};
    (-3,9);(-3,-9) **\crv{(-4,6) & (-6,4) & (-6,-3) & (-4,-6)};?(1)*\dir{>};
}};
(10,4)*{\scs (1^n)};(-5.5,-7)*{\scs i}; (5.5,-7)*{\scs i}
\endxy\ $}.
\end{equation}
The first diagram is zero, because the middle region has label 
$(1,\ldots,3,-1,\ldots,1)\not\in\Lambda(n,n)$, with $3$ on the $i$th coordinate. 
Therefore
\begin{equation*}
\labellist
\hair 2pt
\pinlabel $\scs i$ at -2 22    
\pinlabel $\scs i$ at -2 130  
\endlabellist
\Sigma_{n,n}\Biggl(\
\figins{-17}{0.55}{matches-ud}\
\Biggr)
=\ 
\xy
(1,0)*{\dblue\xybox{%
    (4,4);(-4,4) **\crv{(4,-3) & (-4,-3)};?(.7)*\dir{}+(-4,3.5)*{\bscs i};
    ?(0)*\dir{<} ?(.95)*\dir{<};
    (4,-12);(-4,-12) **\crv{(4,-5) & (-4,-5)};?(.7)*\dir{}+(-4,-3.5)*{\bscs i};
    ?(1)*\dir{>} ?(.05)*\dir{>};
}};
(6,0)*{\scs (1^n)};
\endxy\ 
=\
\xy
(4,0)*{\dblue\xybox{%
    ( 3,9);( 3,-9) **\crv{( 4,1) & ( 4,-1)};?(0)*\dir{<}; (-3.75,0)*{\bullet};
    (-3,9);(-3,-9) **\crv{(-4,1) & (-4,-1)};?(1)*\dir{>};
}};
(12,4)*{\scs (1^n)};(-1.5,-7)*{\scs i}; (9.2,-7)*{\scs i}
\endxy
+\
\xy
(4,0)*{\dblue\xybox{%
    ( 3,9);( 3,-9) **\crv{( 4,1) & ( 4,-1)};?(0)*\dir{<}; (3.75,0)*{\bullet};
    (-3,9);(-3,-9) **\crv{(-4,1) & (-4,-1)};?(1)*\dir{>};
}};
(12,4)*{\scs (1^n)};(-1.5,-7)*{\scs i}; (8.2,-7)*{\scs i}
\endxy
-\ \ \ \ \
\xy
(0,0)*{\dblue\xybox{%
    (3,0);(-3,0) **\crv{(3,4.2) & (-3,4.2)};?(.7)*\dir{}+(-2,0)*{\bscs i};
    ?(.05)*\dir{>} ?(1)*\dir{>}; (0,-3.1)*{\bullet}+(-0.4,-2)*{\bscs -2};
    (3,0);(-3,0) **\crv{(3,-4.2) & (-3,-4.2)};
    (3,9);(3,-9) **\crv{(4,6) & (6,4) & (6,-3) & (4,-6)};?(0)*\dir{<};
    (-3,9);(-3,-9) **\crv{(-4,6) & (-6,4) & (-6,-3) & (-4,-6)};?(1)*\dir{>};
}};
(10,4)*{\scs (1^n)};(-5.5,-7)*{\scs i}; (5.5,-7)*{\scs i}
\endxy
\end{equation*}
Using~\eqref{eq:bub_slides} and the bubble evaluation~\eqref{eq:bubb_deg0} 
we obtain
\begin{align}
\label{eq:edge-dots}
\labellist
\hair 2pt
\pinlabel $\scs i$ at -10  60
\endlabellist
\Sigma_{n,n}\Biggl(\
\figins{-17}{0.55}{edge-startenddot}\
\Biggr)
&=\ 2\
\xy
(4,0)*{\dblue\xybox{%
    ( 3,9);( 3,-9) **\crv{( 4,1) & ( 4,-1)};?(0)*\dir{<}; (3.75,0)*{\bullet};
    (-3,9);(-3,-9) **\crv{(-4,1) & (-4,-1)};?(1)*\dir{>};
}};
(12,4)*{\scs (1^n)};(-1.5,-7)*{\scs i}; (8.2,-7)*{\scs i}
\endxy
-\ \ \ \ \
\xy
(0,0)*{\dblue\xybox{%
    (3,0);(-3,0) **\crv{(3,4.2) & (-3,4.2)};?(.7)*\dir{}+(-2,0)*{\bscs i};
    ?(.05)*\dir{>} ?(1)*\dir{>}; (0,-3.1)*{\bullet}+(-0.4,-2)*{\bscs -2};
    (3,0);(-3,0) **\crv{(3,-4.2) & (-3,-4.2)};
    (3,9);(3,-9) **\crv{(4,6) & (6,4) & (6,-3) & (4,-6)};?(0)*\dir{<};
    (-3,9);(-3,-9) **\crv{(-4,6) & (-6,4) & (-6,-3) & (-4,-6)};?(1)*\dir{>};
}};
(10,4)*{\scs (1^n)};(-5.5,-7)*{\scs i}; (5.5,-7)*{\scs i}
\endxy
\displaybreak[0]
\intertext{and}
\label{eq:dots-edge}
\labellist
\hair 2pt
\pinlabel $\scs i$ at  65  60
\endlabellist
\Sigma_{n,n}\Biggl(\
\figins{-17}{0.55}{startenddot-edge}\
\Biggr)
&=\ 2\
\xy
(4,0)*{\dblue\xybox{%
    ( 3,9);( 3,-9) **\crv{( 4,1) & ( 4,-1)};?(0)*\dir{<}; (-3.75,0)*{\bullet};
    (-3,9);(-3,-9) **\crv{(-4,1) & (-4,-1)};?(1)*\dir{>};
}};
(12,4)*{\scs (1^n)};(-1.5,-7)*{\scs i}; (9.2,-7)*{\scs i}
\endxy
-\ \ \ \ \
\xy
(0,0)*{\dblue\xybox{%
    (3,0);(-3,0) **\crv{(3,4.2) & (-3,4.2)};?(.7)*\dir{}+(-2,0)*{\bscs i};
    ?(.05)*\dir{>} ?(1)*\dir{>}; (0,-3.1)*{\bullet}+(-0.4,-2)*{\bscs -2};
    (3,0);(-3,0) **\crv{(3,-4.2) & (-3,-4.2)};
    (3,9);(3,-9) **\crv{(4,6) & (6,4) & (6,-3) & (4,-6)};?(0)*\dir{<};
    (-3,9);(-3,-9) **\crv{(-4,6) & (-6,4) & (-6,-3) & (-4,-6)};?(1)*\dir{>};
}};
(10,4)*{\scs (1^n)};(-5.5,-7)*{\scs i}; (5.5,-7)*{\scs i}
\endxy.
\end{align}
This establishes that
\begin{equation*}
\Sigma_{n,n}
\Biggl(
\figins{-17}{0.55}{startenddot-edge}\
\Biggr)
+\
\Sigma_{n,n}
\Biggl(\
\figins{-17}{0.55}{edge-startenddot}\
\Biggr)
=\ 2\ 
\Sigma_{n,n}
\Biggl(\
\figins{-17}{0.55}{matches-ud}\
\Biggr).
\end{equation*}

\medskip
\n\emph{Two distant colors:} 
Checking relations~\eqref{eq:reid2dist} to~\eqref{eq:slide3v} is straightforward and only uses relations~\eqref{eq_downup_ij-gen} and~\eqref{eq_r2_ij-gen} 
with distant colors $i$ and $j$. 

\medskip
\n\emph{Adjacent colors:} 
To prove relation~\eqref{eq:dot6v} we first notice 
that using~\eqref{eq_r3_hard-gen} we get
\begin{equation*}
\labellist
\hair 2pt
\pinlabel $\scs i$   at  0 -10
\pinlabel $\scs i+1$ at 69 -11
\endlabellist
\Sigma_{n,n}\Biggl(\
\figins{-12}{0.4}{6vertdotd}\
\Biggr)\
=
\xy 0;/r.16pc/; 
    (16,0)*{(1^n) };    
    ( 0,0)*{\dblue\xybox{
    (-7.5,-10)*{}; (  5, 10) **\crv{(-4.5,-7) & ( 7.5,0) & ( 5,9)}?(1)*\dir{>};
    (12.5,-10)*{}; (  0, 10) **\crv{( 9.5,-7) & (-2.5,0) & ( 0,9)}?(0)*\dir{<};
    (17.5,-10)*{}; (-12.5,-10) **\crv{(  8,  0) & (2.5,6) & (-3, 0)}?(1)*\dir{>};
}};
    ( 0,1)*{\dred\xybox{  
    (-10,0)*{};(10,0) **\crv{(-9,-1) & (0,-8) & (8,-1)}?(.2)*\dir{<} ?(.8)*\dir{<};
    (-15,0)*{}; (15, 0) **\crv{(-4,-10) & (-2.5,-18) & (2.5,-18) & (4,-10)};
    ?(1)*\dir{>};?(0.01)*\dir{>};
}};  
    (-16,-12)*{\scs   i}; (-11,-12)*{\scs   i};
    ( 11,-12)*{\scs i  }; ( 16,-12)*{\scs   i}; 
    ( -9, 12)*{\scs i+1}; (-16, 12)*{\scs i+1};
    \endxy\
=\ 
\xy 0;/r.16pc/; 
    (16,0)*{(1^n) };    
    ( 0,0)*{\dblue\xybox{
    (-7.5,-10)*{}; (12.5,-10) **\crv{(-4,-5) & ( 9,-5)}?(1)*\dir{>};
    (5, 10)*{}; ( 0, 10) **\crv{( 6,0) & (2.5,-4) & ( -1,0)}?(0)*\dir{<};
    (17.5,-10)*{}; (-12.5,-10) **\crv{(  8,  0) & (2.5,6) & (-3, 0)}?(1)*\dir{>};
}};
    ( 0,2.5)*{\dred\xybox{  
    (-10,0)*{};(10,0) **\crv{(-9,-1) & (0,-8) & (8,-1)}?(.2)*\dir{<} ?(.8)*\dir{<};
    (-15,0)*{}; (15, 0) **\crv{(-6,-8) & (-2.5,-15) & (2.5,-15) & (6,-8)};
    ?(1)*\dir{>};?(0.01)*\dir{>};
}};  
    (-16,-12)*{\scs   i}; (-11,-12)*{\scs   i};
    ( 11,-12)*{\scs i  }; ( 16,-12)*{\scs   i}; 
    ( -9, 12)*{\scs i+1}; (-16, 12)*{\scs i+1};
    \endxy\ .
\vspace*{2ex}
\end{equation*}
Note that the other term on the r.h.s. of~\eqref{eq_r3_hard-gen} is equal to 
zero, because it contains a region whose label has a negative entry, i.e. 
does not belong to $\Lambda(n,n)$. 

Using~\eqref{eq_ident_decomp0} followed by~\eqref{eq_downup_ij-gen} and~\eqref{eq:bubb_deg0}
gives
\begin{equation*}
\xy 0;/r.16pc/; 
    (16,2)*{(1^n) };    
    ( 0,0)*{\dblue\xybox{
    (-7.5,-10)*{}; (12.5,-10) **\crv{(-2,-6) & (2.5,-5) & ( 7,-6)}?(1)*\dir{>};
    (5, 10)*{}; ( 0, 10) **\crv{( 4.5,9) & (2.5,8) & ( .5,9)}?(0)*\dir{<};
    (17.5,-10)*{}; (-12.5,-10) **\crv{(8, -3) & (2.5,-2) & (-3, -3)}?(1)*\dir{>};
}};
    ( 0,5.5)*{\dred\xybox{  
    (-10,0)*{};(10, 0) **\crv{( -8,-2) & (0, -9) & ( 8,-2)}?(0)*\dir{<};
    (-15,0)*{};(15, 0) **\crv{(-10,-5) & (0,-13) & (10,-5)}?(1)*\dir{>};
}};  
    (-16,-12)*{\scs   i}; (-11,-12)*{\scs   i};
    ( -9, 12)*{\scs i+1}; (-16, 12)*{\scs i+1};
    \endxy\
\ +\
\xy 0;/r.16pc/; 
    (16,2)*{(1^n) };    
    ( 0,0)*{\dblue\xybox{
    (-7.5,-10)*{}; (12.5,-10) **\crv{(-2,-6) & (2.5,-5) & ( 7,-6)}?(1)*\dir{>};
    (5,10)*{};( 17.5,-10) **\crv{(5,10) & ( 4.5,0) & ( 17.5,-10)}?(0)*\dir{<};
    (0,10)*{};(-12.5,-10) **\crv{(0,10) & ( 0.5,0) & (-12.5,-10)}?(1)*\dir{>};
}};
    ( 0,5.5)*{\dred\xybox{  
    (-10,0)*{};(10, 0) **\crv{( -8,-2) & (0, -9) & ( 8,-2)}?(0)*\dir{<};
    (-15,0)*{};(15, 0) **\crv{(-10,-5) & (0,-13) & (10,-5)}?(1)*\dir{>};
}};  
    (-16,-12)*{\scs   i}; (-11,-12)*{\scs   i};
    ( -9, 12)*{\scs i+1}; (-16, 12)*{\scs i+1};
    \endxy\ .
\end{equation*}
Applying~\eqref{eq_ident_decomp} to the two red strands in the middle region of the second 
term (only one term survives) followed by~\eqref{eq_downup_ij-gen} and~\eqref{eq:bubb_deg0}
gives
\begin{equation*}
\labellist
\hair 2pt
\pinlabel $\scs i$   at  0 -10
\pinlabel $\scs i+1$ at 69 -11
\endlabellist
\Sigma_{n,n}\Biggl(\
\figins{-12}{0.4}{6vertdotd}\
\Biggr)\
=\
\xy 0;/r.16pc/; 
    (16,2)*{(1^n) };    
    ( 0,0)*{\dblue\xybox{
    (-7.5,-10)*{}; (12.5,-10) **\crv{(-2,-6) & (2.5,-5) & ( 7,-6)}?(1)*\dir{>};
    (5, 10)*{}; ( 0, 10) **\crv{( 4.5,9) & (2.5,8) & ( .5,9)}?(0)*\dir{<};
    (17.5,-10)*{}; (-12.5,-10) **\crv{(8, -3) & (2.5,-2) & (-3, -3)}?(1)*\dir{>};
}};
    ( 0,5.5)*{\dred\xybox{  
    (-10,0)*{};(10, 0) **\crv{( -8,-2) & (0, -9) & ( 8,-2)}?(0)*\dir{<};
    (-15,0)*{};(15, 0) **\crv{(-10,-5) & (0,-13) & (10,-5)}?(1)*\dir{>};
}};  
    (-16,-12)*{\scs   i}; (-11,-12)*{\scs   i};
    ( -9, 12)*{\scs i+1}; (-16, 12)*{\scs i+1};
    \endxy\
\ +\
\xy 0;/r.16pc/; 
    (18,2)*{(1^n) };    
    ( 0,0)*{\dblue\xybox{
    (-7.5,-10)*{}; (12.5,-10) **\crv{(-2,-6) & (2.5,-5) & ( 7,-6)}?(1)*\dir{>};
    (5,10)*{};( 17.5,-10) **\crv{(5,10) & ( 4.5,0) & ( 17.5,-10)}?(0)*\dir{<};
    (0,10)*{};(-12.5,-10) **\crv{(0,10) & ( 0.5,0) & (-12.5,-10)}?(1)*\dir{>};
}};
    ( 0,7.5)*{\dred\xybox{  
    (-15,0)*{};(-10, 0) **\crv{(-15,-4) & (-12.5,-6) & (-10,-4)}?(1)*\dir{>};
    ( 15,0)*{};( 10, 0) **\crv{( 15,-4) & ( 12.5,-6) & ( 10,-4)}?(0)*\dir{<};
}};  
    (-16,-12)*{\scs   i}; (-11,-12)*{\scs   i};
    ( 10, 12)*{\scs i+1}; (-15, 12)*{\scs i+1};
    \endxy\ ,
\end{equation*}
which is equal to 
$\Sigma_{n,n}\bigl(\figins{-5}{0.2}{capcupdot}\bigr)+\Sigma_{n,n}\bigl(\figins{-5}{0.2}{mergedots}\bigr)$ .

The corresponding relation with colors switched is not difficult to prove.
We have 
\begin{equation*}
\labellist
\hair 2pt
\pinlabel $\scs i+1$ at  5 -12
\pinlabel $\scs i$   at 69 -10
\endlabellist
\Sigma_{n,n}\Biggl(\ \
\figins{-12}{0.4}{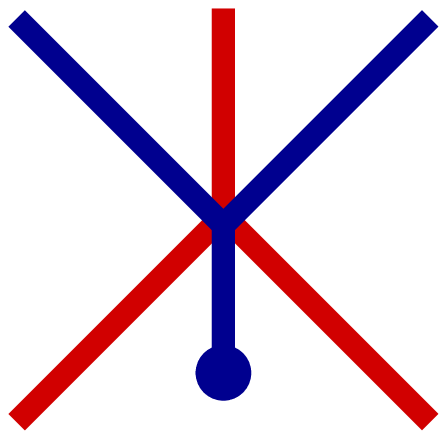}\ 
\Biggr)\
=
\xy 0;/r.16pc/; 
    (20,0)*{(1^n) };    
    ( 0,0)*{\dblue\xybox{
    (17.5,10)*{}; (-12.5, 10) **\crv{(   8, 0) & ( 2.5,-6) & (-3,0)}?(0)*\dir{<};
    (12.5,10)*{}; ( -7.5,10) **\crv{(9.5, 7) & (-2.5,0) & (0,-11) & (5,-11) & (7.5,0) & (-4.5,7) }?(1)*\dir{>}?(0.49)*\dir{>};
}};
    ( 0,0)*{\dred\xybox{  
    (-10,-20)*{};( 10,-20) **\crv{(-9,-19) & (0,-12) & (8,-19)}?(.2)*\dir{>} ?(.8)*\dir{>};
    ( 2.5, 0)*{};( 15,-20) **\crv{( 2.5,0) & ( 2,-10) & ( 15,-20)}?(0)*\dir{<};
    (-2.5, 0)*{};(-15,-20) **\crv{(-2.5,0) & (-2,-10) & (-15,-20)}?(1)*\dir{>};
}};  
    ( -17,-12)*{\scs i+1}; (-10,-12)*{\scs i+1};
    ( 10, 12)*{\scs i };
    ( 16,-12)*{\scs i+1}; 
    ( 16, 12)*{\scs i};
    \endxy
.
\end{equation*}
Use~\eqref{eq_r2_ij-gen} on the bottom part of the diagram. 
Only one of the resulting terms survives (use the first relation in~\eqref{eq_nil_rels}), which in turn equals
\begin{equation*}
\xy 0;/r.16pc/; 
    (20,0)*{(1^n) };    
    ( 0,2)*{\dblue\xybox{
    (7.5,15)*{}; (12.5,15) **\crv{(4.5,12) & (-6.5,5) & (-3,1) & (0,1) & (2.5,5) }?(1)*\dir{>};
    (-12.5,15)*{}; (-17.5,15) **\crv{(-9.5,12) & (1.5,5) & (-2,1) & (-5,1) & (-7.5,5) }?(0)*\dir{<};
}};
    ( 0,0)*{\dred\xybox{  
    (-10,-20)*{};( 10,-20) **\crv{(-9,-19) & (0,-12) & (8,-19)}?(.2)*\dir{>} ?(.8)*\dir{>};
    ( 2.5, 0)*{};( 15,-20) **\crv{( 2.5,0) & ( 2,-10) & ( 15,-20)}?(0)*\dir{<};
    (-2.5, 0)*{};(-15,-20) **\crv{(-2.5,0) & (-2,-10) & (-15,-20)}?(1)*\dir{>};
}};  
    ( -17,-12)*{\scs i+1}; (-10,-12)*{\scs i+1};
    ( 10, 12)*{\scs i };
    ( 16,-12)*{\scs i+1}; 
    ( 16, 12)*{\scs i};
    \endxy
\end{equation*}
(use the first relation in~\eqref{eq_nil_rels} combined with~\eqref{eq_nil_dotslide}).
Applying~\eqref{eq:EF} we get two terms, one of which is
\begin{equation*}
\xy 0;/r.16pc/; 
    (18,2)*{(1^n) };    
    ( 0,0)*{\dred\xybox{
    (-7.5,-10)*{}; (12.5,-10) **\crv{(-2,-6) & (2.5,-5) & ( 7,-6)}?(1)*\dir{>};
    (5,10)*{};( 17.5,-10) **\crv{(5,10) & ( 4.5,0) & ( 17.5,-10)}?(0)*\dir{<};
    (0,10)*{};(-12.5,-10) **\crv{(0,10) & ( 0.5,0) & (-12.5,-10)}?(1)*\dir{>};
}};
    ( 0,7.5)*{\dblue\xybox{  
    (-15,0)*{};(-10, 0) **\crv{(-15,-4) & (-12.5,-6) & (-10,-4)}?(1)*\dir{>};
    ( 15,0)*{};( 10, 0) **\crv{( 15,-4) & ( 12.5,-6) & ( 10,-4)}?(0)*\dir{<};
}};  
    (-16,-12)*{\scs   i}; (-11,-12)*{\scs   i};
    ( 10, 12)*{\scs i+1}; (-15, 12)*{\scs i+1};
    \endxy\
 \end{equation*}
(this follows easily from~\eqref{eq_downup_ij-gen}) 
and the other equals
\begin{equation*}
\xy 0;/r.16pc/; 
    (16,2)*{(1^n) };    
    ( 0,0)*{\dred\xybox{
    (-7.5,-10)*{}; (12.5,-10) **\crv{(-2,-6) & (2.5,-5) & ( 7,-6)}?(1)*\dir{>};
    (5,10)*{};( 17.5,-10) **\crv{(5,10) & ( 4.5,0) & ( 17.5,-10)}?(0)*\dir{<};
    (0,10)*{};(-12.5,-10) **\crv{(0,10) & ( 0.5,0) & (-12.5,-10)}?(1)*\dir{>};
}};
    ( 0,5.5)*{\dblue\xybox{  
    (-10,0)*{};(10, 0) **\crv{( -8,-2) & (0, -9) & ( 8,-2)}?(0)*\dir{<};
    (-15,0)*{};(15, 0) **\crv{(-10,-5) & (0,-13) & (10,-5)}?(1)*\dir{>};
}};  
    (-16,-12)*{\scs   i}; (-11,-12)*{\scs   i};
    ( -9, 12)*{\scs i+1}; (-16, 12)*{\scs i+1};
    \endxy\
.
\end{equation*}
Here we used~\eqref{eq:bubb_deg0}.
The rest of the computation is the same as in the previous case.

We now prove relation~\eqref{eq:reid3}.
We only prove the case where ''blue`` corresponds to $i$ and ''red`` 
corresponds to $i+1$. The relation with the colors reversed is proved 
in the same way.
Start with 
\begin{equation*}
\Sigma_{n,n}
\left(\
\labellist
\tiny\hair 2pt
\pinlabel $i+1$  at  70  -14
\pinlabel $i$    at   0  -12
\endlabellist
\figins{-26}{0.8}{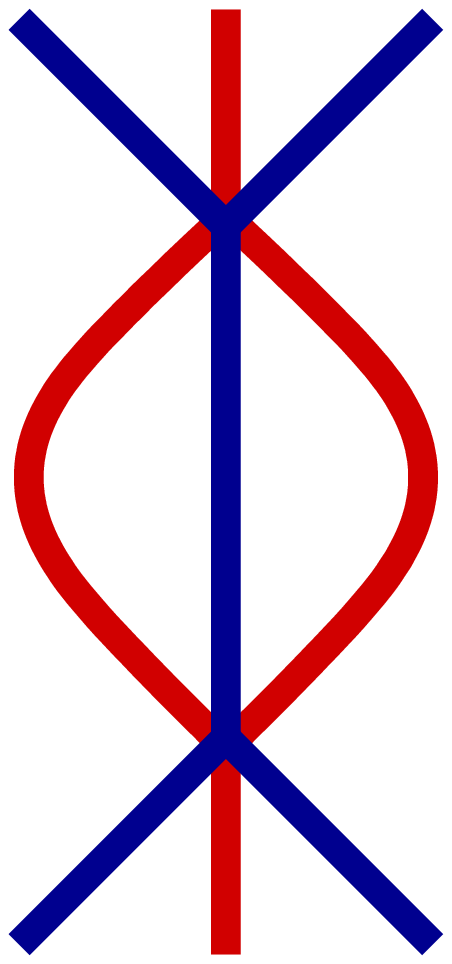}\
\right)
=\
\xy 0;/r.16pc/; 
( 0,0)*{\dblue\xybox{
    (-7,10)*{};(-7,-26) **\crv{(-4.5, 8) & ( 4.5,3) & ( 5.5,-8.5) & (4.5,-20) & (-4.5,-24)}
   ?(0)*\dir{<}?(0.47)*\dir{<};
    ( 12,10)*{};(12,-26) **\crv{(9.5, 8) & (0.5,3) & (-0.5,-8.5) & (0.5,-20) & (9.5,-24)}
   ?(1)*\dir{>}?(0.5)*\dir{>};
    (17.5,10)*{};(-12.5, 10) **\crv{(7,1.5) & (6,1) & ( 2.5,-2) & (-1,1) & (-2,1.5)}
   ?(0)*\dir{<};
    (17.5,-26)*{};(-12.5,-26) **\crv{(7,-17.5) & (6,-17) & (2.5,-14) & (-1,-17) & (-2,-17.5)}
   ?(1)*\dir{>};
}};
( 0,0)*{\dred\xybox{  
    ( 2.5,-20)*{}; ( 2.5, 16) 
   **\crv{( 2.5,-18) & ( 6,-14) & (12,-6) & ( 12,-2) & (12,2) & (6,10) & (2.5,14) }
   ?(1)*\dir{>}?(0.5)*\dir{>};
    (-2.5,-20)*{}; (-2.5, 16) 
   **\crv{(-2.5,-18) & (-6,-14) & (-12,-6) & (-12,-2) & (-12,2) & (-6,10) & (-2.5,14) }
   ?(0)*\dir{<}?(0.5)*\dir{<};
    (-6,-2)*{}; (6,-2) **\crv{(-6,0) & (0,5) & (6,0)}?(0.01)*\dir{>};
    (-6,-2)*{}; (6,-2) **\crv{(-6,-4) & (0,-9) & (6,-4)};
}};
    (16,10)*{(1^n) };   
    ( -16,-20)*{\scs   i}; (-10,-20)*{\scs   i};
    (-3.5,-20.1)*{\scs i+1}; (3.5,-20.1)*{\scs i+1};
    (  10,-20)*{\scs   i}; ( 16,-20)*{\scs   i}; 
    ( -16, 20)*{\scs i}; (7,-4)*{\scs i+1};
\endxy\
= \
\xy 0;/r.16pc/; 
( 0,0)*{\dblue\xybox{
    (-7,10)*{};(-7,-26) **\crv{(-4.5, 7) & ( 7.5,-1) & ( 11.5,-8.5) & (7.5,-16) & (-4.5,-23)}
   ?(0)*\dir{<}?(0.5)*\dir{<};
    ( 12,10)*{};(12,-26) **\crv{(9.5, 7) & (-2.5,-1) & (-6.5,-8.5) & (-2.5,-16) & (9.5,-23)}
   ?(1)*\dir{>}?(0.5)*\dir{>};
    (17.5,10)*{};(-12.5, 10) **\crv{(7, 0) & (6,-1) & ( 2.5,-4) & (-1,-1) & (-2,0)}
   ?(0)*\dir{<};
    (17.5,-26)*{};(-12.5,-26) **\crv{(7,-16) & (6,-15) & (2.5,-13) & (-1,-15) & (-2,-16)}
   ?(1)*\dir{>};
}};
( 0,0)*{\dred\xybox{  
    ( 2.5,-20)*{}; ( 2.5, 16) 
   **\crv{( 2.5,-18) & ( 6,-14) & (12,-6) & ( 12,-2) & (12,2) & (6,10) & (2.5,14) }
   ?(1)*\dir{>}?(0.5)*\dir{>};
    (-2.5,-20)*{}; (-2.5, 16) 
   **\crv{(-2.5,-18) & (-6,-14) & (-12,-6) & (-12,-2) & (-12,2) & (-6,10) & (-2.5,14) }
   ?(0)*\dir{<}?(0.5)*\dir{<};
}};
    (16,10)*{(1^n) };   
    ( -16,-20)*{\scs   i}; (-10,-20)*{\scs   i};
    (-3.5,-20.1)*{\scs i+1}; (3.5,-20.1)*{\scs i+1};
    (  10,-20)*{\scs   i}; ( 16,-20)*{\scs   i}; 
    ( -16, 20)*{\scs i};
\endxy,
\end{equation*}
where the second equality follows from~\eqref{eq_downup_ij-gen} and~\eqref{eq:bubb_deg0}.
Now notice that
\begin{equation}
\label{eq:reid3proof}
0 =
\xy 0;/r.16pc/; 
( 0,0)*{\dblue\xybox{
    (-7,10)*{};(-7,-26) **\crv{(-4.5, 7) & ( 7.5,-1) & ( 11.5,-8.5) & (7.5,-16) & (-4.5,-23)}
   ?(0)*\dir{<}?(0.5)*\dir{<};
    ( 12,10)*{};(12,-26) **\crv{(9.5, 7) & (-2.5,-1) & (-6.5,-8.5) & (-2.5,-16) & (9.5,-23)}
   ?(1)*\dir{>}?(0.5)*\dir{>};
    (17.5,10)*{};(-12.5, 10) **\crv{(7, 0) & (7.5,-4) & ( 2.5,-14) & (-2.5,-4) & (-2,0)}
   ?(0)*\dir{<};
    (17.5,-26)*{};(-12.5,-26) **\crv{(7,-16) & (7.5,-12) & (2.5,-2) & (-2.5,-12) & (-2,-16)}
   ?(1)*\dir{>};
}};
( 0,0)*{\dred\xybox{  
    ( 2.5,-20)*{}; ( 2.5, 16) 
   **\crv{( 2.5,-18) & ( 6,-14) & (12,-6) & ( 12,-2) & (12,2) & (6,10) & (2.5,14) }
   ?(1)*\dir{>}?(0.5)*\dir{>};
    (-2.5,-20)*{}; (-2.5, 16) 
   **\crv{(-2.5,-18) & (-6,-14) & (-12,-6) & (-12,-2) & (-12,2) & (-6,10) & (-2.5,14) }
   ?(0)*\dir{<}?(0.5)*\dir{<};
}};
    (16,10)*{(1^n) };   
    ( -16,-20)*{\scs   i}; (-10,-20)*{\scs   i};
    (-3.5,-20.1)*{\scs i+1}; (3.5,-20.1)*{\scs i+1};
    (  10,-20)*{\scs   i}; ( 16,-20)*{\scs   i}; 
    ( -16, 20)*{\scs i};
\endxy
= 
\xy 0;/r.16pc/; 
( 0,0)*{\dblue\xybox{
    (-7,10)*{};(-7,-26) **\crv{(-4.5, 7) & ( 7.5,-1) & ( 11.5,-8.5) & (7.5,-16) & (-4.5,-23)}
   ?(0)*\dir{<}?(0.5)*\dir{<};
    ( 12,10)*{};(12,-26) **\crv{(9.5, 7) & (-2.5,-1) & (-6.5,-8.5) & (-2.5,-16) & (9.5,-23)}
   ?(1)*\dir{>}?(0.5)*\dir{>};
    (17.5,10)*{};(-12.5, 10) **\crv{(7, 0) & (6,-4) & ( 2.5,-8) & (-1,-4) & (-2,0)}
   ?(0)*\dir{<};
    (17.5,-26)*{};(-12.5,-26) **\crv{(7,-16) & (6,-12) & (2.5,-8) & (-1,-12) & (-2,-16)}
   ?(1)*\dir{>};
}};
( 0,0)*{\dred\xybox{  
    ( 2.5,-20)*{}; ( 2.5, 16) 
   **\crv{( 2.5,-18) & ( 6,-14) & (12,-6) & ( 12,-2) & (12,2) & (6,10) & (2.5,14) }
   ?(1)*\dir{>}?(0.5)*\dir{>};
    (-2.5,-20)*{}; (-2.5, 16) 
   **\crv{(-2.5,-18) & (-6,-14) & (-12,-6) & (-12,-2) & (-12,2) & (-6,10) & (-2.5,14) }
   ?(0)*\dir{<}?(0.5)*\dir{<};
}};
    (16,10)*{(1^n) };   
    ( -16,-20)*{\scs   i}; (-10,-20)*{\scs   i};
    (-3.5,-20.1)*{\scs i+1}; (3.5,-20.1)*{\scs i+1};
    (  10,-20)*{\scs   i}; ( 16,-20)*{\scs   i}; 
    ( -16, 20)*{\scs i};
\endxy
-
\xy 0;/r.16pc/; 
( 0,0)*{\dblue\xybox{
    (-7,10)*{};(-7,-26) **\crv{(-4.5, 7) & ( 7.5,-1) & ( 11.5,-8.5) & (7.5,-16) & (-4.5,-23)}
   ?(0)*\dir{<};
    (-12,10)*{};(-12,-26) 
   **\crv{(-11.5,7) & (-2.5,-1) & (0.5,-8.5) & (-2.5,-16) & (-11.5,-23)}
   ?(1)*\dir{>}?(0.5)*{\bullet};
    ( 12,10)*{};(12,-26) **\crv{(9.5, 7) & (-2.5,-1) & (-6.5,-8.5) & (-2.5,-16) & (9.5,-23)}
   ?(1)*\dir{>};
    ( 17,10)*{};(17,-26) **\crv{(16.5, 7) & (6.5,-1) & (4.5,-8.5) & (6.5,-16) & (16.5,-23)}
   ?(0)*\dir{<}?(0.5)*{\bullet};
}};   
( 0,0)*{\dred\xybox{  
    ( 2.5,-20)*{}; ( 2.5, 16) 
   **\crv{( 2.5,-18) & ( 6,-14) & (12,-6) & ( 12,-2) & (12,2) & (6,10) & (2.5,14) }
   ?(1)*\dir{>}?(0.5)*\dir{>};
    (-2.5,-20)*{}; (-2.5, 16) 
   **\crv{(-2.5,-18) & (-6,-14) & (-12,-6) & (-12,-2) & (-12,2) & (-6,10) & (-2.5,14) }
   ?(0)*\dir{<}?(0.5)*\dir{<};
}}; 
    (16,10)*{(1^n) };   
    ( -16,-20)*{\scs   i}; (-10,-20)*{\scs   i};
    (-3.5,-20.1)*{\scs i+1}; (3.5,-20.1)*{\scs i+1};
    (  10,-20)*{\scs   i}; ( 16,-20)*{\scs   i}; 
\endxy.
\end{equation}
The first equality in~\eqref{eq:reid3proof} comes from the fact that 
the inner most region of the diagram has a label outside $\Lambda(n,n)$. 
The second equality follows from~\eqref{eq_ident_decomp0}. 
The last term is the only non-zero term coming from the sum in
~\eqref{eq_ident_decomp0} (this is a consequence of~\eqref{eq_nil_rels}). 

Applying~\eqref{eq_nil_rels} and~\eqref{eq_nil_dotslide} to 
the last term, we obtain a diagram that can be simplified further by successive application 
of~\eqref{eq_r3_hard-gen},~\eqref{eq_ident_decomp0},~\eqref{eq_downup_ij-gen} and 
again~\eqref{eq:bubb_deg0}.
\begin{align*} 
&
\xy 0;/r.16pc/; 
( 0,0)*{\dblue\xybox{
   (-12,8)*{};(12,-24) **\crv{(-6,4) & (-8,-6) & (2,-16)}?(1)*\dir{>};
   (-12,-24)*{};(12, 8) **\crv{(-6,-20) & (-8,-10) & (2,0)}?(0)*\dir{<};
   ( 17,8)*{};(-7,-24) **\crv{( 11,4) & ( 13,-6) & (3,-16)}?(0)*\dir{<};
   ( 17,-24)*{};(-7,8) **\crv{( 11,-20) & ( 13,-10) & (3,0)}?(1)*\dir{>};
}};
 ( 0,0)*{\dred\xybox{  
    ( 2.5,-18)*{}; ( 2.5, 14) 
  **\crv{( 2.5,-15) & ( 6,-14) & (12,-4) & ( 12,-2) & (12,2) & (6,8) & (2.5,11) }
  ?(1)*\dir{>}?(0.5)*\dir{>};
    (-2.5,-18)*{}; (-2.5, 14) 
  **\crv{(-2.5,-15) & (-6,-14) & (-12,-4) & (-12,-2) & (-12,2) & (-6,8) & (-2.5,11) }
  ?(0)*\dir{<}?(0.45)*\dir{<};
}}; 
    (18,10)*{(1^n) };   
    ( -15,-18)*{\scs   i}; (-10,-18)*{\scs   i};
    (-3.5,-18.1)*{\scs i+1}; (3.5,-18.1)*{\scs i+1};
    (  10,-18)*{\scs   i}; ( 16,-18)*{\scs   i}; 
\endxy\
\overset{\eqref{eq_r3_hard-gen}}{=}\ \ \
\xy 0;/r.16pc/; 
( 0,0)*{\dblue\xybox{
   (-12,8)*{};(-12,-24) **\crv{(-11,4) & (-13,-8) & (-11,-20)}?(1)*\dir{>};
   ( -7,8)*{};(-7,-24) **\crv{( 3,0) & ( 14,-8) & (3,-16)}?(0)*\dir{<};
   ( 12,8)*{};(12,-24) **\crv{( 2,0) & (-9,-8) & (2,-16)}?(1)*\dir{>};
   ( 17,8)*{};( 17,-24) **\crv{(16,4) & (19,-8) & (16,-20)}?(0)*\dir{<};
}};
 ( 0,0)*{\dred\xybox{  
    ( 2.5,-18)*{}; ( 2.5, 14) 
  **\crv{( 2.5,-15) & ( 6,-14) & (12,-4) & ( 12,-2) & (12,2) & (6,8) & (2.5,11) }
  ?(1)*\dir{>}?(0.5)*\dir{>};
    (-2.5,-18)*{}; (-2.5, 14) 
  **\crv{(-2.5,-15) & (-6,-14) & (-12,-4) & (-12,-2) & (-12,2) & (-6,8) & (-2.5,11) }
  ?(0)*\dir{<}?(0.45)*\dir{<};
}};
    (20,10)*{(1^n) };   
    ( -15,-18)*{\scs   i}; (-10,-18)*{\scs   i};
    (-3.5,-18.1)*{\scs i+1}; (3.5,-18.1)*{\scs i+1};
    (  10,-18)*{\scs   i}; ( 14.5,-18)*{\scs   i}; 
\endxy\
\\[1.5ex]\displaybreak[0]
&
\mspace{45mu}
\overset{\eqref{eq_ident_decomp0}+\eqref{eq:bubb_deg0}}{=}\
\ \ \
\xy 0;/r.16pc/; 
( 0,0)*{\dblue\xybox{
   (-12,8)*{};(-12,-24) **\crv{(-11,4) & (-13,-8) & (-11,-20)}?(1)*\dir{>};
   ( -7,8)*{};(-7,-24) **\crv{(-5,0) & ( 2,-8) & (-5,-16)}?(0)*\dir{<};
   ( 12,8)*{};(12,-24) **\crv{(10,0) & (3,-8) & (10,-16)}?(1)*\dir{>};
   ( 17,8)*{};(17,-24) **\crv{(16,4) & (19,-8) & (16,-20)}?(0)*\dir{<};
}};
(0,0)*{\dred\xybox{
   (  7,8)*{};( 7,-24) **\crv{(9,0) & (16,-8) & (9,-16)}?(0)*\dir{<};
   ( -2,8)*{};(-2,-24) **\crv{(-4,0) & (-11,-8) & (-4,-16)}?(1)*\dir{>};
}};
    (20,10)*{(1^n) };   
    ( -15,-18)*{\scs   i}; (-10,-18)*{\scs   i};
    (-3.5,-18.1)*{\scs i+1}; (3.5,-18.1)*{\scs i+1};
    (  10,-18)*{\scs   i}; ( 14.5,-18)*{\scs   i}; 
\endxy\
+\ \ \
\xy 0;/r.16pc/; 
( 0,0)*{\dblue\xybox{
   (-12,8)*{};(-12,-24) **\crv{(-11,4) & (-13,-8) & (-11,-20)}?(1)*\dir{>};
   ( 17,8)*{};( 17,-24) **\crv{(16,4) & (19,-8) & (16,-20)}?(0)*\dir{<};
   (-7, 8)*{};(12, 8) **\crv{(-5,  4) & ( 2.5, -5) & (12,  4)}?(0)*\dir{<};
   (-7,-24)*{};(12,-24) **\crv{(-5,-20) & ( 2.5,-11) & (12,-20)}?(1)*\dir{>};
}};
(0,0)*{\dred\xybox{
   (  7,8)*{};( 7,-24) **\crv{(9,0) & (11,-8) & (9,-16)}?(0)*\dir{<};
   ( -2,8)*{};(-2,-24) **\crv{(-4,0) & (-6,-8) & (-4,-16)}?(1)*\dir{>};
}};
    (20,10)*{(1^n) };   
    ( -15,-18)*{\scs   i}; (-10,-18)*{\scs   i};
    (-3.5,-18.1)*{\scs i+1}; (3.5,-18.1)*{\scs i+1};
    (  10,-18)*{\scs   i}; ( 14.5,-18)*{\scs   i}; 
\endxy\
\\[1.5ex]\displaybreak[0]
&
\mspace{60mu}
\overset{\eqref{eq_downup_ij-gen}}{=}
\mspace{20mu}
\ \ \
\xy 0;/r.16pc/; 
( 0,0)*{\dblue\xybox{
   (-12,8)*{};(-12,-24) **\crv{(-11,4) & (-13,-8) & (-11,-20)}?(1)*\dir{>};
   ( -7,8)*{};( -7,-24) **\crv{( -8,4) & ( -6,-8) & ( -8,-20)}?(0)*\dir{<};
   ( 12,8)*{};( 12,-24) **\crv{(13,4) & (11,-8) & (13,-20)}?(1)*\dir{>};
   ( 17,8)*{};( 17,-24) **\crv{(16,4) & (19,-8) & (16,-20)}?(0)*\dir{<};
}};
(0,0)*{\dred\xybox{
   (  7,8)*{};( 7,-24) **\crv{(7,0) & ( 6,-8) & (7,-16)}?(0)*\dir{<};
   ( -2,8)*{};(-2,-24) **\crv{(-2,0) & (0,-8) & (-2,-16)}?(1)*\dir{>};
}};
    (20,10)*{(1^n) };   
    ( -15,-18)*{\scs   i}; (-10,-18)*{\scs   i};
    (-3.5,-18.1)*{\scs i+1}; (3.5,-18.1)*{\scs i+1};
    (  10,-18)*{\scs   i}; ( 14.5,-18)*{\scs   i}; 
\endxy\
+\ \ \
\xy 0;/r.16pc/; 
( 0,0)*{\dblue\xybox{
   (-12,8)*{};(-12,-24) **\crv{(-11,4) & (-13,-8) & (-11,-20)}?(1)*\dir{>};
   ( 17,8)*{};( 17,-24) **\crv{(16,4) & (19,-8) & (16,-20)}?(0)*\dir{<};
   (-7, 8)*{};(12, 8) **\crv{(-5,  4) & ( 2.5, -5) & (12,  4)}?(0)*\dir{<};
   (-7,-24)*{};(12,-24) **\crv{(-5,-20) & ( 2.5,-11) & (12,-20)}?(1)*\dir{>};
}};
(0,0)*{\dred\xybox{
   (  7,8)*{};( 7,-24) **\crv{(9,0) & (11,-8) & (9,-16)}?(0)*\dir{<};
   ( -2,8)*{};(-2,-24) **\crv{(-4,0) & (-6,-8) & (-4,-16)}?(1)*\dir{>};
}};
    (20,10)*{(1^n) };   
    ( -15,-18)*{\scs   i}; (-10,-18)*{\scs   i};
    (-3.5,-18.1)*{\scs i+1}; (3.5,-18.1)*{\scs i+1};
    (  10,-18)*{\scs   i}; ( 14.5,-18)*{\scs   i}; 
\endxy\ .
\end{align*}
Applying~\eqref{eq_ident_decomp} to the vertical red strands in the second term, 
followed by~\eqref{eq_downup_ij-gen} and~\eqref{eq:bubb_deg0}, we get that it is equal to
\begin{equation*}
\xy 0;/r.16pc/; 
( 0,0)*{\dblue\xybox{
   (-12,10)*{};(-12,-16) **\crv{(-11,4) & (-13,-3) & (-11,-10)}?(1)*\dir{>};
   ( 17,10)*{};( 17,-16) **\crv{(16,4) & (19,-3) & (16,-10)}?(0)*\dir{<};
   (-7, 10)*{};(12, 10) **\crv{(-7,  4) & ( 2.5, -5) & (12,  4)}?(0)*\dir{<};
   (-7,-16)*{};(12,-16) **\crv{(-7,-10) & ( 2.5,-1) & (12,-10)}?(1)*\dir{>};
}};
(0,0)*{\dred\xybox{
   ( -2, 10)*{};(7, 10) **\crv{(-2,8) & (2.5,1) & (7,8)}?(1)*\dir{>};
   ( -2,-16)*{};(7,-16) **\crv{(-2,-14) & (2.5,-7) & (7,-14)}?(0)*\dir{<};
}};
    (20,8)*{(1^n) };   
    ( -15,-15)*{\scs   i}; (-10,-15)*{\scs   i};
    (-3.5,-15.1)*{\scs i+1}; (3.5,-15.1)*{\scs i+1};
    (  10,-15)*{\scs   i}; ( 14.5,-15)*{\scs   i}; 
\endxy\ ,
\end{equation*}
which equals $\Sigma_{n,n}\bigl(\figins{-5}{0.22}{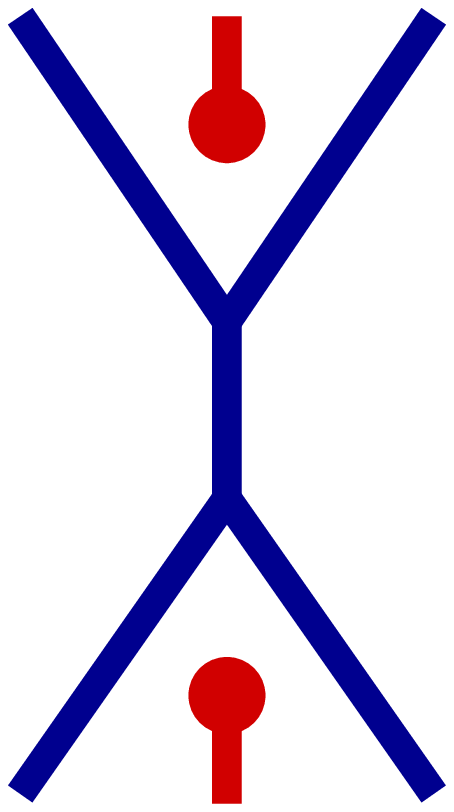}\bigr)$.
Therefore $\Sigma_{n,n}\bigl(\figins{-5}{0.22}{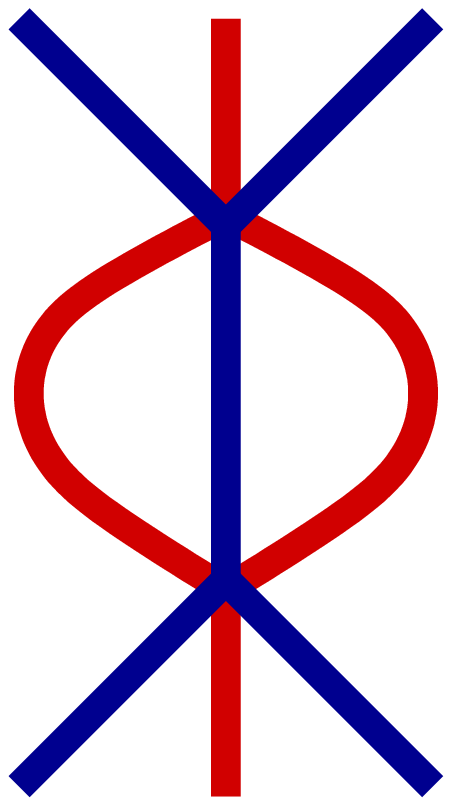}\bigr)
= \Sigma_{n,n}\bigl(\figins{-5}{0.22}{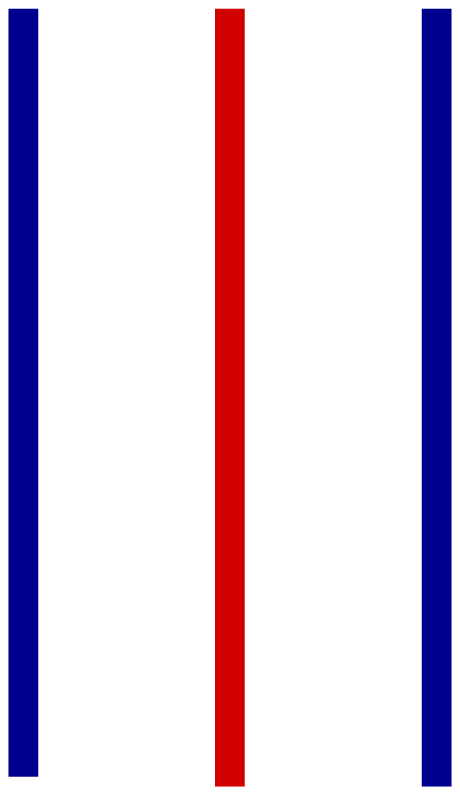}\bigr) +
\Sigma_{n,n}\bigl(\figins{-5}{0.22}{dumbell-dd-short-inv}\bigr)$.

We now prove relation~\eqref{eq:dumbsq}. We denote the left and right hand 
sides of~\eqref{eq:dumbsq} $L$ and $R$, respectively. We have
\begin{equation*}
\Sigma_{n,n}(L)\
=\ \ \
\xy 0;/r.16pc/; 
( 0,0)*{\dblue\xybox{
    (-7,10)*{};(-7,-26) **\crv{(-4.5, 8) & ( 4.5,3) & ( 5.5,-8.5) & (4.5,-20) & (-4.5,-24)}
   ?(0)*\dir{<}?(0.47)*\dir{<};
    ( 12,10)*{};(12,-26) **\crv{(9.5, 8) & (0.5,3) & (-0.5,-8.5) & (0.5,-20) & (9.5,-24)}
   ?(1)*\dir{>}?(0.5)*\dir{>};
    (17.5,10)*{};(-12.5, 10) **\crv{(7,1.5) & (6,1) & ( 2.5,-2) & (-1,1) & (-2,1.5)}
   ?(0)*\dir{<};
    (17.5,-26)*{};(-12.5,-26) **\crv{(7,-17.5) & (6,-17) & (2.5,-14) & (-1,-17) & (-2,-17.5)}
   ?(1)*\dir{>};
}};
( 0,0)*{\dred\xybox{  
      ( 2.5,-20)*{}; (15.5,-4.5) **\crv{( 2.5,-18) & ( 6,-13) & (10,-6) & ( 12,-4.5)}
  ?(1)*\dir{>};
      ( 2.5, 16)*{}; (15.5, 0.5) **\crv{( 2.5,14) & ( 6,9) & (10,2) & ( 12, 0.5)}
  ?(0)*\dir{<};
      (-2.5,-20)*{}; (-15.5,-4.5) **\crv{(-2.5,-18) & (-6,-13) & (-10,-6) & (-12,-4.5)}
  ?(0)*\dir{<};
      (-2.5, 16)*{}; (-15.5, 0.5) **\crv{(-2.5,14) & (-6,9) & (-10,2) & (-12, 0.5)}
  ?(1)*\dir{>};
      (-6,-2)*{}; (6,-2) **\crv{(-6,0) & (0,5) & (6,0)}?(0.01)*\dir{>};
      (-6,-2)*{}; (6,-2) **\crv{(-6,-4) & (0,-9) & (6,-4)};
}};
    (18,10)*{(1^n) };   
    ( -16,-20)*{\scs   i}; (-10,-20)*{\scs   i};
    (-3.5,-20.1)*{\scs i+1}; (3.5,-20.1)*{\scs i+1};
    (  10,-20)*{\scs   i}; ( 16,-20)*{\scs   i}; 
    ( -16, 20)*{\scs i}; (7,-4)*{\scs i+1};
\endxy\
=\ \ \
\xy 0;/r.16pc/; 
( 0,0)*{\dblue\xybox{
    (-7.5,-26)*{};( 17.5,10) **\crv{(-7.5,-26) & (7,-10) & ( 17.5,10)}?(1)*\dir{>};
    (-12.5,-26)*{};( 12.5,10) **\crv{(-12.5,-26) & (-2,-6) & ( 12.5,10)}?(0)*\dir{<};
    (12,-26)*{};(-12.5,10) **\crv{(12,-26) & (-2,-10) & (-12.5,10)}?(0)*\dir{<};
    (17.5,-26)*{};(-7.5,10) **\crv{(17.5,-26) & (7,-6) & (-7.5,10)}?(1)*\dir{>};
}};
( 0,0)*{\dred\xybox{  
      ( 2.5,-20)*{}; (15.5,-4.5) **\crv{( 2.5,-18) & ( 6,-13) & (10,-6) & ( 12,-4.5)}
  ?(1)*\dir{>};
      ( 2.5, 16)*{}; (15.5, 0.5) **\crv{( 2.5,14) & ( 6,9) & (10,2) & ( 12, 0.5)}
  ?(0)*\dir{<};
      (-2.5,-20)*{}; (-15.5,-4.5) **\crv{(-2.5,-18) & (-6,-13) & (-10,-6) & (-12,-4.5)}
  ?(0)*\dir{<};
      (-2.5, 16)*{}; (-15.5, 0.5) **\crv{(-2.5,14) & (-6,9) & (-10,2) & (-12, 0.5)}
  ?(1)*\dir{>};
}};
    (20,10)*{(1^n) };   
    ( -16,-20)*{\scs   i}; (-10,-20)*{\scs   i};
    (-3.5,-20.1)*{\scs i+1}; (3.5,-20.1)*{\scs i+1};
    (  10,-20)*{\scs   i}; ( 16,-20)*{\scs   i}; 
    ( -16, 20)*{\scs i};
\endxy
\end{equation*}
The second equality is obtained as in Equation~\eqref{eq:reid3proof}. 
The same argument shows that this equals 
$\Sigma_{n,n}(R)$.

Relation~\eqref{eq:slidenext} is straightforward to check (it only uses bubble 
slides).

\medskip
\n\emph{Relations involving three colors:}
Relations~\eqref{eq:slide6v} and~\eqref{eq:slide4v} are easy because the 
green strands have 
to be distant from red and blue and so we have all Reidemeister 2 and 
Reidemeister 3 like 
moves between green and one of the other colors.

It remains to prove that $\Sigma_{n,n}$ respects relation~\eqref{eq:dumbdumbsquare}. 
First notice that the diagrams on the left- and right-hand side of~\eqref{eq:dumbdumbsquare}
 are invariant under $180^{\circ}$ rotations and that they can be obtained from one 
 another using a $90^{\circ}$ rotation.
Therefore it suffices to show that the image of one of them is invariant under 
$90^{\circ}$ rotations.
Denote by $L$ the diagram on the left-hand-side of~\eqref{eq:dumbdumbsquare}. 
Then
\begin{equation*}
\Sigma_{n,n}(L)\
=\ \ \
\xy 0;/r.16pc/; 
( 0,0)*{\dblue\xybox{
    (-7,10)*{};(-15,-26) **\crv{(-4.5, 8) & ( 4.5,3) & ( 5.5,-8.5) & (4.5,-20) & (-6.5,-22)}
   ?(0)*\dir{<};
    ( 20,10)*{};(12,-26) **\crv{(11.5, 6) & (0.5,3) & (-0.5,-8.5) & (0.5,-20) & (9.5,-24)}
   ?(1)*\dir{>};
    (25.5,10)*{};(-12.5, 10) **\crv{(11,1) & (6,-1) & ( 2,-2) & (-1,1) & (-2,1.5)}
   ?(0)*\dir{<};
    (17.5,-26)*{};(-20.5,-26) **\crv{(7,-17.5) & (6,-17) & (2.5,-14) & (-1,-16) & (-6,-17.5)}
   ?(1)*\dir{>};
}};
( 0,0)*{\dred\xybox{  
      ( 2.5,-20)*{}; ( 2.5,16) **\crv{( 2.5,-14) & ( 22,-3) & (2.5,10)}
  ?(1)*\dir{>};
      (-2.5,-20)*{}; (-2.5,16) **\crv{(-2.5,-14) & (-22,-3) & (-2.5,10)}
  ?(0)*\dir{<};
      (-23,-5)*{}; (23,-5) **\crv{(-9,-4.5) & (-2.8,3) & (3.0,3) & (9,-4.5)}
  ?(1)*\dir{>};
      (-23,0)*{}; (23,0) **\crv{(-9,1.5) & (-2.8,-7) & (3.0,-7) & (9,1.5)}
  ?(0)*\dir{<};
}};
(0,0)*{\dgreen\xybox{
      (-20.5, 10)*{}; (-10.5,-26) **\crv{(-14.5,4) & (-18.5,-16)}?(1)*\dir{>};
      ( 15.5, 10)*{}; ( 25.5,-26) **\crv{( 25.5,-2) & (19.5,-20)}?(0)*\dir{<};
      (-16.5, 10)*{}; ( 11.5, 10) **\crv{(-13,-2) & (-8.5,-4) & (1,-8) & (5.5,-7) & (22,-2)}?(0)*\dir{<};
      (- 6.5,-26)*{}; ( 21.5,-26) **\crv{(-16.5,-14) & (-1,-9) & (4,-9) & (13,-11) & (17.5,-13)}?(1)*\dir{>};
}};
    (   24, 10)*{(1^n) };   
    (-24.5,-20)*{\scs   i}; (-18,-20)*{\scs   i};
    (  24 , 20)*{\scs   i}; ( 18,20)*{\scs   i};
    (- 3.5,-20.1)*{\scs i+1}; (3.5,-20.1)*{\scs i+1};
    (  -27, -3)*{\scs i+1};(-27,2)*{\scs i+1};
    (-26.5, 17)*{\scs i+2};(-20,20)*{\scs i+2};
    ( 26.5,-17)*{\scs i+2};(20,-20)*{\scs i+2};
\endxy\ .
\end{equation*}
Taking into account that the green strands are distant from the blue ones, we apply~\eqref{eq_downup_ij-gen} 
and a sequence of Reidemeister 3 like moves to obtain 
\begin{equation*}
\Sigma_{n,n}(L)\
=\ \ \
\xy 0;/r.16pc/; 
( 0,0)*{\dblue\xybox{
    (-7,10)*{};(-15,-26) **\crv{(-4.5, 8) & ( 4.5,3) & ( 5.5,-8.5) & (4.5,-20) & (-6.5,-22)}
   ?(0)*\dir{<};
    ( 20,10)*{};(12,-26) **\crv{(11.5, 6) & (0.5,3) & (-0.5,-8.5) & (0.5,-20) & (9.5,-24)}
   ?(1)*\dir{>};
    (25.5,10)*{};(-12.5, 10) **\crv{(11,1) & (6,-1) & ( 2,-2) & (-1,1) & (-2,1.5)}
   ?(0)*\dir{<};
    (17.5,-26)*{};(-20.5,-26) **\crv{(7,-17.5) & (6,-17) & (2.5,-14) & (-1,-16) & (-6,-17.5)}
   ?(1)*\dir{>};
}};
( 0,0)*{\dred\xybox{  
      ( 2.5,-20)*{}; ( 2.5,16) **\crv{( 2.5,-14) & ( 18,-3) & (2.5,10)}
  ?(1)*\dir{>};
      (-2.5,-20)*{}; (-2.5,16) **\crv{(-2.5,-14) & (-18,-3) & (-2.5,10)}
  ?(0)*\dir{<};
      (-23,-5)*{}; (23,-5) **\crv{(-9,-5) & (0,4.5) & (9,-5)}
  ?(1)*\dir{>};
      (-23,0)*{}; (23,0) **\crv{(-9,0) & (0,-9.5) & (9,0)}
  ?(0)*\dir{<};
}};
(0,0)*{\dgreen\xybox{
      (-20.5, 10)*{}; (-10.5,-26) **\crv{(-14.5,4) & (-10.5,-20)}?(1)*\dir{>};
      ( 15.5, 10)*{}; ( 25.5,-26) **\crv{( 15.5,4) & (19.5,-20)}?(0)*\dir{<};
      (-16.5, 10)*{}; ( 11.5, 10) **\crv{(-11.5,7) & (4.5,7)}?(0)*\dir{<};
      (- 6.5,-26)*{}; ( 21.5,-26) **\crv{( 1.5,-23) & (14.5,-23)}?(1)*\dir{>};
}};
    (   24, 10)*{(1^n) };   
    (-24.5,-20)*{\scs   i}; (-18,-20)*{\scs   i};
    (  24 , 20)*{\scs   i}; ( 18,20)*{\scs   i};
    (- 3.5,-20.1)*{\scs i+1}; (3.5,-20.1)*{\scs i+1};
    (  -27, -3)*{\scs i+1};(-27,2)*{\scs i+1};
    (-26.5, 17)*{\scs i+2};(-20,20)*{\scs i+2};
    ( 26.5,-17)*{\scs i+2};(20,-20)*{\scs i+2};
\endxy\ .
\end{equation*}
Using \eqref{eq_other_r3_1} twice 
between the two horizontal red lines and a vertical blue line, followed by
~\eqref{eq_ident_decomp} gives
\begin{equation*}
\Sigma_{n,n}(L)\
=\ \ \
\xy 0;/r.16pc/; 
( 0,0)*{\dblue\xybox{
    (-7,10)*{};(-15,-26) **\crv{(-4.5, 8) & ( 4.5,3) & ( 5.5,-8.5) & (4.5,-20) & (-6.5,-22)}
   ?(0)*\dir{<};
    ( 20,10)*{};(12,-26) **\crv{(11.5, 6) & (0.5,3) & (-0.5,-8.5) & (0.5,-20) & (9.5,-24)}
   ?(1)*\dir{>};
    (25.5,10)*{};(-12.5, 10) **\crv{(11,1) & (6,-1) & ( 2,-2) & (-1,1) & (-2,1.5)}
   ?(0)*\dir{<};
    (17.5,-26)*{};(-20.5,-26) **\crv{(7,-17.5) & (6,-17) & (2.5,-14) & (-1,-16) & (-6,-17.5)}
   ?(1)*\dir{>};
}};
( 0,0)*{\dred\xybox{  
      ( 2.5,-20)*{}; ( 2.5,16) **\crv{( 2.5,-14) & ( 18,-3) & (2.5,10)}
  ?(1)*\dir{>};
      (-2.5,-20)*{}; (-2.5,16) **\crv{(-2.5,-14) & (-18,-3) & (-2.5,10)}
  ?(0)*\dir{<};
      (-23,-5)*{}; (23,-5) **\crv{(-9,-5) & (0,-3) & (9,-5)}?(1)*\dir{>};
      (-23,0)*{}; (23,0) **\crv{(-9,0) & (0,-2) & (9,0)}?(0)*\dir{<};
}};
(0,0)*{\dgreen\xybox{
      (-20.5, 10)*{}; (-10.5,-26) **\crv{(-14.5,4) & (-10.5,-20)}?(1)*\dir{>};
      ( 15.5, 10)*{}; ( 25.5,-26) **\crv{( 15.5,4) & (19.5,-20)}?(0)*\dir{<};
      (-16.5, 10)*{}; ( 11.5, 10) **\crv{(-11.5,7) & (4.5,7)}?(0)*\dir{<};
      (- 6.5,-26)*{}; ( 21.5,-26) **\crv{( 1.5,-23) & (14.5,-23)}?(1)*\dir{>};
}};
    (   24, 10)*{(1^n) };   
    (-24.5,-20)*{\scs   i}; (-18,-20)*{\scs   i};
    (  24 , 20)*{\scs   i}; ( 18,20)*{\scs   i};
    (- 3.5,-20.1)*{\scs i+1}; (3.5,-20.1)*{\scs i+1};
    (  -27, -3)*{\scs i+1};(-27,2)*{\scs i+1};
    (-26.5, 17)*{\scs i+2};(-20,20)*{\scs i+2};
    ( 26.5,-17)*{\scs i+2};(20,-20)*{\scs i+2};
\endxy\ .
\end{equation*}
Notice that the sums in~\eqref{eq_ident_decomp} are not increasing 
and therefore there are no terms with dots here.
Applying~\eqref{eq_other_r3_1} and~\eqref{eq_r3_extra} to the top and bottom 
 we can pass the top and bottom $(i,i)$ crossings to the middle of the diagram 
(the terms coming from the sums in~\eqref{eq_r3_extra} are zero).
We get  
\begin{equation*}
\Sigma_{n,n}(L)\
=\ \ \ 
\xy 0;/r.16pc/; 
( 0,0)*{\dblue\xybox{
    (-7,10)*{};(-15,-26) **\crv{(-3.5, 8) & ( 5.5,-6) & ( 5,-8.5) & (5.5,-11) & (-4.5,-22)}
   ?(0)*\dir{<};
    ( 20,10)*{};(12,-26) **\crv{(11.5, 6) & (-0.5,-6) & ( 0,-8.5) & (-0.5,-11) & (9.5,-24)}
   ?(1)*\dir{>};
    (25.5,10)*{};(-12.5, 10) **\crv{(11,3) & (6,1) & ( 2,2) & (-1,3) & (-2,3.5)}
   ?(0)*\dir{<};
    (17.5,-26)*{};(-20.5,-26) **\crv{(7,-19.5) & (6,-19) & (2.5,-17.5) & (-1,-18) & (-6,-20)}
   ?(1)*\dir{>};
}};
( 0,0)*{\dred\xybox{  
      ( 2.5,-20)*{}; ( 2.5,16) **\crv{( 6,-17) & ( 12,-14) & (12, 8) & (6,13)}
  ?(1)*\dir{>};
      (-2.5,-20)*{}; (-2.5,16) **\crv{(-6,-17) & (-12,-14) & (-12, 8) & (-6,13)}
  ?(0)*\dir{<};
      (-23,-5)*{}; (23,-5) **\crv{(-19,-5) & (0,-16) & (19,-5)}?(1)*\dir{>};
      (-23,0)*{}; (23,0) **\crv{(-19,0) & (0,11) & (19,0)}?(0)*\dir{<};
}};
(0,0)*{\dgreen\xybox{
      (-20.5, 10)*{}; (-10.5,-26) **\crv{(-14.5,4) & (-10.5,-20)}?(1)*\dir{>};
      ( 15.5, 10)*{}; ( 25.5,-26) **\crv{( 15.5,4) & (19.5,-20)}?(0)*\dir{<};
      (-16.5, 10)*{}; ( 11.5, 10) **\crv{(-11.5,8) & (4.5,6)}?(0)*\dir{<};
      (- 6.5,-26)*{}; ( 21.5,-26) **\crv{( 1.5,-22) & (14.5,-24)}?(1)*\dir{>};
}};
    (   24, 10)*{(1^n) };   
    (-24.5,-20)*{\scs   i}; (-18,-20)*{\scs   i};
    (  24 , 20)*{\scs   i}; ( 18,20)*{\scs   i};
    (- 3.5,-20.1)*{\scs i+1}; (3.5,-20.1)*{\scs i+1};
    (  -27, -3)*{\scs i+1};(-27,2)*{\scs i+1};
    (-26.5, 17)*{\scs i+2};(-20,20)*{\scs i+2};
    ( 26.5,-17)*{\scs i+2};(20,-20)*{\scs i+2};
\endxy\ .
\end{equation*}
Using~\eqref{eq_ident_decomp0} in the middle of the diagram followed 
by~\eqref{eq_downup_ij-gen} and a sequence of 
Reidemeister 3 like moves to pass the vertical red strands to the middle gives
\begin{equation*}
\Sigma_{n,n}(L)\
=\ \ \ 
\xy 0;/r.16pc/; 
( 0,0)*{\dblue\xybox{
    (-7,10)*{};(-15,-26) **\crv{(-6.5, 8) & (-4.5,-6) & (-4,-11) & (-4.5,-22)}
   ?(0)*\dir{<};
    ( 20,10)*{};(12,-26) **\crv{(11.5, 7) & (11,-6) & ( 11,-8.5) & (11,-11)}
   ?(1)*\dir{>};
    (24.5,8)*{};(-12.5, 10) **\crv{(11,3) & (6,1) & ( 2,2) & (-1,3) & (-2,3.5)}
   ?(0)*\dir{<};
    (17.5,-26)*{};(-19.5,-24) **\crv{(7,-19.5) & (6,-19) & (2.5,-17.5) & (-1,-18) & (-6,-20)}
   ?(1)*\dir{>};
}};
( 0,0)*{\dred\xybox{  
      ( 2.5,-20)*{}; ( 2.5,16) **\crv{( 3.5,-6) & ( 3.5,6)}
  ?(1)*\dir{>};
      (-2.5,-20)*{}; (-2.5,16) **\crv{(-3.5,-6) & (-3.5,6)}
  ?(0)*\dir{<};
      (-23,-5)*{}; (23,-5) **\crv{(-9,-6) & (9,-6)}?(1)*\dir{>};
      (-23,0)*{}; (23,0) **\crv{(-9,1) & (9,1)}?(0)*\dir{<};
}};
(0,0)*{\dgreen\xybox{
      (-19.5, 8)*{}; (-10.5,-26) **\crv{(-14.5,4) & (-10.5,-20)}?(1)*\dir{>};
      ( 15.5, 10)*{}; ( 24.5,-24) **\crv{( 15.5,4) & (19.5,-20)}?(0)*\dir{<};
      (-16.5, 10)*{}; ( 11.5, 10) **\crv{(-11.5,8) & (4.5,6)}?(0)*\dir{<};
      (- 6.5,-26)*{}; ( 21.5,-26) **\crv{( 1.5,-22) & (14.5,-24)}?(1)*\dir{>};
}};
    (   24, 10)*{(1^n) };   
    (-24.5,-20)*{\scs   i}; (-18,-20)*{\scs   i};
    (  24 , 20)*{\scs   i}; ( 18,20)*{\scs   i};
    (- 3.5,-20.1)*{\scs i+1}; (3.5,-20.1)*{\scs i+1};
    (  -27, -3)*{\scs i+1};(-27,2)*{\scs i+1};
    (-26.5, 17)*{\scs i+2};(-20,20)*{\scs i+2};
    ( 26.5,-17)*{\scs i+2};(20,-20)*{\scs i+2};
\endxy\ ,
\end{equation*}
which is symmetric under $90^{\circ}$ rotations.
\end{proof}

\subsection{$\mathcal{SC}_1(n)$ is a full sub-2-category of $\Scat(n,n)$}

\begin{lem}
\label{lem:comm}
The following diagram commutes  
\begin{equation*}
\xymatrix{
\mathcal{SC}_1(n)\ar[rr]^{\fek}\ar[dr]_{\Sigma_{n,n}} && \bim(n)^*
\\
& \Scat(n,n)^*((1^n),(1^n))\ar[ur]_{\fbim} &
}
\end{equation*}
\end{lem}
\begin{proof}
The commutativity of the diagram can be checked by direct computation. Most of the 
computation is straightforward except for the 6-valent vertex. 
To compute its image under $\mathcal{F}_{Bim}\Sigma_{n,n}$ we divide it in layers and 
compute the bimodule maps for each layer.
We do the case with the colors as in Equation~\ref{eq:sixval}, the other case being similar. 
Remember that
\begin{equation*}
\labellist
\tiny\hair 2pt
\pinlabel $i+1$ at -5 -10
\pinlabel $i$   at 65 -10
\endlabellist
\figins{-18}{0.6}{6vertu}\ \ \
\xmapsto{\Sigma_{n,n}} \ \ 
\text{$
\xy 0;/r.16pc/; 
    (16,0)*{(1^n) };    
    ( 0,0)*{\dblue\xybox{
    (-7.5,10)*{}; (    5,-10) **\crv{(-4.5, 7) & ( 7.5,0) & ( 5,-9)}?(0)*\dir{<};
    (12.5,10)*{}; (    0,-10) **\crv{( 9.5, 7) & (-2.5,0) & ( 0,-9)}?(1)*\dir{>};
    (17.5,10)*{}; (-12.5, 10) **\crv{(   8, 0) & ( 2.5,-6) & (-3,0)}?(0)*\dir{<};
}};
    ( 0,0)*{\dred\xybox{  
    (-10,-20)*{};( 10,-20) **\crv{(-9,-19) & (0,-12) & (8,-19)}?(.2)*\dir{>} ?(.8)*\dir{>};
    ( 2.5, 0)*{};( 15,-20) **\crv{( 2.5,0) & ( 2,-10) & ( 15,-20)}?(0)*\dir{<};
    (-2.5, 0)*{};(-15,-20) **\crv{(-2.5,0) & (-2,-10) & (-15,-20)}?(1)*\dir{>};
}};  
    ( -17,-12)*{\scs i+1}; (-10,-12)*{\scs i+1};
    (-2.5,-12)*{\scs i }; (2.5,-12)*{\scs i };
    ( 16,-12)*{\scs i+1}; 
    ( 16, 12)*{\scs i};
    \endxy
$}\ \ .
\vspace*{2ex}
\end{equation*}
It is easy to see that the map corresponding to the layer
\begin{equation*}
\text{$
\xy 0;/r.16pc/; 
    (20,0)*{(1^n) };    
    ( 0,0)*{\dblue\xybox{
    ( 8.5,8)*{}; (5,-8) **\crv{( 8.5, 7) & ( 7.5,0) & ( 5,-8)}?(0)*\dir{<};
    (-3.5,8)*{}; (0,-8) **\crv{(-3.5, 7) & (-2.5,0) & ( 0,-8)}?(1)*\dir{>};
}};
    ( 0,0.5)*{\dred\xybox{  
    (-10,-18)*{};( 10,-18) **\crv{(-8,-14) & (0,0) & (8,-14)}?(.2)*\dir{>} ?(.83)*\dir{>};
    ( 12, -2)*{};( 15,-18) **\crv{(12,-10) & (15,-18)}?(0)*\dir{<};
    (-12, -2)*{};(-15,-18) **\crv{(-12,-10) & (-15,-18)}?(1)*\dir{>};
}};  
    ( -17,-10)*{\scs i+1}; (-10,-10)*{\scs i+1};
    (-2.5,-10)*{\scs i }; (2.5,-10)*{\scs i };
    ( 16,-10)*{\scs i+1}; 
    \endxy
$}
\end{equation*}
consists only of a relabeling of variables.
The next one is
\begin{align*}
\xy 0;/r.12pc/; 
    (0,5)*{\dblue\bbpef{\black i}};
    (16,-2)*{(1^n) };    
    (-12,3)*{};(12,3)*{};
    (-6.5,2)*{\dblue\xybox{
    (-3,-5)*{}; (-10,8.5) **\crv{(-3,1) & (-10,3)}?(0)*\dir{<};}};
    ( 6.5,2)*{\dblue\xybox{
    ( 3,-5)*{}; ( 10,8.5) **\crv{( 3,1) & ( 10,3)}?(1)*\dir{>};}};
    (-12.5,2)*{\dred\xybox{
    (-3,-5)*{}; (-10,8.5) **\crv{(-3,1) & (-10,3)}?(0)*\dir{<};}};
    (12.5,2)*{\dred\xybox{
    ( 3,-5)*{}; ( 10,8.5) **\crv{( 3,1) & ( 10,3)}?(1)*\dir{>};}};
    (-3,-7)*{\scs i};(3,-7)*{\scs i};
    (-9,-7.1)*{\scs i+1};(9,-7.1)*{\scs i+1};
    \endxy
\ \
\mapsto &
\ \
\left(\quad
\begin{aligned}
\labellist
\tiny\hair 2pt
\pinlabel $1$ at -20 483 \pinlabel $1$ at 140 483 \pinlabel $1$ at 295 483 
\pinlabel $y$ at 235 90 \pinlabel $x$ at 140 245 \pinlabel $z$ at 235 385
\endlabellist
\figins{-42}{1.20}{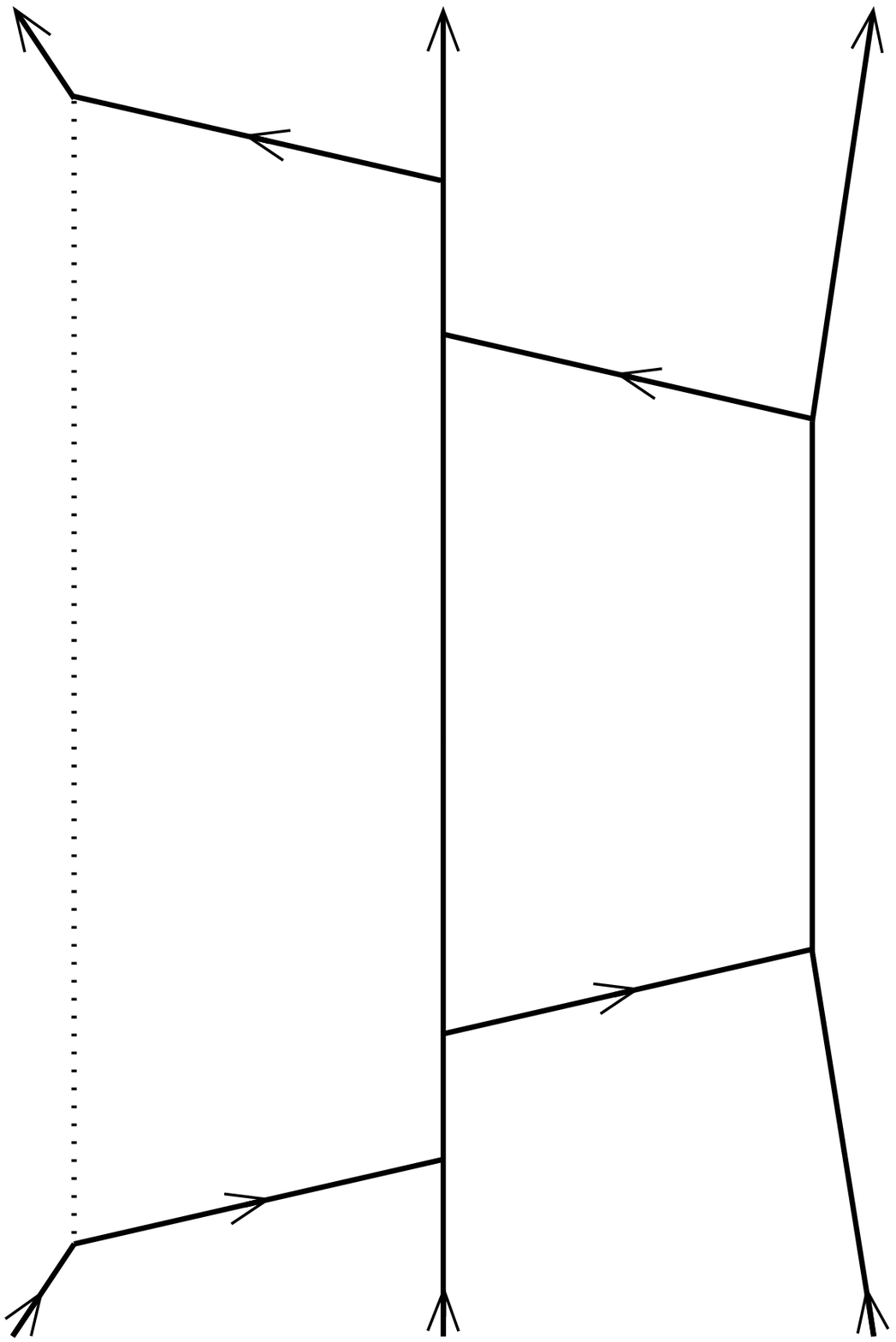}
\qra
\labellist
\tiny\hair 2pt
\pinlabel $1$ at -20 483 \pinlabel $1$ at 140 483 \pinlabel $1$ at 295 483 
\pinlabel $y$ at 235 90  \pinlabel $x'$ at 140 320 \pinlabel $x$ at 140 180
\pinlabel $z$ at 235 385 \pinlabel $t_1,t_2$ at 370 170
\endlabellist
\figins{-42}{1.20}{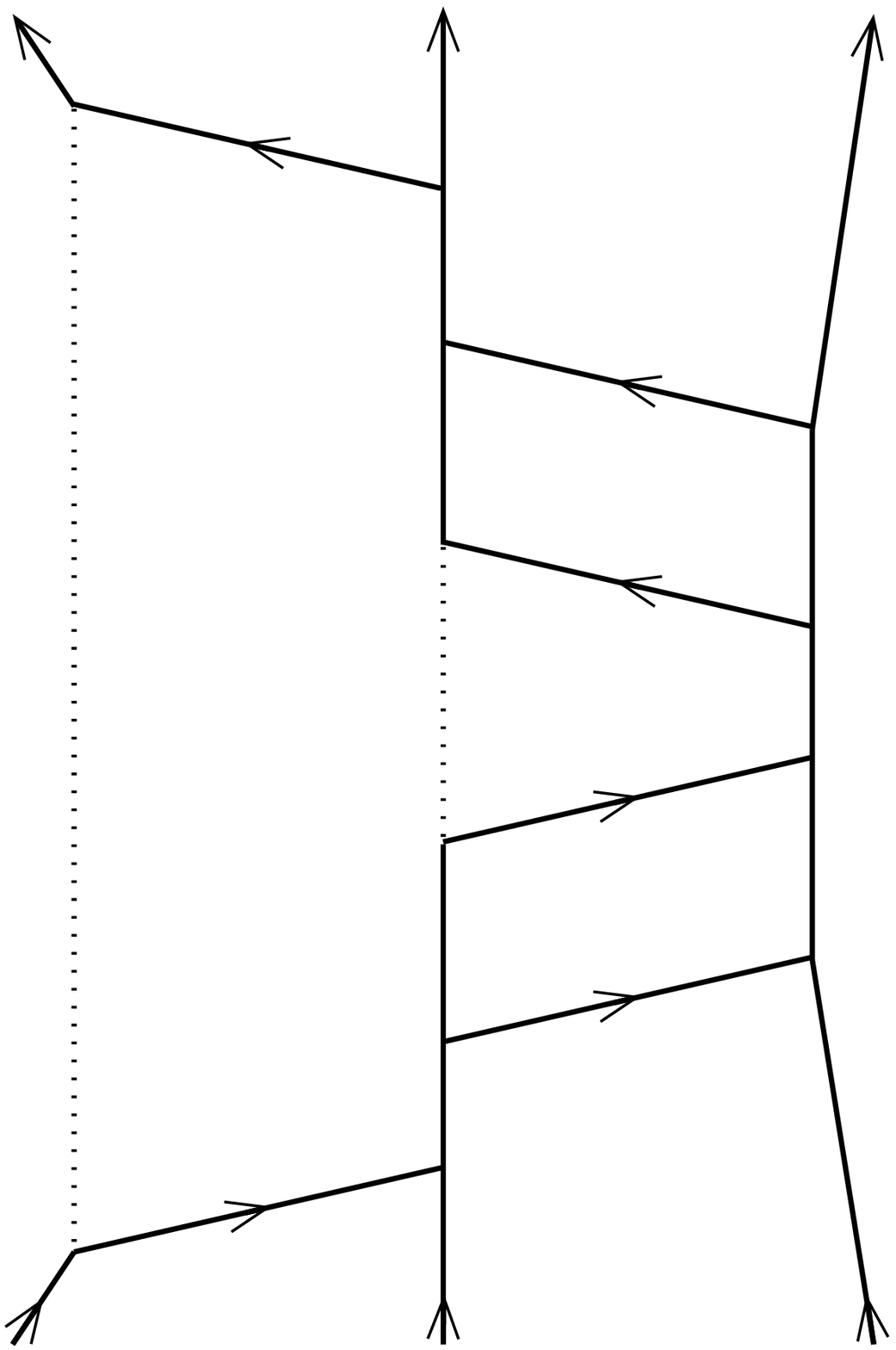}
&\quad\ 
\\[1.0ex]\
p\mapsto \sum\limits_{\ell=0}^{2}(-1)^{\ell}x'^{2-\ell}\varepsilon_{\ell}(t_1,t_2)p
\quad &
\end{aligned}
\right).
\end{align*}
The next step consists of the two crossings between strands labeled $i$,
\begin{equation*}
\text{$
\xy 0;/r.16pc/; 
    (20,0)*{(1^n) };    
    ( 0,0)*{\dblue\xybox{
    ( 12.5,8)*{}; (5,-8) **\crv{(8.5,0)}?(0)*\dir{<};
    ( 12.5,-8)*{}; (5,8) **\crv{(8.5,0)}?(1)*\dir{>};
    (-7.5,8)*{}; (0,-8) **\crv{(-3.5,0)}?(1)*\dir{>};
    (-7.5,-8)*{};(0,8) **\crv{(-3.5,0)}?(0)*\dir{<};
}};
    ( 0,0.5)*{\dred\xybox{  
    ( 15, -2)*{};( 15,-18) **\crv{(15,-10) & (15,-18)}?(0)*\dir{<};
    (-15, -2)*{};(-15,-18) **\crv{(-15,-10) & (-15,-18)}?(1)*\dir{>}; 
}};  
    ( -17,-10)*{\scs i+1}; (-10,-10)*{\scs i};
    (-2.5,-10)*{\scs i }; (2.5,-10)*{\scs i };
    (10,-10)*{\scs i}; 
    ( 16,-10)*{\scs i+1}; 
    \endxy
$}
\end{equation*}
corresponding to the map $p\mapsto \partial_{zx'}\partial_{xy}p$. 
The left pointing $(i,i)$-crossing and the remaining $(i,i+1)$ crossings consist only of
relabeling of variables and shifts. 
Putting everything together, the reader 
can check that this map coincides with the one obtained from $\cF_{EK}$ 
by a straightforward computation.
\end{proof}

We now get to the main result of this subsection. 
\begin{prop}
\label{prop:fulfaith}
The functor $\Sigma_{n,n}$ is an equivalence of categories.
\end{prop}
\begin{proof} We have to show that $\Sigma_{n,n}$ is essentially surjective and 
fully faithful. By the commutation 2-isomorphisms, i.e. the relations 
involving Reidemeister II and III type moves between diagrams in 
$\Scat(n,n)$, we can 
commute the factors of any object $x$ in $\Scat(n,n)^*((1^n),(1^n))$ so 
that it becomes a direct sum of objects whose factors are all of the form 
$\mathcal{E}_{-i}\mathcal{E}_{+i}1_n$. This is always possible because $x$ 
has to have as many factors $\mathcal{E}_{-j}$ as $\mathcal{E}_{+j}$, 
for any $j=1,\ldots,n-1$, or else $x$ contains a factor $1_{\lambda}$ with 
$\lambda\not\in\Lambda(n,n)$ and is therefore equal to zero. 
This shows that $\Sigma_{n,n}$ is essentially 
surjective.  

Since the functor $\fek$ is 
faithful~\cite{E-Kh}, it follows from Lemma~\ref{lem:comm} that 
$\Sigma_{n,n}$ is faithful too. Therefore it only remains to show that $\Sigma_{n,n}$ is 
full. To this end we first note that 
$$\tilde{\tau}(\mathcal{E}_{-i}\mathcal{E}_{+i}1_n)=
\mathcal{E}_{-i}\mathcal{E}_{+i}1_n.$$
By simply checking the definitions one sees that the natural isomorphisms 
in Corollary 4.12 
in~\cite{E-Kh} and the ones in Lemma~\ref{lem:tildetau} in this paper 
intertwine $\Sigma_{n,n}$. For example, we have a commutative square
$$
\begin{CD}
\HOM_{\mathcal{SC}_1(n)}(i\kk,\jj)@>{\cong}>>\HOM_{\mathcal{SC}_1(n)}(\kk,i\jj)\\
@V{\Sigma_{n,n}}VV @V{\Sigma_{n,n}}VV\\
\HOM_{\Scat(n,n)}(\mathcal{E}_{-i}\mathcal{E}_{+i}1_n\Sigma_{n,n}(\kk),\Sigma_{n,n}(\jj))
@>{\cong}>>
\HOM_{\Scat(n,n)}(\Sigma_{n,n}(\kk),
\mathcal{E}_{-i}\mathcal{E}_{+i}1_n\Sigma_{n,n}(\jj)).
\end{CD} 
$$
This observation together with the results 
after Corollary 4.12 in Section 4.3 in \cite{E-Kh} and the fact that 
$\Sigma_{n,n}$ is additive and $\Q$-linear implies that it is enough 
to prove that 
$$\Sigma_{n,n}\colon \HOM_{\mathcal{SC}_1(n)}(\emptyset,\ii)\to 
\HOM_{\Scat(n,n)}(1_n,{\mathcal E}_{-i_1}{\mathcal E}_{+i_1}\cdots 
{\mathcal E}_{-i_t}{\mathcal E}_{+i_t}1_n)$$
is surjective, where $\ii=(i_1,\ldots,i_t)$ is a sequence of $t$ points of 
strictly increasing color $1\leq i_1 < i_2 <\cdots<i_t\leq n-1$. If $t=0$, 
then this is true, because $\HOM_{\mathcal{SC}_1(n)}(\emptyset,\emptyset)\cong 
\Q[x_1-x_2,x_2-x_3,\ldots,x_{n-1}-x_n]$ by Elias and Khovanov's Theorem 1. 
Note that 
$$\Scat(n,d)^*((1^n),(1^n))\cong \Q[x_1-x_2,x_2-x_3,\ldots,x_{n-1}-x_n]$$ 
is exactly the ring generated by the colored bubbles, as we proved in 
Lemma~\ref{lem:superschur}. The functor $\Sigma_{n,n}$ sends double dots to 
colored bubbles. 

Note also that 
$S\Pi_{(1^n)}\cong \Q[x_1-x_2,x_2-x_3,\ldots,x_{n-1}-x_n]$ and 
the surjective map $S\Pi_{(1^n)}\to \END_{\Scat(n,n)}(1_n)$, which 
we explained in Section~\ref{sec:struct}, is 
equal to $\Sigma_{n,n}$. This actually shows that 
$\END_{\Scat(n,n)}(1_n)\cong S\Pi_{(1^n)}$, which is compatible with our 
Conjecture~\ref{conj:bubbles}. 

For $t>0$, note that by Corollary 
4.11 in \cite{E-Kh} $\HOM_{\mathcal{SC}_1(n)}(\emptyset,\ii)$ is a free 
left $\HOM_{\mathcal{SC}_1(n)}(\emptyset,\emptyset)$-module of rank one, generated by 
the diagram consisting of $t$ StartDots colored $i_1,\ldots,i_t$ respectively. 
Note also that, by the fullness of $\Psi_{n,n}$ and by 
Theorem 1.3, Proposition 1.4 and Theorem 2.7 in \cite{K-L3}, we know 
that 
$$\HOM_{\Scat(n,n)}(1_n,{\mathcal E}_{-i_1}{\mathcal E}_{+i_1}\cdots 
{\mathcal E}_{-i_t}{\mathcal E}_{+i_t}1_n)$$ 
is a free 
$\END_{\Scat(n,n)}(1_n)$-module 
of rank one generated by the diagram consisting of $t$ cups colored 
$i_1,\ldots,i_t$ respectively. 
Our functor $\Sigma_{n,n}$ maps the StartDots to the cups, so 
we get that 
$$\Sigma_{n,n}\colon \HOM_{\mathcal{SC}_1(n)}(\emptyset,\ii)\to 
\HOM_{\Scat(n,n)}(1_n,{\mathcal E}_{-i_1}{\mathcal E}_{+i_1}\cdots 
{\mathcal E}_{-i_t}{\mathcal E}_{+i_t}1_n)$$
is an isomorphism. 
\end{proof}

\subsection{A functor from $\mathcal{SC}_1'(d)$ to $\Scat(n,d)^*((1^d),(1^d))$ 
for $d<n$}
\label{ssec:uscqs}

Let $d<n$ be arbitrary but fixed. For $(1^d)\in\Lambda(n,d)$, we write 
$1_d=1_{(1^d)}$. We define a monoidal additive $\Q$-linear functor 
$$\Sigma_{n,d}\colon\mathcal{SC}_1'(d)\to \Scat(n,d)^*((1^d),(1^d)),$$
which is very similar to $\Sigma_{n,n}$ from the previous subsection and 
categorifies $\sigma_{n,d}$ of Section~\ref{sec:hecke-schur}. 
Recall that $\mathcal{SC}_1(d)\subseteq\mathcal{SC}_1'(d)$ is a faithful 
subcategory. So we define $\Sigma_{n,d}$ in exactly the same way as $\Sigma_{n,n}$, 
but restricting to the colors $1\leq i\leq d-1$ and sending $\emptyset$ to the 
empty diagram in the region labeled $(1^d)$ instead of $(1^n)$. The only 
new ingredient for the definition of $\Sigma_{n,d}$ is the image of the boxes, 
which we define by 
\begin{equation*}
\Sigma_{n,d}\bigl(\;\bbox{i}\;\bigr) = 
\sum\limits_{j=i}^{d-1}\
\xy
(0,0)*{\dblue\xybox{%
    (3,0);(-3,0) **\crv{(3,4.2) & (-3,4.2)};
    ?(.05)*\dir{>} ?(1)*\dir{>}; 
    (3,0);(-3,0) **\crv{(3,-4.2) & (-3,-4.2)} ?(.3)*\dir{}+(2,0)*{\bscs j};
}};
(6,3)*{\scs (1^d)},
\endxy-\xy 0;/r.18pc/:
 (0,-1)*{\dred\ccbub{\black -1}{\black d}};
  (8,4)*{\scs(1^d)};
 \endxy
\end{equation*}
for any $i=1,\ldots,d$.
Note that we have 
\begin{equation*}
\Sigma_{n,d}\bigl(\;\bbox{i}-\bbox{i+1}\;\bigr) = \ \
\xy
(0,0)*{\dblue\xybox{%
    (3,0);(-3,0) **\crv{(3,4.2) & (-3,4.2)};
    ?(.05)*\dir{>} ?(1)*\dir{>}; 
    (3,0);(-3,0) **\crv{(3,-4.2) & (-3,-4.2)} ?(.3)*\dir{}+(2,0)*{\bscs i};
}};
(6,3)*{\scs (1^d)},
\endxy
\end{equation*}
which agrees with the first box relation~\eqref{eq:box1}. 
One easily checks that $\Sigma_{n,d}$ preserves 
the other box relations as well. 
The rest of the proof that $\Sigma_{n,d}$ is 
well-defined uses the same arguments as in the previous subsection. 

As in Subsection~\ref{ssec:scqs} we have
\begin{lem}
\label{lem:ucomm}
There is a commutative diagram 
\begin{equation*}
\xymatrix{
\mathcal{SC}_1'(d)\ar[rr]^{\fek'}\ar[dr]_{\Sigma_{n,d}} && \bim(d)^*
\\
& \Scat(n,d)^*((1^d),(1^d))\ar[ur]_{\fbim} & 
}
\end{equation*}
\end{lem}
\begin{prop}
\label{prop:ufull}
The functor $\Sigma_{n,d}$ is an equivalence of categories.
\end{prop}
\begin{proof} Note that $\mathcal{E}_{+k}1_d=0$, for any $k\geq d$, 
so by the commutation 
isomorphisms in $\Scat(n,d)$ we see that any object $x$ in  
$\Scat(n,d)^*((1^d),(1^d))$ is isomorphic to a direct sum of objects whose 
factors are all of the form $\mathcal{E}_{-i}\mathcal{E}_{+i}1_d$ with 
$1\leq i\leq d-1$. This is a consequence of the commutation relations on
the decategorified level~\cite{D-G} which become commutation
isomorphisms on the category level.  
Therefore $\Sigma_{n,d}$ is essentially surjective. 
Faithfulness follows from Elias and Khovanov's results and the 
commuting triangle in Lemma~\ref{lem:ucomm}, just as in the previous 
subsection. 

The arguments which show that $\Sigma_{n,d}$ is full are almost identical 
to the ones in the previous subsection. The only difference is that 
we now have 
$$\HOM_{\mathcal{SC}_1'(d)}(\emptyset,\emptyset)\cong\Q[x_1,\ldots,x_d]\cong 
\ENDS(1_d).$$
The first isomorphism follows from Elias and Khovanov's results 
in~\cite{E-Kh}. The 
second isomorphism follows from the fact that the $i$-colored bubbles 
of positive degree are all zero for $i>d$, since their inner regions 
are labeled by elements that do not belong to $\Lambda(n,d)$, and the 
$d$-colored bubble with a dot is 
mapped to $x_d$. Therefore we have $$\ENDS(1_d)\cong 
\Q[x_1-x_2,\ldots,x_{d-1}-x_d,x_d]\cong \Q[x_1,\ldots,x_d].$$ 
\end{proof}


\section{Grothendieck algebras}    
\label{sec:grothendieck}           

\subsection{The Grothendieck algebra of $\Scat(n,d)$}
To begin with, let us introduce some notions and notations analogous to 
Khovanov and Lauda's in Section 3.5 in \cite{K-L3}. 
Let $\UcatD$ and $\ScatD(n,d)$ 
denote the Karoubi envelopes of $\Ucat$ and $\Scat(n,d)$ respectively.  
We define the objects of 
$\ScatD(n,d)$ to be the elements in $\Lambda(n,d)$ and 
we define the hom-category $\ScatD(n,d)(\lambda,\mu)$ to be the usual 
Karoubi envelope of $\Scat(n,d)(\lambda,\mu)$, for any 
$\lambda,\mu\in\Lambda(n,d)$. There exist 
idempotents 
$e\in\EndS({\mathcal E}_{\mathbf i}1_{\lambda})$, 
so that $({\mathcal E}_{\mathbf i},e)$ is a direct summand of 
${\mathcal E}_{\mathbf i}$ in $\ScatD(n,d)$. For example, we can define 
the idempotents  
\[
e_{+i,m,\lambda} = \xy 0;/r.15pc/:
 (0,0)*{\dblue\xybox{
(-12,-20)*{}; (12,20) **\crv{(-12,-8) & (12,8)}?(1)*\dir{>};
 (-4,-20)*{}; (4,20) **\crv{(-4,-13) & (12,2) & (12,8)&(4,13)}?(1)*\dir{>};?(.88)*\dir{}+(0.1,0)*{\bullet};
 (4,-20)*{}; (-4,20) **\crv{(4,-13) & (12,-8) & (12,-2)&(-4,13)}?(1)*\dir{>}?(.86)*\dir{}+(0.1,0)*{\bullet};
 ?(.92)*\dir{}+(0.1,0)*{\bullet};
 (12,-20)*{}; (-12,20) **\crv{(12,-8) & (-12,8)}?(1)*\dir{>}?(.70)*\dir{}+(0.1,0)*{\bullet};
 ?(.90)*\dir{}+(0.1,0)*{\bullet};?(.80)*\dir{}+(0.1,0)*{\bullet};
  (16,0)*{\lambda};
  }};
\endxy
,
\qquad  e_{-i,m,\lambda} = (-1)^{\frac{m(m-1)}{2}}\;\xy 0;/r.15pc/:
 (0,0)*{\dblue\xybox{
 (-12,-20)*{}; (12,20) **\crv{(-12,-8) & (12,8)}?(0)*\dir{<};
 (-4,-20)*{}; (4,20) **\crv{(-4,-13) & (12,2) & (12,8)&(4,13)}?(0)*\dir{<};?(.88)*\dir{}+(0.1,0)*{\bullet};
 (4,-20)*{}; (-4,20) **\crv{(4,-13) & (12,-8) & (12,-2)&(-4,13)}?(0)*\dir{<}?(.86)*\dir{}+(0.1,0)*{\bullet};
 ?(.92)*\dir{}+(0.1,0)*{\bullet};
 (12,-20)*{}; (-12,20) **\crv{(12,-8) & (-12,8)}?(0)*\dir{<}?(.70)*\dir{}+(0.1,0)*{\bullet};
 ?(.90)*\dir{}+(0.1,0)*{\bullet};?(.80)*\dir{}+(0.1,0)*{\bullet};
 (16,0)*{\lambda} }};
 \endxy
\]
\noindent in $\EndS(\mathcal{E}_{+i^m}1_{\lambda})$ and 
$\EndS(\mathcal{E}_{-i^m}1_{\lambda})$ respectively.
We can 
define the $1$-morphisms in $\ScatD(n,d)$ 
$$
{\mathcal E}_{\pm i^{(m)}}1_{\lambda}:=({\mathcal E}_{\pm i^m}1_{\lambda}, 
e_{\pm i,m,\lambda})\left\{ \dfrac{m(1-m)}{2}\right\} 
$$ 
and have 
$$
{\mathcal E}_{\pm i^m}1_{\lambda}\cong \left({\mathcal E}_{\pm i^{(m)}}1_{\lambda}
\right)^{\oplus [m]!}. 
$$
Recall that $[m]!\in\bN[q,q^{-1}]$ is the $q$-factorial $[m][m-1]\cdots 1$, 
with $[s]=(q^s-q^{-s})/(q-q^{-1})$. For any $q$-integer 
$\oplus_{n=-j}^{k} a_{n}q^{n}\in\bN[q,q^{-1}]$, we define    
$$A^{\oplus_{n=-j}^{k} a_{n}q^{n}}= \bigoplus_{n=-j}^{k}
\left(\oplus_{i=1}^{a_n}A\{n\}\right).$$ 
Note that $e_{+i,m,\lambda}=0$ for $m>\lambda_{i+1}$ and 
$e_{-i,m,\lambda}=0$ for $m>\lambda_i$, because for those values of $m$ 
the left-most region of their defining diagrams has a label with 
a negative entry. This shows that these idempotents depend on $\lambda$, 
which was not the case in~\cite{K-L3}. Note that these lower bounds for 
$m$ are sharp, i.e. 
\begin{align*}
{\mathcal E}_{+i^{(m)}}1_{\lambda}=0&\Leftrightarrow   
m>\lambda_{i+1}\\
{\mathcal E}_{-i^{(m)}}1_{\lambda}=0&\Leftrightarrow m>\lambda_i. 
\end{align*}
This follows from observing the image of 
${\mathcal E}_{\pm i^{(m)}}1_{\lambda}$ under the $2$-functor 
$\fbim\colon \Scat(n,d)^*\to \bim^*$. 

Before we go on, let us make the remark alluded to above 
Conjecture~\ref{conj:bubbles}, when we showed that
\begin{align}
\label{eq:non-inj_two} 
\xy 0;/r.18pc/:
 (0,0)*{\dred\cbub{\black 1}{2}};
  (4,9)*{\scs(0,1,0)};
 \endxy
- \ \
\xy 0;/r.18pc/:
 (0,-1)*{\dblue\cbub{\black -1}{1}};
  (4,8)*{\scs(0,1,0)};
 \endxy
\
=&\quad 0.
\end{align} 

\begin{rem}
\label{rem:zeros}
Suppose $\lambda=(\ldots,a,0,\ldots)\in\Lambda(n,d)$, with $a$ in the 
$i$th position. Let $\mu=(\ldots,0,a,\ldots)$ 
be obtained from $\lambda$ by switching $a$ and $0$. 
From Theorem 5.6 and Corollary 5.8 in~\cite{K-L-M-S} it follows that 
$$\mathcal{E}_{i^{(a)}}\mathcal{E}_{{-i}^{(a)}}1_{\lambda}\cong 
1_{\lambda}\quad\mbox{and}\quad 
\mathcal{E}_{{-i}^{(a)}}\mathcal{E}_{i^{(a)}}1_{\mu}\cong 1_{\mu},$$
because we have $\mathcal{E}_{i^{(j)}}1_{\lambda}=0$ and 
$\mathcal{E}_{{-i}^{(j)}}1_{\mu}=0$ in $\ScatD(n,d)$ 
for any $j>0$. Therefore $\lambda$ and $\mu$ are isomorphic objects in 
the 2-category $\ScatD(n,d)$. Our proof of~\eqref{eq:non-inj_two} 
used the $2$-isomorphism between $1_{(0,1,0)}$ and $\mathcal{E}_{-1}\mathcal{E}_11_{(0,1,0)}$ explicitly in the first step.
\end{rem}
 
Note that $\ScatD(n,d)$ is Krull-Schmidt, just as $\UcatD$. Therefore, we 
can take the split Grothendieck algebras/categories $K_0^{\Q(q)}(\UcatD)$ and 
$K_0^{\bQ(q)}(\ScatD(n,d))$. Considering the latter as a category, we follow Khovanov 
and Lauda~\cite{K-L3} and define  
$\Lambda(n,d)$ to be the set of objects. The hom-space 
$\hom(\lambda,\mu)$ we define to be the split Grothendieck algebra 
of the additive category $\ScatD(\lambda,\mu)$. 
Alternatively, we can see this as an (idempotented) algebra rather than a 
category. In the sequel we will use both points of view interchangeably. 
Note that the remark above shows that there are objects in $K_0^{\bQ(q)}(\ScatD(n,d))$ 
which are isomorphic, e.g. $(1,0,0), (0,1,0)$ and $(0,0,1)$ 
are all isomorphic in $K_0^{\bQ(q)}(\ScatD(3,1))$.  

Analogous 
to Khovanov and Lauda's homomorphism $\gamma=\gamma_U\colon 
\dot{\mathbf U}(\mathfrak{sl}_n)\to K_0^{\bQ(q)}(\UcatD)$, we define a homomorphism 
$\gamma_S\colon \SD(n,d)\to K_0^{\bQ(q)}(\ScatD(n,d))$
by 
$$E_{s_1}\cdots E_{s_m}1_{\lambda}\mapsto 
\left[\mathcal{E}_{s_1}\cdots\mathcal{E}_{s_m}
1_{\lambda}\right].$$
Our main goal in this section is to prove that $\gamma_S$ is an isomorphism. 
Recall that in order to show that $\gamma_U$ is an isomorphism, 
Khovanov and Lauda had to 
determine the indecomposable direct summands of certain $1$-morphisms 
$x$ in $\UcatD$. 
They did this 
by looking at $K_0^{\bQ(q)}(\ENDU(x))$, which is the Grothendieck group of the finitely 
generated graded projective $\ENDU(x)$-modules. This allowed them to 
use known results about the Grothendieck groups of graded algebras, which 
we recall below. The connection between the two sorts of Grothendieck groups 
relies on the 
fact that a finitely-generated graded projective $\ENDU(x)$-module 
is determined by an idempotent $e$ in $\ENDU(x)$ and $[(x,e)]$ is 
an element of $K_0^{\bQ(q)}(\UcatD)$. The isomorphism classes of 
indecomposable projective modules 
form a basis of $K_0^{\bQ(q)}(\ENDU(x))$ and correspond to the minimal idempotents 
in $\ENDU(x)$. We refer to \cite{K-L3} for more details. 
We will follow Khovanov and 
Lauda's approach closely to show that $\gamma_S$ is surjective, but will use a 
completely different method to show that $\gamma_S$ is injective. Although we 
have tried to explain our results clearly, we suspect that the part 
of this section which deals with the surjectivity of $\gamma_S$ will be quite 
hard to understand for someone unfamiliar with \cite{K-L1, K-L2, K-L3, L1}. 
The part on the injectivity of $\gamma_S$ can probably be read independently. 

Before we move on to our results in this section, we should recall 
the basic facts about Grothendieck groups of (graded) algebras which 
Khovanov and Lauda explained in Subsections 3.8.1 and 3.8.2 in~\cite{K-L3}.   
If $A$ is a finite-dimensional algebra over a field, let $K_0(A)$ be the 
Grothendieck group of the category of the finitely generated projective 
$A$-modules. 
\begin{prop}
\label{prop:surj}
Let $f\colon A\to B$ be a surjective homomorphism between two 
finite-dimensional algebras, then 
$K_0(f)\colon K_0(A)\to K_0(B)$ is surjective.
\end{prop}
\noindent Unfortunately in the applications in \cite{K-L3} and in our paper, 
the algebras involved are not finite-dimensional. But fortunately 
they are $\bZ$-graded and we can resort to finite-dimensional quotients 
which do not alter the Grothendieck groups. Let $A$ be a $\bZ$-graded algebra 
over a field, such that in each degree it has finite dimension and the 
grading is bounded from below.   
\begin{defn}
\label{defn:virtnilpot}
Let $I\subset A$ be a two-sided homogeneous 
ideal. We say that $I$ is {\em virtually nilpotent} if for each degree 
$a\in\bZ$ there exists an $N>0$ such that the degree $a$ summand of $I^N$ is 
equal to zero. 
\end{defn} 
\begin{lem}
\label{lem:virtnilpot} 
Let $I\subset A$ be a virtually nilpotent 
ideal. Then $K_0(A)\cong K_0(A/I)$.
\end{lem}
\begin{cor}
\label{cor:virtnilpot}
Let $f\colon A\to B$ be a degree preserving homomorphism of 
$\bZ$-graded algebras of the type described above, and 
$I\subset A$ a virtually nilpotent ideal of finite codimension. If $f$ 
is surjective, then $K_0(f)\colon K_0(A)\cong K_0(A/I)\to K_0(B/f(I))\cong 
K_0(B)$ is surjective. 
\end{cor}
We also need a fact about the split Grothendieck group of Krull-Schmidt 
categories. This result is not recalled in~\cite{K-L3}, but is well 
known in homological algebra. We thank Mikhail Khovanov for explaining it to 
us. To help the reader, we briefly sketch the proof below. 
\begin{prop}
\label{prop:inj}
Let ${\mathcal F}\colon C\to D$ be an additive $\Q$-linear degree preserving 
functor between two graded Krull-Schmidt categories, whose hom-spaces 
are finite-dimensional in each degree and whose gradings are bounded from 
below. If ${\mathcal F}$ is fully faithful, 
then $K_0({\mathcal F})\colon K_0(C)\to K_0(D)$ is injective.
\end{prop}
\noindent Since $C$ and $D$ are Krull-Schmidt, each object in 
$C$ or $D$ can 
be uniquely decomposed into indecomposables, which generate $K_0(C)$ and 
$K_0(D)$ respectively. Being fully faithful, $\mathcal F$ maps the set of 
indecomposables in $C$ injectively into the set of indecomposables in $D$.      

We now get to the main part of this section. By simply checking the 
definitions, we see that the following square commutes:
\begin{equation}
\begin{CD}
\label{cd:groth}
\U @>{\gamma_U}>>K_0^{\bQ(q)}(\UcatD)\\
@V{\phi_{n,d}}VV @VV{K_0^{\bQ(q)}(\Psi_{n,d})}V\\
\SD(n,d)@>{\gamma_S}>> K_0^{\bQ(q)}(\ScatD(n,d)).
\end{CD} 
\end{equation}
We know that $\phi_{n,d}$ is surjective and $\gamma_U$ is an isomorphism. 
We also know that $\Psi_{n,d}$ is full, but we cannot automatically conclude 
that $K_0^{\bQ(q)}(\Psi_{n,d})$ is surjective, because $\ENDS(x)$ is infinite-dimensional 
for any $1$-morphism $x$. We want to prove that 
$K_0^{\bQ(q)}(\Psi_{n,d})$ and $\gamma_S$ are surjective. Of course it suffices to 
prove that $\gamma_S$ is surjective. 

Let us first sketch the chain of arguments that leads to the proof of 
the surjectivity of $\gamma_U$ in Theorem 1.1 in~\cite{K-L3}. 
The proof is by induction with respect to the width of an 
indecomposable $1$-morphism $P$ in $\Ucat$, which by definition is the smallest 
non-negative integer $m$ such that $P$ is isomorphic 
to a direct summand of ${\mathcal E}_{\mathbf i}1_{\lambda}\{t\}$ 
with $||{\mathbf i}||=m$. In Lemma 3.38 Khovanov and Lauda prove 
that any indecomposable object of 
width $m$ is isomorphic to a direct summand of 
$\mathcal{E}_{\nu,-\nu'}1_{\lambda}\{t\}$, for certain $\lambda\in\bZ^{n-1}$, 
$t\in\bZ$ and $\nu,\nu'\in\bN[I]$, such that 
$\vert\vert\nu\vert\vert+\vert\vert\nu'\vert\vert=m$. This narrows down the 
number of cases that need to be considered in the proof of Theorem 1.1. 

Next, suppose $P$ has width zero, then 
$P\cong 1_{\lambda}$ up to a shift, and $K_0^{\bQ(q)}(\ENDU(1_{\lambda}))$ 
lies in the image of $\gamma_U$, because it is isomorphic to $\Q$ 
with generator $[1_{\lambda}]$. The induction step relies on 
the exact sequence of rings (3.38)
\begin{equation}
\label{eq:KLExSeq}
0\to I_{\nu,-\nu',\lambda}\to \ENDU({\mathcal E}_{\nu,-\nu'}
1_{\lambda})\to R_{\nu,-\nu',\lambda}\to 0.
\end{equation}
Recall that for $\mathfrak{g}=\mathfrak{sl}_n$, the ring 
$R_{\nu,-\nu',\lambda}$ is isomorphic to that of 2-morphisms whose 
diagrams are split into upward strands with source and target belonging to 
$\nu$, downward strands with source and target belonging to $-\nu'$, and 
bubbles on the right-hand side. The ideal 
$I_{\nu,-\nu',\lambda}$ is generated by diagrams which contain 
at least one cup or cap between $\nu$ and $-\nu'$. Note that the latter are 
precisely the 2-morphisms which factor through a direct sum of objects with 
width smaller than $\vert\vert\nu\vert\vert+\vert\vert\nu'\vert\vert$. 
As they remark in Remark 3.18, this exact sequence is 
split for $\mathfrak{g}=\mathfrak{sl}_n$. Therefore there is a direct sum 
decomposition  
\begin{equation}
\label{eq:KLDirSum}
K_0^{\bQ(q)}(\ENDU({\mathcal E}_{\nu,-\nu'}
1_{\lambda}))\cong K_0^{\bQ(q)}(I_{\nu,-\nu',\lambda})
\oplus K_0^{\bQ(q)}(R_{\nu,-\nu',\lambda}).
\end{equation}
The fact that $K_0^{\bQ(q)}(R_{\nu,-\nu',\lambda})$ lies in the image of 
$\gamma_U$ is essentially a consequence of the results in~\cite{K-L1, K-L2} 
and a technical result involving a virtually nilpotent ideal, the details 
of which we do not need here. On the other hand, The 2-morphisms in 
$I_{\nu,-\nu',\lambda}$ factor through direct sums of objects 
of smaller width, so any minimal idempotent in this ideal corresponds to an 
object of smaller width. Therefore 
$K_0^{\bQ(q)}(I_{\nu,-\nu',\lambda})$ 
lies in the image of $\gamma_U$ by induction. This shows that 
$K_0^{\bQ(q)}(\ENDU({\mathcal E}_{\nu,-\nu'}
1_{\lambda}))$ lies in the image of $\gamma_U$, as had to be proved. 
We should warn the reader that, contrary to what might seem at a first 
reading, the direct sum decomposition in~\eqref{eq:KLDirSum} does not 
preserve indecomposability. For example, consider the direct sum 
$EF1_1\cong FE1_1\oplus 1_1$ for $n=2$. This corresponds to the diagrammatic 
equation 
\begin{eqnarray}
\label{eq:EFagain}
 \vcenter{\xy 0;/r.18pc/:
  (0,0)*{\dblue\xybox{
  (-8,0)*{};
  (8,0)*{};
  (-4,10)*{}="t1";
  (4,10)*{}="t2";
  (-4,-10)*{}="b1";
  (4,-10)*{}="b2";
  "t1";"b1" **\dir{-} ?(.5)*\dir{<};
  "t2";"b2" **\dir{-} ?(.5)*\dir{>};}};
  (-6,-8)*{};
  (6,-8)*{};
  (10,2)*{(1)};
  (-10,2)*{(1)};
  \endxy}
&\quad = \quad&
 \vcenter{   \xy 0;/r.18pc/:
    (0,0)*{\dblue\xybox{
    (-4,-4)*{};(4,4)*{} **\crv{(-4,-1) & (4,1)}?(1)*\dir{>};
    (4,-4)*{};(-4,4)*{} **\crv{(4,-1) & (-4,1)}?(1)*\dir{<};?(0)*\dir{<};
    (-4,4)*{};(4,12)*{} **\crv{(-4,7) & (4,9)};
    (4,4)*{};(-4,12)*{} **\crv{(4,7) & (-4,9)}?(1)*\dir{>};}};
    (10,2)*{(1)};(-6,-7)*{};(6.8,-7)*{};
 \endxy}
  \quad - \quad
    \vcenter{\xy 0;/r.18pc/:
    (10,2)*{(1)};
  (0,0)*{\dblue\xybox{
  (-4,-8)*{}="b1";
  (4,-8)*{}="b2";
  "b2";"b1" **\crv{(5,-2) & (-5,-2)}; ?(.05)*\dir{<} ?(.93)*\dir{<}
  ?(.8)*\dir{}+(0,-.1)*{}+(-5,2)*{};
  (-4,8)*{}="t1";
  (4,8)*{}="t2";
  "t2";"t1" **\crv{(5,2) & (-5,2)}; ?(.15)*\dir{>} ?(.95)*\dir{>}
  ?(.4)*\dir{};}}
  \endxy}
\end{eqnarray}
The identity on $EF1_1$ is an indecomposable idempotent in $R_{+,-,(1)}$, 
but can be decomposed in $\ENDU(\mathcal{E}_{+,-}1_1)$ into the two 
indecomposable idempotents on the right-hand side of~\eqref{eq:EFagain}, 
which have width 2 and 0 respectively. Note that the 
second term on the right-hand side belongs to $I_{+,-,(1)}$. So 
Khovanov and Lauda's homomorphism
$$\beta\colon \ENDU(\mathcal{E}_{+,-}1_1)\to R_{+,-,(1)}$$ 
maps the first term on the right-hand side to the identity on $EF1_1$. The map 
backwards, which they call $\alpha$, is simply the inclusion, so it maps the 
identity to the identity. In the induction step above, one therefore 
writes 
\begin{eqnarray*}
\vcenter{\xy 0;/r.18pc/:
    (0,0)*{\dblue\xybox{
    (-4,-4)*{};(4,4)*{} **\crv{(-4,-1) & (4,1)}?(1)*\dir{>};
    (4,-4)*{};(-4,4)*{} **\crv{(4,-1) & (-4,1)}?(1)*\dir{<};?(0)*\dir{<};
    (-4,4)*{};(4,12)*{} **\crv{(-4,7) & (4,9)};
    (4,4)*{};(-4,12)*{} **\crv{(4,7) & (-4,9)}?(1)*\dir{>};}};
    (10,2)*{(1)};(-6,-7)*{};(6.8,-7)*{};
 \endxy}
&\quad = \quad
\vcenter{\xy 0;/r.18pc/:
  (0,0)*{\dblue\xybox{
  (-8,0)*{};
  (8,0)*{};
  (-4,10)*{}="t1";
  (4,10)*{}="t2";
  (-4,-10)*{}="b1";
  (4,-10)*{}="b2";
  "t1";"b1" **\dir{-} ?(.5)*\dir{<};
  "t2";"b2" **\dir{-} ?(.5)*\dir{>};}};
  (-6,-8)*{};
  (6,-8)*{};
  (10,2)*{(1)};
  (-10,2)*{(1)};
  \endxy}
\quad - \left( - \quad
    \vcenter{\xy 0;/r.18pc/:
    (10,2)*{(1)};
  (0,0)*{\dblue\xybox{
  (-4,-8)*{}="b1";
  (4,-8)*{}="b2";
  "b2";"b1" **\crv{(5,-2) & (-5,-2)}; ?(.05)*\dir{<} ?(.93)*\dir{<}
  ?(.8)*\dir{}+(0,-.1)*{}+(-5,2)*{};
  (-4,8)*{}="t1";
  (4,8)*{}="t2";
  "t2";"t1" **\crv{(5,2) & (-5,2)}; ?(.15)*\dir{>} ?(.95)*\dir{>}
  ?(.4)*\dir{};}}
  \endxy}\right)
\end{eqnarray*}
to prove that the class of the indecomposable summand of $EF1_1$ of width 2 
corresponding to the idempotent on the left-hand side, belongs to the image of 
$\gamma_U$. 
      
Next let us see how Khovanov and Lauda's proofs can be adapted to our setting. 
In the first place, note that all results in Section 3.5 of~\cite{K-L3} 
continue to be true. More precisely, the statements in their 
Propositions 3.24, 3.25 and 3.26 are still true, although some 
direct summands might now be zero depending on the labels of the 
regions in the diagrams. The crucial Lemma 3.38 in Section 3.8 in~\cite{K-L3} 
holds literally true in our case just as well. 

Let us now prove the analogue of their Theorem 1.1. Our proof is essentially 
the same, except that we use the fact that $\gamma_U$ is an isomorphism and 
$\Psi_{n,d}$ is full to avoid having to formulate and use analogues of 
the results in~\cite{K-L1} and~\cite{K-L2}, which might be hard. 
This is the reason why we did not go into the details of those results above.  

\begin{lem} 
\label{lem:surj}
The homomorphism 
$$\gamma_S\colon \SD(n,d)\to K_0^{\bQ(q)}(\ScatD(n,d))$$
is surjective. 
\end{lem}
\begin{proof}
For the basis of the induction, recall our surjection 
$S\Pi_{\lambda}\to \ENDS(1_{\lambda})$ explained in Section~\ref{sec:struct}. The 
ideal of elements of positive degree $S\Pi_{\lambda}^+$ is virtually nilpotent 
of codimension one, so by Corollary~\ref{cor:virtnilpot} it follows that 
$$\Q\cong K_0^{\bQ(q)}(S\Pi_{\lambda})\to K_0^{\bQ(q)}(\ENDS(1_{\lambda}))$$
is surjective. Therefore $K_0^{\bQ(q)}(\ENDS(1_{\lambda}))$ is generated by 
$[1_{\lambda}]$, i.e. $1_{\lambda}$ is also indecomposable in our case. 
Since $\gamma_S(1_{\lambda})=[1_{\lambda}]$, we see that 
$K_0^{\bQ(q)}(\ENDS(1_{\lambda}))$ lies in the image of $\gamma_S$. Note that 
we have not yet proved that $[1_{\lambda}]\ne 0$. After we have proved that 
$K_0^{\bQ(q)}(\ScatD(n,d))\cong \SD(n,d)$ in Theorem~\ref{thm:groth}, it 
follows that $K_0^{\bQ(q)}(\ENDS(1_{\lambda}))\cong \Q$ with $[1_{\lambda}]\ne 0$ 
being the generator.

For the induction step, note that $\Psi_{n,d}$ maps the exact 
sequence (\ref{eq:KLExSeq}) 
surjectively onto the exact sequence
\begin{equation}
\label{eq:MSVExSeq}
0\to \Psi_{n,d}(I_{\nu,-\nu',\overline{\lambda}})\to 
\ENDS({\mathcal E}_{\nu,-\nu'}
1_{\lambda})\to \ENDS({\mathcal E}_{\nu,-\nu'}1_{\lambda})/
\Psi_{n,d}(I_{\nu,-\nu',\overline{\lambda}})\to 0.
\end{equation}  
We do not know if this exact sequence is split, but fortunately it does not 
matter for our purpose. 

Note also that $\Psi_{n,d}$ induces a surjective map 
$$
R_{\nu,-\nu',\lambda}\to \ENDS({\mathcal E}_{\nu,-\nu'}1_{\lambda})/
\Psi_{n,d}(I_{\nu,-\nu',\overline{\lambda}}).
$$ 
Recall that Khovanov and Lauda defined a virtually nilpotent ideal $\beta\alpha(J)\subset R_{\nu,-\nu',\overline{\lambda}}$ of codimension one in 
Section 3.8.3 in~\cite{K-L3}, alluded to 
above. By Corollary~\ref{cor:virtnilpot} this 
implies that 
\begin{equation}
\label{eq:Rsurj}
K_0^{\bQ(q)}(R_{\nu,-\nu',\overline{\lambda}})\to 
K_0^{\bQ(q)}(\ENDS
({\mathcal E}_{\nu,-\nu'}1_{\lambda})/
\Psi_{n,d}(I_{\nu,-\nu',\overline{\lambda}}))
\end{equation}
is surjective. Now, just as in the proof of Theorem 1.1, let 
$e\in \ENDS({\mathcal E}_{\nu,-\nu'}1_{\lambda})$ be a minimal 
idempotent of width $m$, with $||\nu||+||\nu'||=m$. We have to show that 
$[({\mathcal E}_{\nu,-\nu'}1_{\lambda},e)]$ lies in the image of 
$\gamma_S$. Let $\overline{e}$ be the image of $e$ in $\ENDS
({\mathcal E}_{\nu,-\nu'}1_{\lambda})/
\Psi_{n,d}(I_{\nu,-\nu',\overline{\lambda}})$. Note that we do not know 
a priori that $\overline{e}$ is indecomposable, but that does not matter. 
By the surjectivity of~\eqref{eq:Rsurj}, we can lift $\overline{e}$ to an 
idempotent $e'\in 
R_{\nu,-\nu',\overline{\lambda}}$. By Khovanov and Lauda's results, we know that 
$$[({\mathcal E}_{\nu,-\nu'}1_{\overline{\lambda}},e')]\in 
K_0^{\bQ(q)}(\ENDU
({\mathcal E}_{\nu,-\nu'}1_{\lambda}))\subseteq 
K_0^{\bQ(q)}(\UcatD)$$ 
is in the 
image of $\gamma_U$. By the commutativity of the square in~\eqref{cd:groth}, 
this implies 
that 
$$[({\mathcal E}_{\nu,-\nu'}1_{\lambda},\Psi_{n,d}(e'))]\in 
K_0^{\bQ(q)}(\ENDS
({\mathcal E}_{\nu,-\nu'}1_{\lambda}))\subseteq 
K_0^{\bQ(q)}(\ScatD(n,d))$$ 
is in the image of $\gamma_S$. 
Note that $e-\Psi_{n,d}(e')$ maps to zero in $\ENDS
({\mathcal E}_{\nu,-\nu'}1_{\lambda})/
\Psi_{n,d}(I_{\nu,-\nu',\overline{\lambda}})$. By the minimality of $e$, 
we therefore have $\Psi_{n,d}(e')=e+e''$, with $e''$ an orthogonal idempotent 
in $\Psi_{n,d}(I_{\nu,-\nu',\overline{\lambda}})$ which can be decomposed 
into minimal idempotents of width $<m$. By induction 
$[({\mathcal E}_{\nu,-\nu'}1_{\lambda},e'')]$ is contained in the 
image of $\gamma_S$. This shows 
that $[({\mathcal E}_{\nu,-\nu'}1_{\lambda},e)]$ is contained in the image of 
$\gamma_S$ too, as we had to show. 
\end{proof}

The following two corollaries are immediate. 

\begin{cor} 
\label{cor:surj1}
The homomorphism 
$$K_0^{\bQ(q)}(\Psi_{n,d})\colon K_0^{\bQ(q)}(\dot{\mathcal U}(\mathfrak{sl}(n)))
\to K_0^{\bQ(q)}(\ScatD(n,d))$$ 
is surjective. 
\end{cor}

\begin{cor} 
\label{cor:surj2}
$K_0^{\bQ(q)}(\ScatD(n,d))$ is a quotient of $\SD(n,d)$. In particular 
$K_0^{\bQ(q)}(\ScatD(n,d))$ is finite-dimensional and semi-simple. 
\end{cor}

Before we prove the main result of this paper, we first categorify the 
homomorphism ${\iota}_{n,m}$ from Section~\ref{sec:hecke-schur}. 
Let $m\geq n$ and $d$ arbitrary. Let $\Xi_{n,m}=\oplus_{\lambda\in
\Lambda(n,d)}1_{\lambda}\in \Scat(m,d)$. Let $\Scat(n,m,d)$ 
be the full sub-2-category of $\Scat(m,d)$ whose objects belong to 
$\Lambda(n,d)\subseteq \Lambda(m,d)$.

\begin{defn}
\label{defn:INCL}
Let $m\geq n$ and $d$ arbitrary. We define a functor 
$${\mathcal I}_{n,m}\colon \Scat(n,d)\to \Scat(n,m,d)$$
by mapping any diagram in $\Scat(n,d)$ to itself, using the inclusion 
$\Lambda(n,d)\subseteq\Lambda(m,d)$ for the labels. 
\end{defn}
\noindent It is easy to see that ${\mathcal I}_{n,m}$ is well-defined and 
essentially surjective. We conjecture it to be faithful, but have no proof. 
It is certainly not full, because $\Scat(n,m,d)$ contains 
$n$-colored bubbles for example. Perhaps there is a virtually nilpotent ideal 
$I\subset \Scat(n,m,d)$ such that $\Scat(n,d)\cong \Scat(n,m,d)/I$, e.g. the 
ideal generated by all diagrams with $n$-colored bubbles of positive degree 
on the right-hand side.    

\begin{thm}
\label{thm:groth} The homomorphism 
$$\gamma_S\colon \SD(n,d)\to K_0^{\bQ(q)}(\ScatD(n,d))$$ is an isomorphism.
\end{thm}
\begin{proof}
After the result of Lemma~\ref{lem:surj} it only remains to show that 
$K_0^{\bQ(q)}(\Scat(n,d))$ and $\SD(n,d)$ have the same dimension. 

We first show the case $n=d$. Let $1_n=1_{(1^n)}$. In 
Proposition~\ref{prop:fulfaith} we proved that 
$\mathcal{SC}_1(n)\cong \Scat(n,n)^*((1^n),(1^n))$ is a full sub-2-category of 
$\Scat(n,n)^*$. By Proposition~\ref{prop:inj} this implies 
$$K_0^{\bQ(q)}(\mathcal{SC}(n))\cong K_0^{\bQ(q)}(\ScatD(n,n)((1^n),(1^n)))\subseteq 
K_0^{\bQ(q)}(\ScatD(n,n)).$$
By Theorem~\ref{thm:e-k-s} we know that $K_0^{\bQ(q)}(\mathcal{SC}(n))$ is isomorphic 
to $H_q(n)$. Thus Lemma~\ref{lem:emb} implies that 
$K_0^{\bQ(q)}(\ScatD(n,n))\cong \SD(n,n)$. 

Now let $d<n$. In Proposition~\ref{prop:ufull} we proved that 
$\mathcal{SC}'_1(d)\cong \Scat(n,d)^*((1^d),(1^d))$ is a full sub-2-category 
of $\Scat(n,d)^*$. By Proposition~\ref{prop:inj} this implies
$$K_0^{\bQ(q)}(\mathcal{SC}'(d))\cong K_0^{\bQ(q)}(\ScatD(n,d)((1^d),(1^d)))
\subseteq K_0^{\bQ(q)}(\ScatD(n,d)).$$
By Theorem~\ref{thm:e-k-s} we know that $K_0^{\bQ(q)}(\mathcal{SC}'(d))$ 
is isomorphic to $H_q(d)$. Thus Lemma~\ref{lem:emb} shows that 
$K_0^{\bQ(q)}(\ScatD(n,d))\cong \SD(n,d)$. 

Next, assume that $n<d$. 
Consider the functor 
$${\cal I}_{n,d}\colon \Scat(n,d)\to \Scat(n,d,d).$$ 
We have the following commuting square
$$
\begin{CD}
\SD(n,d)@>{\iota_{n,d}}>>\xi_{n,d}\SD(d,d)\xi_{n,d}\\
@V{\gamma_S(n)}VV @VV{\gamma_S(d)}V\\
K_0^{\bQ(q)}(\ScatD(n,d))@>{K_0^{\bQ(q)}({\mathcal I}_{n,d})}>>K_0^{\bQ(q)}(\ScatD(n,d,d)).
\end{CD}
$$
We already know that $\gamma_S(d)\colon \SD(d,d)\to K_0^{\bQ(q)}(\Scat)(d,d))$ 
is an isomorphism from the first case we proved. 
Therefore $\gamma_S(d)\colon\xi_{n,d}\SD(n,d)\xi_{n,d}\to 
[\Xi_{n,d}]K_0^{\bQ(q)}(\ScatD(d,d))[\Xi_{n,d}]\cong K_0^{\bQ(q)}(\ScatD(n,d,d))$ 
is an isomorphism as well. Recall that $\iota_{n,d}$ is an isomorphism. 
It follows that 
$\gamma_S(n)$ is injective. Recall that $\gamma_S(n)$ is surjective, 
by Lemma~\ref{lem:surj}. It follows that $K_0^{\bQ(q)}(\ScatD(n,d))\cong \SD(n,d)$.    
\end{proof}

Note that we did not follow Khovanov and Lauda's approach to prove injectivity 
of $\gamma_S$. Recall that they defined a non-degenerate $\Q$-semilinear 
form on 
$\dot{U}(\mathfrak{sl}_n)$, which is closely related to Lusztig's bilinear 
form in~\cite{Lu}, and defined an inner product on $K_0^{\bQ(q)}(\UcatD)$ 
by 
$$\langle[x],[y]\rangle=\dim_q(\HOMUD(x,y)).$$
They showed that $\gamma_U$ is injective by proving that it is an isometry. 
We could not prove that $\gamma_S$ is injective in this way, because 
we could not find such a $\Q$-semilinear form on $\SD(n,d)$ in the 
literature.\footnote{Williamson defines such a form in~\cite{Will} 
for $n=d$, but we do 
not know of any diagrammatic interpretation of his form even in that 
restricted case. We conjecture that his form is equivalent to ours for $n=d$. 
This is the only related form in the literature that we could find, even after 
asking numerous experts.} 
By our Theorem~\ref{thm:groth}, we can define one now. 
We first define a non-degenerate $\Q$-semilinear form on 
$K_0^{\bQ(q)}(\ScatD(n,d))$ as above
$$\langle [x],[y]\rangle=\dim_q(\HOMSD(x,y)).$$
\begin{defn} We define a non-degenerate $\Q$-semilinear form on $\SD(n,d)$ by 
$$\langle x,y\rangle=\langle\gamma_S(x),\gamma_S(y)\rangle.$$
\end{defn}
\noindent By definition $\gamma_S$ is an isometry. It is easy to see that the 
semilinear form on $\SD(n,d)$ has the following properties 
(compare to Proposition 2.4 in~\cite{K-L3}):
\begin{cor} We have  
\begin{enumerate} 
\item $\langle 1_{\lambda_1}x1_{\lambda_2},1_{\lambda'_1}x1_{\lambda'_2}\rangle=0$ 
for all $x,y$ unless $\lambda_1=\lambda'_1$ and $\lambda_2=\lambda'_2$.
\item $\langle ux,y\rangle=\langle x,\tau(u)y\rangle.$ 
\end{enumerate}
\end{cor}
\noindent However, Khovanov and Lauda's interpretation of the semilinear form 
on $\dot{U}(\mathfrak{sl}_n)$ in Theorem 2.7 in \cite{K-L3}, which shows 
that $\langle[x],[y]\rangle=\dim_q(\HOMUD(x,y))$ can be obtained by counting 
the number of minimal diagrams in each degree in $\HOMUD(x,y)$, does not hold 
in our case. This is because minimal diagrams in $\Scat(n,d)$ are not linearly 
independent in general. 
For example, consider relation~\eqref{eq:EF} for $n=2$ and $\lambda=(1,0)$. 
Note that the sum on the right-hand side only contains one term. 
The first term on the right-hand side, i.e. the one with the two crossings, has 
a middle region with label $(2,-1)\not\in\Lambda(2,1)$, so it is equal to zero. 
This shows that the minimal diagram on the left-hand side is equivalent to the 
minimal diagram on the right-hand side.     

\subsection{Categorical Weyl modules}
We conjecture that it is easy to categorify the irreducible 
representations $V_{\lambda}$, for $\lambda\in\Lambda^+(n,d)$, using the 
category $\Scat(n,d)$. Recall from Lemma~\ref{lem:weyl} that 
$$V_{\lambda}\cong \SD(n,d)1_{\lambda}/[\mu>\lambda].$$
\begin{defn} For any $\lambda\in\Lambda^+(n,d)$, let 
$\Scat(n,d)1_{\lambda}$ be the category whose objects are the $1$-morphisms 
in $\Scat(n,d)$ of the form $x1_{\lambda}$ and whose morphisms are the 
$2$-morphisms in $\Scat(n,d)$ between such 1-morphisms. Note that 
$\Scat(n,d)1_{\lambda}$ does not have a monoidal structure, because 
two $1$-morphisms $x1_{\lambda}$ and $y1_{\lambda}$ cannot be composed in general. 
Alternatively one can see $\Scat(n,d)1_{\lambda}$ as a graded ring, whose 
elements are the morphisms. 

Let ${\mathcal V}_{\lambda}$ be the quotient of $\Scat(n,d)1_{\lambda}$ by 
the ideal generated by all diagrams which contain a region labeled by 
$\mu>\lambda$. 
\end{defn}

Note that there is a natural categorical action of $\Scat(n,d)$, and therefore 
of $\Ucat$, on ${\mathcal V}_{\lambda}$, 
defined by putting a diagram in $\Scat(n,d)$ on the left-hand side of a 
diagram in ${\mathcal V}_{\lambda}$. This action descends to an action of 
$\SD(n,d)\cong K_0^{\bQ(q)}(\ScatD(n,d))$ on $K_0^{\bQ(q)}(\dot{{\mathcal V}}_{\lambda})$, 
where $\dot{{\mathcal V}}_{\lambda}$ is the Karoubi envelope of 
${\mathcal V}_{\lambda}$. Note that 
$\gamma_S$ induces a well-defined linear map 
$\gamma_{\lambda}\colon V_{\lambda}\to K_0^{\bQ(q)}(\dot{{\mathcal V}}_{\lambda})$, which 
intertwines the $\SD(n,d)$-actions.  

\begin{lem} 
\label{lem:repsurj}
The linear map $\gamma_{\lambda}$ is surjective. 
\end{lem}
\begin{proof} We first show that $K_0^{\bQ(q)}(\ScatD(n,d)1_{\lambda})\to 
K_0^{\bQ(q)}(\dot{{\mathcal V}}_{\lambda})$ is surjective. 
Again, we want to use Proposition~\ref{prop:surj}, but 
have to be careful because the graded rings involved are not 
finite-dimensional. Choose an object 
$x\in \Scat(n,d)1_{\lambda}$. Recall that $\ENDS(x)$ is finitely generated 
as a right module over $\ENDS(1_{\lambda})$. Let $\ENDS(1_{\lambda})^+\subseteq 
\ENDS(1_{\lambda})$ be the two-sided ideal of $2$-morphisms of strictly positive 
degree. Note that  
$\ENDS(1_{\lambda})^+$ is a codimension one virtually nilpotent ideal. 
Let $\mbox{END}^+_{\Scat(n,d)}(x)\subseteq \ENDS(x)$ be the image of 
$\ENDS(x)\otimes \ENDS(1_{\lambda})^+$ under the right action. Then 
$\mbox{END}^+_{\Scat(n,d)}(x)$ is a two-sided ideal of finite codimension and 
is virtually nilpotent, because the grading of $\ENDS(1_{\lambda})$ is bounded 
from below.   

Now let $\mbox{END}^{>\lambda}_{\Scat(n,d)}(x)
\subseteq \ENDS(x)$ be the two-sided ideal generated by all diagrams 
with at least one region labeled by a $\mu>\lambda$. By 
Corollary~\ref{cor:virtnilpot}, the projection 
$$\ENDS(x)\to \ENDS(x)/\mbox{END}^{>\lambda}_{\Scat(n,d)}(x)$$
induces a surjective homomorphism 
$$K_0^{\bQ(q)}(\ENDS(x))\to K_0^{\bQ(q)}(\ENDS(x)/\mbox{END}^{>\lambda}_{\Scat(n,d)}(x)).$$ 
Since $x$ was arbitrary, it follows that 
$$K_0^{\bQ(q)}(\ScatD(n,d)1_{\lambda})\to K_0^{\bQ(q)}(\dot{{\mathcal V}}_{\lambda})$$ 
is surjective. Thus, the composite linear map 
$$\SD(n,d)1_{\lambda}\cong K_0^{\bQ(q)}(\ScatD(n,d)1_{\lambda})\to 
K_0^{\bQ(q)}(\dot{{\mathcal V}}_{\lambda})$$ 
is surjective. Note that $[\mu>\lambda]$ is contained in the kernel of this 
map, which proves this lemma.  
\end{proof}

\begin{conj} For any $\lambda\in\Lambda^+(n,d)$, we have 
$$K_0^{\bQ(q)}(\dot{{\mathcal V}}_{\lambda})\cong V_{\lambda}.$$
\end{conj}

We do not know how to prove the conjecture in general. Note that by 
Lemma~\ref{lem:repsurj}, we have a surjective linear map 
$\gamma_{\lambda}\colon V_{\lambda}\to K_0^{\bQ(q)}(\dot{\mathcal V}_{\lambda})$, which 
intertwines the $\SD(n,d)$-actions. Since $V_{\lambda}$ 
is irreducible, we have $K_0^{\bQ(q)}(\dot{\mathcal V}_{\lambda})\cong V_{\lambda}$ or 
$K_0^{\bQ(q)}(\dot{\mathcal V}_{\lambda})=0$. So it suffices to show that 
$K_0^{\bQ(q)}(\dot{\mathcal V}_{\lambda})\ne 0$. Particular cases can be proved easily. 
For example, if $\lambda=(d)$, then ${\mathcal V}_{\lambda}=
\Scat(n,d)1_{\lambda}$, because there are no weights higher than $(d)$. 
By Theorem~\ref{thm:groth} we have $K_0^{\bQ(q)}(\ScatD(n,d)1_{\lambda})\cong 
\SD(n,d)1_{\lambda}$, which proves the conjecture in this case. 

We can also prove the case $n=2$. If $\lambda=(d,0)$, then the result 
follows from the previous case. Suppose $\lambda=(d-c,c)$, for $0<2c\leq d$.    
Note that $(d-2c,0)=(d-c,c)-(c,c)\in\Lambda^+(2,d-2c)$. Recall that we have a 
functor 
$$\Pi_{d,d-2c}\colon \Scat(2,d)\to\Scat(2,d-2c),$$ 
which induces a functor 
$$\Pi_{d,d-2c}\colon {\mathcal V}_{(d-c,c)}\to {\mathcal V}_{(d-2c,0)}.$$
Thus we have the following commuting square:
$$
\begin{CD}
V_{(d-c,c)}@>{\pi_{d,d-2c}}>> V_{(d-2c,0)}\\
@V{\gamma_{(d-c,c)}}VV @VV{\gamma_{(d-2c,0)}}V\\
K_0^{\bQ(q)}(\dot{\mathcal V}_{(d-c,c)})@>{K_0^{\bQ(q)}(\Pi_{d,d-2c})}>>K_0^{\bQ(q)}
(\dot{\mathcal V}_{(d-2c,0)})
\end{CD}
$$
We know that $\pi_{d,d-2c}$ and $\gamma_{(d-2c,0)}$ are isomorphisms and 
$\gamma_{(d-c,c)}$ is surjective. Therefore 
$\gamma_{(d-c,c)}$ is an isomorphism too, so 
$K_0^{\bQ(q)}(\dot{\mathcal V}_{(d-c,c)})\cong V_{(d-c,c)}$.

There is an obvious functor from the Khovanov-Lauda~\cite{K-L1} 
cyclotomic quotient 
category $R(*,\lambda)$ to a quotient of our ${\mathcal V}_{\lambda}$. The 
quotient is obtained by putting all bubbles of positive degree in the 
right-most region of the diagrams, labeled $\lambda$, equal to zero. By 
our observations above about $\END^+_{\Scat(n,d)}(x)$, this quotient has 
the same Grothendieck group as ${\mathcal V}_{\lambda}$. The 
functor is the ``identity'' on objects and morphisms. The reduction to 
bubbles argument before Conjecture~\ref{conj:struct} shows that our 
quotient satisfies the cyclotomic 
condition. The functor is clearly essentially surjective and full and we 
conjecture it to be faithful, so that it would be an equivalence of 
categories.     

\smallskip


\vspace*{1cm}

\noindent {\bf Acknowledgements} 

Foremost, we want to thank Mikhail Khovanov for his remark that we should be 
able to define the Chuang-Rouquier complex 
for a colored braid using his and Lauda's diagrammatic calculus~\cite{K-L3}. 
That observation formed the 
starting point of this paper. We also thank 
Mikhail Khovanov and Aaron Lauda for further inspiring conversations and for 
letting us copy their LaTeX 
version of the definition of $\Ucat$ almost literally. 

Secondly, we thank the anonymous referee for his or her detailed suggestions 
and comments which have helped us to improve the readability of the paper 
significantly (we hope).

Furthermore, the authors were supported by the 
Funda\c {c}\~{a}o para a Ci\^{e}ncia e Tecnologia (ISR/IST plurianual funding) through the
programme ``Programa Operacional Ci\^{e}ncia, Tecnologia, Inova\-\c{c}\~{a}o'' 
(POCTI) and the POS Conhecimento programme, cofinanced by the European Community 
fund FEDER.

PV was also financially supported by the Funda\c{c}\~ao para a Ci\^encia e Tecnologia 
through the post-doctoral fellowship SFRH/BPD/46299/ 2008.
 
MS was also partially supported by the Ministry of Science of Serbia, project 
174012.


\end{document}